\documentclass{amsart}
\usepackage{amsmath}
\usepackage{amssymb}
\usepackage[all]{xy}
\usepackage{graphicx}
\usepackage{mathrsfs}

\usepackage{hyperref}
\newcommand{\eqrefh}[1]{{\textup{(\ref*{#1})}}}

\let\oldeqn\equation
\let\endoldeqn\endequation
\renewenvironment{equation}{\oldeqn\setlength{\thickmuskip}{10mu plus 5mu}}{\endoldeqn\relax}
\catcode`\*11
\let\oldeqn*\equation*
\let\endoldeqn*\endequation*
\renewenvironment{equation*}{\oldeqn*\setlength{\thickmuskip}{10mu plus 5mu}}{\endoldeqn*\relax}
\catcode`\*12

\usepackage{tikz}
\numberwithin{figure}{section}
\usepackage{subfigure}
\usetikzlibrary{arrows}
\usetikzlibrary{shapes}

\newdimen{\compunum}
\def\getxscale#1#2#3#4#5#6{#1}
\def\getyscale#1#2#3#4#5#6{#4}
\newcommand{\compuinit}{%
  \pgfgettransform{\ta}%
  \edef\xscreg{\expandafter\getxscale\ta}%
  \edef\yscreg{\expandafter\getyscale\ta}%
  \pgfextractx{\compunum}{\pgfpointxyz{0}{0}{1}}%
  \pgfmathsetmacro{\pasx}{\compunum / 1cm}%
  \pgfextracty{\compunum}{\pgfpointxyz{0}{0}{1}}%
  \pgfmathsetmacro{\pasy}{\compunum / 1cm}%
  \pgfmathsetmacro{\pasw}{\yscreg / \xscreg}}
\newcommand{\al}{\scriptstyle}
\newcommand{\qc}[1]{$\vcenter{\hbox{#1}}$}
\newcommand{\vc}[1]{\vcenter{\hbox{#1}}}

\tikzstyle{rearcell}=[white,xscale=-1,fill=gray!20,fill opacity=0.9]
\tikzstyle{frontcell}=[white,fill=gray!75,fill opacity=0.7]

\tikzstyle{patch}=[minimum width=0.5cm*\xscreg,minimum height=0.5cm*\yscreg]
\tikzstyle{rpatch}=[patch,rearcell]
\tikzstyle{fpatch}=[patch,frontcell]

\tikzstyle{flippatch}=[minimum width=0.5cm*\yscreg,minimum height=0.5cm*\xscreg]
\tikzstyle{rflippatch}=[flippatch,rearcell]
\tikzstyle{fflippatch}=[flippatch,frontcell]

\tikzstyle{horizc}=[xslant=\pasx / (\pasy * \pasw),yscale=-\pasy]
\tikzstyle{vertc}=[yslant= \pasy  * \pasw / \pasx,xscale=\pasx]
\tikzstyle{cubef}=[fpatch]
\tikzstyle{cubel}=[vertc,fpatch]
\tikzstyle{cubed}=[vertc,rpatch]
\tikzstyle{cubet}=[horizc,fpatch]
\tikzstyle{cubeb}=[horizc,rpatch]
\tikzstyle{cuber}=[rpatch]

\tikzstyle{diagup}=[xslant=\pasx / ((\pasy-0.5) * \pasw),yscale=0.5-\pasy]
\tikzstyle{diagdp}=[xslant=\pasx / ((\pasy+0.5) * \pasw),yscale=0.5+\pasy]
\tikzstyle{vertp}=[isosceles triangle stretches=true,shape=isosceles triangle,yslant=\pasy*\pasw / \pasx,xscale=\pasx]
\tikzstyle{prismtf}=[diagup,fpatch]
\tikzstyle{prismbf}=[diagdp,fpatch]
\tikzstyle{prismlf}=[vertp,fflippatch]
\tikzstyle{prismdf}=[vertp,shape border rotate=180,rflippatch]

\tikzstyle{prismtr}=[diagdp,rpatch]
\tikzstyle{prismbr}=[diagup,rpatch]
\tikzstyle{prismlr}=[vertp,shape border rotate=180,fpatch]
\tikzstyle{prismdr}=[vertp,rpatch]

\tikzstyle{diagut}=[isosceles triangle stretches=true,shape=isosceles triangle,shape border rotate=90,xslant=\pasx / ((\pasy-0.5) * \pasw),yscale=0.5-\pasy]
\tikzstyle{diagdt}=[isosceles triangle stretches=true,shape=isosceles triangle,shape border rotate=-90,xslant= \pasx / ((\pasy+0.5) * \pasw),yscale=0.5+\pasy]
\tikzstyle{diaglt}=[isosceles triangle stretches=true,shape=isosceles triangle,yslant=(\pasy * \pasw) / (\pasx-0.5),xscale=0.5-\pasx]
\tikzstyle{diagrt}=[isosceles triangle stretches=true,shape=isosceles triangle,yslant=(\pasy * \pasw) / (\pasx+0.5),xscale=0.5+\pasx]

\tikzstyle{tetrtf}=[diagut,fpatch]
\tikzstyle{tetrbf}=[diagdt,fpatch]
\tikzstyle{tetrlr}=[diaglt,rflippatch]
\tikzstyle{tetrdr}=[diagrt,shape border rotate=180,rflippatch]

\tikzstyle{pyraf}=[isosceles triangle stretches=true,shape=isosceles triangle,shape border rotate=180,cm={(1+0.5*\pasx),0.5*\pasy*\pasw,0,1,(0,0)},fflippatch]
\tikzstyle{pyrar}=[isosceles triangle stretches=true,shape=isosceles triangle,cm={1-0.5*\pasx,-0.5*\pasy*\pasw,0,1,(0,0)},rflippatch]
\tikzstyle{pyrat}=[isosceles triangle stretches=true,shape=isosceles triangle,shape border rotate=180,cm={1,0.5*\pasw,-\pasx/\pasw,-\pasy,(0,0)},fflippatch]
\tikzstyle{pyrab}=[isosceles triangle stretches=true,shape=isosceles triangle,cm={1,-0.5*\pasw,-\pasx/\pasw,-\pasy,(0,0)},rflippatch]

\tikzstyle{tricell}=[isosceles triangle stretches=true,shape=isosceles triangle,fflippatch]
\tikzstyle{bigtricell}=[isosceles triangle stretches=true,shape=isosceles triangle,fflippatch,minimum width=1cm*\yscreg]
\tikzstyle{bigtricellr}=[isosceles triangle stretches=true,shape=isosceles triangle,shape border rotate=180,fflippatch,minimum width=1cm*\yscreg]
\tikzstyle{bigtricellu}=[isosceles triangle stretches=true,shape=isosceles triangle,shape border rotate=90,fpatch,minimum width=1cm*\xscreg]
\tikzstyle{bigtricelld}=[isosceles triangle stretches=true,shape=isosceles triangle,shape border rotate=-90,fpatch,minimum width=1cm*\xscreg]
\tikzstyle{rttricellu}=[isosceles triangle stretches=false,shape=isosceles triangle,shape border uses incircle=true,shape border rotate=225,frontcell,minimum width=0.3cm*sqrt(\yscreg*\yscreg + \xscreg*\xscreg),isosceles triangle apex angle=90]
\tikzstyle{rttricelld}=[isosceles triangle stretches=false,shape=isosceles triangle,shape border uses incircle=true,shape border rotate=135,frontcell,minimum width=0.3cm*sqrt(\yscreg*\yscreg + \xscreg*\xscreg),isosceles triangle apex angle=90]
\tikzstyle{lttricellu}=[isosceles triangle stretches=false,shape=isosceles triangle,shape border uses incircle=true,shape border rotate=-45,frontcell,minimum width=0.3cm*sqrt(\yscreg*\yscreg + \xscreg*\xscreg),isosceles triangle apex angle=90]
\tikzstyle{lttricelld}=[isosceles triangle stretches=false,shape=isosceles triangle,shape border uses incircle=true,shape border rotate=45,frontcell,minimum width=0.3cm*sqrt(\yscreg*\yscreg + \xscreg*\xscreg),isosceles triangle apex angle=90]

\pgfmathsetmacro{\tric}{1 / (1 + sqrt(5))}

\tikzstyle{liner}=[densely dotted]
\tikzstyle{linef}=[preaction={draw=white,-,line width=6pt}]
\tikzstyle{cart}=[minimum width=0.25cm*\xscreg,minimum height=0.25cm*\yscreg,draw=gray]
\tikzstyle{arl}=[auto=left,inner sep=2pt]
\tikzstyle{arr}=[auto=right,inner sep=2pt]

\tikzstyle{stdcube}=[xscale=3,yscale=2.5,z={(0.5,-0.4)}]
\tikzstyle{longcube}=[xscale=6,yscale=2.5,z={(0.35,-0.5)}]
\tikzstyle{longcube2}=[xscale=6.5,yscale=2,z={(0.35,-0.5)}]
\tikzstyle{smallcube}=[xscale=2.2,yscale=1.5,z={(0.5,-0.4)}]
\tikzstyle{smallcube2}=[xscale=2.8,yscale=1.2,z={(0.5,-0.4)}]
\tikzstyle{smallercube}=[xscale=1.8,yscale=1.5,z={(0.5,-0.4)}]
\tikzstyle{vsmallcube}=[xscale=1.5,yscale=1.25,z={(0.5,-0.4)}]
\tikzstyle{vsmallcube2}=[xscale=1.5,yscale=1,z={(0.5,-0.4)}]
\tikzstyle{stdprism}=[xscale=3,yscale=2.5,z={(0.5,-0.1)}]
\tikzstyle{longprism}=[xscale=6,yscale=2.5,z={(0.35,-0.1)}]
\tikzstyle{longprism2}=[xscale=5.5,yscale=2,z={(0.35,-0.1)}]
\tikzstyle{smallprism}=[xscale=2.4,yscale=2,z={(0.5,-0.1)}]
\tikzstyle{stdtetr}=[scale=3,z={(-0.2,0.01)}]
\tikzstyle{stdtriangle}=[scale=1.9]
\tikzstyle{stdtriangle2}=[xscale=2.5, yscale=1.5]
\tikzstyle{vsmalltriangle}=[scale=1.2]
\tikzstyle{vsmalltriangle2}=[scale=1.8]
\tikzstyle{stdpyra}=[xscale=4,yscale=2.5,z={(-0.4,-0.6)}]

\newdir{| |}{{}*!/:a(0)15pt/@{}}
\newdir{|o|}{{}*!/:a(0)30pt/@{}}
\newdir{|oo|}{{}*!/:a(0)45pt/@{}}
\newdir{>| |}{!/3pt/@{}*:(1,-.2)@^{>}*:(1,+.2)@_{>}*!/:a(180)15pt/@{}}
\newdir{>|o|}{!/3pt/@{}*:(1,-.2)@^{>}*:(1,+.2)@_{>}*!/:a(180)30pt/@{}}
\newdir{>|oo|}{!/3pt/@{}*:(1,-.2)@^{>}*:(1,+.2)@_{>}*!/:a(180)45pt/@{}}

\newcommand{\C}{\mathbb{C}}
\newcommand{\bk}{\Bbbk}
\newcommand{\Qlb}{\overline{\mathbb{Q}}_\ell}
\newcommand{\fK}{\mathfrak{K}}
\newcommand{\fO}{\mathfrak{O}}
\newcommand{\Z}{\mathbb{Z}}
\newcommand{\Gm}{\mathbb{G}_m}
\newcommand{\bA}{\mathbb{A}}
\newcommand{\pt}{\mathrm{pt}}

\newcommand{\Gv}{{\check G}}
\newcommand{\Tv}{{\check T}}
\newcommand{\Bv}{{\check B}}
\newcommand{\Lv}{{\check L}}
\newcommand{\GO}{G(\fO)}
\newcommand{\LO}{L(\fO)}
\newcommand{\TO}{T(\fO)}
\newcommand{\fg}{\mathfrak{g}}
\newcommand{\fn}{\mathfrak{n}}
\newcommand{\fb}{\mathfrak{b}}
\newcommand{\ft}{\mathfrak{t}}
\newcommand{\fl}{\mathfrak{l}}
\newcommand{\fp}{\mathfrak{p}}

\newcommand{\fu}{\mathfrak{u}}
\newcommand{\fh}{\mathfrak{h}}
\newcommand{\fc}{\mathfrak{c}}
\newcommand{\W}{{W}}

\newcommand{\dT}{\mathbf{T}}
\newcommand{\dTv}{\check{\dT}}
\newcommand{\bW}{\mathbf{W}}
\newcommand{\fH}{\mathfrak{H}}

\newcommand{\PGL}{\mathrm{PGL}}
\newcommand{\GL}{\mathrm{GL}}
\newcommand{\SL}{\mathrm{SL}}
\newcommand{\fgl}{\mathfrak{gl}}
\newcommand{\tglt}{\widetilde{\fgl}(2)}

\newcommand{\bX}{\mathbf{X}}

\newcommand{\Rv}{{\check R}}
\newcommand{\la}{\lambda}
\newcommand{\rhov}{{\check \rho}}
\newcommand{\alv}{{\check \alpha}}

\newcommand{\Gr}{\mathsf{Gr}}
\newcommand{\cL}{\mathcal{L}}
\newcommand{\tcN}{\widetilde{\mathcal{N}}}
\newcommand{\cN}{{\mathcal{N}}}
\newcommand{\cM}{\mathcal{M}}
\newcommand{\fT}{\mathfrak{T}}
\newcommand{\sm}{\mathsf{sm}}
\newcommand{\rs}{\mathrm{rs}}

\newcommand{\cD}{\mathcal{D}}
\newcommand{\cDb}{\cD^{\mathrm{b}}}
\newcommand{\sInd}{\mathsf{Ind}}
\newcommand{\hamma}{\gamma}
\newcommand{\Hamma}{\Gamma}
\newcommand{\Perv}{\mathsf{Perv}}
\newcommand{\PPerv}{\mathsf{P}}
\newcommand{\sP}{\mathsf{P}}
\newcommand{\IC}{\mathrm{IC}}
\newcommand{\p}{{}^p\!}

\newcommand{\Spr}{\underline{\mathsf{Spr}}}
\newcommand{\Groth}{\underline{\mathsf{Groth}}}
\newcommand{\ubk}{\underline{\Bbbk}}
\newcommand{\cT}{\mathcal{T}}

\newcommand{\Rep}{\mathsf{Rep}}
\newcommand{\Vect}{\mathsf{M}}
\newcommand{\sV}{\mathsf{M}}
\newcommand{\For}{\mathsf{For}}

\newcommand{\cA}{\mathcal{A}}
\newcommand{\sCat}{\mathsf{Cat}}
\newcommand{\Yon}{\mathbb{Y}}
\newcommand{\op}{\mathrm{op}}
\newcommand{\bb}{\mathbf{1}}
\newcommand{\sfF}{\mathsf{F}}
\newcommand{\sfG}{\mathsf{G}}
\newcommand{\sfH}{\mathsf{H}}

\newcommand{\BC}{\mathrm{(BC)}}
\newcommand{\Co}{\mathrm{(Co)}}
\newcommand{\Comp}{\hbox to 1pt{\hss$\mathrm{(Co)}$\hss}}
\newcommand{\Adj}{\mathrm{(Adj)}}
\newcommand{\IT}{\mathrm{(InTw)}}
\newcommand{\mFor}{\mathrm{(For)}}
\newcommand{\mInt}{\mathrm{(Int)}}
\newcommand{\Int}{\hbox to 1pt{\hss$\mathrm{(Int)}$\hss}}
\newcommand{\mTr}{\hbox to 1pt{\hss$\mathrm{(Tr)}$\hss}}

\newcommand{\ITr}{\hbox to 1pt{\hss$\mathrm{(ITr)}$\hss}}
\newcommand{\CnstRes}{\mathrm{(CII)}}
\newcommand{\Cnst}{\hbox to 1pt{\hss$\mathrm{(CII)}$\hss}}

\newcommand{\CInt}{\hbox to 1pt{\hss$\mathrm{(CI)}$\hss}}

\newcommand{\CFor}{\hbox to 1pt{\hss$\mathrm{(CF)}$\hss}}
\newcommand{\Poin}{\hbox to 1pt{\hss$\mathrm{(Poin)}$\hss}}
\newcommand{\mIE}{\mathrm{(IE)}}
\newcommand{\IE}{\hbox to 1pt{\hss$\mathrm{(IE)}$\hss}}
\newcommand{\mIEI}{\mathrm{(IEI)}}
\newcommand{\IEI}{\hbox to 1pt{\hss$\mathrm{(IEI)}$\hss}}
\newcommand{\mIBC}{\mathrm{(IBC)}}
\newcommand{\IBC}{\hbox to 1pt{\hss$\mathrm{(IBC)}$\hss}}
\newcommand{\CnstIE}{\mathrm{(CIE)}}
\newcommand{\Rel}{\hbox to 1pt{\hss$\mathrm{(CIE)}$\hss}}

\def\Satake#1{\mathscr{S}_{#1}}
\def\Satakesm#1{\mathscr{S}^{\sm}_{#1}}
\def\Springer#1{\mathbb{S}_{#1}}
\newcommand{\bT}{\mathbb{T}}
\def\ResW#1#2{\mathsf{R}^{#1}_{#2}}
\def\ResG#1#2{\mathrm{R}^{#1}_{#2}}
\def\ResGnonsm#1#2{\overline{\mathrm{R}}^{#1}_{#2}}
\def\ResGzero#1#2{\underline{\mathrm{R}}^{#1}_{#2}}
\def\ResN#1#2{\mathcal{R}^{#1}_{#2}}
\def\ResNderiv#1#2{\widetilde{\mathcal{R}}^{#1}_{#2}}
\def\IndNderiv#1#2{\widetilde{\mathcal{I}}^{#1}_{#2}}
\def\Indg#1#2{\overline{\mathcal{I}}^{#1}_{#2}}
\def\Indgtilde#1#2{\underline{\mathcal{I}}^{#1}_{#2}}

\def\ResGr#1#2{\mathfrak{R}^{#1}_{#2}}
\def\ResGrnonsm#1#2{\overline{\mathfrak{R}}^{#1}_{#2}}
\def\ResGreasy#1#2{\widetilde{\mathfrak{R}}^{#1}_{#2}}
\def\ResGrzero#1#2{\underline{\mathfrak{R}}^{#1}_{#2}}

\newcommand{\simto}{\mathrel{\overset{\sim}{\to}}}
\newcommand{\natisom}{\mathrel{\Longleftrightarrow}}
\newcommand{\directednatisom}{\mathrel{\overset{\sim}{\Longrightarrow}}}
\newcommand{\End}{\mathrm{End}}
\DeclareMathOperator{\Hom}{Hom}
\DeclareMathOperator{\Aut}{Aut}
\DeclareMathOperator{\rank}{rk}
\newcommand{\id}{\mathrm{id}}
\newcommand{\wtimes}{\mathbin{\widetilde{\times}}}
\newcommand{\tboxtimes}{\mathbin{\widetilde{\boxtimes}}}

\newtheorem{thm*}{Theorem}
\numberwithin{equation}{section}
\newtheorem{thm}{Theorem}[section]
\newtheorem{lem}[thm]{Lemma}
\newtheorem{prop}[thm]{Proposition}
\newtheorem{cor}[thm]{Corollary}
\theoremstyle{definition}

\theoremstyle{remark}
\newtheorem{rmk}[thm]{Remark}
\newtheorem{ex}[thm]{Example}

\title[Satake, Springer, Small II]{Geometric Satake, Springer correspondence,\\ and small representations II}

\author{Pramod N. Achar}
\address{Department of Mathematics\\
  Louisiana State University\\
  Baton Rouge, LA 70803\\
  U.S.A.}
\email{pramod@math.lsu.edu}

\author{Anthony Henderson}
\address{School of Mathematics and Statistics\\
  University of Sydney, NSW 2006\\
  Australia}
\email{anthony.henderson@sydney.edu.au}

\author{Simon Riche}
\address{Universit{\'e} Blaise Pascal et CNRS, Laboratoire de  
Math{\'e}ma\-tiques (UMR 6620), Campus universitaire des C{\'e}zeaux,
F-63177 Aubi{\`e}re Cedex, France
}
\email{simon.riche@math.univ-bpclermont.fr}


\dedicatory{In memoriam T.~A.~Springer (1926--2011)}

\thanks{P.A. was supported by NSF Grant No.~DMS-1001594.  A.H. was supported by ARC Future Fellowship Grant No.~FT110100504. S.R. was supported by ANR Grants No.~ANR-09-JCJC-0102-01 and No.~ANR-2010-BLAN-110-02.}

\begin{document}

\begin{abstract}
For a split reductive group scheme $\Gv$ over a commutative ring $\bk$ with Weyl group $\W$, there is an important functor $\Rep(\Gv,\bk)\to\Rep(\W,\bk)$ defined by taking the zero weight space. We prove that the restriction of this functor to the subcategory of small representations has an alternative geometric description, in terms of the affine Grassmannian and the nilpotent cone of the Langlands dual group $G$. The translation from representation theory to geometry is via the Satake equivalence and the Springer correspondence. This generalizes the result for the $\bk=\C$ case proved by the first two authors, and also provides a better explanation than in that earlier paper, since the current proof is uniform across all types.
\end{abstract}

\maketitle

%%%%%%%%%%%%%%%%%%%%%%%%%%%%%%%%%%%%%%%%%%%%%%%%%%%%%%%%%%%%%%%%%%%%%%%%%%%%%%%%%%%%%%%%%%

\section{Introduction}
\label{sect:intro}

\subsection{}

The close relationship between the geometry of (a portion of) the affine Grassmannian $\Gr$ of a reductive group $G$ and of its nilpotent cone $\cN$, and the implications of that relationship for the representation theory of the dual reductive group $\Gv$, have been much studied in type $A$: see \cite{lus:gp,mvy,mautner2}. Here, continuing the work \cite{ah} of the first two authors, we explore this phenomenon in arbitrary type. 

The `portion' of $\Gr$ that has to be considered in general is the closed subvariety $\Gr^\sm$ consisting of $\GO$-orbits corresponding to \emph{small} representations of $\Gv$ (i.e., those whose weights lie in the root lattice and are such that their convex hull does not contain twice a root). The bridge between this subvariety and the nilpotent cone is provided by a finite map $\pi : \cM \to \cN$, where $\cM$ is an open dense subvariety of $\Gr^\sm$. In type $A$, $\cM$ has two irreducible components, and $\pi$ restricts to an isomorphism on each of these components. In other types, the map $\pi$ is slightly more complicated, but its fibres and its image are described explicitly in \cite{ah}. 

The present paper focuses on the representation-theoretic implications of this map $\pi$. Consider these four functors (which will be defined fully in Section~\ref{sect:preliminaries}):
\begin{itemize}
\item The geometric Satake equivalence $\Satake{G}$ defined in~\cite{mv} restricts to an equivalence $\Satakesm{G}$ between the category $\Perv_{\GO}(\Gr^{\sm},\bk)$ of $\bk$-perverse shea\-ves on $\Gr^\sm$ and the category $\Rep(\Gv,\bk)_{\sm}$ of small $\bk$-representations of $\Gv$.
\item The map $\pi:\cM\to\cN$ gives rise to a functor $\Psi_G: \Perv_{\GO}(\Gr^{\sm},\bk) \to \Perv_G(\cN,\bk)$.
\item The Weyl group $\W$ acts on the zero weight space of any representation of $\Gv$. Tensoring this action with the sign character, we obtain a functor $\Phi_{\Gv}: \Rep(\Gv,\bk)_{\sm} \to \Rep(\W,\bk)$.
\item $\W$ also acts on the Springer sheaf $\Spr$ in $\Perv_G(\cN,\bk)$, giving rise to a functor $\Springer{G}=\Hom(\Spr,-):\Perv_G(\cN,\bk)\to\Rep(\W,\bk)$, which implements the Springer correspondence over $\bk$. 
\end{itemize}
These functors form the diagram:
\begin{equation}\label{eqn:main}
\vcenter{\xymatrix@C=2.5cm@R=0.7cm{
\Perv_{\GO}(\Gr^{\sm},\bk) \ar[d]_-{\Psi_G} \ar[r]^-{\Satakesm{G}}_-{\sim} & \Rep(\Gv,\bk)_{\sm} \ar[d]^-{\Phi_{\Gv}} \\
\Perv_G(\cN,\bk) \ar[r]^-{\Springer G} & \Rep(\W,\bk).
}}
\end{equation}
By~\cite[Theorem 1.3]{ah}, this diagram commutes when $\bk = \C$. But the proof in~\cite{ah} was not totally satisfactory: after reducing to the case of simple $G$ and irreducible small representations, it relied on case-by-case arguments, including Reeder's computations of zero weight spaces~\cite{reeder,reeder-additional,reeder2}. The main result of this paper is that \eqref{eqn:main} commutes for any ring $\bk$ for which the geometric Satake equivalence holds.

\begin{thm}  \label{thm:main-result}
Let $\bk$ be any Noetherian commutative ring of finite global dimension. Then there is a canonical isomorphism of functors:
\[
\Phi_{\Gv}\circ\Satakesm{G}\;\natisom\;\Springer{G}\circ\Psi_G.
\]
\end{thm}
\noindent
(The sense in which the isomorphism is canonical is explained in~\S\ref{subsect:canonical}.)
Theorem~\ref{thm:main-result} provides a geometric construction of the functor $\Phi_{\Gv}$, valid in much greater generality than in~\cite{ah}. Notably, our result applies in the setting of modular representation theory, when $\bk$ is a field of positive characteristic; see~\S\ref{subsect:positive-characteristic}. When $\bk=\C$, it provides a new proof of Reeder's results and Broer's covariant restriction theorem; see~\S\ref{subsect:characteristic-zero}.

Moreover, our proof of Theorem~\ref{thm:main-result} is uniform, and thus provides a better explanation of the commutativity of \eqref{eqn:main} than~\cite{ah} did. Indeed, for general $\bk$, a case-by-case argument does not seem feasible: the irreducibles in $\Rep(\Gv,\bk)_{\sm}$ and $\Rep(\W,\bk)$ are poorly understood, and in any case, calculations with irreducibles would be insufficient, since the categories in~\eqref{eqn:main} need not be semisimple.

Note that the geometric results of~\cite{ah} are not superseded by this paper, and that~\cite[Theorem 1.1]{ah} is required in order to define the functor $\Psi_G$.

\subsection{}

Our approach is based on the following elementary observation: \emph{Any representation of $\W$ is determined by the action of the simple reflections.}  The proof of Theorem~\ref{thm:main-result} can be thought of as having just two steps:
\begin{enumerate}
\item For $G$ of semisimple rank~$1$,~\eqref{eqn:main} commutes by direct computation.  
\item Every functor in~\eqref{eqn:main} commutes with `restriction to a Levi subgroup'. 
\end{enumerate}
Together, these two statements imply that~\eqref{eqn:main} becomes commutative after composition with any forgetful functor $\Rep(\W,\bk) \to \Rep(\W_L,\bk)$, where $\W_L$ is the Weyl group of a rank-$1$ Levi subgroup.  The elementary observation above says that an object of $\Rep(\W,\bk)$ can be recovered from its images in the various $\Rep(\W_L,\bk)$, so one might think that the commutativity of~\eqref{eqn:main} follows.

However, there is a subtlety here, which makes the proof far more difficult than this sketch suggests. Of course, a representation of $\W$ is \emph{not} determined by objects in the various $\Rep(\W_L,\bk)$; rather, we also need identifications of their underlying $\bk$-modules.  The two paths around~\eqref{eqn:main} each yield such an identification, but we need to know that the two identifications are the same.  Hence, when showing that a diagram of functors `commutes', as in Step~(2), it is not enough to show the existence of an isomorphism of functors: we must keep track of what the isomorphism is.

Our arguments are therefore forced to be $2$-categorical.
Most of the `commutative diagrams' in the paper are not the ordinary $1$-dimensional kind, but rather `labelled $2$-computads', which contain $0$-cells (categories), $1$-cells (functors), and $2$-cells (natural transformations).  For such a diagram, commutativity is an assertion about equality of compositions of $2$-cells, rather than isomorphism of compositions of $1$-cells. We explain the necessary $2$-categorical background in Appendix~\ref{sect:cubes}.

Note that the idea of reducing to the rank-$1$ case using geometric restriction functors is not new in the context of the geometric Satake equivalence, see e.g.~\cite{bfm, bf}. However, in these instances this idea is used to prove isomorphisms of objects rather than of functors, so the $2$-categorical subtleties do not arise.

We believe that our method will be useful in proving other isomorphisms of functors in geometric representation theory. With this in mind, we have collected in Appendix~\ref{sect:lemmas} the commutativity lemmas that we invoke throughout the paper, expressing the compatibilities of fundamental functors between derived categories.

The most difficult part of our proof is the result, proved in Section~\ref{sect:springer}, that the functor $\Springer{G}$ commutes with restriction to a Levi subgroup. We should emphasize that this is of independent interest in the theory of modular representations of $W$ (independent, that is, of any consideration of $\Gv$ or its small representations): it is a generalization of the restriction theorem in characteristic-$0$ Springer theory. The restriction-to-a-Levi functor for perverse sheaves on the nilpotent cone, which we prove to be exact in Proposition~\ref{prop:ResN-exact}, is studied further in~\cite{am,genspring1,genspring2}.

\subsection{}

Consider the case when $G = \GL(n,\C)$, so that $\W=\mathfrak{S}_n$ and $\Gv \cong \GL(n,\bk)$. In this case, $\Gr^{\sm}$ has two irreducible components (at least when $n\geq 3$ -- see~\cite[\S4.1]{ah} for details). For convenience, replace $\Gr^{\sm}$ with its irreducible component $\Gr^{\sm,+}$, which is essentially the compactification of $\cN$ introduced by Lusztig in~\cite{lus:gp}. The corresponding category $\Rep(\Gv,\bk)_{\sm,+}$ consists of representations of $\GL(n,\bk)$ whose dominant weights are of the form $(\lambda_1-1, \ldots, \lambda_n-1)$ where $\lambda=(\lambda_1 \ge \cdots \ge \lambda_n)$ is a partition of $n$. An important object of this category is $E=(\bk^n)^{\otimes n}\otimes\det^{-1}$.

What makes the $\GL(n)$ case special is that the functor $\Perv_{\GO}(\Gr^{\sm,+},\bk) \to \Perv_G(\cN,\bk)$ obtained by restricting $\Psi_G$ is an equivalence of categories, as shown by Mautner~\cite[Theorem 4.1]{mautner2}. Moreover, $\Psi_G(\Satake{G}^{-1}(E))\cong \Spr$, and the action of $\mathfrak{S}_n$ on $\Spr$ corresponds to the action of $\mathfrak{S}_n$ on $E$ defined by permutation of the tensor factors~\cite[(6.1)]{mautner2}. Given this, the commutativity of~\eqref{eqn:main} (or rather, its analogue for $\Gr^{\sm,+}$) is equivalent to a purely representation-theoretic statement:
\begin{equation} \label{eqn:gln-case}
\Phi_{\Gv}:\Rep(\Gv,\bk)_{\sm,+}\to\Rep(\W,\bk)\text{ is isomorphic to }\Hom(E,-).
\end{equation} 
This follows easily from a well-known analogous isomorphism between two definitions of the Schur functor; see~\cite[A.23(5)]{jantzen}. 

In a sense, then, Theorem~\ref{thm:main-result} can be regarded as a generalization to all $\Gv$ of the property \eqref{eqn:gln-case} of $\GL(n)$, with the Springer sheaf $\Spr$ playing the role of $E$. 

\subsection{}
\label{subsect:positive-characteristic}

Suppose that $\bk$ is a field of characteristic $\ell$. The irreducible representations of $\Gv$ are parametrized by their highest weights: let $L(\lambda)$ denote a small irreducible representation with highest weight $\lambda$. We have $L(\lambda)\cong\Satakesm{G}(\IC(\Gr^{\lambda},\bk))$ where $\IC(\Gr^{\lambda},\bk)$ is the simple perverse sheaf supported on the closure of the $\GO$-orbit $\Gr^{\lambda}$. Applying Theorem~\ref{thm:main-result}, we obtain an isomorphism of representations of $\W$:
\begin{equation} \label{eqn:field-case}
\Phi_{\Gv}(L(\lambda))\cong\Springer{G}(\Psi_G(\IC(\Gr^{\lambda},\bk))).
\end{equation}
The conceptual value of~\eqref{eqn:field-case} is best appreciated by considering the dependence of each side on the characteristic $\ell$. On the left-hand side, the dependence of the zero weight space of $L(\lambda)$ on $\ell$ is part of a famously hard problem of modular representation theory. On the right-hand side, the computation of $\Psi_G(\IC(\Gr^{\lambda},\bk))$ is essentially independent of $\ell$ as long as $\ell\neq 2$, as explained in~\cite[Corollary 5.7]{ahjr} (the reason is that the finite map $\pi:\cM\to\cN$, for simple $G$, is generically $1$-to-$1$ or $2$-to-$1$). Thus, setting aside the $\ell=2$ case,~\eqref{eqn:field-case} says that the dependence on $\ell$ of the left-hand side comes about purely through the dependence on $\ell$ of the Springer correspondence. In~\cite[Section 5.2]{ahjr}, explicit knowledge of the modular Springer correspondence is applied to~\eqref{eqn:field-case} in order to determine $\Phi_{\Gv}(L(\lambda))$ for small $L(\lambda)$ when $\ell\neq 2$.

\subsection{}
\label{subsect:characteristic-zero}

When $\bk = \C$, our main result, Theorem~\ref{thm:main-result}, is very similar to~\cite[Theorem~1.3]{ah}. The difference is that the horizontal arrows have been reversed: our current equivalence $\Satakesm{G}$ is inverse to the equivalence `$\mathrm{Satake}$' of \cite{ah}, and our current $\Springer{G}$ is left inverse to the functor `$\mathrm{Springer}$' of \cite{ah} (which in general has no right inverse). Hence the $\bk=\C$ case of Theorem~\ref{thm:main-result} is slightly weaker than~\cite[Theorem~1.3]{ah}. The additional content of the latter result may be restated as follows: when $\bk = \C$, $\Springer{G}$ is faithful on the image of $\Psi_G$, unless $G$ has factors of type $G_2$.

However, as mentioned above, our new proof of Theorem~\ref{thm:main-result} has an advantage even in the $\bk=\C$ case: it is independent of Reeder's calculation of the functor $\Phi_{\Gv}$ in~\cite{reeder,reeder-additional,reeder2}, and thus provides an alternative way to carry out that calculation. Namely, one can compute the right-hand side of \eqref{eqn:field-case} by combining the computation of $\Psi_G(\IC(\Gr^{\lambda},\bk))$ done in~\cite{ah} with the known values of $\Springer{G}$ on simple objects (dictated by the ordinary Springer correspondence). For the exceptional groups, this is not markedly more complex than Reeder's method. 

Finally, we remark that one of the motivations for~\cite{ah} was the search for a geometric proof of Broer's covariant theorem~\cite{broer}.  This theorem can be interpreted in terms of local equivariant cohomology on $\Gr$ and on $\cN$, and~\cite[\S6.4]{ah} explains how to deduce Broer's result from the commutativity of~\eqref{eqn:main} for $\bk = \C$.  In the context of~\cite{ah}, this argument was circular, because some of Reeder's calculations used Broer's result. With our independent proof of Theorem~\ref{thm:main-result}, the geometric proof of Broer's covariant theorem is now complete.

\subsection{}

The main arguments of this paper could also be carried out in the framework of $\infty$-categories developed by Boardman--Vogt~\cite{bv}, Joyal~\cite{joyal}, and Lurie~\cite{lurie}, among others.  Working with $\infty$-categories rather than $2$-categories would offer certain advantages: for instance, uniqueness questions such as those treated in~\cite{power, powern} are automatically subsumed by `higher homotopies'.  The formalism of Grothendieck's six operations has been developed in an $\infty$-categorical context by Liu--Zheng~\cite{lz} (but in the {\'e}tale setting rather than the classical setting used in this paper). The authors construct all the usual isomorphisms between sheaf functors in an $\infty$-categorical way, which means that their construction simultaneously encodes all higher relationships between those isomorphisms.  At least some of the results in Appendix~\ref{sect:lemmas} can be derived easily from~\cite{lz}, see Remark~\ref{rmk:lz}. However, we believe our $2$-categorical setting is more accessible to nonexperts than $\infty$-categories.

\subsection*{Outline of the paper}

In Section~\ref{sect:preliminaries} we set forth our notation and conventions, and define the categories and functors in diagram~\eqref{eqn:main}.  In Section~\ref{sect:plan} we explain the method of proof of Theorem~\ref{thm:main-result}, showing how to reduce to the case when $G$ has semisimple rank $1$, modulo a certain property of our functors: in essence, what we need is that each functor in~\eqref{eqn:main} commutes with restriction to a Levi subgroup, in a way that is compatible with transitivity of restriction. The remainder of the paper verifies the various ingredients of the proof. In Section~\ref{sect:restriction} we define restriction functors for each of the four categories in~\eqref{eqn:main}, and the transitivity isomorphisms that they satisfy. In Sections~\ref{sect:phi-psi}, \ref{sect:satake}, \ref{sect:springer} we prove the required commutativity statements for the functors in~\eqref{eqn:main}.  In Section~\ref{sect:rankone} we complete the proof by considering the rank-$1$ case. Finally, Appendix~\ref{sect:cubes} is a survey of the $2$-categorical formalism that is used in the paper, and Appendix~\ref{sect:lemmas} contains the basic commutativity lemmas for sheaf functors on which our arguments rely.

\subsection*{Acknowledgments}

This work was greatly assisted by discussions with D.~Juteau, whose modular Springer correspondence~\cite{j} was a key inspiration. The authors are also grateful to S.~Lack for helpful advice on $2$-categories, to C.~Mautner for explaining the results in~\cite{mautner} before their appearance in~\cite{mautner2}, and to D.~Ben-Zvi for drawing their attention to~\cite{lz}.

%%%%%%%%%%%%%%%%%%%%%%%%%%%%%%%%%%%%%%%%%%%%%%%%%%%%%%%%%%%%%%%%%%%%%%%%%%%%%%%%%%%%%%%%%%%%%%%%

\section{Preliminaries} \label{sect:preliminaries}

%In this section, we recall and define the principal notation and conventions of this paper, with the goal of explaining~\eqref{eqn:main}.

\subsection{Notation}
\label{ss:not}

Fix a Noetherian commutative ring $\bk$ of finite global dimension. All our sheaves will have coefficients in $\bk$. If $X$ is a complex algebraic variety (or ind-variety) and $H$ is a complex algebraic group (or pro-algebraic group) acting on $X$, we write $\cDb(X,\bk)$ for the bounded constructible derived category of $X$ with coefficients in $\bk$ (for the strong topology), and $\Perv_H(X,\bk)$ for its full abelian subcategory of $H$-equivariant perverse $\bk$-sheaves on $X$, as considered, for example, by Mirkovi\'c--Vilonen~\cite{mv}. We write $\cDb_H(X,\bk)$ for the constructible equivariant derived category, defined by Bernstein--Lunts~\cite{bl}. On occasion, it will be convenient to consider the perverse subcategory $\Perv_H'(X,\bk)$ of this equivariant derived category, which is equivalent to $\Perv_H(X,\bk)$ when $H$ is connected. For brevity, we sometimes omit $\bk$ from the notation for these categories.

Some of our results have known analogues in the context of $\Qlb$-sheaves for the \'etale topology, as we will mention in remarks. However, the proofs of those analogues often do not carry across: for instance, the Decomposition Theorem of \cite{bbd} does not hold in the setting of $\bk$-sheaves for general $\bk$.

Given a morphism $f:X\to Y$ of varieties, we have functors $f_*,f_!:\cDb(X,\bk)\to\cDb(Y,\bk)$ and $f^*,f^!:\cDb(Y,\bk)\to\cDb(X,\bk)$ as defined in \cite{kas}, and equivariant versions of these defined in \cite{bl}. (We omit the letter $R$ indicating derived functors; instead we use subscripts or exponents ``$0$'' when considering non-derived analogues of these functors.) The isomorphisms and adjunctions satisfied by these functors, and the compatibilities between these, will be our basic computational tools; Appendix~\ref{sect:lemmas} contains the precise statements that we need. 

We use double arrows 
%$\Longrightarrow$, $\directednatisom$, $\natisom$ 
for natural transformations and natural isomorphisms of functors, except in specific sorts of diagrams explained in Appendix~\ref{sect:cubes}. If $\alpha:\mathsf{G}\Longrightarrow\mathsf{H}$ is a natural transformation, and the domain of the functor $\mathsf{F}$ equals the codomain of $\mathsf{G}$ and $\mathsf{H}$, then the induced natural transformation $\mathsf{F}\circ\mathsf{G}\Longrightarrow \mathsf{F}\circ\mathsf{H}$ is written $\mathsf{F}\circ\alpha$ (following~\cite[\S XII.3]{maclane}); similarly for composition on the other side.

We write $\Vect(\bk)$ for the category of finitely-generated $\bk$-modules. If $\Gamma$ is a group scheme over $\bk$ (for instance, a finite group), we write $\Rep(\Gamma,\bk)$ for the category of representations of $\Gamma$ over $\bk$ that are finitely generated over $\bk$, and $\For^{\Gamma}$ for the forgetful functor $\Rep(\Gamma,\bk)\to\Vect(\bk)$. 

Throughout the paper, we let $G$ be a connected reductive algebraic group over $\C$. We choose a Borel subgroup $B$ of $G$ and a maximal torus $T$ of $B$. Let $\fg\supset\fb\supset\ft$ denote the Lie algebras of these groups. Let $U$ be the unipotent radical of $B$, and $\fn$ its Lie algebra. We write $\W_G$ for the Weyl group $N_G(T)/T$. 

We will often consider a parabolic setting, where we have chosen a parabolic subgroup $P$ of $G$ containing $B$, with Levi decomposition $P=LU_P$ where the Levi subgroup $L$ contains $T$. In this context, we let $C$ denote $B\cap L$, which is a Borel subgroup of $L$ containing $T$. 

Of course, $L$ and $T$ are also connected reductive groups, so any notation we define in terms of the triple $G\supset B\supset T$ applies also to $L\supset C\supset T$ and to $T\supset T\supset T$. We generally use subscripts to indicate which group is meant, as for example in the Weyl groups $\W_G$, $\W_L$ and $\W_T$. When only the one group $G$ is under consideration, the subscript $G$ may be omitted (as in Section~\ref{sect:intro}, where we wrote $\W$ for $\W_G$).

\subsection{The geometric Satake equivalence}
\label{ss:Satake}

Let $\fK=\C((t))$, $\fO=\C[[t]]$. The affine Grassmannian $\Gr_G$ is defined to be the ind-variety $G(\fK)/\GO$, on which $\GO$ acts by left translation. We define $\Gr_H$ for an arbitrary algebraic group $H$ in the same way; observe that any homomorphism $H\to H'$ of algebraic groups induces a morphism $\Gr_H\to\Gr_{H'}$, which is injective if $H\to H'$ is injective.

Recall that $\Perv_{\GO}(\Gr_{G},\bk)$ has the structure of a tensor category under the convolution product $\star$ (see \cite{mv}), and that the functor
\[
\sfF_G := H^{\bullet}(\Gr_{G},-) : \Perv_{\GO}(\Gr_{G},\bk) \to \Vect(\bk)
\]
is a tensor functor (see \cite[Proposition 6.3]{mv}). Consider the $\bk$-group scheme
\[
\Gv := \Aut^{\star}(\sfF_G)
\]
of automorphisms of the tensor functor $\sfF_G$. It follows from \cite{mv} and \cite[Proposition 2.8]{dm} that $\Gv$ is a split connected reductive group scheme over $\bk$, dual to $G$ in the sense of Langlands. Moreover, the action of $\Gv$ on $\sfF_G$ gives rise to an equivalence of tensor categories
\[
\Satake{G} : \Perv_{\GO}(\Gr_G,\bk) \xrightarrow{\sim} \Rep(\Gv,\bk),
\]
known as the geometric Satake equivalence. By definition, $\For^{\Gv}\circ\Satake{G}=\sfF_G$.

Let $\bX$ be the cocharacter lattice of $T$, which we identify with $\Gr_T$. We let $\mathbf{t}_{\lambda}$ be the image of $\lambda$ under the embedding $\bX=\Gr_T \hookrightarrow \Gr_G$. Recall that $\Gr_G$ is the union of the $\GO$-orbits 
\[ \Gr^\lambda := \GO\cdot\mathbf{t}_{\lambda}, \] 
and that $\Gr^\lambda=\Gr^\mu$ if and only if $\lambda$ and $\mu$ are in the same $\W_G$-orbit.
Furthermore, $\Gr_G$ is the disjoint union of the $U(\fK)$-orbits
\[
\fT_{\lambda} := U(\fK) \cdot \mathbf{t}_{\lambda},
\]
as $\lambda$ runs over $\bX$. Let $t_{\lambda} : \fT_{\lambda} \hookrightarrow \Gr_G$ be the inclusion.

Using the identification of $\Gr_T$ with $\bX$, the group $\Tv$ (of automorphisms of the tensor functor $\sfF_T$) is identified with the $\bk$-torus $\Hom_{\Z}(\bX,\bk^{\times})$. In particular, the character lattice $X^*(\Tv)$ is canonically identified with $\bX$. Define the functor
\[
\sfF^{\bX} := \bigoplus_{\lambda \in \bX} H^{\bullet} \bigl( \fT_{\lambda}, (t_{\lambda})^!(-) \bigr) : \Perv_{\GO}(\Gr_G,\bk) \to \Rep(\Tv,\bk),
\]
where we identify $\Rep(\Tv,\bk)$ with the category of $\bX$-graded finitely-generated $\bk$-modules. By~\cite[Theorems 3.5, 3.6]{mv}, we have a canonical isomorphism
\begin{equation}
\label{eqn:fibre-functor-Tv}
\For^{\Tv} \circ \sfF^{\bX} \natisom \sfF_G.
\end{equation}
Moreover, $\sfF^{\bX}$ is a tensor functor, and \eqref{eqn:fibre-functor-Tv} is an isomorphism of tensor functors. So $\sfF^{\bX}$ is the composition of $\Satake{G}$ with a tensor functor $\Rep(\Gv,\bk) \to \Rep(\Tv,\bk)$ compatible with forgetful functors. By \cite[Corollary 2.9]{dm}, the latter functor is induced by a group morphism $\iota_{\Tv}^{\Gv}:\Tv \to \Gv$. It is proved in \cite{mv} that $\iota_{\Tv}^{\Gv}$ is injective, and identifies $\Tv$ with a maximal torus of $\Gv$. 

Let $\Rv\subset\bX$ denote the set of roots of $(\Gv,\Tv)$, or in other words coroots of $(G,T)$. The Weyl group $\W_{\Gv}=N_{\Gv}(\Tv)/\Tv$ is identified, as a subgroup of the group of automorphisms of $\bX$, with $\W_G$. We will call it $\W_G$ (or $\W$) rather than $\W_{\Gv}$.

\subsection{The base connected component of $\Gr$}
\label{ss:conn-compt-Gr}

Let $\Gr^\circ$ be the connected component of $\Gr$ containing $\mathbf{t}_0$. This is the union of the $\GO$-orbits $\Gr^\lambda$ where $\lambda$ runs over $\Z\Rv$. Let $z_G : \Gr^{\circ} \hookrightarrow \Gr$ denote the inclusion. We have a fully faithful functor
\[
(z_G)_* : \Perv_{\GO}(\Gr^{\circ},\bk) \to \Perv_{\GO}(\Gr,\bk).
\]
The essential image of $\Satake{G}\circ(z_G)_*$ is the subcategory $\Rep(\Gv,\bk)^{Z(\Gv)}$ of $\Rep(\Gv,\bk)$ consisting of representations whose $\Tv$-weights belong to $\Z\Rv$, or in other words representations on which the centre $Z(\Gv)$ acts trivially. Let $\mathbf{I}_{\Gv} :  \Rep(\Gv,\bk)^{Z(\Gv)} \hookrightarrow \Rep(\Gv,\bk)$
denote the inclusion; then by definition there is a unique equivalence
\[
\Satake{G}^\circ : \Perv_{\GO}(\Gr^\circ,\bk) \xrightarrow{\sim} \Rep(\Gv,\bk)^{Z(\Gv)}
\]
such that
\begin{equation}
\label{eqn:Satake-z-1}
\mathbf{I}_{\Gv} \circ \Satake{G}^\circ = \Satake{G} \circ (z_G)_*.
\end{equation}
Now $(z_G)_*$ is left adjoint to $(z_G)^!$, and $\mathbf{I}_{\Gv}$ is left adjoint to
\[
(-)^{Z(\Gv)} : \Rep(\Gv,\bk) \to \Rep(\Gv,\bk)^{Z(\Gv)},
\]
the functor of taking $Z(\Gv)$-invariants. We therefore obtain a canonical isomorphism
\begin{equation}
\label{eqn:Satake-z-2}
(-)^{Z(\Gv)} \circ \Satake{G} \natisom \Satake{G}^\circ \circ (z_G)^!.
\end{equation}

\subsection{The functor $\Satakesm{G}$}
\label{ss:small-part-Gr}

Recall that $\lambda\in\bX$ is said to be \emph{small} for $\Gv$ if it belongs to the root lattice $\Z \Rv$ and if the convex hull of $\W \cdot \lambda$ does not contain any element of the form $2\alv$ for $\alv \in \Rv$. We denote by $\Gr^{\sm}$ the closed subvariety of $\Gr$ which is the union of the $\GO$-orbits $\Gr^\lambda$ for small $\lambda\in\bX$. Let $f_G : \Gr^{\sm} \hookrightarrow \Gr$ be the inclusion. We have a fully faithful functor
\[
(f_G)_* : \Perv_{\GO}(\Gr^{\sm},\bk) \to \Perv_{\GO}(\Gr,\bk).
\]
The essential image of $\Satake{G}\circ(f_G)_*$ is the subcategory $\Rep(\Gv,\bk)_{\sm}$ of $\Rep(\Gv,\bk)$ consisting of small representations, i.e.~representations whose $\Tv$-weights are small. Let $
\mathbb{I}_{\Gv} : \Rep(\Gv,\bk)_{\sm} \hookrightarrow \Rep(\Gv,\bk)$
denote the inclusion; then by definition there is a unique equivalence of categories
\[
\Satakesm{G} : \Perv_{\GO}(\Gr^{\sm},\bk) \xrightarrow{\sim} \Rep(\Gv,\bk)_{\sm}
\]
such that
\begin{equation}
\label{eqn:Satake-sm-non-sm}
\mathbb{I}_{\Gv} \circ \Satakesm{G} = \Satake{G} \circ (f_G)_*.
\end{equation}

We denote by $f_G^\circ : \Gr^{\sm} \hookrightarrow \Gr^\circ$ and $\mathbb{I}_{\Gv}^\circ : \Rep(\Gv,\bk)_{\sm} \hookrightarrow \Rep(\Gv,\bk)^{Z(\Gv)}$ the inclusions, so that $f_G=z_G f_G^\circ$ and $\mathbb{I}_{\Gv}=\mathbf{I}_{\Gv}\circ\mathbb{I}_{\Gv}^\circ$. Then there is a unique isomorphism
\begin{equation}
\label{eqn:Satake-sm-non-sm-0}
\mathbb{I}_{\Gv}^\circ \circ \Satakesm{G} \natisom \Satake{G}^\circ \circ (f_G^\circ)_*
\end{equation}
that makes the following diagram of isomorphisms commutative, where $\Co$ denotes the composition isomorphism defined in \S\ref{subsect:composition-definitions}:
\begin{equation}
\label{eqn:diagram-f_G-f_G^0}
\vcenter{
\xymatrix@R=15pt{
\mathbf{I}_{\Gv}\circ \Satake{G}^\circ\circ (f_G^\circ)_* \ar@{<=>}[rr]^-{\eqref{eqn:Satake-z-1}} \ar@{<=>}[d]_-{\eqref{eqn:Satake-sm-non-sm-0}} & & \Satake{G} \circ (z_G)_* \circ (f_G^\circ)_* \ar@{<=>}[d]^-{\Co} \\
\mathbf{I}_{\Gv} \circ \mathbb{I}_{\Gv}^\circ \circ \Satakesm{G} \ar@{=}[r] & \mathbb{I}_{\Gv} \circ \Satakesm{G} \ar@{<=>}[r]^-{\eqref{eqn:Satake-sm-non-sm}} & \Satake{G}\circ (f_G)_*.
}
}
\end{equation} 

\subsection{The functor $\Phi_{\Gv}$}

For any $V\in\Rep(\Gv,\bk)$, we have a natural action of $\W$ on the zero weight space $V^{\Tv}$. It is convenient to tensor this action by the sign character $\epsilon:\W\to\bk^\times$;
the resulting map from representations of $\Gv$ to representations of $\W$, together with the obvious map on morphisms, constitutes an exact functor
\[
\Phi^0_{\Gv} : \Rep(\Gv,\bk) \to \Rep(\W,\bk).
\]
The composition $\For^{\W}\circ\Phi^0_{\Gv}:\Rep(\Gv,\bk)\to\Vect(\bk)$ is the functor of $\Tv$-invariants. 
Let
\[
\Phi_{\Gv} := \Phi^0_{\Gv}\circ\mathbb{I}_{\Gv} : \Rep(\Gv,\bk)_{\sm} \to \Rep(\W,\bk)
\]
be the restriction of $\Phi^0_{\Gv}$ to the subcategory of small representations.

\subsection{The functor $\Psi_G$}

Following~\cite{ah}, we set
$\Gr_0^-:= G(\fO^-) \cdot \mathbf{t}_0$, where $\fO^-:=\C[t^{-1}] \subset \fK$, and let $\mathfrak{G}$ be the kernel of the evaluation map $G(\fO^-) \to G$ at $t=\infty$. There is a natural map from $\mathfrak{G}$ to the kernel of the evaluation map $G(\C[t^{-1}]/t^{-2}) \to G$, which we identify with the Lie algebra $\fg$ of $G$. Moreover, we have an isomorphism $\mathfrak{G} \xrightarrow{\sim} \Gr_0^-:g\mapsto g\cdot\mathbf{t}_0$. Hence we obtain a $G$-equivariant morphism
$\pi^{\dag}_G : \Gr_0^- \to \fg$.

We define the open subvariety
\[
\cM := \Gr^{\sm} \cap \Gr_0^-
\]
of $\Gr^{\sm}$, and let $j_G : \cM \hookrightarrow \Gr^{\sm}$ be the inclusion. As $\cM$ is $G$-stable, we have an exact functor
\[
(j_G)^!:\Perv_{\GO}(\Gr^{\sm},\bk)\to\Perv_G(\cM,\bk).
\]
Let $\cN\subset\fg$ be the nilpotent cone. By \cite[Theorem 1.1]{ah}, we have $\pi^{\dag}_G(\cM)\subseteq\cN$, and the restriction $\pi_G : \cM \to \cN$
is a finite morphism. (The assumption in \cite{ah} is that $G$ is simply connected and simple, but the result for general $G$ follows.) Hence $\pi_G$ induces an exact functor $(\pi_G)_* : \Perv_{G}(\cM,\bk) \to \Perv_G(\cN,\bk)$ (see \cite[Corollaire 4.1.3]{bbd}). We then obtain an exact functor
\[
\Psi_G := (\pi_G)_* \circ (j_G)^! : \Perv_{\GO}(\Gr^{\sm},\bk) \to \Perv_G(\cN,\bk).
\]

\subsection{The functor $\Springer{G}$}
\label{ss:functor-Springer}

Recall the Grothendieck--Springer simultaneous resolution 
\[
\mu_\fg:G\times^B\fb\to\fg:(g,x)\mapsto g \cdot x.
\]
It is well known that $\mu_\fg$ is proper and small, so
\[
\Groth := (\mu_\fg)_! \ubk_{G\times^B\fb}[\dim\fg]
\]
is an object of $\Perv_G(\fg,\bk)$. More explicitly, we have a canonical isomorphism
\[
\Groth \cong (j_\fg)_{!*} \bigl( (\mu_\fg^\rs)_! \ubk_{\mu_\fg^{-1}(\fg^\rs)}[\dim\fg] \bigr)
\]
where $j_\fg:\fg^\rs\hookrightarrow\fg$ is the inclusion of the open set of regular semisimple elements, and $\mu_\fg^\rs$ is the restriction of $\mu_\fg$ to $\mu_\fg^{-1}(\fg^\rs)$. Since $\mu_\fg^\rs$ is a Galois covering with group $\W$, we obtain an action of $\W$ on $\Groth$ by automorphisms in $\Perv_G(\fg,\bk)$.

Let $i_\fg:\cN\hookrightarrow\fg$ be the inclusion of the nilpotent cone, and let $r=\dim\fg-\dim\cN$ be the rank of $G$. 
Let $\mu_{\cN}:G\times^B\fn\to\cN$ be the Springer resolution, i.e.~the restriction of $\mu_\fg$ to $G\times^B\fn$. Since $\mu_{\cN}$ is proper and semi-small, the \emph{Springer sheaf}
\[
\Spr := (\mu_\cN)_! \ubk_{G\times^B\fn}[\dim\cN]
\]
is an object of $\Perv_G(\cN,\bk)$. By base change applied to the cartesian square
\[
\vc{\begin{tikzpicture}[vsmallcube]
\compuinit
% 2-cells
\node[cart] at (0.75,0.5) {};
% 0-cells
\node (lu) at (0,1) {$G \times^B \fn$};
\node (ru) at (1.5,1) {$G \times^B \fb$};
\node (ld) at (0,0) {$\cN_G$};
\node (rd) at (1.5,0) {$\fg$};
% 1-cells
\draw[right hook->] (lu) --  (ru);
\draw[->] (lu) -- node[arr] {$\al \mu_{\cN}$} (ld);
\draw[->] (ru) -- node[arl] {$\al \mu_{\fg}$} (rd);
\draw[right hook->] (ld) -- node[arr] {$\al i_\fg$} (rd);
\end{tikzpicture}}
\]
we obtain a canonical isomorphism
\begin{equation}
\Spr \cong (i_\fg)^*\Groth[-r]. 
\end{equation}
We use this isomorphism to define an action of $\W$ on $\Spr$, which
induces a functor
\[
\Springer{G} : \Perv_G(\cN,\bk) \to \Rep(\W,\bk),
\]
defined on objects by $M\mapsto\Hom_{\Perv_G(\cN)}(\Spr,M)$. We will show in Proposition~\ref{prop:springer-exactness} that $\Springer{G}$ is exact, or in other words that $\Spr$ is a projective object in $\Perv_G(\cN,\bk)$.

\begin{rmk} \label{rmk:springer-sign}
The $\W$-action on $\Groth$ was defined by Lusztig~\cite{lus:gp}. From it one may obtain a $\W$-action on $\Spr$ either via the restriction functor $(i_\fg)^*$ as above (following~\cite{bm}), or via a Fourier transform (following~\cite{brylinski}). It is well known in the $\bk=\Qlb$ case that these two actions coincide up to tensoring with the sign character; in~\cite[Theorem 1.1]{ahjr} this result is generalized to $\bk$ as in the current paper. Moreover,~\cite[Corollary 5.3]{ahjr} shows that the functor $\Springer{G}$ implements a Springer correspondence over $\bk$: for any simple object $M$ of $\Perv_G(\cN,\bk)$, $\Springer{G}(M)$ is either an irreducible representation of $\W$ or zero, with each irreducible representation of $\W$ over $\bk$ occurring for a unique $M$. These results are not required for the proof of Theorem~\ref{thm:main-result}.
\end{rmk}

%%%%%%%%%%%%%%%%%%%%%%%%%%%%%%%%%%%%%%%%%%%%%%%%%%%%%%%%%%%%%%%%%%%%%%%%%%%%%%%%%%%%%%%%%

\section{Plan of the proof of Theorem \ref{thm:main-result}}
\label{sect:plan}

%%%%%%%%%%%%%%%%%%%%%%%%%%%%%%%%%%%%%%%%%%%%%%%%%%%%%%%%%%%%%%%%%%%%%%%%%%%%%%%%%%%%%%%%%

From now on, $2$-categorical methods will be ubiquitous.  Before proceeding, the reader may wish to consult Appendix~\ref{sect:cubes} for a survey of the notions we need.

\subsection{An easy result}

For any subgroup $\W'$ of $\W$, let
\[
\ResW{\W}{\W'} : \Rep(\W,\bk) \to \Rep(\W',\bk)
\]
denote the restriction functor. Note that we have $\For^{\W'} \circ \ResW{\W}{\W'} = \For^{\W}$. In particular, we will use the functor $\ResW{\W}{\W'}$ in the case where $\W'$ is the subgroup $\W_s$ generated by a simple reflection $s$. This following proposition is very easy; its proof is left to the reader.

\begin{prop}
\label{prop:functors-Rep-W}
Suppose we have two $\bk$-linear functors $\sfG,\sfH : \cA \to \Rep(\W,\bk)$, where $\cA$ is some $\bk$-linear category,
and a given isomorphism of functors
\[
\phi : \For^{\W} \circ \sfG \directednatisom \For^{\W} \circ \sfH.
\]
Assume that for any simple reflection $s\in\W$ there is an isomorphism of functors 
\[
\phi^{\W_s} : \ResW{\W}{\W_s} \circ \sfG \directednatisom \ResW{\W}{\W_s} \circ \sfH\;\text{ such that }\;\For^{\W_s} \circ \phi^{\W_s} = \phi.
\] 
There is a unique isomorphism of functors $\phi^{\W} : \sfG \directednatisom \sfH$
such that $\For^{\W} \circ \phi^{\W} = \phi$.
\end{prop}

%\begin{proof}
%The isomorphism $\phi$ consists of isomorphisms of $\bk$-modules
%\[
%\phi_X : \sfG(X) \xrightarrow{\sim} \sfH(X)
%\]
%for all $X$ in $\cA$ (with compatibility conditions). This isomorphism can be lifted to an isomorphism $\phi^{\W}$ as in the statement if and only if $\phi_X$ is a morphism of $\bk\W$-modules for any $X$. However, our assumption ensures that $\phi_X$ commutes with the action of any simple reflection in $\W$. As simple reflections generate $\W$, this implies that $\phi_X$ commutes with the $\W$-actions. The uniqueness of $\phi^{\W}$ is obvious.
%\end{proof}

\subsection{Restriction, transitivity and intertwining}
\label{subsect:four-prisms}

To prove Theorem~\ref{thm:main-result}, we must define an isomorphism of functors 
\begin{equation} \label{eqn:main-square}
\alpha_G:\Phi_{\Gv}\circ\Satakesm{G}\directednatisom\Springer{G}\circ\Psi_G.
\end{equation}
As foreshadowed in the introduction, we will construct $\alpha_G$ in a way that is compatible with certain restriction functors from each of the four categories involved to the corresponding category for a Levi subgroup $L$:
\[
\begin{split}
\ResGr{G}{L}&:\Perv_{\GO}(\Gr_G^{\sm},\bk)\to\Perv_{\LO}(\Gr_L^{\sm},\bk),\\
\ResG{\Gv}{\Lv}&:\Rep(\Gv,\bk)_{\sm}\to\Rep(\Lv,\bk)_{\sm},\\
\ResN{G}{L}&:\Perv_G(\cN_G,\bk)\to\Perv_L(\cN_L,\bk),\\
\ResW{\W_G}{\W_L}&:\Rep(\W_G,\bk)\to\Rep(\W_L,\bk).
\end{split}
\]
The functor $\ResW{\W_G}{\W_L}$ was defined above; the other ones will be defined in Section~\ref{sect:restriction}.

As a foretaste of the general definitions, consider the special case where $L = T$.  Note that $\W_T$ is trivial, so $\Rep(\W_T,\bk)=\Vect(\bk)$, and $\ResW{\W_G}{\W_T}=\For^{\W_G}$. Also $\Rep(\Tv,\bk)_{\sm}=\Vect(\bk)$. We define $\ResG{\Gv}{\Tv}=(-)^{\Tv}$.  Next, $\Gr^\sm_T$ and $\cN_T$ are both just single points, so $H^0:\Perv_{\TO}(\Gr_T^{\sm},\bk)\to\Vect(\bk)$ and $H^0:\Perv_T(\cN_T,\bk)\to\Vect(\bk)$ are equivalences of categories.  We will define $\ResGr{G}{T}$ and $\ResN{G}{T}$ in such a way that
\[
H^0\circ\ResGr{G}{T} = H^0_{\fT_0}(\Gr_G, {-})
\qquad\text{and}\qquad
H^0\circ\ResN{G}{T} = H^0_{\fn}(\cN_G, {-}).
\]

%Notice that the names of these functors use four different fonts for the letter `R'. %We will consistently use fraktur letters for the names of functors between categories of perverse sheaves on Grassmannians, Roman letters for categories of representations of algebraic groups, calligraphic letters for categories of perverse sheaves on nilpotent cones, and sans-serif letters for categories of representations of Weyl groups. (We have no such convention for functors from one of these types of categories to another.)

For all four restriction functors we will define \emph{transitivity isomorphisms}: 
\[
\begin{array}{ccccccc}
\ResGr{G}{T} & \natisom & \ResGr{L}{T}\circ\ResGr{G}{L}, & \quad & \ResG{\Gv}{\Tv} & \natisom & \ResG{\Lv}{\Tv}\circ\ResG{\Gv}{\Lv},\\
\ResN{G}{T} & \natisom & \ResN{L}{T}\circ\ResN{G}{L}, & \quad & \ResW{\W_G}{\W_T} & \natisom & \ResW{\W_L}{\W_T}\circ\ResW{\W_G}{\W_L}.
\end{array}
\]
The last of these transitivity isomorphisms is simply the identity isomorphism from $\For^{\W_G}$ to itself. The other three will be defined in Section~\ref{sect:restriction}.
(In each case, a more general transitivity isomorphism exists, replacing $T$ by a Levi subgroup contained in $L$. This generality is not needed for the proof of Theorem \ref{thm:main-result}, hence not considered.)

The bulk of our work will be in showing that the four functors in \eqref{eqn:main} intertwine the restriction functors in a way that is compatible with the transitivity isomorphisms. More precisely, we will define \emph{intertwining isomorphisms}:
\[
\begin{array}{ccccccc}
\ResW{\W_G}{\W_L}\circ\Phi_{\Gv}& \natisom & \Phi_{\Lv}\circ\ResG{\Gv}{\Lv}, & \quad & \ResN{G}{L}\circ\Psi_G & \natisom & \Psi_L\circ\ResGr{G}{L},\\
\ResG{\Gv}{\Lv}\circ\Satakesm{G} & \natisom & \Satakesm{L}\circ\ResGr{G}{L}, & \quad & \ResW{\W_G}{\W_L}\circ\Springer{G}& \natisom & \Springer{L}\circ\ResN{G}{L}
\end{array}
\]
and show that the following four prisms commute, in the sense explained in Example~\ref{ex:prism}. Here we label each triangular face by the appropriate transitivity isomorphism, and each square face by the appropriate intertwining isomorphism, whether that is the general $(G,L)$ version or either of the $(G,T)$ and $(L,T)$ versions that are entailed as special cases.
\begin{equation} \label{eqn:phi-prism}
\vc{\begin{tikzpicture}[longprism]
\compuinit
% hidden 2-cells
\node[cuber] at (0.5,0.5,0) {$\IT$};
\node[prismdf] at (1,0.5,\tric) {$\mTr$};
% outer 0-cells
\node (rlu) at (0,1,0) {$\Rep(\Gv,\bk)_\sm$};
\node (rru) at (1,1,0) {$\Rep(\W_G,\bk)$};
\node (fr) at (	1,0.5,1) {$\Rep(\W_L,\bk)$};
\node (rld) at (0,0,0) {$\Rep(\Tv,\bk)_\sm$};
\node (rrd) at (1,0,0) {$\Rep(\W_T,\bk)$};
% hidden 1-cells
\draw[liner,->] (rru) -- node[arr,pos=.3] {$\al \ResW{\W_G}{\W_T}$} (rrd);
% outer 1-cells
\draw[->] (rlu) -- node[arl] {$\al \Phi_{\Gv}$} (rru);
\draw[->] (rlu) -- node[arr] {$\al \ResG{\Gv}{\Tv}$} (rld);
\draw[->] (rru) -- node[arl] {$\al \ResW{\W_G}{\W_L}$} (fr);
\draw[->] (rld) -- node[arr] {$\al \Phi_{\Tv}$} (rrd);
\draw[->] (fr) -- node[arl] {$\al \ResW{\W_L}{\W_T}$} (rrd);
% visible 2-cells
\node[prismtf] at (0.5,0.75,0.5) {$\IT$};
\node[prismbf] at (0.5,0.25,0.5) {$\IT$};
\node[prismlf] at (0,0.5,\tric) {$\mTr$};
% visible 0- and 1-cells
\node (fl) at (0,0.5,1) {$\Rep(\Lv,\bk)_\sm$};
\draw[->] (rlu) -- node[arl] {$\al \ResG{\Gv}{\Lv}$} (fl);
\draw[->] (fl) -- node[arl,pos=.3] {$\al \ResG{\Lv}{\Tv}$} (rld);
\draw[->] (fl)  -- node[arr,pos=.3] {$\al \Phi_{\Lv}$} (fr);
\end{tikzpicture}}
\end{equation}
\vspace{-8pt}
\begin{equation} \label{eqn:psi-prism}
\vc{\begin{tikzpicture}[longprism]
\compuinit
% hidden 2-cells
\node[cuber] at (0.5,0.5,0) {$\IT$};
\node[prismdf] at (1,0.5,\tric) {$\mTr$};
% outer 0-cells
\node (rlu) at (0,1,0) {$\Perv_{\GO}(\Gr_G^{\sm},\bk)$};
\node (rru) at (1,1,0) {$\Perv_G(\cN_G,\bk)$};
\node (fr) at (	1,0.5,1) {$\Perv_L(\cN_L,\bk)$};
\node (rld) at (0,0,0) {$\Perv_{\TO}(\Gr_T^{\sm},\bk)$};
\node (rrd) at (1,0,0) {$\Perv_T(\cN_T,\bk)$};
% hidden 1-cells
\draw[liner,->] (rru) -- node[arr,pos=.3] {$\al \ResN{G}{T}$} (rrd);
% outer 1-cells
\draw[->] (rlu) -- node[arl] {$\al \Psi_G$} (rru);
\draw[->] (rlu) -- node[arr] {$\al \ResGr{G}{T}$} (rld);
\draw[->] (rru) -- node[arl] {$\al \ResN{G}{L}$} (fr);
\draw[->] (rld) -- node[arr] {$\al \Psi_T$} (rrd);
\draw[->] (fr) -- node[arl] {$\al \ResN{L}{T}$} (rrd);
% visible 2-cells
\node[prismtf] at (0.5,0.75,0.5) {$\IT$};
\node[prismbf] at (0.5,0.25,0.5) {$\IT$};
\node[prismlf] at (0,0.5,\tric) {$\mTr$};
% visible 0- and 1-cells
\node (fl) at (0,0.5,1) {$\Perv_{\LO}(\Gr_L^{\sm},\bk)$};
\draw[->] (rlu) -- node[arl] {$\al \ResGr{G}{L}$} (fl);
\draw[->] (fl) -- node[arl,pos=.3] {$\al \ResGr{L}{T}$} (rld);
\draw[->] (fl)  -- node[arr,pos=.3] {$\al \Psi_L$} (fr);
\end{tikzpicture}}
\end{equation}
\vspace{-8pt}
\begin{equation} \label{eqn:satake-prism}
\vc{\begin{tikzpicture}[longprism]
\compuinit
% hidden 2-cells
\node[cuber] at (0.5,0.5,0) {$\IT$};
\node[prismdf] at (1,0.5,\tric) {$\mTr$};
% outer 0-cells
\node (rlu) at (0,1,0) {$\Perv_{\GO}(\Gr_G^{\sm},\bk)$};
\node (rru) at (1,1,0) {$\Rep(\Gv,\bk)_\sm$};
\node (fr) at (	1,0.5,1) {$\Rep(\Lv,\bk)_\sm$};
\node (rld) at (0,0,0) {$\Perv_{\TO}(\Gr_T^{\sm},\bk)$};
\node (rrd) at (1,0,0) {$\Rep(\Tv,\bk)_\sm$};
% hidden 1-cells
\draw[liner,->] (rru) -- node[arr,pos=.3] {$\al \ResG{\Gv}{\Tv}$} (rrd);
% outer 1-cells
\draw[->] (rlu) -- node[arl] {$\al \Satakesm{G}$} (rru);
\draw[->] (rlu) -- node[arr] {$\al \ResGr{G}{T}$} (rld);
\draw[->] (rru) -- node[arl] {$\al \ResG{\Gv}{\Lv}$} (fr);
\draw[->] (rld) -- node[arr] {$\al \Satakesm{T}$} (rrd);
\draw[->] (fr) -- node[arl] {$\al \ResG{\Lv}{\Tv}$} (rrd);
% visible 2-cells
\node[prismtf] at (0.5,0.75,0.5) {$\IT$};
\node[prismbf] at (0.5,0.25,0.5) {$\IT$};
\node[prismlf] at (0,0.5,\tric) {$\mTr$};
% visible 0- and 1-cells
\node (fl) at (0,0.5,1) {$\Perv_{\LO}(\Gr_L^{\sm},\bk)$};
\draw[->] (rlu) -- node[arl] {$\al \ResGr{G}{L}$} (fl);
\draw[->] (fl) -- node[arl,pos=.3] {$\al \ResGr{L}{T}$} (rld);
\draw[->] (fl)  -- node[arr,pos=.3] {$\al \Satakesm{L}$} (fr);
\end{tikzpicture}}
\end{equation}
\vspace{-8pt}
\begin{equation} \label{eqn:springer-prism}
\vc{\begin{tikzpicture}[longprism]
\compuinit
% hidden 2-cells
\node[cuber] at (0.5,0.5,0) {$\IT$};
\node[prismdf] at (1,0.5,\tric) {$\mTr$};
% outer 0-cells
\node (rlu) at (0,1,0) {$\Perv_G(\cN_G,\bk)$};
\node (rru) at (1,1,0) {$\Rep(\W_G,\bk)$};
\node (fr) at (	1,0.5,1) {$\Rep(\W_L,\bk)$};
\node (rld) at (0,0,0) {$\Perv_T(\cN_T,\bk)$};
\node (rrd) at (1,0,0) {$\Rep(\W_T,\bk)$};
% hidden 1-cells
\draw[liner,->] (rru) -- node[arr,pos=.3] {$\al \ResW{\W_G}{\W_T}$} (rrd);
% outer 1-cells
\draw[->] (rlu) -- node[arl] {$\al \Springer{G}$} (rru);
\draw[->] (rlu) -- node[arr] {$\al \ResN{G}{T}$} (rld);
\draw[->] (rru) -- node[arl] {$\al \ResW{\W_G}{\W_L}$} (fr);
\draw[->] (rld) -- node[arr] {$\al \Springer{T}$} (rrd);
\draw[->] (fr) -- node[arl] {$\al \ResW{\W_L}{\W_T}$} (rrd);
% visible 2-cells
\node[prismtf] at (0.5,0.75,0.5) {$\IT$};
\node[prismbf] at (0.5,0.25,0.5) {$\IT$};
\node[prismlf] at (0,0.5,\tric) {$\mTr$};
% visible 0- and 1-cells
\node (fl) at (0,0.5,1) {$\Perv_L(\cN_L,\bk)$};
\draw[->] (rlu) -- node[arl] {$\al \ResN{G}{L}$} (fl);
\draw[->] (fl) -- node[arl,pos=.3] {$\al \ResN{L}{T}$} (rld);
\draw[->] (fl)  -- node[arr,pos=.3] {$\al \Springer{L}$} (fr);
\end{tikzpicture}}
\end{equation}
The definitions of the intertwining isomorphisms for $\Phi_{\Gv}$ and $\Psi_G$, and the proofs that \eqref{eqn:phi-prism} and \eqref{eqn:psi-prism} commute, will be given in Section~\ref{sect:phi-psi}. The definition of the intertwining isomorphism for $\Satakesm{G}$ and the proof that \eqref{eqn:satake-prism} commutes will be given in Section~\ref{sect:satake}. The definition of the intertwining isomorphism for $\Springer{G}$ and the proof that \eqref{eqn:springer-prism} commutes will be given in Section~\ref{sect:springer}.

To illustrate the meaning of the intertwining isomorphism for $\Springer{G}$ (the most difficult to construct), take $L=T$. The functor $\Springer{T}$ is canonically isomorphic to $H^0:\Perv_T(\cN_T,\bk)\simto\Vect(\bk)$. So we obtain an isomorphism $\For^{W_G}\circ\Springer{G} \cong H^0_{\fn}(\cN_G,-)$. The analogous isomorphism in the case of $\Qlb$-sheaves was found in~\cite{a}.

\subsection{Constructing $\alpha_G$}
\label{subsect:main-proof}

Assuming all the definitions and commutativity results referred to in \S\ref{subsect:four-prisms}, the construction of the isomorphism \eqref{eqn:main-square} proceeds as follows. 

First, we construct an analogous isomorphism for $T$.  Recall that $\Gr_T^{\sm}$ and $\cN_T$ are both single points. The composition $\Phi_{\Tv}\circ\Satakesm{T}$ is the equivalence $H^0:\Perv_{\TO}(\Gr_T^{\sm},\bk)\to\Vect(\bk)$. As observed above, $\Springer{T}$ is canonically isomorphic to the equivalence $H^0:\Perv_T(\cN_T,\bk)\to\Vect(\bk)$. Since $\Psi_T:\Perv_{\TO}(\Gr_T^{\sm},\bk)\to\Perv_T(\cN_T,\bk)$ is the obvious identification, we have a canonical isomorphism 
\begin{equation}
\alpha_T:\Phi_{\Tv}\circ\Satakesm{T}\directednatisom\Springer{T}\circ\Psi_T.
\end{equation}

We can now state a more precise version of Theorem~\ref{thm:main-result}.

\begin{thm} \label{thm:main-result-again} 
There is a unique isomorphism $\alpha_G:\Phi_{\Gv}\circ\Satakesm{G}\directednatisom\Springer{G}\circ\Psi_G$ that makes the following cube commutative:
\begin{equation} \label{eqn:main-cube}
\vc{\begin{tikzpicture}[longcube]
\compuinit
% hidden 0-cell
\node (rrd) at (1,0,0) {$\Rep(\Tv,\bk)_{\sm}$};
% hidden 2-cells
\node[cuber] at (0.5,0.5,0) {$\IT$};
\node[cubed] at (1,0.5,0.5) {$\IT$};
\node[cubeb] at (0.5,0,0.5) {\Large$\alpha_T$};
% outer 0-cells
\node (rlu) at (0,1,0) {$\Perv_{\GO}(\Gr_G^{\sm},\bk)$};
\node (rru) at (1,1,0) {$\Rep(\Gv,\bk)_{\sm}$};
\node (fru) at (1,1,1) {$\Rep(\W_G,\bk)$};
\node (frd) at (1,0,1) {$\Rep(\W_T,\bk)$};
\node (fld) at (0,0,1) {$\Perv_T(\cN_T,\bk)$};
\node (rld) at (0,0,0) {$\Perv_{\TO}(\Gr_T^{\sm},\bk)$};
% hidden 1-cells
\draw[liner,->] (rru) -- node[arl,pos=.7] {$\al \ResG{\Gv}{\Tv}$} (rrd);
\draw[liner,->] (rld) -- node[arl] {$\al \Satakesm{T}$} (rrd);
\draw[liner,->] (rrd) -- node[arl] {$\al \Phi_{\Tv}$} (frd);
% outer 1-cells
\draw[->] (rlu) -- node[arl] {$\al \Satakesm{G}$} (rru);
\draw[->] (rru) -- node[arl] {$\al \Phi_{\Gv}$} (fru);
\draw[->] (fru) -- node[arl] {$\al \ResW{\W_G}{\W_T}$}(frd);
\draw[->] (fld) -- node[arr] {$\al \Springer{T}$} (frd);
\draw[->] (rld) -- node[arr] {$\al \Psi_T$} (fld);
\draw[->] (rlu) -- node[arr] {$\al \ResGr{G}{T}$} (rld);
% visible 2-cells
\node[cubel] at (0,0.5,0.5) {$\IT$};
\node[cubet] at (0.5,1,0.5) {\Large$\alpha_G$};
\node[cubef] at (0.5,0.5,1) {$\IT$};
% visible 0- and 1-cells
\node (flu) at (0,1,1) {$\Perv_G(\cN_G,\bk)$};
\draw[->] (rlu) -- node[arl,pos=.7] {$\al \Psi_G$} (flu);
\draw[->] (flu) -- node[arr,pos=.3] {$\al \Springer{G}$} (fru);
\draw[->] (flu) -- node[arl,pos=.3] {$\al \ResN{G}{T}$} (fld);
\end{tikzpicture}}
\end{equation}
Here the top face is to be labelled by $\alpha_G$, the bottom face by $\alpha_T$, and the other faces by the appropriate intertwining isomorphisms.
\end{thm}

In Section~\ref{sect:rankone}, we will prove Theorem~\ref{thm:main-result-again} in the special case that $G$ has semisimple rank $1$. Assuming that, the proof of Theorem~\ref{thm:main-result-again} in general is as follows.

\begin{proof}
From the isomorphisms already defined we have an isomorphism 
\begin{equation}
\phi_{G,T}:\ResW{\W_G}{\W_T}\circ\Phi_{\Gv}\circ\Satakesm{G}\directednatisom\ResW{\W_G}{\W_T}\circ\Springer{G}\circ\Psi_G,
\end{equation}
namely that obtained as the composition of the five already constructed edges of the hexagon \eqref{eqn:hexagon} associated to our cube:
\begin{equation}
\vcenter{
\xymatrix@C=8pt@R=10pt{
&&\ResW{\W_G}{\W_T}\circ\Phi_{\Gv}\circ\Satakesm{G}\ar@{<=>}[drr]\\
\ResW{\W_G}{\W_T}\circ\Springer{G}\circ\Psi_G\ar@{<=>}[d]&&&&\Phi_{\Tv}\circ\ResG{\Gv}{\Tv}\circ\Satakesm{G}\ar@{<=>}[d]\\
\Springer{T}\circ\ResN{G}{T}\circ\Psi_G\ar@{<=>}[drr]&&&&\Phi_{\Tv}\circ\Satakesm{T}\circ\ResGr{G}{T}\ar@{<=>}[dll]\\
&&\Springer{T}\circ\Psi_T\circ\ResGr{G}{T}
}
}
\end{equation} 
Saying that $\alpha_G$ makes \eqref{eqn:main-cube} commutative amounts to the equality $\ResW{\W_G}{\W_T}\circ\alpha_G=\phi_{G,T}$. 

By Proposition \ref{prop:functors-Rep-W}, the existence and uniqueness of such $\alpha_G$ will follow if we can show that whenever $L$ has semisimple rank $1$, there exists an isomorphism 
\[ 
\phi^{\W_L}:\ResW{\W_G}{\W_L}\circ\Phi_{\Gv}\circ\Satakesm{G}\directednatisom\ResW{\W_G}{\W_L}\circ\Springer{G}\circ\Psi_G\;\text{ such that }\;\ResW{\W_L}{\W_T}\circ\phi^{\W_L}=\phi_{G,T}.
\]
From now on, let $L$ have semisimple rank $1$. By the special case of Theorem~\ref{thm:main-result-again} we are assuming, there is an isomorphism $\alpha_L:\Phi_{\Lv}\circ\Satakesm{L}\directednatisom\Springer{L}\circ\Psi_L$ such that the cube \eqref{eqn:main-cube}, with $G$ replaced by $L$, is commutative, i.e.\ such that $\ResW{\W_L}{\W_T}\circ\alpha_L=\phi_{L,T}$. Then we can glue to this commutative cube the four commutative prisms \eqref{eqn:phi-prism}, \eqref{eqn:psi-prism}, \eqref{eqn:satake-prism}, \eqref{eqn:springer-prism} to produce the labelled $2$-computad shown in Figure~\ref{fig:main-picture}.

\begin{figure}
\scriptsize
\[
\vcenter{
\xymatrix@R=8pt@C=0pt{
\PPerv_{\GO}(\Gr_G^{\sm})\ar[rrrrrrrr]^{\Satakesm{G}}\ar[dddddddd]_{\ResGr{G}{T}}\ar[ddrrrr]^(.6){\Psi_G}\ar[ddddrr]_{\ResGr{G}{L}}&&&&&&&&\Rep(\Gv)_{\sm}\ar'[ddl][ddddll]^(.3){\ResG{\Gv}{\Lv}}\ar'[dd]'[ddddd]^{\ResG{\Gv}{\Tv}}'[dddddd][dddddddd]\ar[ddrrrrrr]^{\Phi_{\Gv}}\\
\\
&&&&\PPerv_G(\cN_G)\ar[rrrrrrrr]^(.3){\Springer{G}}\ar[dddddddd]_(.6){\ResN{G}{T}}\ar[ddddrr]^(.3){\ResN{G}{L}}&&&&&&&&\Rep(\W_G)\ar[dddddddd]_{\ResW{\W_G}{\W_T}}\ar[ddddll]_{\ResW{\W_G}{\W_L}}&&\Rep(\W_G)\ar[ddddddddll]^{\ResW{\W_G}{\W_T}}\ar'[ddll]^(.65){\ResW{\W_G}{\W_L}}[ddddllll]\\
\\
&&\PPerv_{\LO}(\Gr_L^{\sm})\ar'[rr]^(.7){\Satakesm{L}}'[rrr][rrrr]\ar[lldddd]_{\ResGr{L}{T}}\ar'[drr]^(.7){\Psi_L}[ddrrrr]&&&&\Rep(\Lv)_{\sm}\ar'[rdd]^(.7){\ResG{\Lv}{\Tv}}[rrdddd]\ar[ddrrrr]^(.3){\Phi_{\Lv}}&&&&&&\\
&&&&&&&&\\
&&&&&&\PPerv_L(\cN_L)\ar[rrrr]^(.4){\Springer{L}}\ar[lldddd]_(.3){\ResN{L}{T}}&&&&\Rep(\W_L)\ar[rrdddd]^{\ResW{\W_L}{\W_T}}\\
\\
\PPerv_{\TO}(\Gr_T^{\sm})\ar'[rrrr]^(.6){\Satakesm{T}}'[rrrrr][rrrrrrrr]\ar[ddrrrr]_{\Psi_T}&&&&&&&&\Rep(\Tv)_{\sm}\ar[ddrrrr]^{\Phi_{\Tv}}\\
\\
&&&&\PPerv_T(\cN_T)\ar[rrrrrrrr]_{\Springer{T}}&&&&&&&&\Rep(\W_T)
}
}
\]
\caption{(To save space, we abbreviate $\PPerv=\Perv$.)}\label{fig:main-picture}
\end{figure}

Notice that we have glued the prisms together along the triangular faces that they share, except that we have left unglued the two copies of the face labelled by the transitivity isomorphism $\ResW{\W_G}{\W_T}\natisom\ResW{\W_L}{\W_T}\circ\ResW{\W_G}{\W_L}$. Recall that this isomorphism is in fact just an equality.

By the gluing principle of \S\ref{subsect:gluing}, the labelled $2$-computad in Figure~\ref{fig:main-picture} is commutative. Its boundary consists of two pasting diagrams with domain $\ResW{\W_G}{\W_T}\circ\Phi_{\Gv}\circ\Satakesm{G}$ and codomain $\ResW{\W_G}{\W_T}\circ\Springer{G}\circ\Psi_G$, one of which (on the underside of the picture)
has composite $\phi_{G,T}$ and the other of which has composite $\ResW{\W_L}{\W_T}\circ\phi_{G,L}$, where $\phi_{G,L}:\ResW{\W_G}{\W_L}\circ\Phi_{\Gv}\circ\Satakesm{G}\directednatisom\ResW{\W_G}{\W_L}\circ\Springer{G}\circ\Psi_G$ is defined in the same way as $\phi_{G,T}$. 
Hence $\ResW{\W_L}{\W_T}\circ\phi_{G,L}=\phi_{G,T}$, and $\phi_{G,L}$ is the required isomorphism $\phi^{\W_L}$.
\end{proof}

\subsection{Canonicity of $\alpha_G$}
\label{subsect:canonical}
%---------------------------------------------------------------------

In Section~\ref{sect:preliminaries}, we fixed a choice $B \supset T$, but the isomorphism $\alpha_G$ of Theorem~\ref{thm:main-result-again} is actually independent of this choice.  We conclude this section by briefly explaining why.

To make sense of this assertion, we must first replace the categories and functors in~\eqref{eqn:main-cube} by versions that do not depend on the choice of $B$ and $T$.  If $G \supset B' \supset T'$ is another choice, then there exists $g \in G$ such that $gB g^{-1} = B'$ and $gT g^{-1} = T'$.  The key observation is that although $g$ is not unique, the induced map
\[
B/[B,B] \to B'/[B',B']
\]
is independent of $g$.  Thus, the groups $B/[B,B]$ and $B'/[B',B']$ are canonically identified.  Let $\dT$ denote either one of them; we call $\dT$ the \emph{universal maximal torus} for $G$.  Its Lie algebra $\fH$, the \emph{universal Cartan algebra}, is acted on by a reflection group $\bW$, the \emph{universal Weyl group}.  (See~\cite[Lemma~3.1.26]{CG} and the discussion following it.)  The pair $B \supset T$ determines a unique isomorphism $\W_G \cong \bW$.  Moreover, the induced action of $\bW$ on $\Spr$ is independent of this choice, so the Springer functor $\Springer{G}$ can be regarded as taking values in $\Rep(\bW,\bk)$.

Similar considerations lead to the notion of the \emph{universal zero weight space} of a $\Gv$-module $V$.  Let $\Tv' \subset \Bv' \subset \Gv$ be a maximal torus and a Borel subgroup.  This choice determines a partial order on the set of characters of $\Tv'$.  Let $V_{\ge 0}$ (resp.~$V_{> 0}$) be the submodule on which $\Tv'$ acts with weights that are${}\ge 0$ (resp.${}>0$) in this order.  Then the quotient $V_{\ge 0}/V_{> 0}$ is independent of our choice. Moreover, the universal Weyl group $\bW$ and the universal maximal torus $\dTv$ act canonically on this space (the latter acting trivially).  We therefore have a universal version of $\Phi_{\Gv}$ taking values in $\Rep(\bW,\bk)$, as well as a functor $\ResG{\Gv}{\dTv}$ taking values in $\Rep(\dTv,\bk)_{\sm}$.

The existence of a universal version of $\ResGr{G}{T}$, taking values in $\Perv_{\dT(\fO)}(\Gr^{\sm}_{\dT},\bk)$, is proved in~\cite[Theorem~3.6]{mv}.  This result is less elementary than the situations considered above: roughly, as the choice $B' \supset T'$ varies, the various functors $\ResGr{G}{T'}$ (perhaps better denoted $\ResGr{G}{B' \supset T'}$) can be assembled into a local system on $G/T$.  That local system is trivial because $G/T$ is simply connected, so the various functors $\ResGr{G}{B' \supset T'}$ are canonically isomorphic to one another.  The same argument shows that $\ResN{G}{T}$ has a universal version as well.

For the remaining functors in~\eqref{eqn:main-cube}, the independence of the choice of $B \supset T$ is obvious.  Taken together, the preceding paragraphs describe how to construct a version of~\eqref{eqn:main-cube} whose $1$-skeleton is universal.  A priori, the top face is labelled by a $2$-cell $\alpha_G = \alpha_{G \supset B \supset T}$ that depends on the choice of $B \supset T$, but the uniqueness asserted in Theorem~\ref{thm:main-result-again} implies that $\alpha_{G \supset B \supset T} = \alpha_{G \supset B' \supset T'}$ for any other choice $B' \supset T'$.  In other words, $\alpha_G$ is independent of this choice.

%%%%%%%%%%%%%%%%%%%%%%%%%%%%%%%%%%%%%%%%%%%%%%%%%%%%%%%%%%%%%%%%%%%%%%%%%%%%%%%%%%%%%%%%

\section{Restriction to a Levi subgroup}
\label{sect:restriction}

%%%%%%%%%%%%%%%%%%%%%%%%%%%%%%%%%%%%%%%%%%%%%%%%%%%%%%%%%%%%%%%%%%%%%%%%%%%%%%%%%%%%%%%%

Throughout Sections \ref{sect:restriction}--\ref{sect:springer}, we fix a parabolic subgroup $P \subset G$ containing $B$, with Lie algebra $\mathfrak{p}$, and we let $L$ be the unique Levi factor of $P$ containing $T$. We denote by $U_P$ the unipotent radical of $P$. Of course, any notation or construction for the triple $G \supset P \supset L$ can be used for $G \supset B \supset T$ or $L\supset C\supset T$, where $C=B\cap L$.

%In this section, after reviewing some well-known properties of the Satake equivalence with respect to restriction to a Levi subgroup, we define the restriction functors (from the category associated to $G$ to the category associated to $L$) for the categories $\Rep(\Gv,\bk)_{\sm}$, $\Perv_{\GO}(\Gr_G^{\sm},\bk)$ and $\Perv_G(\cN_G,\bk)$, and the transitivity isomorphisms for these restriction functors.

%---------------------------------------------------------------------
\subsection{Review of the Satake equivalence and restriction}
%---------------------------------------------------------------------

Consider the diagram
\begin{equation} \label{eqn:affine-Grassmannian-diagram}
\xymatrix{
\Gr_L & \Gr_P \ar[l]_-{q_P} \ar[r]^-{i_P} & \Gr_G
}
\end{equation}
where $q_P$ is induced by the projection $P \twoheadrightarrow L$ whose kernel is the unipotent radical $U_P$, and $i_P$ is induced by the embedding $P \hookrightarrow G$. Define the functor
\[
\ResGreasy{G}{L} := (q_P)_* \circ (i_P)^! : \cDb(\Gr_G,\bk) \to \cDb(\Gr_L,\bk).
\]
Recall that the connected components of $\Gr_L$ are parametrized by characters of $Z(\widetilde{L})$, where $\widetilde{L} \subset \Gv$ is the Levi subgroup containing $\Tv$ whose roots are dual to those of $L$ (and $Z(\widetilde{L})$ is its centre), see \cite[Proposition 4.5.4]{bd}. If $M$ is in $\cDb(\Gr_L,\bk)$ and $\chi \in X^*(Z(\widetilde{L}))$, we denote by $M_{\chi}$ the restriction of $M$ to the corresponding connected component. Define the functor $\ResGrnonsm{G}{L} : \cDb(\Gr_G,\bk) \to \cDb(\Gr_L,\bk)$
by the formula
\[
\ResGrnonsm{G}{L}(M) = \bigoplus_{\chi \in X^*(Z(\widetilde{L}))} \, \bigl( \ResGreasy{G}{L}(M) \bigr)_{\chi}[\langle \chi, 2\rho_L -2\rho_G \rangle],
\]
where $\rho_G$ and $\rho_L$ are the half sums of positive roots of $G$ and $L$. 
It is proved in \cite[Proposition 5.3.29]{bd} that $\ResGrnonsm{G}{L}$ restricts to a functor
\[
\ResGrnonsm{G}{L} : \Perv_{\GO}(\Gr_G,\bk) \to \Perv_{\LO}(\Gr_L,\bk).
\]
Moreover, it is explained in \cite[\S5.3.30]{bd} that this functor is a tensor functor.

Applying base change for the cartesian square
\begin{equation}
\label{eqn:affine-Grassmannian-square}
\vc{\begin{tikzpicture}[vsmallcube]
\compuinit
% 2-cells
\node[cart] at (0.5,0.5) {};
% 0-cells
\node (lu) at (0,1) {$\Gr_B$};
\node (ru) at (1,1) {$\Gr_P$};
\node (ld) at (0,0) {$\Gr_C$};
\node (rd) at (1,0) {$\Gr_L$};
% 1-cells
\draw[->] (lu) -- (ru);
\draw[->] (lu) -- (ld);
\draw[->] (ru) -- (rd);
\draw[->] (ld) -- (rd);
\end{tikzpicture}}
\end{equation}
we obtain a natural isomorphism of functors:
\begin{equation}
\label{eqn:composition-restriction-Satake-no-shifts}
\ResGreasy{G}{T} \natisom \ResGreasy{L}{T} \circ \ResGreasy{G}{L} : \cDb(\Gr_G,\bk) \to \cDb(\Gr_T,\bk).
\end{equation}
More precisely, this isomorphism is defined by the following pasting diagram:
\begin{equation} \label{eqn:composition-restriction-Satake-no-shifts-db-version}
\vc{\begin{tikzpicture}[stdtriangle]
\compuinit
% 2-cells
\node[lttricelld] at (0.5+0.5*\tric, 1.5+0.5*\tric) {\tiny$\Comp$};
\node[lttricelld] at (1.5+0.5*\tric, 0.5+0.5*\tric) {\tiny$\Comp$};
\node[cubef] at (1.5,1.5) {$\BC$};
% 0-cells
\node (lu) at (0,2) {$\cDb(\Gr_G)$};
\node (mu) at (1,2) {$\cDb(\Gr_P)$};
\node (ru) at (2,2) {$\cDb(\Gr_L)$};
\node (mm) at (1,1) {$\cDb(\Gr_B)$};
\node (rm) at (2,1) {$\cDb(\Gr_C)$};
\node (rd) at (2,0) {$\cDb(\Gr_T)$};
% 1-cells
\draw[->] (ru) -- node[arl] {{\tiny $\al (\cdot)^!$}} (rm);
\draw[->] (lu) -- node[arl] {{\tiny $\al (\cdot)^!$}} (mu);
\draw[->] (lu) -- node[arr] {{\tiny $\al (\cdot)^!$}} (mm);
\draw[->] (mu) -- node[arl] {{\tiny $(\cdot)_*$}} (ru);
\draw[->] (mu) -- node[arr] {{\tiny $(\cdot)^!$}} (mm);
\draw[->] (rm) -- node[arl] {{\tiny $(\cdot)_*$}} (rd);
\draw[->] (mm) -- node[arl] {{\tiny $(\cdot)_*$}} (rm);
\draw[->] (mm) -- node[arr] {{\tiny $(\cdot)_*$}} (rd);
\end{tikzpicture}}
\end{equation}
For simplicity, we have not indicated the morphisms; all of them are the obvious ones.
The notations $\Co$ and $\BC$, and similar notations used in later diagrams, are explained in Appendix~\ref{sect:lemmas}. Restricting to perverse sheaves and taking shifts into account, one can check that \eqref{eqn:composition-restriction-Satake-no-shifts} induces an isomorphism
\begin{equation}
\label{eqn:composition-restriction-Satake}
\ResGrnonsm{G}{T} \natisom \ResGrnonsm{L}{T} \circ \ResGrnonsm{G}{L} : \Perv_{\GO}(\Gr_G,\bk) \to \Perv_{\TO}(\Gr_T,\bk).
\end{equation}

Consider the case $P=B$, $L=T$. The morphism $i_B : \Gr_B \to \Gr_G$ is a bijection and a locally closed embedding, which factors through a natural identification
\[
\Gr_B \xrightarrow{\sim} \bigsqcup_{\la \in \bX} \, \fT_{\la}.
\]
Using this identification, the composition of $\ResGrnonsm{G}{T}$ with the equivalence
\[
\Satake{T} : \Perv_{\TO}(\Gr_T,\bk) \xrightarrow{\sim} \Rep(\Tv,\bk)
\]
is identified with the functor $\sfF^{\bX}$ of \S\ref{ss:Satake}, so that \eqref{eqn:fibre-functor-Tv} induces an isomorphism
\begin{equation}
\label{eqn:composition-fibre-ResGr}
\sfF_G\natisom \For^{\Tv} \circ \Satake{T} \circ \ResGrnonsm{G}{T}= \sfF_T \circ \ResGrnonsm{G}{T}.
\end{equation}
Hence, composing isomorphism \eqref{eqn:composition-restriction-Satake} with $\sfF_T$ provides an isomorphism of functors
\begin{equation}
\label{eqn:composition-fibre-ResGr-L}
\sfF_G \natisom \sfF_L \circ \ResGrnonsm{G}{L}.
\end{equation}
It is explained in \cite[\S5.3.30]{bd} that this isomorphism is an isomorphism of tensor functors. If $\Lv$ is the $\bk$-algebraic group provided by the constructions of \S\ref{ss:Satake} for the group $L$, we obtain using \eqref{eqn:composition-fibre-ResGr-L} a morphism of algebraic groups 
\[
\iota_{\Lv}^{\Gv} : \Lv = \Aut^{\star}(\sfF_L) \to \Aut^{\star}(\sfF_L \circ \ResGrnonsm{G}{L}) \cong \Aut^{\star}(\sfF_G) = \Gv.
\]
It is known that $\iota_{\Lv}^{\Gv}$ is injective, and that its image is $\widetilde{L}$ (see \cite[Lemma 5.3.31]{bd}); we can therefore identify $\Lv$ and $\widetilde{L}$. Note that the following diagram of isomorphisms of functors is commutative by construction of isomorphism \eqref{eqn:composition-fibre-ResGr-L}:
\begin{equation}
\label{eqn:diag-iota}
\vcenter{
\xymatrix@C=2cm@R=12pt{
\sfF_G \ar@{<=>}[r]^-{\eqref{eqn:composition-fibre-ResGr}_G} \ar@{<=>}[d]_-{\eqref{eqn:composition-fibre-ResGr-L}} & \sfF_T \circ \ResGrnonsm{G}{T} \ar@{<=>}[d]^-{\eqref{eqn:composition-restriction-Satake}} \\
\sfF_L \circ \ResGrnonsm{G}{L} \ar@{<=>}[r]^-{\eqref{eqn:composition-fibre-ResGr}_L} & \sfF_T \circ \ResGrnonsm{L}{T} \circ \ResGrnonsm{G}{L} 
}
}
\end{equation}

Let $\ResGnonsm{\Gv}{\Lv} : \Rep(\Gv,\bk) \to \Rep(\Lv,\bk)$
be the restriction functor (i.e.~inverse image for the morphism $\iota_{\Lv}^{\Gv}$). We have
\begin{equation}
\label{eqn:restGnonsm}
\For^{\Lv} \circ \ResGnonsm{\Gv}{\Lv} = \For^{\Gv}.
\end{equation}
By construction, isomorphism \eqref{eqn:composition-fibre-ResGr-L} lifts to an isomorphism of functors
\begin{equation}
\label{eqn:restriction-Satake-non-sm}
\ResGnonsm{\Gv}{\Lv} \circ \Satake{G} \natisom \Satake{L} \circ \ResGrnonsm{G}{L}.
\end{equation}

In the case $P=B$, $L=T$ the morphism $\iota_{\Tv}^{\Gv} : \Tv \to \Gv$ is the morphism considered in \S\ref{ss:Satake}. Moreover, by commutativity of \eqref{eqn:diag-iota} we have $\iota_{\Lv}^{\Gv} \circ \iota_{\Tv}^{\Lv} = \iota_{\Tv}^{\Gv}$.
It follows that
\begin{equation}
\label{eqn:composition-restriction-G-non-sm}
\ResGnonsm{\Gv}{\Tv} = \ResGnonsm{\Lv}{\Tv} \circ \ResGnonsm{\Gv}{\Lv} : \Rep(\Gv,\bk) \to \Rep(\Tv,\bk).
\end{equation}

\begin{lem}
\label{lem:restriction-Satake-Levi}

The following prism is commutative:
\[
\vc{\begin{tikzpicture}[longprism]
\compuinit
% hidden 2-cells
\node[cuber] at (0.5,0.5,0) {\eqrefh{eqn:restriction-Satake-non-sm}};
\node[prismdf] at (1,0.5,\tric) {\eqrefh{eqn:composition-restriction-G-non-sm}};
% outer 0-cells
\node (rlu) at (0,1,0) {$\Perv_{\GO}(\Gr_G,\bk)$};
\node (rru) at (1,1,0) {$\Rep(\Gv,\bk)$};
\node (fr) at (	1,0.5,1) {$\Rep(\Lv,\bk)$};
\node (rld) at (0,0,0) {$\Perv_{\TO}(\Gr_T,\bk)$};
\node (rrd) at (1,0,0) {$\Rep(\Tv,\bk)$};
% hidden 1-cells
\draw[liner,->] (rru) -- node[arl,pos=.3] {$\al \ResGnonsm{\Gv}{\Tv}$} (rrd);
% outer 1-cells
\draw[->] (rlu) -- node[arl] {$\al \Satake{G}$} (rru);
\draw[->] (rlu) -- node[arr] {$\al \ResGrnonsm{G}{T}$} (rld);
\draw[->] (rru) -- node[arl] {$\al \ResGnonsm{\Gv}{\Lv}$} (fr);
\draw[->] (rld) -- node[arr] {$\al \Satake{T}$} (rrd);
\draw[->] (fr) -- node[arl] {$\al \ResGnonsm{\Lv}{\Tv}$} (rrd);
% visible 2-cells
\node[prismtf] at (0.5,0.75,0.5) {\eqrefh{eqn:restriction-Satake-non-sm}};
\node[prismbf] at (0.5,0.25,0.5) {\eqrefh{eqn:restriction-Satake-non-sm}};
\node[prismlf] at (0,0.5,\tric) {\eqrefh{eqn:composition-restriction-Satake}};
% visible 0- and 1-cells
\node (fl) at (0,0.5,1) {$\Perv_{\LO}(\Gr_L,\bk)$};
\draw[->] (rlu) -- node[arl] {$\al \ResGrnonsm{G}{L}$} (fl);
\draw[->] (fl) -- node[arl,pos=.3] {$\al \ResGrnonsm{L}{T}$} (rld);
\draw[->] (fl)  -- node[arr,pos=.3] {$\al \Satake{L}$} (fr);
\end{tikzpicture}}
\]
\end{lem}

\begin{proof}
We have to prove that the following diagram is commutative:
\[
\xymatrix@R=15pt{
\ResGnonsm{\Lv}{\Tv} \circ \ResGnonsm{\Gv}{\Lv} \circ \Satake{G} \ar@{<=>}[r]^-{\eqref{eqn:composition-restriction-G-non-sm}} \ar@{<=>}[d]_-{\eqref{eqn:restriction-Satake-non-sm}_{G,L}} & \ResGnonsm{\Gv}{\Tv} \circ \Satake{G} \ar@{<=>}[r]^-{\eqref{eqn:restriction-Satake-non-sm}_{G,T}} & \Satake{T} \circ \ResGrnonsm{G}{T} \ar@{<=>}[d]^-{\eqref{eqn:composition-restriction-Satake}} \\
\ResGnonsm{\Lv}{\Tv} \circ \Satake{L} \circ \ResGrnonsm{G}{L} \ar@{<=>}[rr]_-{\eqref{eqn:restriction-Satake-non-sm}_{L,T}} & & \Satake{T} \circ \ResGrnonsm{L}{T} \circ \ResGrnonsm{G}{L}.
}
\]
As the functor $\For^{\Tv} : \Rep(\Tv,\bk) \to \Vect(\bk)$ is faithful, it is enough to prove the commutativity of the diagram obtained by composing each functor with $\For^{\Tv}$. But the resulting diagram can be identified (using \eqref{eqn:restGnonsm}) with diagram \eqref{eqn:diag-iota}, which is commutative by construction.
\end{proof}

%---------------------------------------------------------------------
\subsection{Restriction functor for small representations}
\label{ss:small-representations-restriction}
%---------------------------------------------------------------------

Consider now the functor
\[
\ResGzero{\Gv}{\Lv} := (-)^{Z(\Lv)} \circ \ResGnonsm{\Gv}{\Lv} \circ \mathbf{I}_{\Gv} : \Rep(\Gv,\bk)^{Z(\Gv)} \to \Rep(\Lv,\bk)^{Z(\Lv)}.
\]
By \eqref{eqn:composition-restriction-G-non-sm} and the fact that $Z(\Gv)\subset Z(\Lv)\subset Z(\Tv)=\Tv$, we have
\begin{equation}
\label{eqn:composition-restriction-G-zero}
\ResGzero{\Gv}{\Tv} = \ResGzero{\Lv}{\Tv} \circ \ResGzero{\Gv}{\Lv} : \Rep(\Gv,\bk)^{Z(\Gv)} \to \Rep(\Tv,\bk)^{Z(\Tv)}.
\end{equation}

\begin{lem}
\label{lem:ResG}

There is a unique functor $\ResG{\Gv}{\Lv} : \Rep(\Gv,\bk)_{\sm} \to \Rep(\Lv,\bk)_{\sm}$
such that
\begin{equation}
\label{eqn:ResG-ResGnonsm}
\ResGzero{\Gv}{\Lv} \circ \mathbb{I}_{\Gv}^0 = \mathbb{I}_{\Lv}^0 \circ \ResG{\Gv}{\Lv}.
\end{equation}

\end{lem}

\begin{proof}
We have to show that for any $V \in \Rep(\Gv,\bk)_{\sm}$, $V':=(\ResGnonsm{\Gv}{\Lv} V)^{Z(\Lv)}$ is in $\Rep(\Lv,\bk)_{\sm}$. By definition, the $\Lv$-action on $V'$ factors through $\Lv/Z(\Lv)$, hence all the $\Tv$-weights of $V'$ are in $\Z \Rv$. Moreover, the convex hull of weights of $V'$ is included in the convex hull of weights of $V$, hence does not contain any weight of the form $2\alv$ for a root $\alv$ of $\Lv$, which proves the lemma.
\end{proof}

We deduce from \eqref{eqn:composition-restriction-G-zero} that we have
\begin{equation}
\label{eqn:composition-restriction-G}
\ResG{\Gv}{\Tv} = \ResG{\Lv}{\Tv} \circ \ResG{\Gv}{\Lv}.
\end{equation}
We therefore define the transitivity isomorphism for $\ResG{\Gv}{\Lv}$ to be simply this equality.

%---------------------------------------------------------------------
\subsection{Restriction functor for $\Perv_{\GO}(\Gr_G^{\sm})$}
\label{ss:restriction-Grsm}
%---------------------------------------------------------------------

%As an intermediate step, we first construct restriction functors for connected components of base points in affine Grassmannians. 
Let us consider the diagram
\begin{equation}
\xymatrix{
\Gr_L^\circ & \Gr_P^\circ \ar[l]_-{q_P^\circ} \ar[r]^-{i_P^\circ} & \Gr_G^\circ
}
\end{equation}
obtained by restriction of diagram \eqref{eqn:affine-Grassmannian-diagram},
and the functor
\[
\ResGrzero{G}{L} := (q_P^\circ)_* \circ (i_P^\circ)^! : \cDb(\Gr_G^\circ,\bk) \to \cDb(\Gr_L^\circ,\bk).
\]
Recall that $z_G$ denotes the inclusion $\Gr_G^\circ\hookrightarrow\Gr_G$; define $z_P$, $z_L$ similarly.

\begin{lem}

There is a canonical isomorphism of functors
\begin{equation} \label{eqn:resGr-nonsm-zero}
(z_L)^! \circ \ResGrnonsm{G}{L} \natisom \ResGrzero{G}{L} \circ (z_G)^!.
\end{equation}
In particular, $\ResGrzero{G}{L}$ restricts to a functor from $\Perv_{\GO}(\Gr_G^\circ,\bk)$ to $\Perv_{\LO}(\Gr_L^\circ,\bk)$.

\end{lem}

\begin{proof}
We have a cartesian square
\begin{equation} \label{eqn:grzero-cartesian}
\vc{\begin{tikzpicture}[vsmallcube]
\compuinit
% 2-cells
\node[cart] at (0.5,0.5) {};
% 0-cells
\node (lu) at (0,1) {$\Gr_P^\circ$};
\node (ru) at (1,1) {$\Gr_L^\circ$};
\node (ld) at (0,0) {$\Gr_P$};
\node (rd) at (1,0) {$\Gr_L$};
% 1-cells
\draw[->] (lu) -- node[arl] {$\al q_P^\circ$} (ru);
\draw[->] (lu) -- node[arr] {$\al z_P$} (ld);
\draw[->] (ru) -- node[arl] {$\al z_L$} (rd);
\draw[->] (ld) -- node[arr] {$\al q_P$} (rd);
\end{tikzpicture}}
\end{equation}
Then the pasting diagram
\begin{equation} \label{eqn:resGr-nonsm-zero-db-version}
\vc{\begin{tikzpicture}[smallcube]
\compuinit
% 2-cells
\node[cubef] at (0.5,0.5) {$\Comp$};
\node[cubef] at (1.5,0.5) {$\BC$};
% 0-cells
\node (lu) at (0,1) {$\cDb(\Gr_G)$};
\node (mu) at (1,1) {$\cDb(\Gr_P)$};
\node (ru) at (2,1) {$\cDb(\Gr_L)$};
\node (ld) at (0,0) {$\cDb(\Gr_G^\circ)$};
\node (md) at (1,0) {$\cDb(\Gr_P^\circ)$};
\node (rd) at (2,0) {$\cDb(\Gr_L^\circ)$};
% 1-cells
\draw[->] (lu) -- node[arl] {{\tiny $\al (i_P)^!$}} (mu);
\draw[->] (ld) -- node[arr] {{\tiny $\al (i_P^\circ)^!$}} (md);
\draw[->] (mu) -- node[arl] {{\tiny $(q_P)_*$}} (ru);
\draw[->] (md) -- node[arr] {{\tiny $(q_P^\circ)_*$}} (rd);
\draw[->] (lu) -- node[arr] {{\tiny $(z_G)^!$}} (ld);
\draw[->] (mu) -- node[arl] {{\tiny $(z_P)^!$}} (md);
\draw[->] (ru) -- node[arl] {{\tiny $(z_L)^!$}} (rd);
\end{tikzpicture}}
\end{equation}
defines the desired isomorphism, since $(z_L)^! \circ \ResGrnonsm{G}{L} = (z_L)^! \circ \ResGreasy{G}{L}$.
\end{proof}

Restricting the cartesian square \eqref{eqn:affine-Grassmannian-square} to connected components of base points produces the cartesian square
\begin{equation}
\label{eqn:affine-Grassmannian-square-0}
\vc{\begin{tikzpicture}[vsmallcube]
\compuinit
% 2-cells
\node[cart] at (0.5,0.5) {};
% 0-cells
\node (lu) at (0,1) {$\Gr_B^\circ$};
\node (ru) at (1,1) {$\Gr_P^\circ$};
\node (ld) at (0,0) {$\Gr_C^\circ$};
\node (rd) at (1,0) {$\Gr_L^\circ$};
% 1-cells
\draw[->] (lu) -- (ru);
\draw[->] (lu) -- (ld);
\draw[->] (ru) -- (rd);
\draw[->] (ld) -- (rd);
\end{tikzpicture}}
\end{equation}
Then, using the pasting diagram
\begin{equation} \label{eqn:composition-restriction-Satake-zero-db-version}
\vc{\begin{tikzpicture}[stdtriangle]
\compuinit
% 2-cells
\node[lttricelld] at (0.5+0.5*\tric, 1.5+0.5*\tric) {\tiny$\Comp$};
\node[lttricelld] at (1.5+0.5*\tric, 0.5+0.5*\tric) {\tiny$\Comp$};
\node[cubef] at (1.5,1.5) {$\BC$};
% 0-cells
\node (lu) at (0,2) {$\cDb(\Gr_G^\circ)$};
\node (mu) at (1,2) {$\cDb(\Gr_P^\circ)$};
\node (ru) at (2,2) {$\cDb(\Gr_L^\circ)$};
\node (mm) at (1,1) {$\cDb(\Gr_B^\circ)$};
\node (rm) at (2,1) {$\cDb(\Gr_C^\circ)$};
\node (rd) at (2,0) {$\cDb(\Gr_T^\circ)$};
% 1-cells
\draw[->] (ru) -- node[arl] {{\tiny $\al (\cdot)^!$}} (rm);
\draw[->] (lu) -- node[arl] {{\tiny $\al (\cdot)^!$}} (mu);
\draw[->] (lu) -- node[arr] {{\tiny $\al (\cdot)^!$}} (mm);
\draw[->] (mu) -- node[arl] {{\tiny $\al (\cdot)_*$}} (ru);
\draw[->] (mu) -- node[arr] {{\tiny $\al (\cdot)^!$}} (mm);
\draw[->] (rm) -- node[arl] {{\tiny $\al (\cdot)_*$}} (rd);
\draw[->] (mm) -- node[arl] {{\tiny $\al (\cdot)_*$}} (rm);
\draw[->] (mm) -- node[arr] {{\tiny $\al (\cdot)_*$}} (rd);
\end{tikzpicture}}
\end{equation}
and restricting to perverse sheaves we obtain a canonical isomorphism of functors
\begin{equation}
\label{eqn:composition-restriction-Satake-0}
\ResGrzero{G}{T} \natisom \ResGrzero{L}{T} \circ \ResGrzero{G}{L} : \Perv_{\GO}(\Gr_G^\circ,\bk) \to \Perv_{\TO}(\Gr_T^\circ,\bk).
\end{equation}

Since $P$ is not reductive, we have not hitherto defined the notation $\Gr_P^{\sm}$. We set
\[
\Gr_P^{\sm} := \Gr_P^\circ \cap (i_P)^{-1}(\Gr_G^{\sm}),
\]
and denote by $f_P : \Gr_P^{\sm} \hookrightarrow \Gr_P$ the inclusion. 
We have analogous definitions of $\Gr_B^{\sm}$ 
and $\Gr_C^{\sm}$. 
The following result is a geometric counterpart of Lemma \ref{lem:ResG}.

\begin{lem}
\label{lem:Grsm-P-L}

There is a unique morphism $q_P^{\sm} : \Gr_P^{\sm} \to \Gr_L^{\sm}$
such that $f_L\circ q_P^{\sm}=q_P\circ f_P$. 

\end{lem}

\begin{proof}
We have to show that $q_P(\Gr_P^{\sm})\subset\Gr_L^{\sm}$; assume the contrary. As $\Gr_P^{\sm}$ is $\LO$-stable and $q_P$ is $\LO$-equivariant, there exists $\la \in \bX$ which is not small for $\Lv$ and such that $\mathbf{t}_{\la} \in q_P(\Gr_P^{\sm})$. Then $q_P(\Gr_P^{\sm}) \cap \fT_{\la}^L \neq \emptyset$, where $\fT_{\la}^L$ is the 
subvariety of $\Gr_L$ defined in \S\ref{ss:Satake} (for $L$). 
This implies that $\Gr_P^{\sm} \cap (q_P)^{-1}(\fT_{\la}^L) \neq \emptyset$, hence that $\Gr_G^{\sm} \cap i_P \bigl( (q_P)^{-1}(\fT_{\la}^L) \bigr) \neq \emptyset$
(since $i_P \bigl( \Gr_P^{\sm} \cap (q_P)^{-1}(\fT_{\la}^L) \bigr) \subset \Gr_G^{\sm} \cap i_P \bigl( (q_P)^{-1}(\fT_{\la}^L) \bigr)$).

However we have $i_P \bigl( (q_P)^{-1}(\fT_{\la}^L) \bigr)=\fT_{\la}^G$ (see \eqref{eqn:affine-Grassmannian-square}), hence $\Gr_G^{\sm} \cap \fT_{\la}^G \neq \emptyset$. This means that there exists $\mu \in \bX$ which is small for $\Gv$ and such that $\Gr_{G}^{\mu} \cap \fT_{\la}^G \neq \emptyset$. By \cite[Theorem 3.2]{mv} we deduce that $\la$ is in the convex hull of $\W_G \cdot \mu$, which contradicts the fact that $\la$ is not small for $\Lv$.
\end{proof}

Using the lemma we can consider the diagram
\begin{equation} \label{eqn:grlsm-grgsm}
\xymatrix{
\Gr_L^{\sm} & \Gr_P^{\sm} \ar[l]_-{q_P^{\sm}} \ar[r]^-{i_P^{\sm}} & \Gr_G^{\sm}
}
\end{equation}
where $i_P^{\sm}$ denotes the restriction of $i_P$ to $\Gr_P^{\sm}$, and thus define the functor
\[
\ResGr{G}{L} := (q_P^{\sm})_* \circ (i_P^{\sm})^! : \cDb(\Gr^{\sm}_G,\bk) \to \cDb(\Gr_L^{\sm},\bk).
\]
Let us denote by $f_P^\circ : \Gr_P^{\sm} \hookrightarrow \Gr_P^\circ$ the (closed) inclusion.

\begin{lem}

There is a canonical isomorphism of functors
\[
(f_L^\circ)_* \circ \ResGr{G}{L} \natisom \ResGrzero{G}{L} \circ (f_G^\circ)_*.
\]
In particular, $\ResGr{G}{L}$ restricts to a functor from $\Perv_{\GO}(\Gr_G^{\sm},\bk)$ to $\Perv_{\LO}(\Gr_L^{\sm},\bk)$.

\end{lem}

\begin{proof}
By definition of $\Gr_P^{\sm}$, we have a cartesian square
\begin{equation} \label{eqn:grsm-grzero-cartesian}
\vc{\begin{tikzpicture}[vsmallcube]
\compuinit
% 2-cells
\node[cart] at (0.5,0.5) {};
% 0-cells
\node (lu) at (0,1) {$\Gr_P^\sm$};
\node (ru) at (1,1) {$\Gr_G^\sm$};
\node (ld) at (0,0) {$\Gr_P^\circ$};
\node (rd) at (1,0) {$\Gr_G^\circ$};
% 1-cells
\draw[->] (lu) -- node[arl] {$\al i_P^\sm$} (ru);
\draw[->] (lu) -- node[arr] {$\al f_P^\circ$} (ld);
\draw[->] (ru) -- node[arl] {$\al f_G^\circ$} (rd);
\draw[->] (ld) -- node[arr] {$\al i_P^\circ$} (rd);
\end{tikzpicture}}
\end{equation}
Then the pasting diagram
\begin{equation} \label{eqn:resGr-resGrzero}
\vc{\begin{tikzpicture}[smallcube]
\compuinit
% 2-cells
\node[cubef] at (0.5,0.5) {$\BC$};
\node[cubef] at (1.5,0.5) {$\Comp$};
% 0-cells
\node (lu) at (0,1) {$\cDb(\Gr_G^{\sm})$};
\node (mu) at (1,1) {$\cDb(\Gr_P^{\sm})$};
\node (ru) at (2,1) {$\cDb(\Gr_L^{\sm})$};
\node (ld) at (0,0) {$\cDb(\Gr_G^\circ)$};
\node (md) at (1,0) {$\cDb(\Gr_P^\circ)$};
\node (rd) at (2,0) {$\cDb(\Gr_L^\circ)$};
% 1-cells
\draw[->] (lu) -- node[arl] {{\tiny $\al (i_P^{\sm})^!$}} (mu);
\draw[->] (ld) -- node[arr] {{\tiny $\al (i_P^\circ)^!$}} (md);
\draw[->] (mu) -- node[arl] {{\tiny $\al (q_P^{\sm})_*$}} (ru);
\draw[->] (md) -- node[arr] {{\tiny $\al (q_P^\circ)_*$}} (rd);
\draw[->] (lu) -- node[arr] {{\tiny $\al (f_G^\circ)_*$}} (ld);
\draw[->] (mu) -- node[arl] {{\tiny $\al (f_P^\circ)_*$}} (md);
\draw[->] (ru) -- node[arl] {{\tiny $\al (f_L^\circ)_*$}} (rd);
\end{tikzpicture}}
\end{equation}
produces the desired isomorphism.
\end{proof}

Now we construct a transitivity isomorphism for $\ResGr{G}{L}$. We need some preparation. First, observe that the morphism $\Gr_B \to \Gr_P$ induced by the inclusion $B \hookrightarrow P$ induces a morphism $\Gr_B^{\sm} \to \Gr_P^{\sm}$. Similarly, as the composition $\Gr_B \to \Gr_C \to \Gr_L$ coincides with the composition $\Gr_B \to \Gr_P \to \Gr_L$, one can deduce from Lemma \ref{lem:Grsm-P-L} that the natural morphism $\Gr_B \to \Gr_C$ induces a morphism $\Gr_B^{\sm} \to \Gr_C^{\sm}$.

\begin{lem}
\label{lem:square-Grsm-cartesian}

The following square is cartesian:
\[
\vc{\begin{tikzpicture}[vsmallcube]
\compuinit
% 2-cells
\node[cart] at (0.5,0.5) {};
% 0-cells
\node (lu) at (0,1) {$\Gr_B^\sm$};
\node (ru) at (1,1) {$\Gr_P^\sm$};
\node (ld) at (0,0) {$\Gr_C^\sm$};
\node (rd) at (1,0) {$\Gr_L^\sm$};
% 1-cells
\draw[->] (lu) -- node[arl] {$\al a$} (ru);
\draw[->] (lu) -- node[arr] {$\al b$} (ld);
\draw[->] (ru) -- node[arl] {$\al q_P^\sm$} (rd);
\draw[->] (ld) -- node[arr] {$\al i_C^\sm$} (rd);
\end{tikzpicture}}
\]
\end{lem}

\begin{proof}
Let $x \in \Gr_P^{\sm}$ and $y \in \Gr_C^{\sm}$ be such that $q_P^{\sm}(x) = i_C^{\sm}(y)$. As \eqref{eqn:affine-Grassmannian-square-0} is cartesian, there exists $z \in \Gr_B^\circ$ such that $a(z)=x$ and $b(z)=y$. The fact that $x \in \Gr_P^{\sm}$ implies that $i_P(x) = i_B(z) \in \Gr_G^{\sm}$, hence that $z \in \Gr_B^{\sm}$.
\end{proof}

Using Lemma \ref{lem:square-Grsm-cartesian}, the pasting diagram
\begin{equation} \label{eqn:transitivity-isom-resGr-db-version}
\vc{\begin{tikzpicture}[stdtriangle]
\compuinit
% 2-cells
\node[lttricelld] at (0.5+0.5*\tric, 1.5+0.5*\tric) {\tiny$\Comp$};
\node[lttricelld] at (1.5+0.5*\tric, 0.5+0.5*\tric) {\tiny$\Comp$};
\node[cubef] at (1.5,1.5) {$\BC$};
% 0-cells
\node (lu) at (0,2) {$\cDb(\Gr_G^{\sm})$};
\node (mu) at (1,2) {$\cDb(\Gr_P^{\sm})$};
\node (ru) at (2,2) {$\cDb(\Gr_L^{\sm})$};
\node (mm) at (1,1) {$\cDb(\Gr_B^{\sm})$};
\node (rm) at (2,1) {$\cDb(\Gr_C^{\sm})$};
\node (rd) at (2,0) {$\cDb(\Gr_T^{\sm})$};
% 1-cells
\draw[->] (ru) -- node[arl] {{\tiny $\al (\cdot)^!$}} (rm);
\draw[->] (lu) -- node[arl] {{\tiny $\al (\cdot)^!$}} (mu);
\draw[->] (lu) -- node[arr] {{\tiny $\al (\cdot)^!$}} (mm);
\draw[->] (mu) -- node[arl] {{\tiny $\al (\cdot)_*$}} (ru);
\draw[->] (mu) -- node[arr] {{\tiny $\al (\cdot)^!$}} (mm);
\draw[->] (rm) -- node[arl] {{\tiny $\al (\cdot)_*$}} (rd);
\draw[->] (mm) -- node[arl] {{\tiny $\al (\cdot)_*$}} (rm);
\draw[->] (mm) -- node[arr] {{\tiny $\al (\cdot)_*$}} (rd);
\end{tikzpicture}}
\end{equation}
produces (by restriction to perverse sheaves) the desired isomorphism of functors
\begin{equation}
\label{eqn:composition-restriction-Satake-sm}
\ResGr{G}{T} \natisom \ResGr{L}{T} \circ \ResGr{G}{L} : \Perv_{\GO}(\Gr_G^{\sm},\bk) \to \Perv_{\TO}(\Gr_T^{\sm},\bk).
\end{equation}

%---------------------------------------------------------------------
\subsection{Restriction functor for $\Perv_G(\cN_G)$}
\label{ss:restriction-N}
%---------------------------------------------------------------------

Consider the diagram
\begin{equation} \label{eqn:cnl-cng}
\xymatrix{
\cN_L & \cN_P \ar[l]_-{p_P} \ar[r]^-{m_P} & \cN_G
}
\end{equation}
where $\cN_P\subset\fp$ denotes the nilpotent cone of $P$ (as with our notation for reductive groups), $p_P$ is induced by the projection $P \to L$, and $m_P$ is induced by the inclusion $P \hookrightarrow G$. We define the functor
\[
\ResN{G}{L} := (p_P)_* \circ (m_P)^! : \cDb(\cN_G,\bk) \to \cDb(\cN_L,\bk).
\]
%Our first goal in this subsection is to prove: 
 
\begin{prop}
\label{prop:ResN-exact}

The functor $\ResN{G}{L}$ restricts to an exact functor (denoted similarly) from $\Perv_G(\cN_G,\bk)$ to $\Perv_L(\cN_L,\bk)$. 

\end{prop}

\begin{rmk}
The analogue of Proposition~\ref{prop:ResN-exact} for $\Qlb$-sheaves follows from Lusztig's results on character sheaves, especially~\cite[Proposition 15.2]{lus:charsh}.
\end{rmk}

To prove Proposition~\ref{prop:ResN-exact}, it is convenient to consider the similar functor for equivariant derived categories. First, a general remark: although we have defined $\Perv_H(X)$ as a full subcategory of $\cDb(X)$, there is also the full subcategory of $\cDb_H(X)$ consisting of perverse sheaves (see \cite[\S5.1]{bl}), which we denote $\Perv_H'(X)$. Recall that for connected $H$, the forgetful functor $\For:\cDb_H(X)\to\cDb(X)$ restricts to an equivalence $\Perv_H'(X)\simto\Perv_H(X)$ (see \cite[Theorem A.3(i)]{mv1}).

We denote by $\ResNderiv{G}{L}$ the composition of functors
\begin{equation*}
\xymatrix{
\cDb_G(\cN_G)\ar[r]^-{\For^G_P}&\cDb_P(\cN_G)\ar[r]^-{(m_P)^!}&\cDb_P(\cN_P)\ar[r]^-{(p_P)_*}&\cDb_P(\cN_L)\ar[r]^-{\For^P_L}&\cDb_L(\cN_L).
}
\end{equation*}
Here, $P$ acts on $\cN_L$ via the projection $P\to L$, and the functors are defined as in \S\ref{subsect:equivariant} and \S\ref{sss:forgetting}.
The functor $\ResNderiv{G}{L}$ lifts $\ResN{G}{L}$ in the sense that there is an isomorphism 
\begin{equation}
\label{eqn:ResN-ResNderiv}
\ResN{G}{L}\circ\For\natisom\For\circ\ResNderiv{G}{L}
\end{equation}
obtained from the following pasting diagram:
\[
\vc{\begin{tikzpicture}[smallcube]
\compuinit
% 2-cells
\node[lttricelld] at (0.5+0.6*\tric,0.5+0.5*\tric) {\tiny$\mTr$};
\node[cubef] at (1.5,0.5) {$\mFor$};
\node[cubef] at (2.5,0.5) {$\mFor$};
\node[rttricelld] at (3.5-0.6*\tric,0.5+0.5*\tric) {\tiny$\mTr$};
% 0-cells
\node (llu) at (0,1) {$\cDb_G(\cN_G)$};
\node (lu) at (1,1) {$\cDb_P(\cN_G)$};
\node (mu) at (2,1) {$\cDb_P(\cN_P)$};
\node (ru) at (3,1) {$\cDb_P(\cN_L)$};
\node (rru) at (4,1) {$\cDb_L(\cN_L)$};
\node (ld) at (1,0) {$\cDb(\cN_G)$};
\node (md) at (2,0) {$\cDb(\cN_P)$};
\node (rd) at (3,0) {$\cDb(\cN_L)$};
% 1-cells
\draw[->] (llu) -- node[arl] {{\tiny $\For^G_P$}} (lu);
\draw[->] (lu) -- node[arl] {{\tiny $(m_P)^!$}} (mu);
\draw[->] (mu) -- node[arl] {{\tiny $(p_P)_*$}} (ru);
\draw[->] (ru) -- node[arl] {{\tiny $\For^P_L$}} (rru);
\draw[->] (ld) -- node[arl] {{\tiny $(m_P)^!$}} (md);
\draw[->] (md) -- node[arl] {{\tiny $(p_P)_*$}} (rd);
\draw[->] (llu) -- node[arr] {{\tiny $\For$}} (ld);
\draw[->] (lu) -- node[arl] {{\tiny $\For$}} (ld);
\draw[->] (mu) -- node[arl] {{\tiny $\For$}} (md);
\draw[->] (ru) -- node[arl] {{\tiny $\For$}} (rd);
\draw[->] (rru) -- node[arl] {{\tiny $\For$}} (rd);
\end{tikzpicture}}
\]

The functor $\ResNderiv{G}{L}$ has a left adjoint $\IndNderiv{G}{L}:\cDb_L(\cN_L)\to\cDb_G(\cN_G)$, defined as the following composition:
\begin{equation*}
\xymatrix{
\cDb_G(\cN_G)\ar@{<-}[r]^-{\hamma^G_P}&\cDb_P(\cN_G)\ar@{<-}[r]^-{(m_P)_!}&\cDb_P(\cN_P)\ar@{<-}[r]^-{(p_P)^*}&\cDb_P(\cN_L)\ar@{<-}[r]^-{\hamma^P_L}&\cDb_L(\cN_L).
}
\end{equation*}
Here, $\hamma^H_K$ is the left adjoint of $\For^H_K$ (see \cite[\S3.7.1]{bl} or \S\ref{sss:forgetting}). Note that since $U_P$ is contractible and acts trivially on $\cN_L$, the functor $\hamma^P_L:\cDb_L(\cN_L)\to\cDb_P(\cN_L)$ is an equivalence, with inverse $\For^P_L$ (see \cite[Theorem 3.7.3]{bl}).

\begin{lem}
\label{lem:ind-Springer-exact}

The functor $\IndNderiv{G}{L}$ is right exact for the perverse $t$-structure.

\end{lem}

\begin{proof}
For any $L$-orbit $\mathscr{O} \subset \cN_L$, we denote by $j_{\mathscr{O}} : \mathscr{O} \hookrightarrow \cN_L$ the inclusion. Then, for any $L$-equivariant local system $\mathcal{E}$ on $\mathscr{O}$, we consider the object
\[
\Delta(\mathscr{O},\mathcal{E}) := (j_{\mathscr{O}})_! \mathcal{E} [\dim \mathscr{O}]
\]
of $\cDb_{L}(\cN_L)$. (By a local system, we mean a locally constant sheaf of finitely-generated $\bk$-modules.) 
Then ${}^p \cD_L^{\leq 0}(\cN_L)$ is the smallest full subcategory of $\cDb_L(\cN_L)$ that contains all $\Delta(\mathscr{O},E)[n]$ with $n \ge 0$ and is stable under extensions. Hence to prove the lemma it is sufficient to prove that for all $\mathscr{O}$ and $\mathcal{E}$
\begin{equation}\label{eqn:res-exact}
\IndNderiv{G}{L} \Delta(\mathscr{O},\mathcal{E}) \in {}^p \cD^{\leq 0}_G(\cN_G).
\end{equation}

Let us fix such a pair $(\mathscr{O},\mathcal{E})$. Consider the map
\[
n_{\mathscr{O}} : G \times^P (\mathscr{O} + \fu_P) \to \cN_G
\]
induced by the $G$-action on $\cN_G$, where $\fu_P:=\mathrm{Lie}(U_P)$. For $x \in \cN_G$, an estimate of the dimension of the fibre $n_{\mathscr{O}}^{-1}(x)$ is given in~\cite[Proposition~1.2(b)]{lus:icc}:
\begin{equation}\label{eqn:semismall}
\dim \bigl( n_{\mathscr{O}}^{-1}(x) \bigr) \le \frac{1}{2} \bigl(\dim G - \dim (G \cdot x) - \dim L  + \dim \mathscr{O} \bigr).
\end{equation}

Now, by definition we have $\IndNderiv{G}{L} \Delta(\mathscr{O},\mathcal{E}) \cong \hamma_P^G M_{\mathscr{O},\mathcal{E}}$, where
\[
M_{\mathscr{O},\mathcal{E}} := (j'_{\mathscr{O}})_! \bigl( \mathcal{E} \boxtimes \ubk_{\fu_P} \bigr) [\dim \mathscr{O}]
\]
and $j'_{\mathscr{O}} : \mathscr{O} + \fu_P \hookrightarrow \cN_G$ is the inclusion. Let also $i_{\mathscr{O}} : \mathscr{O} + \fu_P \hookrightarrow G \times^P (\mathscr{O} + \fu_P)$ be the natural inclusion. Then we have
\[
\hamma_P^G M_{\mathscr{O},\mathcal{E}} \cong \hamma_P^G (n_{\mathscr{O}})_! (i_{\mathscr{O}})_! \bigl( \mathcal{E} \boxtimes \ubk \bigr) [\dim \mathscr{O}] \mathrel{\overset{\mInt}{\cong}} (n_{\mathscr{O}})_! \hamma_P^G (i_{\mathscr{O}})_! \bigl( \mathcal{E} \boxtimes \ubk \bigr) [\dim \mathscr{O}]
\]
where $\mInt$ is defined in \S\ref{sss:forgetting}. As explained in \S\ref{subsect:induction-equiv}, the composition $\hamma_P^G (i_{\mathscr{O}})_!:\cDb_P(\mathscr{O} + \fu_P)\to\cDb_G(G \times^P (\mathscr{O} + \fu_P))$ is an equivalence, and is inverse to the functor $(i_{\mathscr{O}})^* \For^G_P[-\dim(G)+\dim(L)]$. Hence $\hamma_P^G (i_{\mathscr{O}})_! \bigl( \mathcal{E} \boxtimes \ubk \bigr) [\dim \mathscr{O}]$ is concentrated in degree $-\dim(\mathscr{O}) - \dim(G) + \dim(L)$. Using \eqref{eqn:semismall}, we deduce that, for any $x \in \cN_G$,
\[
H^i \bigl( (\IndNderiv{G}{L} \Delta(\mathscr{O},\mathcal{E}))|_x \bigr) \cong H^i_c \bigl( n_{\mathscr{O}}^{-1}(x), (\hamma_P^G (i_{\mathscr{O}})_! \bigl( \mathcal{E} \boxtimes \ubk \bigr) [\dim \mathscr{O}])|_{n_{\mathscr{O}}^{-1}(x)} \bigr)
\]
vanishes unless $i \leq -\dim(G \cdot x)$, see \cite[Proposition X.1.4]{iv}, which proves \eqref{eqn:res-exact}.
\end{proof}

\begin{rmk}
The dimension estimate~\eqref{eqn:semismall} amounts to saying that $n_{\mathscr{O}}$ is \emph{semismall}.  That notion is usually applied to proper maps, where it implies that the push-forward of the constant sheaf is (a suitable shift of) a perverse sheaf.  Here, since $n_{\mathscr{O}}$ is not proper, we obtain only a one-sided statement.
\end{rmk}

Let $P^-$ be the parabolic subgroup of $G$ which is opposite to $P$ (i.e.~the $T$-weights of the Lie algebra of $P^-$ are opposite to those of $\mathfrak{p}$). We have a diagram
\[
\xymatrix{
\cN_L & \cN_{P^-} \ar[l]_-{p_{P^-}} \ar[r]^-{m_{P^-}} & \cN_G
}
\]
hence we can consider the functor
\[
{}'\ResN{G}{L} := (p_{P^-})_! \circ (m_{P^-})^* : \cDb(\cN_G) \to \cDb(\cN_L).
\]
As for $\ResN{G}{L}$, this functor has a lift ${}'\ResNderiv{G}{L}$ to equivariant derived categories, which is the composition
\begin{equation*}
\xymatrix{
\cDb_G(\cN_G)\ar[r]^-{\For^G_{P^-}}&\cDb_{P^-}(\cN_G)\ar[r]^-{(m_{P^-})^*} &\cDb_{P^-}(\cN_{P^-})\ar[r]^-{(p_{P^-})_!}&\cDb_{P^-}(\cN_L)\ar[r]^-{\For^{P^-}_L}&\cDb_L(\cN_L).
}
\end{equation*}
The functor ${}'\ResNderiv{G}{L}$ has a right adjoint ${}'\IndNderiv{G}{L}$, defined as the composition
\begin{equation*}
\xymatrix{
\cDb_G(\cN_G)\ar@{<-}[r]^-{\Hamma^G_{P^-}}&\cDb_{P^-}(\cN_G)\ar@{<-}[r]^-{(m_{P^-})_*} &\cDb_{P^-}(\cN_{P^-})\ar@{<-}[r]^-{(p_{P^-})^!}&\cDb_{P^-}(\cN_L)\ar@{<-}[r]^-{\Hamma^{P^-}_L}&\cDb_L(\cN_L).
}
\end{equation*}
Here, $\Hamma_K^H$ is the right adjoint of $\For_K^H$ (see \cite[\S3.7.1]{bl}).

\begin{lem}
\label{lem:ind-Springer-exact-2}

The functor ${}'\IndNderiv{G}{L}$ is left exact for the perverse $t$-structure.

\end{lem}

\begin{proof}
Similar to the proof of Lemma \ref{lem:ind-Springer-exact}, using $\nabla(\mathscr{O},\mathcal{E}) := (j_{\mathscr{O}})_* \mathcal{E}[\dim \mathscr{O}]$
instead of $\Delta(\mathscr{O},\mathcal{E})$. The required vanishing statement is provided by Lemma \ref{lem:vanishing-cohomology-support} below.
\end{proof}

\begin{lem}
\label{lem:vanishing-cohomology-support}

Let $X$ be a smooth variety, and $Y \subset X$ a closed subvariety. Then for any local system $\mathcal{E}$ on $X$ we have $H^i_Y(X,\mathcal{E})=0$
unless $i \geq 2\,\mathrm{codim}_X(Y)$.

\end{lem}

\begin{proof}[Sketch of proof]
If $\mathcal{E}$ is constant this follows from \cite[Theorem X.2.1]{iv}. One deduces the general case using a covering of $X$ which trivializes $\mathcal{E}$, together with the excision exact sequence and isomorphism~\cite[II.9.5 and II.9.6]{iv}.
\end{proof}

\begin{proof}[Proof of Proposition {\rm \ref{prop:ResN-exact}}]
As the left adjoint $\IndNderiv{G}{L}$ of $\ResNderiv{G}{L}$ is right exact (see Lemma \ref{lem:ind-Springer-exact}), $\ResNderiv{G}{L}$ is left exact. As the functor $\For : \Perv_G'(\cN_G) \to \Perv_G(\cN_G)$ is an equivalence, using \eqref{eqn:ResN-ResNderiv} and the definition of the perverse $t$-structure on $\cDb_G(\cN_G)$, it follows that $\ResN{G}{L}$ sends $\Perv_G(\cN_G)$ inside ${}^p \cD^{\geq 0}(\cN_L)$. By the same argument (using Lemma \ref{lem:ind-Springer-exact-2}), the functor ${}' \ResNderiv{G}{L}$ is right exact. As above, it follows that ${}'\ResN{G}{L}$ sends $\Perv_G(\cN_G)$ inside ${}^p \cD^{\leq 0}(\cN_L)$. Finally, by \cite[Theorem 1]{br}, for any $M$ in $\Perv_G(\cN_G)$ we have $\ResN{G}{L}(M) \cong {}' \ResN{G}{L}(M)$,
hence both of these objects are in $\Perv_L(\cN_L)$.
\end{proof}

\begin{cor}
\label{cor:ResNderiv-exact}
The functor $\ResNderiv{G}{L}$ restricts to an exact functor from $\Perv_G'(\cN_G)$ to $\Perv_L'(\cN_L)$.
\end{cor}
\begin{proof}
This follows from Proposition \ref{prop:ResN-exact} and \eqref{eqn:ResN-ResNderiv}.
\end{proof}

Finally we must explain how to construct a transitivity isomorphism
\begin{equation}
\label{eqn:composition-restriction-Springer}
\ResN{G}{T} \natisom \ResN{L}{T} \circ \ResN{G}{L} : \Perv_G(\cN_G,\bk) \to \Perv_T(\cN_T,\bk).
\end{equation}
In fact, using the cartesian square
\begin{equation}
\vc{\begin{tikzpicture}[vsmallcube]
\compuinit
% 2-cells
\node[cart] at (0.5,0.5) {};
% 0-cells
\node (lu) at (0,1) {$\cN_B$};
\node (ru) at (1,1) {$\cN_P$};
\node (ld) at (0,0) {$\cN_C$};
\node (rd) at (1,0) {$\cN_L$};
% 1-cells
\draw[->] (lu) -- (ru);
\draw[->] (lu) -- (ld);
\draw[->] (ru) -- (rd);
\draw[->] (ld) -- (rd);
\end{tikzpicture}}
\end{equation}
(where all morphisms are the natural ones), the pasting diagram
\begin{equation} \label{eqn:transitivity-isom-resN-db-version}
\vc{\begin{tikzpicture}[stdtriangle]
\compuinit
% 2-cells
\node[lttricelld] at (0.5+0.5*\tric, 1.5+0.5*\tric) {\tiny$\Comp$};
\node[lttricelld] at (1.5+0.5*\tric, 0.5+0.5*\tric) {\tiny$\Comp$};
\node[cubef] at (1.5,1.5) {$\BC$};
% 0-cells
\node (lu) at (0,2) {$\cDb(\cN_G)$};
\node (mu) at (1,2) {$\cDb(\cN_P)$};
\node (ru) at (2,2) {$\cDb(\cN_L)$};
\node (mm) at (1,1) {$\cDb(\cN_B)$};
\node (rm) at (2,1) {$\cDb(\cN_C)$};
\node (rd) at (2,0) {$\cDb(\cN_T)$};
% 1-cells
\draw[->] (ru) -- node[arl] {{\tiny $(\cdot)^!$}} (rm);
\draw[->] (lu) -- node[arl] {{\tiny $(\cdot)^!$}} (mu);
\draw[->] (lu) -- node[arr] {{\tiny $(\cdot)^!$}} (mm);
\draw[->] (mu) -- node[arl] {{\tiny $(\cdot)_*$}} (ru);
\draw[->] (mu) -- node[arr] {{\tiny $(\cdot)^!$}} (mm);
\draw[->] (rm) -- node[arl] {{\tiny $(\cdot)_*$}} (rd);
\draw[->] (mm) -- node[arl] {{\tiny $(\cdot)_*$}} (rm);
\draw[->] (mm) -- node[arr] {{\tiny $(\cdot)_*$}} (rd);
\end{tikzpicture}}
\end{equation}
produces the desired isomorphism of functors (by restriction to perverse sheaves).

%%%%%%%%%%%%%%%%%%%%%%%%%%%%%%%%%%%%%%%%%%%%%%%%%%%%%%%%%%%%%%%%%%%%%%%%%%%%%%%%%%%%%%%%

\section{The functors $\Phi_{\Gv}$ and $\Psi_G$ and restriction to a Levi}
\label{sect:phi-psi}

%%%%%%%%%%%%%%%%%%%%%%%%%%%%%%%%%%%%%%%%%%%%%%%%%%%%%%%%%%%%%%%%%%%%%%%%%%%%%%%%%%%%%%%%

%Our aim in this section is to define intertwining isomorphisms $\ResW{\W_G}{\W_L}\circ\Phi_{\Gv}\natisom\Phi_{\Lv}\circ\ResG{\Gv}{\Lv}$ and $\ResN{G}{L}\circ\Psi_G\natisom\Psi_L\circ\ResGr{G}{L}$ that are compatible with the transitivity isomorphisms we have defined, in the sense that the prisms \eqref{eqn:phi-prism} and \eqref{eqn:psi-prism} are commutative.

\subsection{Intertwining isomorphism for the functor $\Phi_{\Gv}$}

Let $V$ be in $\Rep(\Gv,\bk)$. Since $Z(\Lv)\subset\Tv$, the zero weight space of $V$ is the same as the zero weight space of $V^{Z(\Lv)}$. Of course, the sign character of $\W_G$ restricts to that of $\W_L$. Hence we have the following equality, which we declare to be the intertwining isomorphism:
\begin{equation}
\ResW{\W_G}{\W_L}\circ\Phi_{\Gv} = \Phi_{\Lv}\circ\ResG{\Gv}{\Lv}.
\end{equation}
The prism \eqref{eqn:phi-prism} is trivially commutative.

\subsection{Intertwining isomorphism for the functor $\Psi_G$}

We need some preparatory results. In the next lemma, we identify $\Gr_L$ with its image in $\Gr_G$.

\begin{lem}
\label{lem:birkhoff-levi}

We have equalities $\Gr_{0,G}^- \cap \Gr_L = \Gr_{0,L}^-$, $\cM_G \cap \Gr_L = \cM_L$.

\end{lem}

\begin{proof}
The first equality follows from the fact that $\Gr_{0,G}^- = \{x \in \Gr_G \mid \lim_{s \to \infty} s \cdot x = \mathbf{t}_0 \}$,
where the $\Gm$-action considered here is the loop rotation. (This fact follows from the Birkhoff decomposition.) The second equality is a consequence, using the obvious inclusion $\Gr_G^\sm \cap \Gr_L^\circ \subset \Gr_L^\sm$.
\end{proof}

\begin{lem}
\label{lem:diagram-cartesian-L-G}

The following square is cartesian:
\[
\vc{\begin{tikzpicture}[vsmallcube]
\compuinit
% 2-cells
\node[cart] at (0.5,0.5) {};
% 0-cells
\node (lu) at (0,1) {$\cM_L$};
\node (ru) at (1,1) {$\cM_G$};
\node (ld) at (0,0) {$\cN_L$};
\node (rd) at (1,0) {$\cN_G$};
% 1-cells
\draw[right hook->] (lu) -- (ru);
\draw[->] (lu) -- node[arr] {$\al \pi_L$} (ld);
\draw[->] (ru) -- node[arl] {$\al \pi_G$} (rd);
\draw[right hook->] (ld) -- (rd);
\end{tikzpicture}}
\]
\end{lem}

\begin{proof}
Note first that the square commutes by~\cite[Lemma~2.5]{ah}.  Let $Z^\circ(L)$ denote the identity component of the center of $L$, and let $x \in \cN_L$.  Since $x$ is fixed by $Z^\circ(L)$ and $\pi_G^{-1}(x)$ is a finite set, each point $y \in \pi_G^{-1}(x)$ must be fixed by $Z^\circ(L)$ as well.  It is known that the fixed-point set of $Z^\circ(L)$ on $\Gr_G$ is precisely $\Gr_L$.  In view of Lemma~\ref{lem:birkhoff-levi}, we have $y \in \cM_G \cap \Gr_L = \cM_L$.  But then $\pi_L(y) = x$.  In other words, $y \in \pi_L^{-1}(x)$, so $\pi_L^{-1}(x) = \pi_G^{-1}(x)$, as desired.
\end{proof}

Now recall diagrams \eqref{eqn:grlsm-grgsm} 
and 
\eqref{eqn:cnl-cng}.
We need a similar diagram relating $\cM_L$ and $\cM_G$. 
First we define
\[
\cM_P:=(q_P^{\sm})^{-1}(\cM_L)
\]
and denote by $j_P : \cM_P \hookrightarrow \Gr_P^{\sm}$ the inclusion. (Note that $\cM_P$ depends on $L$ and $G$.) 
We have analogous definitions of $\cM_B$ (when $(G,L)$ is replaced by $(G,T)$) and $\cM_C$ (when $(G,L)$ is replaced by $(L,T)$). 

%The following is a generalization of \cite[Proposition 6.9]{ah} (with a similar proof).

\begin{prop}
\label{prop:M_P}
We have $i_P(\cM_P) \subset \cM_G$, and there is a morphism $\pi_P : \cM_P \to \cN_P$ making the following square cartesian:
\[
\vc{\begin{tikzpicture}[vsmallcube]
\compuinit
% 2-cells
\node[cart] at (0.5,0.5) {};
% 0-cells
\node (lu) at (0,1) {$\cM_P$};
\node (ru) at (1,1) {$\cM_G$};
\node (ld) at (0,0) {$\cN_P$};
\node (rd) at (1,0) {$\cN_G$};
% 1-cells
\draw[right hook->] (lu) -- node[arl] {$\al i_P$} (ru);
\draw[->] (lu) -- node[arr] {$\al \pi_P$} (ld);
\draw[->] (ru) -- node[arl] {$\al \pi_G$} (rd);
\draw[right hook->] (ld) -- node[arr] {$\al m_P$} (rd);
\end{tikzpicture}}
\]
\end{prop}

\begin{proof}
By definition, $i_P(\cM_P)$ is contained in
$i_P(q_P^{-1}(\Gr_{0,L}^-)) = L(\fO^-) \cdot U_P(\fK) \cdot \mathbf{t}_0$.
Now we have $U_P(\fK) = U_P(\fO^-) \cdot U_P(\fO)$ since $U_P$ is unipotent, so $U_P(\fK) \cdot \mathbf{t}_0 = U_P(\fO^-) \cdot \mathbf{t}_0$. Therefore 
\begin{equation} \label{eqn:parabolic-preserved}
i_P(\cM_P)\subset P(\fO^-) \cdot \mathbf{t}_0 \subset \Gr_{0,G}^-.
\end{equation} 
Also $i_P(\cM_P)\subset i_P(\Gr^{\sm}_P)\subset\Gr^{\sm}_G$, so $i_P(\cM_P) \subset \cM_G$. Moreover, \eqref{eqn:parabolic-preserved} implies that $\pi_G(i_P(\cM_P))\subset \fp\cap\cN_G=\cN_P$. Let $\pi_P:\cM_P\to\cN_P$ be the restriction of $\pi_G\circ i_P$.

To prove that the square is cartesian, we have to show that if $x \in \cM_G$ and $\pi_G(x) \in \cN_P$, then $x \in i_P(\cM_P)$.  So, consider some $x \in \cM_G$ such that $\pi_G(x) \in \cN_P$. For convenience of notation, we identify $\cM_L$ and $\cM_P$ with their images in $\cM_G$. First, if $\pi_G(x) \in \cN_L$, then by Lemma \ref{lem:diagram-cartesian-L-G} we have $x \in \cM_L$, which proves the result. 

Assume now that $\pi_G(x) \in \cN_P \smallsetminus \cN_L$. Assume also, for a contradiction, that $x \notin \cM_P$. Let $\lambda=(2\rhov_G - 2\rhov_L) \in \bX$, where $\rhov_G$, respectively $\rhov_L$, is the half sum of positive coroots of $G$, respectively of $L$. Consider the point $y := \lim_{s \to 0} \lambda(s) \cdot x$.
As $x \notin \cM_P$, we have $y \notin \cM_L$. As $y \in \Gr_L$, we deduce from the second equality in Lemma \ref{lem:birkhoff-levi} that $y \notin \cM_G$. Similarly, consider $z := \lim_{s \to \infty} \lambda(s) \cdot x$.
If $z \in \cM_L$, then $x \in \cM_{P^-}$ (where $\cM_{P^-}$ is defined in the same way as $\cM_P$, but for the parabolic $P^-$), hence we would have $\pi_G(x) \in \cN_P \cap \cN_{P^-} = \cN_L$, which is not the case by assumption. Hence $z \notin \cM_L$, which implies as above that $z \notin \cM_G$. 

It follows from these considerations that the orbit $\{\lambda(s) \cdot x \mid s \in \C^{\times}\} \subset \cM_G$
is closed in $\cM_G$. As $\pi_G$ is a finite morphism, we deduce that the orbit $\{\lambda(s) \cdot \pi_G(x) \mid s \in \C^{\times}\} \subset \cN_G$.
is closed in $\cN_G$. This is absurd since $\pi_G(x) \in \cN_P \smallsetminus \cN_L$, which finishes the proof.
\end{proof}

Let $i_P^{\cM}:\cM_P\to\cM_G$ and $q_P^{\cM}:\cM_P\to\cM_L$ be the restrictions of $i_P$ and $q_P$ respectively. We now have a diagram of commutative squares
\begin{equation} \label{eqn:squareofsquares}
\vc{\begin{tikzpicture}[vsmallcube]
\compuinit
% 2-cells
\node[cart] at (0.5,0.5) {};
\node[cart] at (1.5,1.5) {};
% 0-cells
\node (lu) at (0,2) {$\Gr_G^{\sm}$};
\node (mu) at (1,2) {$\cM_G$};
\node (ru) at (2,2) {$\cN_G$};
\node (lm) at (0,1) {$\Gr_P^{\sm}$};
\node (mm) at (1,1) {$\cM_P$};
\node (rm) at (2,1) {$\cN_P$};
\node (ld) at (0,0) {$\Gr_L^{\sm}$};
\node (md) at (1,0) {$\cM_L$};
\node (rd) at (2,0) {$\cN_L$};
% 1-cells
\draw[<-right hook] (ru) -- node[arl] {$\al m_P$} (rm);
\draw[<-right hook] (mu) -- node[arr] {$\al i_P^\cM$} (mm);
\draw[<-right hook] (lu) -- node[arr] {$\al i_P^\sm$} (lm);
\draw[->] (rm) -- node[arl] {$\al p_P$} (rd);
\draw[->] (mm) -- node[arl] {$\al q_P^\cM$} (md);
\draw[->] (lm) -- node[arr] {$\al q_P^\sm$} (ld);
\draw[<-left hook] (lu) -- node[arl] {$\al j_G$} (mu);
\draw[->] (mu) -- node[arl] {$\al \pi_G$} (ru);
\draw[<-left hook] (lm) -- node[arl] {$\al j_P$} (mm);
\draw[->] (mm) -- node[arl] {$\al \pi_P$} (rm);
\draw[<-left hook] (ld) -- node[arr] {$\al j_L$} (md);
\draw[->] (md) -- node[arr] {$\al \pi_L$} (rd);
\end{tikzpicture}}
\end{equation}
where the top right square is cartesian by Proposition \ref{prop:M_P} and the bottom left square is cartesian by definition of $\cM_P$.

Recall that the functors $\Psi_G$, $\Psi_L$, $\ResGr{G}{L}$, and $\ResN{G}{L}$ are obtained by restricting
functors that are defined on the level of the derived categories. So to define our intertwining isomorphism, it suffices to define an isomorphism $\ResN{G}{L}\circ\Psi_G\natisom\Psi_L\circ\ResGr{G}{L}$ of functors from $\cDb(\Gr_G^{\sm})$ to $\cDb(\cN_L)$. We define this isomorphism by the following pasting diagram, where the morphisms are those in \eqref{eqn:squareofsquares}:
\begin{equation} \label{eqn:intertwining-isom-psi}
\vc{\begin{tikzpicture}[smallcube]
\compuinit
% 2-cells
\node[cubef] at (0.5,0.5) {$\BC$};
\node[cubef] at (1.5,0.5) {$\Comp$};
\node[cubef] at (0.5,1.5) {$\Comp$};
\node[cubef] at (1.5,1.5) {$\BC$};
% 0-cells
\node (lu) at (0,2) {$\cDb(\Gr_G^{\sm})$};
\node (mu) at (1,2) {$\cDb(\cM_G)$};
\node (ru) at (2,2) {$\cDb(\cN_G)$};
\node (lm) at (0,1) {$\cDb(\Gr_P^{\sm})$};
\node (mm) at (1,1) {$\cDb(\cM_P)$};
\node (rm) at (2,1) {$\cDb(\cN_P)$};
\node (ld) at (0,0) {$\cDb(\Gr_L^{\sm})$};
\node (md) at (1,0) {$\cDb(\cM_L)$};
\node (rd) at (2,0) {$\cDb(\cN_L)$};
% 1-cells
\draw[->] (ru) -- node[arl] {{\tiny $(\cdot)^!$}} (rm);
\draw[->] (mu) -- node[arl] {{\tiny $(\cdot)^!$}} (mm);
\draw[->] (lu) -- node[arr] {{\tiny $(\cdot)^!$}} (lm);
\draw[->] (rm) -- node[arl] {{\tiny $(\cdot)_*$}} (rd);
\draw[->] (mm) -- node[arl] {{\tiny $(\cdot)_*$}} (md);
\draw[->] (lm) -- node[arr] {{\tiny $(\cdot)_*$}} (ld);
\draw[->] (lu) -- node[arl] {{\tiny $(\cdot)^!$}} (mu);
\draw[->] (mu) -- node[arl] {{\tiny $(\cdot)_*$}} (ru);
\draw[->] (lm) -- node[arl] {{\tiny $(\cdot)^!$}} (mm);
\draw[->] (mm) -- node[arl] {{\tiny $(\cdot)_*$}} (rm);
\draw[->] (ld) -- node[arr] {{\tiny $(\cdot)^!$}} (md);
\draw[->] (md) -- node[arr] {{\tiny $(\cdot)_*$}} (rd);
\end{tikzpicture}}
\end{equation}

\subsection{Proof that the prism \eqref{eqn:psi-prism} is commutative}

It suffices to prove the analogous statement with the categories of perverse sheaves replaced by their ambient derived categories.

\begin{prop} \label{prop:psi-prism-db-version}
The following prism is commutative:
\begin{equation*}
\vc{\begin{tikzpicture}[longprism]
\compuinit
% hidden 2-cells
\node[cuber] at (0.5,0.5,0) {$\eqrefh{eqn:intertwining-isom-psi}$};
\node[prismdf] at (1,0.5,\tric) {$\eqrefh{eqn:transitivity-isom-resN-db-version}$};
% outer 0-cells
\node (rlu) at (0,1,0) {$\cDb(\Gr_G^{\sm})$};
\node (rru) at (1,1,0) {$\cDb(\cN_G)$};
\node (fr) at (	1,0.5,1) {$\cDb(\cN_L)$};
\node (rld) at (0,0,0) {$\cDb(\Gr_T^{\sm})$};
\node (rrd) at (1,0,0) {$\cDb(\cN_T)$};
% hidden 1-cells
\draw[liner,->] (rru) -- node[arr,pos=.3] {$\al \ResN{G}{T}$} (rrd);
% outer 1-cells
\draw[->] (rlu) -- node[arl] {$\al \Psi_G$} (rru);
\draw[->] (rlu) -- node[arr] {$\al \ResGr{G}{T}$} (rld);
\draw[->] (rru) -- node[arl] {$\al \ResN{G}{L}$} (fr);
\draw[->] (rld) -- node[arr] {$\al \Psi_T$} (rrd);
\draw[->] (fr) -- node[arl] {$\al \ResN{L}{T}$} (rrd);
% visible 2-cells
\node[prismtf] at (0.5,0.75,0.5) {$\eqrefh{eqn:intertwining-isom-psi}$};
\node[prismbf] at (0.5,0.25,0.5) {$\eqrefh{eqn:intertwining-isom-psi}$};
\node[prismlf] at (0,0.5,\tric) {$\eqrefh{eqn:transitivity-isom-resGr-db-version}$};
% visible 0- and 1-cells
\node (fl) at (0,0.5,1) {$\cDb(\Gr_L^{\sm})$};
\draw[->] (rlu) -- node[arl] {$\al \ResGr{G}{L}$} (fl);
\draw[->] (fl) -- node[arl,pos=.3] {$\al \ResGr{L}{T}$} (rld);
\draw[->] (fl)  -- node[arr,pos=.3] {$\al \Psi_L$} (fr);
\end{tikzpicture}}
\end{equation*}
\end{prop}

\begin{proof}
By Lemmas  \ref{lem:^!composition^!}, \ref{lem:_*composition^!} and \ref{lem:_*^!basechange^!}, the constituent prisms and cube in the following prism are all commutative, so the prism as a whole is commutative by the gluing principle:
\begin{equation} \label{eqn:j-prism}
\vc{\begin{tikzpicture}[longprism2]
\compuinit
% outer 0-cells
\node (rlu) at (0,2,0) {$\cDb(\Gr_G^\sm)$};
\node (rru) at (1,2,0) {$\cDb(\cM_G)$};
\node (rlm) at (0,1,0) {$\cDb(\Gr_B^\sm)$};
\node (rrm) at (1,1,0) {$\cDb(\cM_B)$};
\node (rld) at (0,0,0) {$\cDb(\Gr_T^\sm)$};
\node (rrd) at (1,0,0) {$\cDb(\cM_T)$};
\node (mlu) at (0,1.5,1) {$\cDb(\Gr_P^\sm)$};
\node (mru) at (1,1.5,1) {$\cDb(\cM_P)$};
\node (mld) at (0,0.5,1) {$\cDb(\Gr_C^\sm)$};
\node (mrd) at (1,0.5,1) {$\cDb(\cM_C)$};
\node (fl) at (0,1,2) {$\cDb(\Gr_L^{\sm})$};
\node (fr) at (	1,1,2) {$\cDb(\cM_L)$};
% outer 1-cells
\draw[->] (rlu) -- node[arl] {$\al (\cdot)^!$} (rlm);
\draw[->] (rlm) -- node[arl] {$\al (\cdot)_*$} (rld);
\draw[->] (rru) -- node[arl,pos=.3] {$\al (\cdot)^!$} (rrm);
\draw[->] (rrm) -- node[arl,pos=.4] {$\al (\cdot)_*$} (rrd);

\draw[->] (mru) -- node[arl] {$\al (\cdot)^!$} (rrm);
\draw[->] (fr)  -- node[arl] {$\al (\cdot)^!$} (mrd);
\draw[->] (mrd) -- node[arl] {$\al (\cdot)_*$} (rrd);

\draw[->] (mlu) -- node[arr] {$\al (\cdot)^!$} (rlm);
\draw[->] (fl)  -- node[arr] {$\al (\cdot)^!$} (mld);
\draw[->] (mld) -- node[arr] {$\al (\cdot)_*$} (rld);

\draw[->] (rru) -- node[arl] {$\al (\cdot)^!$} (mru);
\draw[->] (mru) -- node[arl] {$\al (\cdot)_*$} (fr);
\draw[->] (rrm) -- node[arl,pos=.8] {$\al (\cdot)_*$} (mrd);

\draw[->] (rlm) -- node[arl,pos=.3] {$\al (j_B)^!$} (rrm);

\draw[->] (rlu) -- node[arl] {$\al (\cdot)^!$} (mlu);
\draw[->,linef] (mlu) -- node[arl] {$\al (\cdot)_*$} (fl);
\draw[->] (rlm) -- node[arl] {$\al (\cdot)_*$} (mld);

\draw[->] (rlu) -- node[arl] {$\al (j_G)^!$} (rru);
\draw[->] (rld) -- node[arl] {$\al (j_T)^!$} (rrd);
\draw[->,linef] (mlu) -- node[arl] {$\al (j_P)^!$} (mru);
\draw[->,linef] (mld) -- node[arl] {$\al (j_C)^!$} (mrd);
\draw[->,linef] (fl)  -- node[arl,pos=.7] {$\al (j_L)^!$} (fr);
\end{tikzpicture}}
\end{equation}
The only new cartesian squares required to define \eqref{eqn:j-prism} are
\begin{equation}
\label{eqn:cartesian-squares-j-prism}
\vc{\begin{tikzpicture}[vsmallcube]
\compuinit
% 2-cells
\node[cart] at (0.5,0.5) {};
% 0-cells
\node (lu) at (0,1) {$\cM_B$};
\node (ru) at (1,1) {$\cM_C$};
\node (ld) at (0,0) {$\Gr_B^\sm$};
\node (rd) at (1,0) {$\Gr_C^\sm$};
% 1-cells
\draw[->] (lu) -- (ru);
\draw[->] (lu) -- (ld);
\draw[->] (ru) -- (rd);
\draw[->] (ld) -- (rd);
\end{tikzpicture}}
\qquad\text{and}\qquad
\vc{\begin{tikzpicture}[vsmallcube]
\compuinit
% 2-cells
\node[cart] at (0.5,0.5) {};
% 0-cells
\node (lu) at (0,1) {$\cM_B$};
\node (ru) at (1,1) {$\cM_P$};
\node (ld) at (0,0) {$\cM_C$};
\node (rd) at (1,0) {$\cM_L$};
% 1-cells
\draw[->] (lu) -- (ru);
\draw[->] (lu) -- (ld);
\draw[->] (ru) -- (rd);
\draw[->] (ld) -- (rd);
\end{tikzpicture}}
\end{equation}
The first one follows from the cartesian squares giving the definitions of $\cM_B$ and $\cM_C$.
%, namely:
%\[
%\vc{\begin{tikzpicture}[vsmallcube]
%\compuinit
%% 2-cells
%\node[cart] at (0.5,0.5) {};
%% 0-cells
%\node (lu) at (0,1) {$\cM_B$};
%\node (ru) at (1,1) {$\cM_T$};
%\node (ld) at (0,0) {$\Gr_B^\sm$};
%\node (rd) at (1,0) {$\Gr_T^\sm$};
%% 1-cells
%\draw[->] (lu) -- (ru);
%\draw[->] (lu) -- (ld);
%\draw[->] (ru) -- (rd);
%\draw[->] (ld) -- (rd);
%\end{tikzpicture}}
%\qquad\text{and}\qquad
%\vc{\begin{tikzpicture}[vsmallcube]
%\compuinit
%% 2-cells
%\node[cart] at (0.5,0.5) {};
%% 0-cells
%\node (lu) at (0,1) {$\cM_C$};
%\node (ru) at (1,1) {$\cM_T$};
%\node (ld) at (0,0) {$\Gr_C^\sm$};
%\node (rd) at (1,0) {$\Gr_T^\sm$};
%% 1-cells
%\draw[->] (lu) -- (ru);
%\draw[->] (lu) -- (ld);
%\draw[->] (ru) -- (rd);
%\draw[->] (ld) -- (rd);
%\end{tikzpicture}}
%\]
The second cartesian square follows from the one of Lemma \ref{lem:square-Grsm-cartesian}, the first cartesian square in \eqref{eqn:cartesian-squares-j-prism} and the bottom left cartesian square in \eqref{eqn:squareofsquares}.

By Lemmas  \ref{lem:_*composition_*}, \ref{lem:^!composition_*} and \ref{lem:_*^!basechange_*}, the constituent prisms and cube in the following prism are all commutative, so the prism as a whole is commutative:
\begin{equation} \label{eqn:pi-prism}
\vc{\begin{tikzpicture}[longprism2]
\compuinit
% outer 0-cells
\node (rlu) at (0,2,0) {$\cDb(\cM_G)$};
\node (rru) at (1,2,0) {$\cDb(\cN_G)$};
\node (rlm) at (0,1,0) {$\cDb(\cM_B)$};
\node (rrm) at (1,1,0) {$\cDb(\cN_B)$};
\node (rld) at (0,0,0) {$\cDb(\cM_T)$};
\node (rrd) at (1,0,0) {$\cDb(\cN_T)$};
\node (mlu) at (0,1.5,1) {$\cDb(\cM_P)$};
\node (mru) at (1,1.5,1) {$\cDb(\cN_P)$};
\node (mld) at (0,0.5,1) {$\cDb(\cM_C)$};
\node (mrd) at (1,0.5,1) {$\cDb(\cN_C)$};
\node (fl) at (0,1,2) {$\cDb(\cM_L)$};
\node (fr) at (	1,1,2) {$\cDb(\cN_L)$};
% outer 1-cells
\draw[->] (rlu) -- node[arl] {$\al (\cdot)^!$} (rlm);
\draw[->] (rlm) -- node[arl] {$\al (\cdot)_*$} (rld);
\draw[->] (rru) -- node[arl,pos=.3] {$\al (\cdot)^!$} (rrm);
\draw[->] (rrm) -- node[arl,pos=.4] {$\al (\cdot)_*$} (rrd);

\draw[->] (mru) -- node[arl] {$\al (\cdot)^!$} (rrm);
\draw[->] (fr)  -- node[arl] {$\al (\cdot)^!$} (mrd);
\draw[->] (mrd) -- node[arl] {$\al (\cdot)_*$} (rrd);

\draw[->] (mlu) -- node[arr] {$\al (\cdot)^!$} (rlm);
\draw[->] (fl)  -- node[arr] {$\al (\cdot)^!$} (mld);
\draw[->] (mld) -- node[arr] {$\al (\cdot)_*$} (rld);

\draw[->] (rru) -- node[arl] {$\al (\cdot)^!$} (mru);
\draw[->] (mru) -- node[arl] {$\al (\cdot)_*$} (fr);
\draw[->] (rrm) -- node[arl,pos=.8] {$\al (\cdot)_*$} (mrd);

\draw[->] (rlm) -- node[arl,pos=.3] {$\al (\pi_B)_*$} (rrm);

\draw[->] (rlu) -- node[arl] {$\al (\cdot)^!$} (mlu);
\draw[->,linef] (mlu) -- node[arl] {$\al (\cdot)_*$} (fl);
\draw[->] (rlm) -- node[arl] {$\al (\cdot)_*$} (mld);

\draw[->] (rlu) -- node[arl] {$\al (\pi_G)_*$} (rru);
\draw[->] (rld) -- node[arl] {$\al (\pi_T)_*$} (rrd);
\draw[->,linef] (mlu) -- node[arl] {$\al (\pi_P)_*$} (mru);
\draw[->,linef] (mld) -- node[arl] {$\al (\pi_C)_*$} (mrd);
\draw[->,linef] (fl)  -- node[arl,pos=.7] {$\al (\pi_L)_*$} (fr);
\end{tikzpicture}}
\end{equation}
The only new cartesian square required to define \eqref{eqn:pi-prism} is
\begin{equation}
\vc{\begin{tikzpicture}[vsmallcube]
\compuinit
% 2-cells
\node[cart] at (0.5,0.5) {};
% 0-cells
\node (lu) at (0,1) {$\cM_B$};
\node (ru) at (1,1) {$\cN_B$};
\node (ld) at (0,0) {$\cM_P$};
\node (rd) at (1,0) {$\cN_P$};
% 1-cells
\draw[->] (lu) -- (ru);
\draw[->] (lu) -- (ld);
\draw[->] (ru) -- (rd);
\draw[->] (ld) -- (rd);
\end{tikzpicture}}
\end{equation}
which follows from the cartesian square in Proposition~\ref{prop:M_P} and its analogue with $B$ in place of $P$.

We can then glue the prisms \eqref{eqn:j-prism} and \eqref{eqn:pi-prism} together along the face with vertices $\cDb(\cM_G),\cDb(\cM_L),\cDb(\cM_T)$ to obtain the desired commutative prism.
\end{proof}

%%%%%%%%%%%%%%%%%%%%%%%%%%%%%%%%%%%%%%%%%%%%%%%%%%%%%%%%%%%%%%%%%%%%%%%%%%%%%%%%%%%%%%%%

\section{The Satake equivalence and restriction to a Levi}
\label{sect:satake}

%%%%%%%%%%%%%%%%%%%%%%%%%%%%%%%%%%%%%%%%%%%%%%%%%%%%%%%%%%%%%%%%%%%%%%%%%%%%%%%%%%%%%%%%

%Our aim in this section is to define an intertwining isomorphism $\ResG{\Gv}{\Lv}\circ\Satakesm{G}\natisom\Satakesm{L}\circ\ResGr{G}{L}$ making the prism \eqref{eqn:satake-prism} commutative. As with the definition of the transitivity isomorphisms for $\ResG{\Gv}{\Lv}$ and $\ResGr{G}{L}$ in Section~\ref{sect:restriction}, we need to consider first the analogous situation for the connected components $\Gr_G^\circ$, $\Gr_L^\circ$ and the categories $\Rep(\Gv,\bk)^{Z(\Gv)}$, $\Rep(\Lv,\bk)^{Z(\Lv)}$.

\subsection{Intertwining isomorphism for $\Satake{G}^\circ$}

We begin with the compatibility between the transitivity isomorphism for $\ResGrnonsm{G}{L}$, defined in \eqref{eqn:composition-restriction-Satake}, and that for $\ResGrzero{G}{L}$, defined in \eqref{eqn:composition-restriction-Satake-0}.

\begin{lem} \label{lem:connectedcomponent-prism}
The following prism is commutative:
\[
\vc{\begin{tikzpicture}[longprism]
\compuinit
% hidden 2-cells
\node[cuber] at (0.5,0.5,0) {$\eqrefh{eqn:resGr-nonsm-zero}$};
\node[prismdf] at (1,0.5,\tric) {\eqrefh{eqn:composition-restriction-Satake-0}};
% outer 0-cells
\node (rlu) at (0,1,0) {$\Perv_{\GO}(\Gr_G)$};
\node (rru) at (1,1,0) {$\Perv_{\GO}(\Gr_G^\circ)$};
\node (fr) at (	1,0.5,1) {$\Perv_{\LO}(\Gr_L^\circ)$};
\node (rld) at (0,0,0) {$\Perv_{\TO}(\Gr_T)$};
\node (rrd) at (1,0,0) {$\Perv_{\TO}(\Gr_T^\circ)$};
% hidden 1-cells
\draw[liner,->] (rru) -- node[arl,pos=.3] {$\al \ResGrzero{G}{T}$} (rrd);
% outer 1-cells
\draw[->] (rlu) -- node[arl] {$\al (z_G)^!$} (rru);
\draw[->] (rlu) -- node[arr] {$\al \ResGrnonsm{G}{T}$} (rld);
\draw[->] (rru) -- node[arl] {$\al \ResGrzero{G}{L}$} (fr);
\draw[->] (rld) -- node[arr] {$\al (z_T)^!$} (rrd);
\draw[->] (fr) -- node[arl] {$\al \ResGrzero{L}{T}$} (rrd);
% visible 2-cells
\node[prismtf] at (0.5,0.75,0.5) {$\eqrefh{eqn:resGr-nonsm-zero}$};
\node[prismbf] at (0.5,0.25,0.5) {$\eqrefh{eqn:resGr-nonsm-zero}$};
\node[prismlf] at (0,0.5,\tric) {\eqrefh{eqn:composition-restriction-Satake}};
% visible 0- and 1-cells
\node (fl) at (0,0.5,1) {$\Perv_{\LO}(\Gr_L)$};
\draw[->] (rlu) -- node[arl] {$\al \ResGrnonsm{G}{L}$} (fl);
\draw[->] (fl) -- node[arl,pos=.3] {$\al \ResGrnonsm{L}{T}$} (rld);
\draw[->] (fl)  -- node[arr,pos=.3] {$\al (z_L)^!$} (fr);
\end{tikzpicture}}
\]
\end{lem}

\begin{proof}
It suffices to prove the commutativity of the prism:
\begin{equation} \label{eqn:conncompt-prism}
\vc{\begin{tikzpicture}[longprism]
\compuinit
% hidden 2-cells
\node[cuber] at (0.5,0.5,0) {$\eqrefh{eqn:resGr-nonsm-zero-db-version}$};
\node[prismdf] at (1,0.5,\tric) {\eqrefh{eqn:composition-restriction-Satake-zero-db-version}};
% outer 0-cells
\node (rlu) at (0,1,0) {$\cDb(\Gr_G)$};
\node (rru) at (1,1,0) {$\cDb(\Gr_G^\circ)$};
\node (fr) at (	1,0.5,1) {$\cDb(\Gr_L^\circ)$};
\node (rld) at (0,0,0) {$\cDb(\Gr_T)$};
\node (rrd) at (1,0,0) {$\cDb(\Gr_T^\circ)$};
% hidden 1-cells
\draw[liner,->] (rru) -- node[arl,pos=.3] {$\al \ResGrzero{G}{T}$} (rrd);
% outer 1-cells
\draw[->] (rlu) -- node[arl] {$\al (z_G)^!$} (rru);
\draw[->] (rlu) -- node[arr] {$\al \ResGreasy{G}{T}$} (rld);
\draw[->] (rru) -- node[arl] {$\al \ResGrzero{G}{L}$} (fr);
\draw[->] (rld) -- node[arr] {$\al (z_T)^!$} (rrd);
\draw[->] (fr) -- node[arl] {$\al \ResGrzero{L}{T}$} (rrd);
% visible 2-cells
\node[prismtf] at (0.5,0.75,0.5) {$\eqrefh{eqn:resGr-nonsm-zero-db-version}$};
\node[prismbf] at (0.5,0.25,0.5) {$\eqrefh{eqn:resGr-nonsm-zero-db-version}$};
\node[prismlf] at (0,0.5,\tric) {\eqrefh{eqn:composition-restriction-Satake-no-shifts-db-version}};
% visible 0- and 1-cells
\node (fl) at (0,0.5,1) {$\cDb(\Gr_L)$};
\draw[->] (rlu) -- node[arl] {$\al \ResGreasy{G}{L}$} (fl);
\draw[->] (fl) -- node[arl,pos=.3] {$\al \ResGreasy{L}{T}$} (rld);
\draw[->] (fl)  -- node[arr,pos=.3] {$\al (z_L)^!$} (fr);
\end{tikzpicture}}
\end{equation}
But by definition, the prism \eqref{eqn:conncompt-prism} is obtained by gluing together two prisms and a cube that are known to be commutative by Lemmas \ref{lem:^!composition^!}, \ref{lem:_*composition^!},  and \ref{lem:_*^!basechange^!}. The gluing picture is identical to \eqref{eqn:j-prism}, but with $j_H:\cM_H\hookrightarrow\Gr_H^{\sm}$ replaced by $z_H:\Gr_H^\circ\hookrightarrow\Gr_H$ for all groups $H$. The only new cartesian square required here is
\begin{equation}
\vc{\begin{tikzpicture}[vsmallcube]
\compuinit
% 2-cells
\node[cart] at (0.5,0.5) {};
% 0-cells
\node (lu) at (0,1) {$\Gr_B^\circ$};
\node (ru) at (1,1) {$\Gr_C^\circ$};
\node (ld) at (0,0) {$\Gr_B$};
\node (rd) at (1,0) {$\Gr_C$};
% 1-cells
\draw[->] (lu) -- (ru);
\draw[->] (lu) -- (ld);
\draw[->] (ru) -- (rd);
\draw[->] (ld) -- (rd);
\end{tikzpicture}}
\end{equation}
which follows from the $(G,T)$ and $(L,T)$ cases of \eqref{eqn:grzero-cartesian}.
\end{proof}

Recall that we have defined an isomorphism $\ResGnonsm{\Gv}{\Lv}\circ\Satake{G}\natisom\Satake{L}\circ\ResGrnonsm{G}{L}$ in \eqref{eqn:restriction-Satake-non-sm}. To define an analogous isomorphism $\ResGzero{\Gv}{\Lv}\circ\Satake{G}^\circ\natisom\Satake{L}^\circ\circ\ResGrzero{G}{L}$, we use the cube:
\begin{equation} \label{eqn:connectedcomponent-cube}
\vc{\begin{tikzpicture}[longcube]
\compuinit
% hidden 0-cell
\node (rrd) at (1,0,0) {$\Rep(\Lv)$};
% hidden 2-cells
\node[cuber] at (0.5,0.5,0) {$\eqrefh{eqn:restriction-Satake-non-sm}$};
\node[cubed] at (1,0.5,0.5) {$=$};
\node[cubeb] at (0.5,0,0.5) {$\eqrefh{eqn:Satake-z-2}$};
% outer 0-cells
\node (rlu) at (0,1,0) {$\Perv_{\GO}(\Gr_G)$};
\node (rru) at (1,1,0) {$\Rep(\Gv)$};
\node (fru) at (1,1,1) {$\Rep(\Gv)^{Z(\Gv)}$};
\node (frd) at (1,0,1) {$\Rep(\Lv)^{Z(\Lv)}$};
\node (fld) at (0,0,1) {$\Perv_{\LO}(\Gr_L^\circ)$};
\node (rld) at (0,0,0) {$\Perv_{\LO}(\Gr_L)$};
% hidden 1-cells
\draw[liner,->] (rru) -- node[arl,pos=.7] {$\al \ResGnonsm{\Gv}{\Lv}$} (rrd);
\draw[liner,->] (rld) -- node[arl] {$\al \Satake{L}$} (rrd);
\draw[liner,->] (rrd) -- node[arl] {$\al (-)^{Z(\Lv)}$} (frd);
% outer 1-cells
\draw[->] (rlu) -- node[arl] {$\al \Satake{G}$} (rru);
\draw[->] (rru) -- node[arl] {$\al (-)^{Z(\Gv)}$} (fru);
\draw[->] (fru) -- node[arl] {$\al \ResGzero{\Gv}{\Lv}$}(frd);
\draw[->] (fld) -- node[arr] {$\al \Satake{L}^\circ$} (frd);
\draw[->] (rld) -- node[arr] {$\al (z_L)^!$} (fld);
\draw[->] (rlu) -- node[arr] {$\al \ResGrnonsm{G}{L}$} (rld);
% visible 2-cells
\node[cubel] at (0,0.5,0.5) {$\eqrefh{eqn:resGr-nonsm-zero-db-version}$};
\node[cubet] at (0.5,1,0.5) {$\eqrefh{eqn:Satake-z-2}$};
\node[cubef] at (0.5,0.5,1) {?};
% visible 0- and 1-cells
\node (flu) at (0,1,1) {$\Perv_{\GO}(\Gr_G^\circ)$};
\draw[->] (rlu) -- node[arl,pos=.7] {$\al (z_G)^!$} (flu);
\draw[->] (flu) -- node[arr,pos=.3] {$\al \Satake{G}^\circ$} (fru);
\draw[->] (flu) -- node[arl,pos=.3] {$\al \ResGrzero{G}{L}$} (fld);
\end{tikzpicture}}
\end{equation}
Here every face is labelled with an already-defined isomorphism of functors 
except the front face marked with `?'. Since $(z_G)^!:\Perv_{\GO}(\Gr_G)\to\Perv_{\GO}(\Gr_G^\circ)$ is full and essentially surjective, 
there is a unique isomorphism with which to label the front face so as to make the cube commutative (see Example \ref{ex:cube2}).

We now prove that the isomorphism $\ResGzero{\Gv}{\Lv}\circ\Satake{G}^\circ\natisom\Satake{L}^\circ\circ\ResGrzero{G}{L}$ defined by \eqref{eqn:connectedcomponent-cube} is compatible with the relevant transitivity isomorphisms.

\begin{lem} \label{lem:satakezero-prism}
The following prism is commutative:
\[
\vc{\begin{tikzpicture}[longprism]
\compuinit
% hidden 2-cells
\node[cuber] at (0.5,0.5,0) {$\eqrefh{eqn:connectedcomponent-cube}$};
\node[prismdf] at (1,0.5,\tric) {\eqrefh{eqn:composition-restriction-G-zero}};
% outer 0-cells
\node (rlu) at (0,1,0) {$\Perv_{\GO}(\Gr_G^{\circ})$};
\node (rru) at (1,1,0) {$\Rep(\Gv)^{Z(\Gv)}$};
\node (fr) at (	1,0.5,1) {$\Rep(\Lv)^{Z(\Lv)}$};
\node (rld) at (0,0,0) {$\Perv_{\TO}(\Gr_T^{\circ})$};
\node (rrd) at (1,0,0) {$\Rep(\Tv)^{Z(\Tv)}$};
% hidden 1-cells
\draw[liner,->] (rru) -- node[arr,pos=.3] {$\al \ResGzero{\Gv}{\Tv}$} (rrd);
% outer 1-cells
\draw[->] (rlu) -- node[arl] {$\al \Satake{G}^\circ$} (rru);
\draw[->] (rlu) -- node[arr] {$\al \ResGrzero{G}{T}$} (rld);
\draw[->] (rru) -- node[arl] {$\al \ResGzero{\Gv}{\Lv}$} (fr);
\draw[->] (rld) -- node[arr] {$\al \Satake{T}^\circ$} (rrd);
\draw[->] (fr) -- node[arl] {$\al \ResGzero{\Lv}{\Tv}$} (rrd);
% visible 2-cells
\node[prismtf] at (0.5,0.75,0.5) {$\eqrefh{eqn:connectedcomponent-cube}$};
\node[prismbf] at (0.5,0.25,0.5) {$\eqrefh{eqn:connectedcomponent-cube}$};
\node[prismlf] at (0,0.5,\tric) {\eqrefh{eqn:composition-restriction-Satake-0}};
% visible 0- and 1-cells
\node (fl) at (0,0.5,1) {$\Perv_{\LO}(\Gr_L^{\circ})$};
\draw[->] (rlu) -- node[arl] {$\al \ResGrzero{G}{L}$} (fl);
\draw[->] (fl) -- node[arl,pos=.3] {$\al \ResGrzero{L}{T}$} (rld);
\draw[->] (fl)  -- node[arr,pos=.3] {$\al \Satake{L}^\circ$} (fr);
\end{tikzpicture}}
\]
\end{lem}

\begin{proof}
By the essential surjectivity of $(z_G)^!$, it suffices to prove the commutativity of the prism obtained by gluing together those in Lemmas \ref{lem:connectedcomponent-prism} and \ref{lem:satakezero-prism} along their common triangular face (see Example \ref{ex:gluing-converse}). 
But this prism can also be obtained by 
gluing 
the commutative prism in Lemma \ref{lem:restriction-Satake-Levi}, the commutative cube \eqref{eqn:connectedcomponent-cube} in its $(G,L)$, $(L,T)$, and $(G,T)$ versions, and the following prism:
\begin{equation} \label{eqn:centralinvariants-prism}
\vc{\begin{tikzpicture}[longprism2]
\compuinit
% outer 0-cells
\node (rlu) at (0,1,0) {$\Rep(\Gv)$};
\node (rru) at (1,1,0) {$\Rep(\Gv)^{Z(\Gv)}$};
\node (fr) at (	1,0.5,1) {$\Rep(\Lv)^{Z(\Lv)}$};
\node (rld) at (0,0,0) {$\Rep(\Tv)$};
\node (rrd) at (1,0,0) {$\Rep(\Tv)^{Z(\Tv)}$};
% hidden 1-cells
\draw[->] (rru) -- node[arr,pos=.3] {$\al \ResGzero{\Gv}{\Tv}$} (rrd);
% outer 1-cells
\draw[->] (rlu) -- node[arl] {$\al (-)^{Z(\Gv)}$} (rru);
\draw[->] (rlu) -- node[arr] {$\al \ResGnonsm{\Gv}{\Tv}$} (rld);
\draw[->] (rru) -- node[arl] {$\al \ResGzero{\Gv}{\Lv}$} (fr);
\draw[->] (rld) -- node[arr] {$\al (-)^{Z(\Tv)}$} (rrd);
\draw[->] (fr) -- node[arl] {$\al \ResGzero{\Lv}{\Tv}$} (rrd);
% visible 0- and 1-cells
\node (fl) at (0,0.5,1) {$\Rep(\Lv)$};
\draw[->] (rlu) -- node[arl] {$\al \ResGnonsm{\Gv}{\Lv}$} (fl);
\draw[->] (fl) -- node[arl,pos=.3] {$\al \ResGnonsm{\Lv}{\Tv}$} (rld);
\draw[linef,->] (fl)  -- node[arr,pos=.3] {$\al (-)^{Z(\Lv)}$} (fr);
\end{tikzpicture}}
\end{equation}
which is trivially commutative because every face is labelled by an equality.
\end{proof}

\subsection{Intertwining isomorphism for $\Satakesm{G}$}

We now want to pass from the setting of the functor $\Satake{G}^\circ$ to that of the functor $\Satakesm{G}$. Recall the transitivity isomorphism for $\ResGr{G}{L}$, defined via the diagram \eqref{eqn:transitivity-isom-resGr-db-version}, and the isomorphism relating $\ResGr{G}{L}$ and $\ResGrzero{G}{L}$, defined via the diagram \eqref{eqn:resGr-resGrzero}.

\begin{lem} \label{lem:small-prism}
The following prism is commutative:
\[
\vc{\begin{tikzpicture}[longprism]
\compuinit
% hidden 2-cells
\node[cuber] at (0.5,0.5,0) {$\eqrefh{eqn:resGr-resGrzero}$};
\node[prismdf] at (1,0.5,\tric) {\eqrefh{eqn:composition-restriction-Satake-zero-db-version}};
% outer 0-cells
\node (rlu) at (0,1,0) {$\Perv_{\GO}(\Gr_G^\sm)$};
\node (rru) at (1,1,0) {$\Perv_{\GO}(\Gr_G^\circ)$};
\node (fr) at (	1,0.5,1) {$\Perv_{\LO}(\Gr_L^\circ)$};
\node (rld) at (0,0,0) {$\Perv_{\TO}(\Gr_T^\sm)$};
\node (rrd) at (1,0,0) {$\Perv_{\TO}(\Gr_T^\circ)$};
% hidden 1-cells
\draw[liner,->] (rru) -- node[arl,pos=.3] {$\al \ResGrzero{G}{T}$} (rrd);
% outer 1-cells
\draw[->] (rlu) -- node[arl] {$\al (f_G^\circ)_*$} (rru);
\draw[->] (rlu) -- node[arr] {$\al \ResGr{G}{T}$} (rld);
\draw[->] (rru) -- node[arl] {$\al \ResGrzero{G}{L}$} (fr);
\draw[->] (rld) -- node[arr] {$\al (f_T^\circ)_*$} (rrd);
\draw[->] (fr) -- node[arl] {$\al \ResGrzero{L}{T}$} (rrd);
% visible 2-cells
\node[prismtf] at (0.5,0.75,0.5) {$\eqrefh{eqn:resGr-resGrzero}$};
\node[prismbf] at (0.5,0.25,0.5) {$\eqrefh{eqn:resGr-resGrzero}$};
\node[prismlf] at (0,0.5,\tric) {\eqrefh{eqn:transitivity-isom-resGr-db-version}};
% visible 0- and 1-cells
\node (fl) at (0,0.5,1) {$\Perv_{\LO}(\Gr_L^\sm)$};
\draw[->] (rlu) -- node[arl] {$\al \ResGr{G}{L}$} (fl);
\draw[->] (fl) -- node[arl,pos=.3] {$\al \ResGr{L}{T}$} (rld);
\draw[->] (fl)  -- node[arr,pos=.3] {$\al (f_L^\circ)_*$} (fr);
\end{tikzpicture}}
\]
\end{lem}

\begin{proof}
It suffices to prove the commutativity of the prism obtained by replacing $\Perv$ with $\cDb$. By definition, that prism is obtained by gluing together two prisms and a cube that are known to be commutative by Lemmas  \ref{lem:_*composition_*}, \ref{lem:^!composition_*} and \ref{lem:_*^!basechange_*}. The gluing picture is identical to \eqref{eqn:pi-prism}, with $\pi_H:\cM_H\rightarrow\cN_H$ replaced by $f_H^\circ:\Gr_H^\sm\hookrightarrow\Gr_H^\circ$ for all groups $H$. The only new cartesian square required here is
\begin{equation}
\vc{\begin{tikzpicture}[vsmallcube]
\compuinit
% 2-cells
\node[cart] at (0.5,0.5) {};
% 0-cells
\node (lu) at (0,1) {$\Gr_B^\sm$};
\node (ru) at (1,1) {$\Gr_P^\sm$};
\node (ld) at (0,0) {$\Gr_B^\circ$};
\node (rd) at (1,0) {$\Gr_P^\circ$};
% 1-cells
\draw[->] (lu) -- (ru);
\draw[->] (lu) -- (ld);
\draw[->] (ru) -- (rd);
\draw[->] (ld) -- (rd);
\end{tikzpicture}}
\end{equation}
which follows from \eqref{eqn:grsm-grzero-cartesian} and its analogue with $P$ replaced by $B$.
\end{proof}

We come now to the definition of the intertwining isomorphism $\ResG{\Gv}{\Lv}\circ\Satakesm{G}\natisom\Satakesm{L}\circ\ResGr{G}{L}$. Consider the following cube:
\begin{equation} \label{eqn:small-cube}
\vc{\begin{tikzpicture}[longcube]
\compuinit
% hidden 0-cell
\node (rrd) at (1,0,0) {$\Rep(\Lv)^{Z(\Lv)}$};
% hidden 2-cells
\node[cuber] at (0.5,0.5,0) {$\eqrefh{eqn:connectedcomponent-cube}$};
\node[cubed] at (1,0.5,0.5) {$\eqrefh{eqn:ResG-ResGnonsm}$};
\node[cubeb] at (0.5,0,0.5) {$\eqrefh{eqn:Satake-sm-non-sm-0}$};
% outer 0-cells
\node (rlu) at (0,1,0) {$\Perv_{\GO}(\Gr_G^\circ)$};
\node (rru) at (1,1,0) {$\Rep(\Gv)^{Z(\Gv)}$};
\node (fru) at (1,1,1) {$\Rep(\Gv)_{\sm}$};
\node (frd) at (1,0,1) {$\Rep(\Lv)_{\sm}$};
\node (fld) at (0,0,1) {$\Perv_{\LO}(\Gr_L^\sm)$};
\node (rld) at (0,0,0) {$\Perv_{\LO}(\Gr_L^\circ)$};
% hidden 1-cells
\draw[liner,->] (rru) -- node[arl,pos=.7] {$\al \ResGzero{\Gv}{\Lv}$} (rrd);
\draw[liner,->] (rld) -- node[arl] {$\al \Satake{L}^\circ$} (rrd);
\draw[liner,<-] (rrd) -- node[arl] {$\al \mathbb{I}_{\Lv}^\circ$} (frd);
% outer 1-cells
\draw[->] (rlu) -- node[arl] {$\al \Satake{G}^\circ$} (rru);
\draw[<-] (rru) -- node[arl] {$\al \mathbb{I}_{\Gv}^\circ$} (fru);
\draw[->] (fru) -- node[arl] {$\al \ResG{\Gv}{\Lv}$}(frd);
\draw[->] (fld) -- node[arr] {$\al \Satakesm{L}$} (frd);
\draw[<-] (rld) -- node[arr] {$\al (f_L^\circ)_*$} (fld);
\draw[->] (rlu) -- node[arr] {$\al \ResGrzero{G}{L}$} (rld);
% visible 2-cells
\node[cubel] at (0,0.5,0.5) {$\eqrefh{eqn:resGr-resGrzero}$};
\node[cubet] at (0.5,1,0.5) {$\eqrefh{eqn:Satake-sm-non-sm-0}$};
\node[cubef] at (0.5,0.5,1) {?};
% visible 0- and 1-cells
\node (flu) at (0,1,1) {$\Perv_{\GO}(\Gr_G^\sm)$};
\draw[<-] (rlu) -- node[arl,pos=.7] {$\al (f_G^\circ)_*$} (flu);
\draw[->] (flu) -- node[arr,pos=.3] {$\al \Satakesm{G}$} (fru);
\draw[->] (flu) -- node[arl,pos=.3] {$\al \ResGr{G}{L}$} (fld);
\end{tikzpicture}}
\end{equation}
Here every face is labelled with an already-defined isomorphism of functors except the front face marked with `?'. Since $\mathbb{I}_{\Lv}^\circ:\Rep(\Lv)_\sm\to\Rep(\Lv)^{Z(\Lv)}$ is full and faithful, there is a unique isomorphism with which to label the front face so as to make the cube commutative (see Example \ref{ex:cube2}). 

\subsection{Proof that the prism \eqref{eqn:satake-prism} is commutative}

Consider the following prism, which is trivially commutative because every face is labelled by an equality:
\begin{equation} \label{eqn:embedding-prism}
\vc{\begin{tikzpicture}[longprism2]
\compuinit
% outer 0-cells
\node (rlu) at (0,1,0) {$\Rep(\Gv)_\sm$};
\node (rru) at (1,1,0) {$\Rep(\Gv)^{Z(\Gv)}$};
\node (fr) at (	1,0.5,1) {$\Rep(\Lv)^{Z(\Lv)}$};
\node (rld) at (0,0,0) {$\Rep(\Tv)_\sm$};
\node (rrd) at (1,0,0) {$\Rep(\Tv)^{Z(\Tv)}$};
% hidden 1-cells
\draw[->] (rru) -- node[arr,pos=.3] {$\al \ResGzero{\Gv}{\Tv}$} (rrd);
% outer 1-cells
\draw[->] (rlu) -- node[arl] {$\al \mathbb{I}_{\Gv}^0$} (rru);
\draw[->] (rlu) -- node[arr] {$\al \ResG{\Gv}{\Tv}$} (rld);
\draw[->] (rru) -- node[arl] {$\al \ResGzero{\Gv}{\Lv}$} (fr);
\draw[->] (rld) -- node[arr] {$\al \mathbb{I}_{\Tv}^0$} (rrd);
\draw[->] (fr) -- node[arl] {$\al \ResGzero{\Lv}{\Tv}$} (rrd);
% visible 0- and 1-cells
\node (fl) at (0,0.5,1) {$\Rep(\Lv)_\sm$};
\draw[->] (rlu) -- node[arl] {$\al \ResG{\Gv}{\Lv}$} (fl);
\draw[->] (fl) -- node[arl,pos=.3] {$\al \ResG{\Lv}{\Tv}$} (rld);
\draw[linef,->] (fl)  -- node[arr,pos=.3] {$\al \mathbb{I}_{\Lv}^0$} (fr);
\end{tikzpicture}}
\end{equation}
Since $\mathbb{I}_{\Tv}^0:\Rep(\Tv)_\sm\to\Rep(\Tv)^{Z(\Tv)}$ is faithful (indeed, an equivalence), to prove that \eqref{eqn:satake-prism} is commutative it suffices, by Example \ref{ex:gluing-converse}, to prove the following result.

\begin{prop} \label{prop:satake-prism-modified}
The prism
\[
\vc{\begin{tikzpicture}[longprism2]
\compuinit
% outer 0-cells
\node (rlu) at (0,1,0) {$\Perv_{\GO}(\Gr_G^\sm)$};
\node (rru) at (1,1,0) {$\Rep(\Gv)^{Z(\Gv)}$};
\node (fr) at (	1,0.5,1) {$\Rep(\Lv)^{Z(\Lv)}$};
\node (rld) at (0,0,0) {$\Perv_{\TO}(\Gr_T^\sm)$};
\node (rrd) at (1,0,0) {$\Rep(\Tv)^{Z(\Tv)}$};
% hidden 1-cells
\draw[->] (rru) -- node[arr,pos=.3] {$\al \ResGzero{\Gv}{\Tv}$} (rrd);
% outer 1-cells
\draw[->] (rlu) -- node[arl] {$\al \mathbb{I}_{\Gv}^0\circ\Satakesm{G}$} (rru);
\draw[->] (rlu) -- node[arr] {$\al \ResGr{G}{T}$} (rld);
\draw[->] (rru) -- node[arl] {$\al \ResGzero{\Gv}{\Lv}$} (fr);
\draw[->] (rld) -- node[arr] {$\al \mathbb{I}_{\Tv}^0\circ\Satakesm{T}$} (rrd);
\draw[->] (fr) -- node[arl] {$\al \ResGzero{\Lv}{\Tv}$} (rrd);
% visible 0- and 1-cells
\node (fl) at (0,0.5,1) {$\Perv_{\LO}(\Gr_L^\sm)$};
\draw[->] (rlu) -- node[arl] {$\al \ResGr{G}{L}$} (fl);
\draw[->] (fl) -- node[arl,pos=.3] {$\al \ResGr{L}{T}$} (rld);
\draw[linef,->] (fl)  -- node[arr,pos=.3] {$\al \mathbb{I}_{\Lv}^0\circ\Satakesm{L}$} (fr);
\end{tikzpicture}}
\]
obtained by gluing \eqref{eqn:satake-prism} and \eqref{eqn:embedding-prism} is commutative.
\end{prop}

\begin{proof}
This prism can be obtained by an alternative gluing procedure, in which the pieces to be glued are the commutative prisms in Lemmas \ref{lem:satakezero-prism} and \ref{lem:small-prism} and the commutative cube \eqref{eqn:small-cube} in its $(G,L)$, $(L,T)$, and $(G,T)$ versions.
\end{proof}

%%%%%%%%%%%%%%%%%%%%%%%%%%%%%%%%%%%%%%%%%%%%%%%%%%%%%%%%%%%%%%%%%%%%%%%%%%%%%%%%%%%%%%%%

\section{The Springer functor and restriction to a Levi}
\label{sect:springer}

%%%%%%%%%%%%%%%%%%%%%%%%%%%%%%%%%%%%%%%%%%%%%%%%%%%%%%%%%%%%%%%%%%%%%%%%%%%%%%%%%%%%%%%%

%In this section our aim is to define an intertwining isomorphism $\ResW{\W_G}{\W_L}\circ\Springer{G}\natisom\Springer{L}\circ\ResN{G}{L}$ that makes the prism \eqref{eqn:springer-prism} commutative, with the transitivity isomorphisms for $\ResW{\W_G}{\W_L}$ and $\ResN{G}{L}$ defined as in Section \ref{sect:restriction}.

%---------------------------------------------------------------------
\subsection{Restriction for equivariant derived categories}
\label{ss:restriction-equivariant}
%---------------------------------------------------------------------

Our first step is to pass from categories of equivariant perverse sheaves to equivariant derived categories. Recall the functor $\ResNderiv{G}{L} : \cDb_G(\cN_G) \to \cDb_L(\cN_L)$ defined in \S\ref{ss:restriction-N}. There is a transitivity isomorphism for this functor, namely an isomorphism
\begin{equation}
\label{eqn:composition-restriction-Springerderiv}
\ResNderiv{G}{T} \natisom \ResNderiv{L}{T} \circ \ResNderiv{G}{L} : \cDb_G(\cN_G) \to \cDb_T(\cN_T),
\end{equation}
defined by the following elaboration of \eqref{eqn:transitivity-isom-resN-db-version}:
\begin{equation}
\label{eqn:transitivity-elaborate}
\vc{\begin{tikzpicture}[stdtriangle]
\compuinit
% 2-cells
\node[lttricelld] at (0.5+0.5*\tric, 3.5+0.5*\tric) {\tiny$\mTr$};
\node[lttricelld] at (1.5+0.5*\tric, 2.5+0.5*\tric) {\tiny$\Comp$};
\node[lttricelld] at (2.5+0.5*\tric, 1.5+0.5*\tric) {\tiny$\Comp$};
\node[lttricelld] at (3.5+0.5*\tric, 0.5+0.5*\tric) {\tiny$\mTr$};
\node[lttricelld] at (3.5+0.5*\tric, 3.5+0.5*\tric) {\tiny$\mTr$};
\node[rttricellu] at (3.5-0.5*\tric, 3.5-0.5*\tric) {\tiny$\mTr$};
\node[cubef] at (1.5,3.5) {$\mFor$};
\node[cubef] at (2.5,3.5) {$\mFor$};
\node[cubef] at (2.5,2.5) {$\BC$};
\node[cubef] at (3.5,2.5) {$\mFor$};
\node[cubef] at (3.5,1.5) {$\mFor$};
% 0-cells
\node (lluu) at (0,4) {$\cDb_G(\cN_G)$};
\node (luu) at (1,4) {$\cDb_P(\cN_G)$};
\node (muu) at (2,4) {$\cDb_P(\cN_P)$};
\node (ruu) at (3,4) {$\cDb_P(\cN_L)$};
\node (rruu) at (4,4) {$\cDb_L(\cN_L)$};
\node (lu) at (1,3) {$\cDb_B(\cN_G)$};
\node (mu) at (2,3) {$\cDb_B(\cN_P)$};
\node (ru) at (3,3) {$\cDb_B(\cN_L)$};
\node (rru) at (4,3) {$\cDb_C(\cN_L)$};
\node (mm) at (2,2) {$\cDb_B(\cN_B)$};
\node (rm) at (3,2) {$\cDb_B(\cN_C)$};
\node (rrm) at (4,2) {$\cDb_C(\cN_C)$};
\node (rd) at (3,1) {$\cDb_B(\cN_T)$};
\node (rrd) at (4,1) {$\cDb_C(\cN_T)$};
\node (rrdd) at (4,0) {$\cDb_T(\cN_T)$};
% 1-cells
\draw[->] (lluu) -- node[arl] {{\tiny $\For^G_P$}} (luu);
\draw[->] (luu) -- node[arl] {{\tiny $(\cdot)^!$}} (muu);
\draw[->] (muu) -- node[arl] {{\tiny $(\cdot)_*$}} (ruu);
\draw[->] (ruu) -- node[arl] {{\tiny $\For^P_L$}} (rruu);
\draw[->] (lu) -- node[arl] {{\tiny $(\cdot)^!$}} (mu);
\draw[->] (mu) -- node[arl] {{\tiny $(\cdot)_*$}} (ru);
\draw[->] (ru) -- node[arl] {{\tiny $\For^B_C$}} (rru);
\draw[->] (mm) -- node[arl] {{\tiny $(\cdot)_*$}} (rm);
\draw[->] (rm) -- node[arl] {{\tiny $\For^B_C$}} (rrm);
\draw[->] (rd) -- node[arl] {{\tiny $\For^B_C$}} (rrd);
\draw[->] (luu) -- node[arl] {{\tiny $\For^P_B$}} (lu);
\draw[->] (muu) -- node[arl] {{\tiny $\For^P_B$}} (mu);
\draw[->] (ruu) -- node[arl] {{\tiny $\For^P_B$}} (ru);
\draw[->] (rruu) -- node[arl] {{\tiny $\For^L_C$}} (rru);
\draw[->] (mu) -- node[arl] {{\tiny $(\cdot)^!$}} (mm);
\draw[->] (ru) -- node[arl] {{\tiny $(\cdot)^!$}} (rm);
\draw[->] (rru) -- node[arl] {{\tiny $(\cdot)^!$}} (rrm);
\draw[->] (rm) -- node[arl] {{\tiny $(\cdot)_*$}} (rd);
\draw[->] (rrm) -- node[arl] {{\tiny $(\cdot)_*$}} (rrd);
\draw[->] (rrd) -- node[arl] {{\tiny $\For^C_T$}} (rrdd);
\draw[->] (lluu) -- node[arr] {{\tiny $\For^G_B$}} (lu);
\draw[->] (lu) -- node[arr] {{\tiny $(\cdot)^!$}} (mm);
\draw[->] (mm) -- node[arr] {{\tiny $(\cdot)_*$}} (rd);
\draw[->] (rd) -- node[arr] {{\tiny $\For^B_T$}} (rrdd);
\draw[->] (ruu) -- node[arl] {{\tiny $\For^P_C$}} (rru);
\end{tikzpicture}}
\end{equation}
Recall also that we have defined an isomorphism $\ResN{G}{L}\circ\For\natisom\For\circ\ResNderiv{G}{L}$ in \eqref{eqn:ResN-ResNderiv}.

\begin{lem} \label{lem:forgetfulness}
Isomorphism \eqref{eqn:ResN-ResNderiv} is compatible with transitivity in the sense that the following prism is commutative:
\[
\vc{\begin{tikzpicture}[longprism]
\compuinit
% hidden 2-cells
\node[cuber] at (0.5,0.5,0) {\eqrefh{eqn:ResN-ResNderiv}};
\node[prismdf] at (1,0.5,\tric) {\eqrefh{eqn:composition-restriction-Springer}};
% outer 0-cells
\node (rlu) at (0,1,0) {$\cDb_G(\cN_G)$};
\node (rru) at (1,1,0) {$\cDb(\cN_G)$};
\node (fr) at (1,0.5,1) {$\cDb(\cN_L)$};
\node (rld) at (0,0,0) {$\cDb_T(\cN_T)$};
\node (rrd) at (1,0,0) {$\cDb(\cN_T)$};
% hidden 1-cells
\draw[liner,->] (rru) -- node[arr,pos=.3] {$\al \ResN{G}{T}$} (rrd);
% outer 1-cells
\draw[->] (rlu) -- node[arl] {$\al \For$} (rru);
\draw[->] (rlu) -- node[arr] {$\al \ResNderiv{G}{T}$} (rld);
\draw[->] (rru) -- node[arl] {$\al \ResN{G}{L}$} (fr);
\draw[->] (rld) -- node[arr] {$\al \For$} (rrd);
\draw[->] (fr) -- node[arl] {$\al \ResN{L}{T}$} (rrd);
% visible 2-cells
\node[prismtf] at (0.5,0.75,0.5) {\eqrefh{eqn:ResN-ResNderiv}};
\node[prismbf] at (0.5,0.25,0.5) {\eqrefh{eqn:ResN-ResNderiv}};
\node[prismlf] at (0,0.5,\tric) {\eqrefh{eqn:composition-restriction-Springerderiv}};
% visible 0- and 1-cells
\node (fl) at (0,0.5,1) {$\cDb_L(\cN_L)$};
\draw[->] (rlu) -- node[arl] {$\al \ResNderiv{G}{L}$} (fl);
\draw[->] (fl) -- node[arl,pos=.3] {$\al \ResNderiv{L}{T}$} (rld);
\draw[->] (fl)  -- node[arr,pos=.3] {$\al \For$} (fr);
\end{tikzpicture}}
\]
\end{lem}

\begin{proof}
This prism is obtained by gluing together cubes and prisms whose left faces are the squares and triangles in \eqref{eqn:transitivity-elaborate}, and whose right faces are the non-equivariant analogues. These are commutative by Lemmas \ref{lem:cocycle-For}, \ref{lem:for_*for}, \ref{lem:for^!for}, \ref{lem:_*compositionfor}, \ref{lem:^!compositionfor} and \ref{lem:basechangefor}. (Recall that $\For:\cDb_H(X)\to\cDb(X)$ is the same as $\For^H_{\{1\}}$.)
\end{proof}

Now consider the diagram
\begin{equation} \label{eqn:fiddle}
\vcenter{
\xymatrix@R=14pt{
\Perv_G'(\cN_G)\ar[d]_{\ResNderiv{G}{L}}\ar[r]^-{\For}_-{\sim}&\Perv_G(\cN_G)\ar[d]_{\ResN{G}{L}}\ar[r]^-{\Springer{G}}&\Rep(\W_G)\ar[d]^{\ResW{\W_G}{\W_L}}\\
\Perv_L'(\cN_L)\ar[r]^-{\For}_-{\sim}&\Perv_L(\cN_L)\ar[r]^-{\Springer{L}}&\Rep(\W_L)
}
}
\end{equation}
(where the left-hand square is well defined by Proposition \ref{prop:ResN-exact} and Corollary \ref{cor:ResNderiv-exact}). Our goal in this section is to define an isomorphism for the right-hand square in \eqref{eqn:fiddle} and show that it is compatible with transitivity. We already have such an isomorphism for the left-hand square, and it is compatible with transitivity by Lemma \ref{lem:forgetfulness}. Since $\For:\Perv_G'(\cN_G)\to\Perv_G(\cN_G)$ is an equivalence, it suffices to define an isomorphism for the outer square in \eqref{eqn:fiddle}, 
and show that it is compatible with transitivity (see Example \ref{ex:gluing-converse}). Now $\Springer{G}\circ\For:\Perv_G'(\cN_G)\to\Rep(\W_G)$ is clearly isomorphic to the functor $\Springer{G}':\Perv_G'(\cN_G)\to\Rep(\W_G)$ defined on objects by $M\mapsto \Hom_{\Perv_G'(\cN_G)}(\Spr_G,M)$.
The following observation, which is immediate from Example \ref{ex:prism}, allows us to consider $\Springer{G}'$ instead of $\Springer{G}\circ\For$.

\begin{lem} \label{lem:whisker}
Suppose we have an isomorphism
\begin{equation} \label{eqn:whisker1}
\ResW{\W_G}{\W_L}\circ\Springer{G}'\natisom\Springer{L}'\circ\ResNderiv{G}{L}
\end{equation}
that is compatible with transitivity in the sense that the following prism:
\[
\vc{\begin{tikzpicture}[longprism]
\compuinit
% hidden 2-cells
\node[cuber] at (0.5,0.5,0) {\eqrefh{eqn:whisker1}};
\node[prismdf] at (1,0.5,\tric) {$\mTr$};
% outer 0-cells
\node (rlu) at (0,1,0) {$\Perv_G'(\cN_G,\bk)$};
\node (rru) at (1,1,0) {$\Rep(\W_G,\bk)$};
\node (fr) at (	1,0.5,1) {$\Rep(\W_L,\bk)$};
\node (rld) at (0,0,0) {$\Perv_T'(\cN_T,\bk)$};
\node (rrd) at (1,0,0) {$\Rep(\W_T,\bk)$};
% hidden 1-cells
\draw[liner,->] (rru) -- node[arr,pos=.3] {$\al \ResW{\W_G}{\W_T}$} (rrd);
% outer 1-cells
\draw[->] (rlu) -- node[arl] {$\al \Springer{G}'$} (rru);
\draw[->] (rlu) -- node[arr] {$\al \ResNderiv{G}{T}$} (rld);
\draw[->] (rru) -- node[arl] {$\al \ResW{\W_G}{\W_L}$} (fr);
\draw[->] (rld) -- node[arr] {$\al \Springer{T}'$} (rrd);
\draw[->] (fr) -- node[arl] {$\al \ResW{\W_L}{\W_T}$} (rrd);
% visible 2-cells
\node[prismtf] at (0.5,0.75,0.5) {\eqrefh{eqn:whisker1}};
\node[prismbf] at (0.5,0.25,0.5) {\eqrefh{eqn:whisker1}};
\node[prismlf] at (0,0.5,\tric) {\eqrefh{eqn:composition-restriction-Springerderiv}};
% visible 0- and 1-cells
\node (fl) at (0,0.5,1) {$\Perv_L'(\cN_L,\bk)$};
\draw[->] (rlu) -- node[arl] {$\al \ResNderiv{G}{L}$} (fl);
\draw[->] (fl) -- node[arl,pos=.3] {$\al \ResNderiv{L}{T}$} (rld);
\draw[->] (fl)  -- node[arr,pos=.3] {$\al \Springer{L}'$} (fr);
\end{tikzpicture}}
\]
is commutative.
Then the isomorphism 
defined as the composition
\begin{equation}
\ResW{\W_G}{\W_L}\circ\Springer{G}\circ\For \natisom \ResW{\W_G}{\W_L}\circ\Springer{G}' \overset{\eqref{eqn:whisker1}}{\natisom} \Springer{L}'\circ\ResNderiv{G}{L} \natisom \Springer{L}\circ\For\circ\ResNderiv{G}{L}
\end{equation}
is also compatible with transitivity.
\end{lem}

The functor $\Springer{G}'$ extends to a functor $\cDb_G(\cN_G)\to\Rep(\W_G)$ defined on objects by $M\mapsto \Hom_{\cDb_G(\cN_G)}(\Spr_G,M)$. We will denote the latter functor by $\Springer{G}'$ also. Our conclusion is that it suffices to define an intertwining isomorphism 
\begin{equation}
\label{eqn:springer'-intertwine}
\vc{\begin{tikzpicture}[smallcube2]
\compuinit
% 2-cells
\node[cubef] at (0.5,0.5) {$?$};
% 0-cells
\node (lu) at (0,1) {$\cDb_G(\cN_G)$};
\node (ld) at (0,0) {$\cDb_L(\cN_L)$};
\node (ru) at (1,1) {$\Rep(\W_G)$};
\node (rd) at (1,0) {$\Rep(\W_L)$};
% 1-cells
\draw[->] (lu) -- node[arl] {$\al \Springer{G}'$} (ru);
\draw[->] (lu) -- node[arr] {$\al \ResNderiv{G}{L}$} (ld);
\draw[->] (ld) -- node[arr] {$\al \Springer{L}'$} (rd);
\draw[->] (ru) -- node[arl] {$\al \ResW{\W_G}{\W_L}$} (rd);
\end{tikzpicture}}
\end{equation}
and show that it is compatible with transitivity.

%---------------------------------------------------------------------
\subsection{Induction}
%---------------------------------------------------------------------

%Now, recall the left adjoint $\IndNderiv{G}{L}$ of $\ResNderiv{G}{L}$ defined in \S\ref{ss:restriction-N}. 
There is a transitivity isomorphism 
\begin{equation}
\label{eqn:composition-induction-Springerderiv}
\IndNderiv{G}{T} \natisom \IndNderiv{L}{T} \circ \IndNderiv{G}{L} : \cDb_T(\cN_T) \to \cDb_G(\cN_G)
\end{equation}
(where $\IndNderiv{G}{L}$ is defined in \S\ref{ss:restriction-N}) defined by the following pasting diagram:
\begin{equation}
\label{eqn:induction-transitivity}
\vc{\begin{tikzpicture}[stdtriangle]
\compuinit
% 2-cells
\node[lttricelld] at (0.5+0.5*\tric, 3.5+0.5*\tric) {\tiny$\mTr$};
\node[lttricelld] at (1.5+0.5*\tric, 2.5+0.5*\tric) {\tiny$\Comp$};
\node[lttricelld] at (2.5+0.5*\tric, 1.5+0.5*\tric) {\tiny$\Comp$};
\node[lttricelld] at (3.5+0.5*\tric, 0.5+0.5*\tric) {\tiny$\mTr$};
\node[lttricelld] at (3.5+0.5*\tric, 3.5+0.5*\tric) {\tiny$\mTr$};
\node[rttricellu] at (3.5-0.5*\tric, 3.5-0.5*\tric) {\tiny$\mTr$};
\node[cubef] at (1.5,3.5) {$\mInt$};
\node[cubef] at (2.5,3.5) {$\mInt$};
\node[cubef] at (2.5,2.5) {$\BC$};
\node[cubef] at (3.5,2.5) {$\mInt$};
\node[cubef] at (3.5,1.5) {$\mInt$};
% 0-cells
\node (lluu) at (0,4) {$\cDb_G(\cN_G)$};
\node (luu) at (1,4) {$\cDb_P(\cN_G)$};
\node (muu) at (2,4) {$\cDb_P(\cN_P)$};
\node (ruu) at (3,4) {$\cDb_P(\cN_L)$};
\node (rruu) at (4,4) {$\cDb_L(\cN_L)$};
\node (lu) at (1,3) {$\cDb_B(\cN_G)$};
\node (mu) at (2,3) {$\cDb_B(\cN_P)$};
\node (ru) at (3,3) {$\cDb_B(\cN_L)$};
\node (rru) at (4,3) {$\cDb_C(\cN_L)$};
\node (mm) at (2,2) {$\cDb_B(\cN_B)$};
\node (rm) at (3,2) {$\cDb_B(\cN_C)$};
\node (rrm) at (4,2) {$\cDb_C(\cN_C)$};
\node (rd) at (3,1) {$\cDb_B(\cN_T)$};
\node (rrd) at (4,1) {$\cDb_C(\cN_T)$};
\node (rrdd) at (4,0) {$\cDb_T(\cN_T)$};
% 1-cells
\draw[<-] (lluu) -- node[arl] {{\tiny $\hamma^G_P$}} (luu);
\draw[<-] (luu) -- node[arl] {{\tiny $(\cdot)_!$}} (muu);
\draw[<-] (muu) -- node[arl] {{\tiny $(\cdot)^*$}} (ruu);
\draw[<-] (ruu) -- node[arl] {{\tiny $\hamma^P_L$}} (rruu);
\draw[<-] (lu) -- node[arl] {{\tiny $(\cdot)_!$}} (mu);
\draw[<-] (mu) -- node[arl] {{\tiny $(\cdot)^*$}} (ru);
\draw[<-] (ru) -- node[arl] {{\tiny $\hamma^B_C$}} (rru);
\draw[<-] (mm) -- node[arl] {{\tiny $(\cdot)^*$}} (rm);
\draw[<-] (rm) -- node[arl] {{\tiny $\hamma^B_C$}} (rrm);
\draw[<-] (rd) -- node[arl] {{\tiny $\hamma^B_C$}} (rrd);
\draw[<-] (luu) -- node[arl] {{\tiny $\hamma^P_B$}} (lu);
\draw[<-] (muu) -- node[arl] {{\tiny $\hamma^P_B$}} (mu);
\draw[<-] (ruu) -- node[arl] {{\tiny $\hamma^P_B$}} (ru);
\draw[<-] (rruu) -- node[arl] {{\tiny $\hamma^L_C$}} (rru);
\draw[<-] (mu) -- node[arl] {{\tiny $(\cdot)_!$}} (mm);
\draw[<-] (ru) -- node[arl] {{\tiny $(\cdot)_!$}} (rm);
\draw[<-] (rru) -- node[arl] {{\tiny $(\cdot)_!$}} (rrm);
\draw[<-] (rm) -- node[arl] {{\tiny $(\cdot)^*$}} (rd);
\draw[<-] (rrm) -- node[arl] {{\tiny $(\cdot)^*$}} (rrd);
\draw[<-] (rrd) -- node[arl] {{\tiny $\hamma^C_T$}} (rrdd);
\draw[<-] (lluu) -- node[arr] {{\tiny $\hamma^G_B$}} (lu);
\draw[<-] (lu) -- node[arr] {{\tiny $(\cdot)_!$}} (mm);
\draw[<-] (mm) -- node[arr] {{\tiny $(\cdot)^*$}} (rd);
\draw[<-] (rd) -- node[arr] {{\tiny $\hamma^B_T$}} (rrdd);
\draw[<-] (ruu) -- node[arl] {{\tiny $\hamma^P_C$}} (rru);
\end{tikzpicture}}
\end{equation}

We can express the functor $\Springer{G}'$ as the following composition:
\begin{equation*}
\cDb_G(\cN_G) \xrightarrow{\Yon} \Vect(\bk)^{\cDb_G(\cN_G)^\op} \xrightarrow{-(\Spr_G)} \Rep(\W_G)
\end{equation*}
where $\Yon$ is the Yoneda embedding for $\cDb_G(\cN_G)$ (see \S\ref{sss:adjunction}) and $-(\Spr_G)$ is the evaluation on the object $\Spr_G$ of $\cDb_G(\cN_G)$ (on which $\W_G$ acts). Consider the diagram:
\begin{equation}\label{eqn:two-isoms}
\vcenter{
\xymatrix@C=2cm@R=15pt{
\cDb_G(\cN_G)\ar[d]_{\ResNderiv{G}{L}}\ar[r]^-{\Yon}&\Vect(\bk)^{\cDb_G(\cN_G)^\op}\ar[d]^{-\circ(\IndNderiv{G}{L})^\op}\ar[r]^-{-(\Spr_G)}&\Rep(\W_G)\ar[d]^{\ResW{\W_G}{\W_L}}\\
\cDb_L(\cN_L)\ar[r]^-{\Yon}&\Vect(\bk)^{\cDb_L(\cN_L)^\op}\ar[r]^-{-(\Spr_L)}&\Rep(\W_L)
}
}
\end{equation}
Note that $-\circ(\IndNderiv{G}{L})^\op$ has its own transitivity isomorphism, defined by the pasting diagram obtained from \eqref{eqn:induction-transitivity} by replacing every $\mathsf{C}$ with $\Vect(\bk)^{\mathsf{C}^\op}$ and every $\alpha$ with $-\circ\alpha^\op$, reversing all arrows. We will refer to this isomorphism as \eqref{eqn:composition-induction-Springerderiv}.

We have an isomorphism for the left-hand square in \eqref{eqn:two-isoms}, namely the following composition of adjunction isomorphisms (where we write $\sV$ for $\Vect(\bk)$.):
\begin{equation}
\label{eqn:adjunction-strip}
\vc{\begin{tikzpicture}[smallcube]
\compuinit
% 2-cells
\node[cubef] at (0.55,0.5) {$\Adj$};
\node[cubef] at (1.65,0.5) {$\Adj$};
\node[cubef] at (2.75,0.5) {$\Adj$};
\node[cubef] at (3.85,0.5) {$\Adj$};
% 0-cells
\node (luu) at (0,1) {{\tiny $\cDb_G(\cN_G)$}};
\node (lu) at (1.1,1) {{\tiny $\cDb_P(\cN_G)$}};
\node (lm) at (2.2,1) {{\tiny $\cDb_P(\cN_P)$}};
\node (ld) at (3.3,1) {{\tiny $\cDb_P(\cN_L)$}};
\node (ldd) at (4.4,1) {{\tiny $\cDb_L(\cN_L)$}};
\node (ruu) at (0,0) {{\tiny $\sV^{\cDb_G(\cN_G)^\op}$}};
\node (ru) at (1.1,0) {{\tiny $\sV^{\cDb_P(\cN_G)^\op}$}};
\node (rm) at (2.2,0) {{\tiny $\sV^{\cDb_P(\cN_P)^\op}$}};
\node (rd) at (3.3,0) {{\tiny $\sV^{\cDb_P(\cN_L)^\op}$}};
\node (rdd) at (4.4,0) {{\tiny $\sV^{\cDb_L(\cN_L)^\op}$}};
% 1-cells
\draw[->] (luu) -- node[arl] {{\tiny $\For^G_P$}} (lu);
\draw[->] (lu) -- node[arl] {{\tiny $(m_P)^!$}} (lm);
\draw[->] (lm) -- node[arl] {{\tiny $(p_P)_*$}} (ld);
\draw[->] (ld) -- node[arl] {{\tiny $\For^P_L$}} (ldd);
\draw[->] (luu) -- node[arr] {{\tiny $\Yon$}} (ruu);
\draw[->] (lu) -- node[arl] {{\tiny $\Yon$}} (ru);
\draw[->] (lm) -- node[arl] {{\tiny $\Yon$}} (rm);
\draw[->] (ld) -- node[arl] {{\tiny $\Yon$}} (rd);
\draw[->] (ldd) -- node[arl] {{\tiny $\Yon$}} (rdd);
\draw[->] (ruu) -- node[arl] {{\tiny $- \circ (\hamma^G_P)^\op$}} (ru);
\draw[->] (ru) -- node[arl] {{\tiny $- \circ (m_P)_!^\op$}} (rm);
\draw[->] (rm) -- node[arl] {{\tiny $- \circ (p_P)^{*,\op}$}} (rd);
\draw[->] (rd) -- node[arl] {{\tiny $- \circ (\hamma^P_L)^\op$}} (rdd);
\end{tikzpicture}}
\end{equation}

\begin{lem}
\label{lem:transitivity-ResN-IndN}

%Isomorphism \eqref{eqn:adjunction-strip} is compatible with transitivity in the sense that 
The following prism is commutative:
\[
\vc{\begin{tikzpicture}[longprism]
\compuinit
% hidden 2-cells
\node[cuber] at (0.5,0.5,0) {\eqrefh{eqn:adjunction-strip}};
\node[prismdf] at (1,0.5,\tric) {\eqrefh{eqn:composition-induction-Springerderiv}};
% outer 0-cells
\node (rlu) at (0,1,0) {$\cDb_G(\cN_G)$};
\node (rru) at (1,1,0) {$\Vect(\bk)^{\cDb_G(\cN_G)^\op}$};
\node (fr) at (1,0.5,1) {$\Vect(\bk)^{\cDb_L(\cN_L)^\op}$};
\node (rld) at (0,0,0) {$\cDb_T(\cN_T)$};
\node (rrd) at (1,0,0) {$\Vect(\bk)^{\cDb_T(\cN_T)^\op}$};
% hidden 1-cells
\draw[liner,->] (rru) -- node[arr,pos=.3] {$\al - \circ (\IndNderiv{G}{T})^\op$} (rrd);
% outer 1-cells
\draw[->] (rlu) -- node[arl] {$\al \Yon$} (rru);
\draw[->] (rlu) -- node[arr] {$\al \ResNderiv{G}{T}$} (rld);
\draw[->] (rru) -- node[arl] {$\al - \circ (\IndNderiv{G}{L})^\op$} (fr);
\draw[->] (rld) -- node[arr] {$\al \Yon$} (rrd);
\draw[->] (fr) -- node[arl] {$\al - \circ (\IndNderiv{L}{T})^\op$} (rrd);
% visible 2-cells
\node[prismtf] at (0.5,0.75,0.5) {\eqrefh{eqn:adjunction-strip}};
\node[prismbf] at (0.5,0.25,0.5) {\eqrefh{eqn:adjunction-strip}};
\node[prismlf] at (0,0.5,\tric) {\eqrefh{eqn:composition-restriction-Springerderiv}};
% visible 0- and 1-cells
\node (fl) at (0,0.5,1) {$\cDb_L(\cN_L)$};
\draw[->] (rlu) -- node[arl] {$\al \ResNderiv{G}{L}$} (fl);
\draw[->] (fl) -- node[arl,pos=.3] {$\al \ResNderiv{L}{T}$} (rld);
\draw[->] (fl)  -- node[arr,pos=.3] {$\al \Yon$} (fr);
\end{tikzpicture}}
\]

\end{lem}

\begin{proof}
This prism is obtained by gluing together cubes and prisms whose left faces are the squares and triangles in \eqref{eqn:transitivity-elaborate}, and whose left-to-right edges are all $\Yon$. These are commutative by Lemmas \ref{lem:^*_*compositionadjunction}, \ref{lem:_!^!compositionadjunction}, \ref{lem:basechangeadjunction}, \ref{lem:for_*gamma^*adjunction}, \ref{lem:for^!gamma_!adjunction}, \ref{lem:transitivityadjunction}.
\end{proof}

By Lemma \ref{lem:transitivity-ResN-IndN} and the gluing principle, what remains in order to construct isomorphism \eqref{eqn:springer'-intertwine} and prove its compatibility with transitivity is to define an isomorphism for the right-hand square in \eqref{eqn:two-isoms} and prove its compatibility with transitivity. Note that we can think of $\W_G$ and $\W_L$ as one-object categories, and then $\Rep(\W_G)=\Vect(\bk)^{\W_G}$, $\Rep(\W_L)=\Vect(\bk)^{\W_L}$. So it suffices to define an isomorphism
\begin{equation}
\label{eqn:induction-Springer}
\vc{\begin{tikzpicture}[smallcube2]
\compuinit
% 2-cells
\node[cubef] at (0.5,0.5) {$?$};
% 0-cells
\node (lu) at (0,1) {$\W_G$};
\node (ld) at (0,0) {$\W_L$};
\node (ru) at (1,1) {$\cDb_G(\cN_G)$};
\node (rd) at (1,0) {$\cDb_L(\cN_L)$};
% 1-cells
\draw[->] (lu) -- node[arl] {$\al \Spr_G$} (ru);
\draw[<-] (lu) -- node[arr] {} (ld);
\draw[->] (ld) -- node[arr] {$\al \Spr_L$} (rd);
\draw[<-] (ru) -- node[arl] {$\al \IndNderiv{G}{L}$} (rd);
\end{tikzpicture}}
\end{equation}
and to prove that this isomorphism is compatible with transitivity in the sense that the prism
\[
\vc{\begin{tikzpicture}[longprism2]
\compuinit
% hidden 2-cells
\node[cuber] at (0.5,0.5,0) {\eqrefh{eqn:induction-Springer}};
\node[prismdf] at (1,0.5,\tric) {\eqrefh{eqn:composition-induction-Springerderiv}};
% outer 0-cells
\node (rlu) at (0,1,0) {$\W_G$};
\node (rru) at (1,1,0) {$\cDb_G(\cN_G)$};
\node (fr) at (1,0.5,1) {$\cDb_L(\cN_L)$};
\node (rld) at (0,0,0) {$\W_T$};
\node (rrd) at (1,0,0) {$\cDb_T(\cN_T)$};
% hidden 1-cells
\draw[liner,<-] (rru) -- node[arr,pos=.3] {$\al \IndNderiv{G}{T}$} (rrd);
% outer 1-cells
\draw[->] (rlu) -- node[arl] {$\al \Spr_G$} (rru);
\draw[<-] (rlu) -- node[arr] {} (rld);
\draw[<-] (rru) -- node[arl] {$\al \IndNderiv{G}{L}$} (fr);
\draw[->] (rld) -- node[arr] {$\al \Spr_T$} (rrd);
\draw[<-] (fr) -- node[arl] {$\al \IndNderiv{L}{T}$} (rrd);
% visible 2-cells
\node[prismtf] at (0.5,0.75,0.5) {\eqrefh{eqn:induction-Springer}};
\node[prismbf] at (0.5,0.25,0.5) {\eqrefh{eqn:induction-Springer}};
\node[prismlf] at (0,0.5,\tric) {$=$};
% visible 0- and 1-cells
\node (fl) at (0,0.5,1) {$\W_L$};
\draw[<-] (rlu) -- node[arl] {} (fl);
\draw[<-] (fl) -- node[arl,pos=.3] {} (rld);
\draw[->] (fl)  -- node[arr,pos=.3] {$\al \Spr_L$} (fr);
\end{tikzpicture}}
\]
is commutative. In plain terms, this amounts to defining a $\W_L$-equivariant isomorphism $\IndNderiv{G}{L}(\Spr_L)\simto\Spr_G$ such that the following square 
commutes:
\begin{equation}
\vcenter{
\xymatrix@C=8pt@R=12pt{
\IndNderiv{G}{L}(\Spr_L)\ar[rrrr]^-{\sim}\ar@{<-}[d]_-{\wr}&&&&\Spr_G\ar@{<-}[d]^-{\wr}\\
\IndNderiv{G}{L}(\IndNderiv{L}{T}(\Spr_T))\ar[rrrr]^-{\sim}&&&&\IndNderiv{G}{T}(\Spr_T)
}
}
\end{equation}

\begin{rmk}
In the setting of $\Qlb$-sheaves, the existence of a $\W_L$-equivariant isomorphism $\IndNderiv{G}{L}(\Spr_L)\simto\Spr_G$ is a special case of \cite[Theorem 8.3]{lus:icc}.
\end{rmk}

%---------------------------------------------------------------------
\subsection{From $\Spr$ to $\Groth$}
%---------------------------------------------------------------------

By definition 
we have a $\W_G$-equivariant isomorphism $\Spr_G\cong (i_\fg)^*\Groth_G[-r]$ where $i_\fg:\cN_G\hookrightarrow\fg$ is the inclusion and $r=\rank(G)$. So the functor $\Spr_G:\W_G\to\cDb_G(\cN_G)$ is isomorphic to the composition
\begin{equation*}
\W_G \xrightarrow{\Groth_G} \cDb_G(\fg) \xrightarrow{(i_\fg)^{\diamondsuit}} \cDb_G(\cN_G).
\end{equation*}
(Here and below we use the notation $(\cdot)^{\diamondsuit}:=(\cdot)^* [-r]$.)
Using the same principle as in Lemma \ref{lem:whisker}, it suffices to define an isomorphism
\begin{equation*}
\vc{\begin{tikzpicture}[smallcube2]
\compuinit
% 2-cells
\node[cubef] at (0.75,0.5) {$?$};
% 0-cells
\node (lu) at (0,1) {$\W_G$};
\node (ld) at (0,0) {$\W_L$};
\node (ru) at (1.5,1) {$\cDb_G(\cN_G)$};
\node (rd) at (1.5,0) {$\cDb_L(\cN_L)$};
% 1-cells
\draw[->] (lu) -- node[arl] {$\al (i_\fg)^{\diamondsuit} \Groth_G$} (ru);
\draw[<-] (lu) -- node[arr] {} (ld);
\draw[->] (ld) -- node[arr] {$\al (i_{\fl})^{\diamondsuit} \Groth_L$} (rd);
\draw[<-] (ru) -- node[arl] {$\al \IndNderiv{G}{L}$} (rd);
\end{tikzpicture}}
\end{equation*}
and show that it is compatible with transitivity. Consider the diagram:
\begin{equation} \label{eqn:spr-groth}
\vcenter{
\xymatrix@C=2cm@R=14pt{
\W_G  \ar[r]^-{\Groth_G} &\cDb_G(\fg)\ar[r]^-{(i_\fg)^{\diamondsuit}} &\cDb_G(\cN_G) \\
\W_L \ar[u]\ar[r]^-{\Groth_L} & \cDb_L(\fl)\ar[r]^-{(i_\fl)^{\diamondsuit}}\ar[u]_{\Indg{G}{L}} &\cDb_L(\cN_L)\ar[u]_{\IndNderiv{G}{L}}
}
}
\end{equation}
where $\Indg{G}{L}$ is defined as the composition
\begin{equation*}
\xymatrix{
\cDb_L(\fl)\ar[r]^-{\hamma^P_L}&\cDb_P(\fl)\ar[r]^-{(\cdot)^*}&\cDb_P(\fp)\ar[r]^-{(\cdot)_!}&\cDb_P(\fg)\ar[r]^-{\hamma^G_P}&\cDb_G(\fg).
}
\end{equation*}
(Here the morphism $\fp \to \fg$, resp.~$\fp \to \fl$, is the inclusion, resp.~ the projection.)
Note that $\Indg{G}{L}$ has its own transitivity isomorphism
\begin{equation}
\label{eqn:transitivity-Indg}
\Indg{G}{T} \natisom \Indg{G}{L} \circ \Indg{L}{T}
\end{equation}
defined by a diagram analogous to \eqref{eqn:induction-transitivity} where $\cN_H$ is replaced by $\fh$ throughout.

We have an isomorphism for the right-hand square in \eqref{eqn:spr-groth}, given by the following pasting diagram (where $i_{\fp} : \cN_P \to \fp$ is the inclusion):
\begin{equation}
\label{eqn:spr-groth-strip}
\vc{\begin{tikzpicture}[smallcube]
\compuinit
% 2-cells
\node[cubef] at (0.55,0.5) {$\mInt$};
\node[cubef] at (1.65,0.5) {$\Comp$};
\node[cubef] at (2.75,0.5) {$\BC$};
\node[cubef] at (3.85,0.5) {$\mInt$};
% 0-cells
\node (luu) at (4.4,1) {$\cDb_G(\fg)$};
\node (lu) at (3.3,1) {$\cDb_P(\fg)$};
\node (lm) at (2.2,1) {$\cDb_P(\fp)$};
\node (ld) at (1.1,1) {$\cDb_P(\fl)$};
\node (ldd) at (0,1) {$\cDb_L(\fl)$};
\node (ruu) at (4.4,0) {$\cDb_G(\cN_G)$};
\node (ru) at (3.3,0) {$\cDb_P(\cN_G)$};
\node (rm) at (2.2,0) {$\cDb_P(\cN_P)$};
\node (rd) at (1.1,0) {$\cDb_P(\cN_L)$};
\node (rdd) at (0,0) {$\cDb_L(\cN_L)$};
% 1-cells
\draw[<-] (luu) -- node[arr] {{\tiny $\hamma^G_P$}} (lu);
\draw[<-] (lu) -- node[arr] {{\tiny $(\cdot)_!$}} (lm);
\draw[<-] (lm) -- node[arr] {{\tiny $(\cdot)^*$}} (ld);
\draw[<-] (ld) -- node[arr] {{\tiny $\hamma^P_L$}} (ldd);
\draw[->] (luu) -- node[arl] {{\tiny $(i_\fg)^{\diamondsuit}$}} (ruu);
\draw[->] (lu) -- node[arr] {{\tiny $(i_\fg)^{\diamondsuit}$}} (ru);
\draw[->] (lm) -- node[arr] {{\tiny $(i_\fp)^{\diamondsuit}$}} (rm);
\draw[->] (ld) -- node[arr] {{\tiny $(i_\fl)^{\diamondsuit}$}} (rd);
\draw[->] (ldd) -- node[arr] {{\tiny $(i_\fl)^{\diamondsuit}$}} (rdd);
\draw[<-] (ruu) -- node[arl] {{\tiny $\hamma^G_P$}} (ru);
\draw[<-] (ru) -- node[arl] {{\tiny $(m_P)_!$}} (rm);
\draw[<-] (rm) -- node[arl] {{\tiny $(p_P)^*$}} (rd);
\draw[<-] (rd) -- node[arl] {{\tiny $\hamma^P_L$}} (rdd);
\end{tikzpicture}}
\end{equation}
\begin{lem}

Isomorphism \eqref{eqn:spr-groth-strip} is compatible with transitivity in the sense that the following prism is commutative:
\[
\vc{\begin{tikzpicture}[longprism2]
\compuinit
% hidden 2-cells
\node[cuber] at (0.5,0.5,0) {\eqrefh{eqn:spr-groth-strip}};
\node[prismdf] at (1,0.5,\tric) {\eqrefh{eqn:composition-induction-Springerderiv}};
% outer 0-cells
\node (rlu) at (0,1,0) {$\cDb_G(\fg)$};
\node (rru) at (1,1,0) {$\cDb_G(\cN_G)$};
\node (fr) at (1,0.5,1) {$\cDb_L(\cN_L)$};
\node (rld) at (0,0,0) {$\cDb_T(\ft)$};
\node (rrd) at (1,0,0) {$\cDb_T(\cN_T)$};
% hidden 1-cells
\draw[liner,<-] (rru) -- node[arr,pos=.3] {$\al \IndNderiv{G}{T}$} (rrd);
% outer 1-cells
\draw[->] (rlu) -- node[arl] {$\al (i_\fg)^{\diamondsuit}$} (rru);
\draw[<-] (rlu) -- node[arr] {$\al \Indg{G}{T}$} (rld);
\draw[<-] (rru) -- node[arl] {$\al \IndNderiv{G}{L}$} (fr);
\draw[->] (rld) -- node[arr] {$\al (i_\ft)^{\diamondsuit}$} (rrd);
\draw[<-] (fr) -- node[arl] {$\al \IndNderiv{L}{T}$} (rrd);
% visible 2-cells
\node[prismtf] at (0.5,0.75,0.5) {\eqrefh{eqn:spr-groth-strip}};
\node[prismbf] at (0.5,0.25,0.5) {\eqrefh{eqn:spr-groth-strip}};
\node[prismlf] at (0,0.5,\tric) {\eqrefh{eqn:transitivity-Indg}};
% visible 0- and 1-cells
\node (fl) at (0,0.5,1) {$\cDb_L(\fl)$};
\draw[<-] (rlu) -- node[arl] {$\al \Indg{G}{L}$} (fl);
\draw[<-] (fl) -- node[arl,pos=.3] {$\al \Indg{L}{T}$} (rld);
\draw[->] (fl)  -- node[arr,pos=.3] {$\al (i_\fl)^{\diamondsuit}$} (fr);
\end{tikzpicture}}
\]

\end{lem}

\begin{proof}
This prism is obtained by gluing together cubes and prisms that are commutative by Lemmas \ref{lem:^*composition^*}, \ref{lem:_!composition^*}, \ref{lem:_!^*basechange^*}, \ref{lem:Gamma^*Gamma}, \ref{lem:^*compositionGamma} and \ref{lem:basechangeGamma}.
\end{proof}

All that remains is to define an isomorphism for the left-hand square in \eqref{eqn:spr-groth} and show its compatibility with transitivity, i.e.~to define a $\W_L$-equivariant isomorphism 
\begin{equation}
\label{eqn:desired-isom}
\Indg{G}{L}(\Groth_L)\simto\Groth_G
\end{equation}
such that the following square of isomorphisms in $\cDb_G(\fg)$ commutes:
\begin{equation} \label{eqn:desired-square}
\vcenter{
\xymatrix@C=8pt@R=12pt{
\Indg{G}{L}(\Groth_L)\ar[rrrr]^-{\sim}\ar@{<-}[d]_-{\wr}&&&&\Groth_G\ar@{<-}[d]^-{\wr}\\
\Indg{G}{L}(\Indg{L}{T}(\Groth_T))\ar[rrrr]^-{\sim}&&&&\Indg{G}{T}(\Groth_T)
}
}
\end{equation}

%---------------------------------------------------------------------
\subsection{Another induction functor}
%---------------------------------------------------------------------
 
Let $\Indgtilde{G}{L}$ be the composition:
{\small
\[
\cDb_L(L\times^C\fc) \xrightarrow{\hamma^P_L} \cDb_P(L\times^C\fc) \xrightarrow{(\cdot)^*} \cDb_P(P\times^B\fb) \xrightarrow{(\cdot)_!} \cDb_P(G\times^B\fb) \\
\xrightarrow{\hamma^G_P} \cDb_G(G\times^B\fb).
\]
}(Here, the morphism $P \times^B \fb \to L \times^C \fc \cong P \times^B \fc$ is induced by the projection $\fb \to \fc$, the morphism $P \times^B \fb \to G \times^B \fb$ is the natural inclusion, and $P$ acts on $L\times^C\fc$ via the projection $P\to L$.) This functor has its own transitivity isomorphism, defined by the following pasting diagram (where all morphisms are the natural ones):
\begin{equation}
\label{eqn:indgtilde-transitivity}
\vc{\begin{tikzpicture}[stdtriangle]
\compuinit
% 2-cells
\node[lttricelld] at (0.5+0.5*\tric, 3.5+0.5*\tric) {\tiny$\mTr$};
\node[lttricelld] at (1.5+0.5*\tric, 2.5+0.5*\tric) {\tiny$\Comp$};
\node[lttricelld] at (2.5+0.5*\tric, 1.5+0.5*\tric) {\tiny$\Comp$};
\node[lttricelld] at (3.5+0.5*\tric, 0.5+0.5*\tric) {\tiny$\mTr$};
\node[lttricelld] at (3.5+0.5*\tric, 3.5+0.5*\tric) {\tiny$\mTr$};
\node[rttricellu] at (3.5-0.5*\tric, 3.5-0.5*\tric) {\tiny$\mTr$};
\node[cubef] at (1.5,3.5) {$\mInt$};
\node[cubef] at (2.5,3.5) {$\mInt$};
\node[cubef] at (2.5,2.5) {$\BC$};
\node[cubef] at (3.5,2.5) {$\mInt$};
\node[cubef] at (3.5,1.5) {$\mInt$};
% 0-cells
\node (lluu) at (0,4) {{\tiny $\cDb_G(G \times^B \fb)$}};
\node (luu) at (1,4) {{\tiny $\cDb_P(G \times^B \fb)$}};
\node (muu) at (2,4) {{\tiny $\cDb_P(P \times^B \fb)$}};
\node (ruu) at (3,4) {{\tiny $\cDb_P(L \times^C \fc)$}};
\node (rruu) at (4,4) {{\tiny $\cDb_L(L \times^C \fc)$}};
\node (lu) at (1,3) {{\tiny $\cDb_B(G \times^B \fb)$}};
\node (mu) at (2,3) {{\tiny $\cDb_B(P \times^B \fb)$}};
\node (ru) at (3,3) {{\tiny $\cDb_B(L \times^C \fc)$}};
\node (rru) at (4,3) {{\tiny $\cDb_C(L \times^C \fc)$}};
\node (mm) at (2,2) {{\tiny $\cDb_B(B \times^B \fb)$}};
\node (rm) at (3,2) {{\tiny $\cDb_B(C \times^C \fc)$}};
\node (rrm) at (4,2) {{\tiny $\cDb_C(C \times^C \fc)$}};
\node (rd) at (3,1) {{\tiny $\cDb_B(T \times^T \ft)$}};
\node (rrd) at (4,1) {{\tiny $\cDb_C(T \times^T \ft)$}};
\node (rrdd) at (4,0) {{\tiny $\cDb_T(T \times^T \ft)$}};
% 1-cells
\draw[<-] (lluu) -- node[arl] {{\tiny $\hamma^G_P$}} (luu);
\draw[<-] (luu) -- node[arl] {{\tiny $(\cdot)_!$}} (muu);
\draw[<-] (muu) -- node[arl] {{\tiny $(\cdot)^*$}} (ruu);
\draw[<-] (ruu) -- node[arl] {{\tiny $\hamma^P_L$}} (rruu);
\draw[<-] (lu) -- node[arl] {{\tiny $(\cdot)_!$}} (mu);
\draw[<-] (mu) -- node[arl] {{\tiny $(\cdot)^*$}} (ru);
\draw[<-] (ru) -- node[arl] {{\tiny $\hamma^B_C$}} (rru);
\draw[<-] (mm) -- node[arl] {{\tiny $(\cdot)^*$}} (rm);
\draw[<-] (rm) -- node[arl] {{\tiny $\hamma^B_C$}} (rrm);
\draw[<-] (rd) -- node[arl] {{\tiny $\hamma^B_C$}} (rrd);
\draw[<-] (luu) -- node[arl] {{\tiny $\hamma^P_B$}} (lu);
\draw[<-] (muu) -- node[arl] {{\tiny $\hamma^P_B$}} (mu);
\draw[<-] (ruu) -- node[arl] {{\tiny $\hamma^P_B$}} (ru);
\draw[<-] (rruu) -- node[arl] {{\tiny $\hamma^L_C$}} (rru);
\draw[<-] (mu) -- node[arl] {{\tiny $(\cdot)_!$}} (mm);
\draw[<-] (ru) -- node[arl] {{\tiny $(\cdot)_!$}} (rm);
\draw[<-] (rru) -- node[arl] {{\tiny $(\cdot)_!$}} (rrm);
\draw[<-] (rm) -- node[arl] {{\tiny $(\cdot)^*$}} (rd);
\draw[<-] (rrm) -- node[arl] {{\tiny $(\cdot)^*$}} (rrd);
\draw[<-] (rrd) -- node[arl] {{\tiny $\hamma^C_T$}} (rrdd);
\draw[<-] (lluu) -- node[arr] {{\tiny $\hamma^G_B$}} (lu);
\draw[<-] (lu) -- node[arr] {{\tiny $(\cdot)_!$}} (mm);
\draw[<-] (mm) -- node[arr] {{\tiny $(\cdot)^*$}} (rd);
\draw[<-] (rd) -- node[arr] {{\tiny $\hamma^B_T$}} (rrdd);
\draw[<-] (ruu) -- node[arl] {{\tiny $\hamma^P_C$}} (rru);
\end{tikzpicture}}
\end{equation}

We have an isomorphism $(\mu_\fg)_!\circ\Indgtilde{G}{L}\natisom\Indg{G}{L}\circ(\mu_\fl)_!$, defined by 
\begin{equation}
\label{eqn:indgtilde-strip}
\vc{\begin{tikzpicture}[smallcube]
\compuinit
% 2-cells
\node[cubef] at (0.55,0.5) {$\mInt$};
\node[cubef] at (1.65,0.5) {$\BC$};
\node[cubef] at (2.75,0.5) {$\Comp$};
\node[cubef] at (3.85,0.5) {$\mInt$};
% 0-cells
\node (luu) at (4.4,1) {{\tiny $\cDb_G(G \times^B \fb)$}};
\node (lu) at (3.3,1) {{\tiny $\cDb_P(G \times^B \fb)$}};
\node (lm) at (2.2,1) {{\tiny $\cDb_P(P \times^B \fb)$}};
\node (ld) at (1.1,1) {{\tiny $\cDb_P(L \times^C \fc)$}};
\node (ldd) at (0,1) {{\tiny $\cDb_L(L \times^C \fc)$}};
\node (ruu) at (4.4,0) {{\tiny $\cDb_G(\fg)$}};
\node (ru) at (3.3,0) {{\tiny $\cDb_P(\fg)$}};
\node (rm) at (2.2,0) {{\tiny $\cDb_P(\fp)$}};
\node (rd) at (1.1,0) {{\tiny $\cDb_P(\fl)$}};
\node (rdd) at (0,0) {{\tiny $\cDb_L(\fl)$}};
% 1-cells
\draw[<-] (luu) -- node[arr] {{\tiny $\hamma^G_P$}} (lu);
\draw[<-] (lu) -- node[arr] {{\tiny $(\cdot)_!$}} (lm);
\draw[<-] (lm) -- node[arr] {{\tiny $(\cdot)^*$}} (ld);
\draw[<-] (ld) -- node[arr] {{\tiny $\hamma^P_L$}} (ldd);
\draw[->] (luu) -- node[arl] {{\tiny $(\mu_\fg)_!$}} (ruu);
\draw[->] (lu) -- node[arl] {{\tiny $(\mu_\fg)_!$}} (ru);
\draw[->] (lm) -- node[arl] {{\tiny $(\mu_\fp)_!$}} (rm);
\draw[->] (ld) -- node[arl] {{\tiny $(\mu_\fl)_!$}} (rd);
\draw[->] (ldd) -- node[arr] {{\tiny $(\mu_\fl)_!$}} (rdd);
\draw[<-] (ruu) -- node[arl] {{\tiny $\hamma^G_P$}} (ru);
\draw[<-] (ru) -- node[arl] {{\tiny $(\cdot)_!$}} (rm);
\draw[<-] (rm) -- node[arl] {{\tiny $(\cdot)^*$}} (rd);
\draw[<-] (rd) -- node[arl] {{\tiny $\hamma^P_L$}} (rdd);
\end{tikzpicture}}
\end{equation}
(Here, $\mu_\fp : P \times^B \fb \to \fp$ is the morphism induced by the adjoint action of $P$ on $\fp$.)

\begin{lem}
\label{lem:indgtilde-prism}
Isomorphism \eqref{eqn:indgtilde-strip} is compatible with transitivity in the sense that the following prism is commutative:
\[
\vc{\begin{tikzpicture}[longprism2]
\compuinit
% hidden 2-cells
\node[cuber] at (0.5,0.5,0) {\eqrefh{eqn:indgtilde-strip}};
\node[prismdf] at (1,0.5,\tric) {\eqrefh{eqn:transitivity-Indg}};
% outer 0-cells
\node (rlu) at (0,1,0) {$\cDb_G(G \times^B \fb)$};
\node (rru) at (1,1,0) {$\cDb_G(\fg)$};
\node (fr) at (1,0.5,1) {$\cDb_L(\fl)$};
\node (rld) at (0,0,0) {$\cDb_T(T \times^T \ft)$};
\node (rrd) at (1,0,0) {$\cDb_T(\ft)$};
% hidden 1-cells
\draw[liner,<-] (rru) -- node[arr,pos=.3] {$\al \Indg{G}{T}$} (rrd);
% outer 1-cells
\draw[->] (rlu) -- node[arl] {$\al (\mu_\fg)_!$} (rru);
\draw[<-] (rlu) -- node[arr] {$\al \Indgtilde{G}{T}$} (rld);
\draw[<-] (rru) -- node[arl] {$\al \Indg{G}{L}$} (fr);
\draw[->] (rld) -- node[arr] {$\al (\mu_\ft)_!$} (rrd);
\draw[<-] (fr) -- node[arl] {$\al \Indg{L}{T}$} (rrd);
% visible 2-cells
\node[prismtf] at (0.5,0.75,0.5) {\eqrefh{eqn:indgtilde-strip}};
\node[prismbf] at (0.5,0.25,0.5) {\eqrefh{eqn:indgtilde-strip}};
\node[prismlf] at (0,0.5,\tric) {\eqrefh{eqn:indgtilde-transitivity}};
% visible 0- and 1-cells
\node (fl) at (0,0.5,1) {$\cDb_L(L \times^C \fc)$};
\draw[<-] (rlu) -- node[arl] {$\al \Indgtilde{G}{L}$} (fl);
\draw[<-] (fl) -- node[arl,pos=.3] {$\al \Indgtilde{L}{T}$} (rld);
\draw[->] (fl)  -- node[arr,pos=.3] {$\al (\mu_\fl)_!$} (fr);
\end{tikzpicture}}
\]
\end{lem}

\begin{proof}
By definition, this prism is obtained by gluing together cubes and prisms that are commutative by Lemmas \ref{lem:_!composition_!}, \ref{lem:^*composition_!},  \ref{lem:_!^*basechange_!}, \ref{lem:Gamma_!Gamma}, \ref{lem:_!compositionGamma} and \ref{lem:basechangeGamma}. All the required cartesian squares are easy.
\end{proof}

%---------------------------------------------------------------------
\subsection{Definition of \eqref{eqn:desired-isom} and commutativity of \eqref{eqn:desired-square}}
%---------------------------------------------------------------------

Neglecting the $\W_G$-action for now, we may think of $\Groth_G$ as the composition
\[
\bb \xrightarrow{\ubk[\dim\fg]} \cDb_G(G\times^B\fb) \xrightarrow{(\mu_\fg)_!} \cDb_G(\fg)
\]
where $\bb$ is the trivial group regarded as a one-object category. So to define an isomorphism $\Indg{G}{L}(\Groth_L)\simto\Groth_G$, we need to consider the diagram:
\begin{equation} \label{eqn:ind-groth}
\vcenter{
\xymatrix@C=2cm@R=15pt{
\bb\ar[r]^-{\ubk[\dim\fg]}\ar[dr]_-{\ubk[\dim\fl]}&\cDb_G(G\times^B\fb)\ar[r]^-{(\mu_\fg)_!}&\cDb_G(\fg)\\
&\cDb_L(L\times^C\fc)\ar[r]^-{(\mu_\fl)_!}\ar[u]^-{\Indgtilde{G}{L}}&\cDb_L(\fl)\ar[u]^-{\Indg{G}{L}}
}
}
\end{equation}
We have just defined an isomorphism for the square in \eqref{eqn:ind-groth}. An isomorphism for the triangle
may be defined by the following pasting diagram (see \S\ref{sss:constantsheaf-res}, \S\ref{sss:constantsheaf-fi} and \S\ref{sss:const-induction-equiv} for the notation):
\begin{equation} \label{eqn:bb-triangle}
\vc{\begin{tikzpicture}[stdtriangle2]
\compuinit
% 2-cells
\node[lttricelld] at (1,0.7) {{\tiny $\CInt$}};
\node[rttricelld] at (0.3,1.7) {{\tiny $\Rel$}};
\node[lttricellu] at (1,1.3) {{\tiny $\Cnst$}};
% 0-cellu
\node (ldd) at (1.4,0) {{\tiny $\cDb_L(L \times^C \fc)$}};
\node (rd) at (1.4,1) {{\tiny $\cDb_P(L \times^C \fc)$}};
\node (lm) at (0,1) {{\tiny $\bb$}};
\node (ru) at (1.4,2) {{\tiny $\cDb_P(P \times^B \fb)$}};
\node (luu) at (0,2) {{\tiny $\cDb_G(G \times^B \fb)$}};
% 1-cells
\draw[->] (ldd) -- node[arr] {$\al \hamma_L^P$} (rd);
\draw[->] (rd) -- node[arr] {$\al (\cdot)^*$} (ru);
\draw[->] (ru) -- node[arr] {$\al \hamma_P^G (\cdot)_!$} (luu);
\draw[->] (lm) -- node[arr] {$\al \ubk[\dim \fl]$} (ldd);
\draw[->] (lm) -- node[arl] {$\al \ubk[\dim \fl]$} (rd);
\draw[->] (lm) -- node[arl] {$\al \ubk[\dim \fl]$} (ru);
\draw[->] (lm) -- node[arl] {$\al \ubk[\dim \fg]$} (luu);
\end{tikzpicture}}
\end{equation}

\begin{lem}
\label{lem:constant-tetrahedron}
Isomorphism \eqref{eqn:bb-triangle} is compatible with transitivity in the sense that the following tetrahedron is commutative:
\[
\vc{\begin{tikzpicture}[stdtetr]
\compuinit
% hidden 2-cells
\node[tetrlr] at ({-\tric/2},0.5,\tric) {\eqrefh{eqn:bb-triangle}};
\node[tetrdr] at ({\tric/2},0.5,\tric) {\eqrefh{eqn:indgtilde-transitivity}};
% outer 0-cells
\node (ru) at (0,1,0) {{\tiny $\cDb_G(G \times^B \fb)$}};
\node (fl) at (-0.5,0.5,1) {{\tiny $\bb$}};
\node (fr) at (0.5,0.5,1) {{\tiny $\cDb_L(L \times^C \fc)$}};
\node (rd) at (0,0,0) {{\tiny $\cDb_T(T \times^T \ft)$}};
% hidden 1-cell
\draw[liner,<-] (ru) -- node[arl,pos=.6] {$\al \Indgtilde{G}{T}$} (rd);
% visible 2-cells
\node[tetrtf] at (0, {(1+\tric)/2}, {1-\tric}) {\eqrefh{eqn:bb-triangle}};
\node[tetrbf] at (0, {(1-\tric)/2}, {1-\tric}) {\eqrefh{eqn:bb-triangle}};
% visible 1-cells
\draw[->] (fl) -- node[arl] {$\al \ubk[\dim\fg]$} (ru);
\draw[->] (fl) -- node[arr] {$\al \ubk[\dim\ft]$} (rd);
\draw[<-] (ru) -- node[arl] {$\al \Indgtilde{G}{L}$} (fr);
\draw[->] (rd) -- node[arr] {$\al \Indgtilde{L}{T}$} (fr);
\draw[->] (fl) -- node[arl] {$\al \ubk[\dim \fl]$} (fr);
\end{tikzpicture}}
\]
\end{lem}

\begin{proof}
By definition, this tetrahedron is obtained by gluing together things that are commutative by Lemmas \ref{lem:^*compositionconstant}, \ref{lem:Gammacompositionconstant}, \ref{lem:Gamma^*constant}, \ref{lem:Gamma_!compositionconstant},  \ref{lem:Gamma_!Gammaconstant} and \ref{lem:Gamma_!^*constant}.
\end{proof}

The diagram \eqref{eqn:ind-groth} is now complete, so we have our isomorphism $\Indg{G}{L}(\Groth_L)\simto\Groth_G$. Gluing together the prism in Lemma \ref{lem:indgtilde-prism} and the tetrahedron in Lemma \ref{lem:constant-tetrahedron}, we obtain a tetrahedron whose commutativity means exactly that diagram \eqref{eqn:desired-square} commutes. At this point, 
all that remains is to prove that our isomorphism is $\W_L$-equivariant.

%---------------------------------------------------------------------
\subsection{$\W_L$-equivariance}
%---------------------------------------------------------------------

Recall that $j_\fg^*:\End(\Groth_G)\to\End(j_\fg^*\Groth_G)$ is injective (and even an isomorphism). So it suffices to prove that the induced isomorphism $j_\fg^*\Indg{G}{L}(\Groth_L)\simto j_\fg^*\Groth_G$ is $\W_L$-equivariant. By base change, we have 
\begin{equation}
\label{eqn:Groth-rs}
j_\fg^*\Groth_G \simto (\mu_\fg^\rs)_!\ubk[\dim\fg],
\end{equation}
where $\mu_\fg^\rs:G\times^B(\fb\cap\fg^\rs)\to\fg^\rs$ denotes the restriction of $\mu_\fg$ to $\mu_\fg^{-1}(\fg^\rs)$. It is well known that $\mu_\fg^\rs$ is a Galois covering with group $\W_G$, so $(\mu_\fg^\rs)_!\ubk$ is isomorphic to a rank-$|\W_G|$ local system on $\fg^\rs$, and carries a natural $\W_G$-action (see e.g.~\S\ref{ss:equivariance-finite}). By definition of the $\W_G$-action on $\Groth_G$, isomorphism \eqref{eqn:Groth-rs} is $\W_G$-equivariant. 

Define a functor ${}^\rs\Indg{G}{L}:\cDb_L(\fl\cap\fg^\rs)\to\cDb_G(\fg^\rs)$ as the composition
\begin{equation*}
\xymatrix{
\cDb_L(\fl\cap\fg^\rs)\ar[r]^-{\hamma^P_L}&\cDb_P(\fl\cap\fg^\rs)\ar[r]^-{(\cdot)^*}&\cDb_P(\fp\cap\fg^\rs)\ar[r]^-{(\cdot)_!}&\cDb_P(\fg^\rs)\ar[r]^-{\hamma^G_P}&\cDb_G(\fg^\rs).
}
\end{equation*}
Note that $\fl\cap\fg^\rs$ is an open subset of $\fl^\rs$. Let $j_\fl'$ denote the inclusion of $\fl\cap\fg^\rs$ in $\fl$, and $\mu_\fl^{\rs,\prime}$ the restriction of $\mu_\fl$ to $\mu_\fl^{-1}(\fl\cap\fg^\rs)$. 
We have an isomorphism $j_\fg^*\circ\Indg{G}{L}\natisom{}^\rs\Indg{G}{L}\circ(j_\fl')^*$, defined by the following pasting diagram (where, $j_{\fp}' : \fp \cap \fg^{\rs} \hookrightarrow \fp$ is the inclusion):
\begin{equation}
\label{eqn:indgrs-strip}
\vc{\begin{tikzpicture}[smallcube]
\compuinit
% 2-cells
\node[cubef] at (0.55,0.5) {$\mInt$};
\node[cubef] at (1.65,0.5) {$\Comp$};
\node[cubef] at (2.75,0.5) {$\BC$};
\node[cubef] at (3.85,0.5) {$\mInt$};
% 0-cells
\node (luu) at (4.4,1) {$\cDb_G(\fg)$};
\node (lu) at (3.3,1) {$\cDb_P(\fg)$};
\node (lm) at (2.2,1) {$\cDb_P(\fp)$};
\node (ld) at (1.1,1) {$\cDb_P(\fl)$};
\node (ldd) at (0,1) {$\cDb_L(\fl)$};
\node (ruu) at (4.4,0) {$\cDb_G(\fg^{\rs})$};
\node (ru) at (3.3,0) {$\cDb_P(\fg^{\rs})$};
\node (rm) at (2.2,0) {$\cDb_P(\fp \cap \fg^{\rs})$};
\node (rd) at (1.1,0) {$\cDb_P(\fl \cap \fg^{\rs})$};
\node (rdd) at (0,0) {$\cDb_L(\fl \cap \fg^{\rs})$};
% 1-cells
\draw[<-] (luu) -- node[arr] {{\tiny $\hamma^G_P$}} (lu);
\draw[<-] (lu) -- node[arr] {{\tiny $(\cdot)_!$}} (lm);
\draw[<-] (lm) -- node[arr] {{\tiny $(\cdot)^*$}} (ld);
\draw[<-] (ld) -- node[arr] {{\tiny $\hamma^P_L$}} (ldd);
\draw[->] (luu) -- node[arl] {{\tiny $(j_\fg)^*$}} (ruu);
\draw[->] (lu) -- node[arl] {{\tiny $(j_\fg)^*$}} (ru);
\draw[->] (lm) -- node[arl] {{\tiny $(j_\fp')^*$}} (rm);
\draw[->] (ld) -- node[arl] {{\tiny $(j_\fl')^*$}} (rd);
\draw[->] (ldd) -- node[arr] {{\tiny $(j_\fl')^*$}} (rdd);
\draw[<-] (ruu) -- node[arl] {{\tiny $\hamma^G_P$}} (ru);
\draw[<-] (ru) -- node[arl] {{\tiny $(\cdot)_!$}} (rm);
\draw[<-] (rm) -- node[arl] {{\tiny $(\cdot)^*$}} (rd);
\draw[<-] (rd) -- node[arl] {{\tiny $\hamma^P_L$}} (rdd);
\end{tikzpicture}}
\end{equation}

We can modify the definition of $\Indgtilde{G}{L}$ in exactly the same way to obtain a functor ${}^\rs\Indgtilde{G}{L}:\cDb_L(L\times^C(\fc\cap\fg^\rs))\to\cDb_G(G\times^B(\fb\cap\fg^\rs))$. This functor is related to $\Indgtilde{G}{L}$ by a diagram analogous to \eqref{eqn:indgrs-strip}, namely we have an isomorphism
\begin{equation}
\label{eqn:Indgtilde-restriction}
(k_\fg)^* \circ \Indgtilde{G}{L} \natisom {}^\rs\Indgtilde{G}{L} \circ (k_{\fl}')^*
\end{equation}
where $k_{\fg} : G \times^B (\fb \cap \fg^{\rs}) \hookrightarrow G \times^B \fb$ and $k_{\fl}' : L \times^C (\fc \cap \fg^{\rs}) \hookrightarrow L \times^C \fc$ are the inclusions. The functor ${}^\rs\Indgtilde{G}{L}$ is also related to ${}^\rs\Indg{G}{L}$ by a diagram analogous to \eqref{eqn:indgtilde-strip}, namely we have an isomorphism
\begin{equation}
\label{eqn:Indgtilde-rs}
(\mu_{\fg}^{\rs})_! \circ {}^\rs\Indgtilde{G}{L} \natisom {}^\rs\Indg{G}{L} \circ (\mu_{\fl}^{\rs,\prime})_!.
\end{equation}

\begin{lem} 
\label{lem:indgrs-cube}

The following cube is commutative: 
\begin{equation}
\vc{\begin{tikzpicture}[longcube2]
\compuinit
% hidden 0-cell
\node (rrd) at (1,0,0) {$\cDb_L(\fl)$};
% hidden 2-cells
\node[cuber] at (0.5,0.5,0) {$\eqrefh{eqn:indgtilde-strip}$};
\node[cubed] at (1,0.5,0.5) {$\eqrefh{eqn:indgrs-strip}$};
\node[cubeb] at (0.5,0,0.5) {$\BC$};
% outer 0-cells
\node (rlu) at (0,1,0) {$\cDb_G(G \times^B \fb)$};
\node (rru) at (1,1,0) {$\cDb_G(\fg)$};
\node (fru) at (1,1,1) {$\cDb_G(\fg^{\rs})$};
\node (frd) at (1,0,1) {$\cDb_L(\fl \cap \fg^{\rs})$};
\node (fld) at (0,0,1) {$\cDb_L(L \times^C (\fc \cap \fg^{\rs}))$};
\node (rld) at (0,0,0) {$\cDb_L(L \times^C \fc)$};
% hidden 1-cells
\draw[liner,<-] (rru) -- node[arl,pos=.7] {$\al \Indg{G}{L}$} (rrd);
\draw[liner,->] (rld) -- node[arl] {$\al (\mu_{\fl})_!$} (rrd);
\draw[liner,->] (rrd) -- node[arl] {$\al (j_{\fl}')^*$} (frd);
% outer 1-cells
\draw[->] (rlu) -- node[arl] {$\al (\mu_{\fg})_!$} (rru);
\draw[->] (rru) -- node[arl] {$\al (j_{\fg})^*$} (fru);
\draw[<-] (fru) -- node[arl] {$\al {}^{\rs} \Indg{G}{L}$}(frd);
\draw[->] (fld) -- node[arr] {$\al (\mu_{\fl}^{\rs,\prime})_!$} (frd);
\draw[->] (rld) -- node[arr] {$\al (k_{\fl}')^*$} (fld);
\draw[<-] (rlu) -- node[arr] {$\al \Indgtilde{G}{L}$} (rld);
% visible 2-cells
\node[cubel] at (0,0.5,0.5) {$\eqrefh{eqn:Indgtilde-restriction}$};
\node[cubet] at (0.5,1,0.5) {$\BC$};
\node[cubef] at (0.5,0.5,1) {$\eqrefh{eqn:Indgtilde-rs}$};
% visible 0- and 1-cells
\node (flu) at (0,1,1) {$\cDb_G(G \times^B (\fb \cap \fg^{\rs}))$};
\draw[->] (rlu) -- node[arl,pos=.7] {$\al (k_{\fg})^*$} (flu);
\draw[->] (flu) -- node[arr,pos=.3] {$\al (\mu_{\fg}^{\rs})_!$} (fru);
\draw[<-] (flu) -- node[arl,pos=.3] {$\al {}^\rs\Indgtilde{G}{L}$} (fld);
\end{tikzpicture}}
\end{equation}

\end{lem}

\begin{proof}
By definition, this cube is obtained by gluing together cubes that are commutative by Lemmas \ref{lem:_!^*basechange_!}, \ref{lem:_!^*basechange^*} and \ref{lem:basechangeGamma} (used twice).
\end{proof}

We also have an isomorphism 
\begin{equation}
\label{eqn:bb-triangle-rs}
{}^\rs\Indgtilde{G}{L}(\ubk[\dim\fl]) \simto \ubk[\dim\fg], 
\end{equation}
defined by the obvious analogue of \eqref{eqn:bb-triangle}.

\begin{lem}
\label{lem:constant-pyramid}

The following pyramid is commutative:
\[
\vc{\begin{tikzpicture}[stdpyra]
\compuinit
% hidden 0-cell
\node (rrd) at (1,0,0) {$\cDb_L(L \times^C \fc)$};
% hidden 2-cells
\node[pyrar] at (1-\tric,0.5,0.5*\tric) {$\eqrefh{eqn:bb-triangle}$};
\node[pyrab] at (1-\tric,0.5*\tric,0.5) {$\Cnst$};
% outer 0-cells
\node (l) at (0,0.5,0.5) {$\bb$};
\node (rru) at (1,1,0) {$\cDb_G(G \times^B \fb)$};
\node (frd) at (1,0,1) {$\cDb_L(L \times^C (\fc \cap \fg^{\rs}))$};
% hidden 1-cells
\draw[liner,->] (l) -- node[arl] {$\al \ubk[\dim \fl]$} (rrd);
% outer 1-cells
\draw[<-] (rru) -- node[arl,pos=.7] {$\al \Indgtilde{G}{L}$} (rrd);
\draw[->] (rrd) -- node[arl] {$\al (k_{\fl}')^*$} (frd);
\draw[->] (l) -- node[arl] {$\al \ubk[\dim \fg]$} (rru);
\draw[->] (l) -- node[arr] {$\al \ubk[\dim \fl]$} (frd);
% visible 2-cells
\node[cubel,xscale=-1] at (1,0.5,0.5) {$\eqrefh{eqn:Indgtilde-restriction}$};
\node[pyraf] at (1-\tric,0.5,1-0.5*\tric) {$\eqrefh{eqn:bb-triangle-rs}$};
\node[pyrat] at (1-\tric,1-0.5*\tric,0.5) {$\Cnst$};
% visible 0- and 1-cells
\node (fru) at (1,1,1) {$\cDb_G(G \times^B (\fb \cap \fg^{\rs}))$};
\draw[->] (rru) -- node[arl] {$\al (k_\fg)^*$} (fru);
\draw[<-] (fru) -- node[arl] {$\al {}^\rs\Indgtilde{G}{L}$}(frd);
\draw[->] (l) -- node[arl,pos=.7] {$\al \ubk[\dim \fg]$} (fru);
\end{tikzpicture}}
\]

\end{lem}

\begin{proof}
By definition, this pyramid is obtained by gluing together things that are commutative by Lemmas \ref{lem:^*compositionconstant}, \ref{lem:Gamma^*constant} and \ref{lem:Gamma_!^*constant}.
\end{proof}

Combining isomorphisms \eqref{eqn:Indgtilde-rs} and \eqref{eqn:bb-triangle-rs} we obtain an isomorphism
\begin{equation}
\label{eqn:bb-triangle-rs-Groth}
{}^\rs\Indgtilde{G}{L}\bigl( (\mu_{\fl}^{\rs,\prime})_! \ubk[\dim\fl] \bigr) \simto (\mu_{\fg}^{\rs})_! \ubk[\dim\fg].
\end{equation}
Gluing together the cube in Lemma \ref{lem:indgrs-cube} and the pyramid in Lemma \ref{lem:constant-pyramid}, we obtain the following commutative pyramid:
\begin{equation}
\label{eqn:final-pyramid}
\vc{\begin{tikzpicture}[stdpyra]
\compuinit
% hidden 0-cell
\node (rrd) at (1,0,0) {$\cDb_L(\fl)$};
% hidden 2-cells
\node[pyrar] at (1-\tric,0.5,0.5*\tric) {$\eqrefh{eqn:desired-isom}$};
\node[pyrab] at (1-\tric,0.5*\tric,0.5) {};
% outer 0-cells
\node (l) at (0,0.5,0.5) {$\bb$};
\node (rru) at (1,1,0) {$\cDb_G(\fg)$};
\node (frd) at (1,0,1) {$\cDb_L(\fl \cap \fg^{\rs})$};
% hidden 1-cells
\draw[liner,->] (l) -- node[arl] {$\al \Groth_L$} (rrd);
% outer 1-cells
\draw[<-] (rru) -- node[arl,pos=.7] {$\al \Indgtilde{G}{L}$} (rrd);
\draw[->] (rrd) -- node[arl] {$\al (j_{\fl}')^*$} (frd);
\draw[->] (l) -- node[arl] {$\al \Groth_G$} (rru);
\draw[->] (l) -- node[arr] {$\al (\mu_{\fl}^{\rs,\prime})_! \ubk[\dim \fl]$} (frd);
% visible 2-cells
\node[cubel,xscale=-1] at (1,0.5,0.5) {$\eqrefh{eqn:Indgtilde-restriction}$};
\node[pyraf] at (1-\tric,0.5,1-0.5*\tric) {$\eqrefh{eqn:bb-triangle-rs-Groth}$};
\node[pyrat] at (1-\tric,1-0.5*\tric,0.5) {$\eqrefh{eqn:Groth-rs}$};
% visible 0- and 1-cells
\node (fru) at (1,1,1) {$\cDb_G(\fg^{\rs})$};
\draw[->] (rru) -- node[arl] {$\al (j_\fg)^*$} (fru);
\draw[<-] (fru) -- node[arl] {$\al {}^\rs\Indgtilde{G}{L}$}(frd);
\draw[->] (l) -- node[arl,pos=.7] {$\al (\mu_{\fg}^{\rs})_! \ubk[\dim \fg]$} (fru);
\end{tikzpicture}}
\end{equation}
where the hidden face on the bottom is labelled by the obvious analogue of \eqref{eqn:Groth-rs}. This means that the following diagram of isomorphisms in $\cDb_G(\fg)$ commutes:
\begin{equation} \label{eqn:final-pentagon}
\vcenter{
\xymatrix@C=8pt@R=12pt{
j_\fg^*\Indg{G}{L}(\Groth_L) \ar[d]_-{{\scriptscriptstyle(\text{II})}\wr} \ar[rrrr]^-{\overset{(\text{I})}{\sim}} &&&& j_\fg^*\Groth_G \ar[d]^-{\wr{\scriptscriptstyle(\text{III})}} \\
{}^\rs\Indg{G}{L} \bigl( (j_\fl')^*\Groth_L \bigr) \ar[rr]^-{\overset{(\text{IV})}{\sim}} && {}^\rs\Indg{G}{L} \bigl( (\mu_\fl^{\rs,\prime})_!\ubk[\dim\fl] \bigr) \ar[rr]^-{\overset{(\text{V})}{\sim}} && (\mu_\fg^\rs)_!\ubk[\dim\fg]
}
}
\end{equation}

All the objects in this diagram are endowed with an action of $\W_L$. 
We want to prove that isomorphism (I) in \eqref{eqn:final-pentagon} is $\W_L$-equivariant. Isomorphism (II) is clearly $\W_L$-equivariant, because it arises from an isomorphism of functors applied to $\Groth_L$.
As remarked above, isomorphism (III) is $\W_G$-equivariant by definition of the $\W_G$-action on $\Groth_G$, and isomorphism (IV) is $\W_L$-equivariant for the same reason. So it suffices to prove that isomorphism (V), namely \eqref{eqn:bb-triangle-rs-Groth}, is $\W_L$-equivariant.

Now \eqref{eqn:bb-triangle-rs-Groth} is by definition the composition
\[
\xymatrix{
{}^\rs\Indg{G}{L} \bigl( (\mu_\fl^{\rs,\prime})_!\ubk[\dim\fl] \bigr) \ar[r]^-{\sim} & (\mu_\fg^\rs)_!{}^\rs\Indgtilde{G}{L}(\bk[\dim\fl]) \ar[r]^-{\sim} & (\mu_\fg^\rs)_!\bk[\dim\fg],
}
\]
where the first isomorphism comes from \eqref{eqn:Indgtilde-rs}, and the second comes from \eqref{eqn:bb-triangle-rs}. The second isomorphism is obviously $\W_G$-equivariant, because the $\W_G$-actions on its domain and codomain come about purely because $\mu_\fg^\rs$ is a Galois covering with group $\W_G$. So it suffices to show that the first isomorphism is $\W_L$-equivariant. Unravelling the definition of this isomorphism similarly, we see that it suffices to prove the $\W_L$-equivariance of the isomorphism $\hamma^G_Pu_!(\mu_\fp^\rs)'_!\ubk[\dim\fl]\simto(\mu_\fg^\rs)_!\hamma^G_Pv_!\ubk[\dim\fl]$ coming from the following pasting diagram:
\[
\vc{\begin{tikzpicture}[smallcube]
\compuinit
% 2-cells
\node[cubef] at (0.75,0.5) {$\Comp$};
\node[cubef] at (2.25,0.5) {$\mInt$};
% 0-cells
\node (lm) at (3,1) {{\tiny $\cDb_G(G \times^B (\fb \cap \fg^{\rs}))$}};
\node (ld) at (1.5,1) {{\tiny $\cDb_P(G \times^B (\fb \cap \fg^{\rs}))$}};
\node (ldd) at (0,1) {{\tiny $\cDb_P(P \times^B (\fb \cap \fg^{\rs}))$}};
\node (rm) at (3,0) {{\tiny $\cDb_G(\fg^{\rs})$}};
\node (rd) at (1.5,0) {{\tiny $\cDb_P(\fg^{\rs})$}};
\node (rdd) at (0,0) {{\tiny $\cDb_P(\fp \cap \fg^{\rs})$}};
% 1-cells
\draw[<-] (lm) -- node[arr] {{\tiny $\hamma_P^G$}} (ld);
\draw[<-] (ld) -- node[arr] {{\tiny $v_!$}} (ldd);
\draw[->] (lm) -- node[arl] {{\tiny $(\mu_{\fg}^{\rs})_!$}} (rm);
\draw[->] (ld) -- node[arl] {{\tiny $(\mu_{\fg}^{\rs})_!$}} (rd);
\draw[->] (ldd) -- node[arr] {{\tiny $(\mu_{\fp}^{\rs,\prime})_!$}} (rdd);
\draw[<-] (rm) -- node[arl] {{\tiny $\hamma_P^G$}} (rd);
\draw[<-] (rd) -- node[arl] {{\tiny $u_!$}} (rdd);
\end{tikzpicture}}
\]
Here, $u$ and $v$ are the inclusions and $\mu_\fp^{\rs,\prime}$ is the obvious restriction of $\mu_{\fp}$, which is a Galois covering with group $\W_L$. This is a special case of Lemma \ref{lem:_!action}.

\subsection{Exactness of $\Springer{G}$} 

As a consequence of our intertwining isomorphism (using the fact that $\ResW{\W_G}{\W_T}$ is exact and faithful and Proposition~\ref{prop:ResN-exact}) we obtain:

\begin{prop} \label{prop:springer-exactness}
The functor $\Springer{G}:\Perv_G(\cN_G,\bk)\to\Rep(\W_G,\bk)$ is exact. 
\end{prop}

%\begin{proof}
%Since $\ResW{\W_G}{\W_T}$ is exact and faithful, it suffices to show that $\ResW{\W_G}{\W_T}\circ\Springer{G}$ is exact. But we now know that $\ResW{\W_G}{\W_T}\circ\Springer{G}$ is isomorphic to $\Springer{T}\circ\ResN{G}{T}$. As seen in \S\ref{sect:plan}, $\Springer{T}$ is an equivalence, and $\ResN{G}{T}$ is exact by Proposition~\ref{prop:ResN-exact}.
%\end{proof}

%%%%%%%%%%%%%%%%%%%%%%%%%%%%%%%%%%%%%%%%%%%%%%%%%%%%%%%%%%%%%%%%%%%%%%%%%%%%%%%%%%%%%%%%

\section{Computations in rank $1$}
\label{sect:rankone}

%%%%%%%%%%%%%%%%%%%%%%%%%%%%%%%%%%%%%%%%%%%%%%%%%%%%%%%%%%%%%%%%%%%%%%%%%%%%%%%%%%%%%%%%

What remains is to prove Theorem~\ref{thm:main-result-again} in the special case where $G$ has semisimple rank~$1$. Since all the functors involved in the statement of Theorem~\ref{thm:main-result-again} are invariant under the replacement of $G$ by $G/Z(G)$, it suffices to consider the case where $G=\PGL(2)$, and we assume this throughout Section~\ref{sect:rankone}.

%--------------------------------------------------------------------------
\subsection{Notation and preliminaries on $\cT_2$}
\label{ss:r1-notation}
%--------------------------------------------------------------------------

For brevity, we will write $\Gr$ for $\Gr_{G}$, $\W$ for $\W_G$, etc.
The nontrivial element of $\W$ is denoted $s$.  Choose $T \subset G$ consisting of images of diagonal matrices, and $B \subset G$ consisting of images of upper-triangular matrices. The coweights (resp.~dominant coweights) of $G$ are naturally identified with $\Z$ (resp.~the nonnegative integers).

For $i=0,1,2$, let $j_i : \Gr^i \hookrightarrow \Gr$ be the inclusion map.  For a finitely-generated $\bk$-module $E$, we write
\[
\IC_i(E) = (j_i)_{!*}(\underline{E}[i]),
\qquad
\Delta_i(E) = \p(j_i)_!(\underline{E}[i]),
\qquad
\nabla_i(E) = \p(j_i)_*(\underline{E}[i]).
\]
These are perverse sheaves supported on $\overline{\Gr^i}$.  Because $\Gr^1 \subset \Gr$ is closed and isomorphic to $\mathbb{P}^1$, there is a canonical isomorphism
\begin{equation}\label{eqn:ic1-constant}
\IC_1(\bk) \cong \ubk_{\Gr^1}[1].
\end{equation}

Set $V:=H^{\bullet}(\IC_1(\bk))$. This is a free $\bk$-module of rank $2$. Moreover, the action of $\Gv$ on $V$ defines a canonical isomorphism
\[
\Gv \xrightarrow{\sim} \SL(V).
\]
The torus $\Tv$ is the subgroup of $\Gv$ consisting of elements acting on $V$ compatibly with the grading.
By definition, the category $\Rep(\Gv,\bk)_{\sm}$ is the category of $\Gv$-modules whose $\Tv$-weights belong to $\{-2,0,2\}$, and $\Gr^{\sm}=\Gr^0 \sqcup \Gr^2$.

The following object will play a key role throughout this section:
\[
\cT_2 := \IC_1(\bk) \star \IC_1(\bk).
\]
Since $\Satake{G}$ is a tensor functor, we have $\Satake{G}(\cT_2)\cong V\otimes V$, which clearly belongs to $\Rep(\Gv,\bk)_{\sm}$. Let
$\eta: \cT_2 \to \cT_2$
be the involution induced by the commutativity constraint on $\Perv_{\GO}(\Gr)$, i.e.~the unique endomorphism of $\cT_2$ such that
 \[
\Satake{G}(\eta): V \otimes V \to V \otimes V
\qquad\text{is given by} \qquad
x \otimes y \mapsto y \otimes x.
\]
The involution $\eta$ defines a $\W$-action on $\cT_2$, and hence a functor $\bT = \Hom(\cT_2,-): \Perv_{\GO}(\Gr^\sm,\bk) \to \Rep(\W,\bk)$.

We now recall the definition of $\eta$ (see~\cite{mv}). The construction involves global versions of the affine Grassmannian.
Consider the diagonal embedding $\bA^1 \to \bA^2$, and let $U \subset \bA^2$ be its complement.  Let $\W$ act on $\bA^2$ by exchanging the two copies of $\bA^1$, and let $\bA^{(2)} = \bA^2/\W$.  Finally, let $U' = U/\W \subset \bA^{(2)}$.  We have the following commutative diagram in which every square is cartesian.
\begin{equation}\label{eqn:commconstraint-diag}
\vc{\xymatrix@R=13pt{
\Gr^1 \wtimes \Gr^1 \ar[r]^-{\tilde e}\ar[d]_-{m} &
  \Gr^1_{\bA^1} \wtimes \Gr^1_{\bA^1} \ar[d]_-{m'} &
  (\Gr^1_{\bA^1} \times \Gr^1_{\bA^1})|_U \ar[l]_-{\tilde u} \ar@{=}[d] \\
\Gr^\sm \ar[r]^-{e'}\ar@{=}[d] &
  \Gr^\sm_{\bA^2} \ar[d]_-{\varpi'} &
  (\Gr^1_{\bA^1} \times \Gr^1_{\bA^1})|_U \ar[l]_-{u'} \ar[d]^-{\varpi} \\
\Gr^\sm \ar[r]_-{e} &
  \Gr^\sm_{\bA^{(2)}} &
  \Gr^\sm_{U'} \ar[l]^{u} }}
\end{equation}
Here, $(\Gr^1_{\bA^1} \times \Gr^1_{\bA^1})|_U$ denotes the preimage of $U \subset \bA^2$ under the natural map $\Gr^1_{\bA^1} \times \Gr^1_{\bA^1} \to \bA^2$.  This diagram is explained in a general setting in~\cite[\S5]{mv}.  For a concrete description in the case of $\PGL(2)$, see the proof of Lemma~\ref{lem:rankone-comm} below.  

Let $\sigma:(\Gr^1_{\bA^1} \times \Gr^1_{\bA^1})|_U\to (\Gr^1_{\bA^1} \times \Gr^1_{\bA^1})|_U$ be the involution of swapping the factors, and let $\sigma': \Gr^\sm_{\bA^2} \to \Gr^\sm_{\bA^2}$ be the involution induced by the $\W$-action on $\bA^2$. We have $\sigma'e'=e'$ and $\sigma'u'=u'\sigma$. By definition, $\cT_2 = m_!(\IC_1(\bk) \tboxtimes \IC_1(\bk)) \cong m_!(\ubk_{\Gr^1\wtimes\Gr^1})[2]$.
By base change, we obtain an isomorphism 
\[
\cT_2 \cong (e')^*(m')_!(\ubk_{\Gr^1_{\bA^1}\wtimes\Gr^1_{\bA^1}})[2].
\]
Since $m'$ is small and proper, this gives rise to an isomorphism
\begin{equation} \label{eqn:e'pu'}
\cT_2 \cong (e')^* u'_{!*}\bigl(\ubk_{(\Gr^1_{\bA^1} \times \Gr^1_{\bA^1})|_U}[4]\bigr)[-2].
\end{equation} 
The natural isomorphism $\ubk_{(\Gr^1_{\bA^1} \times \Gr^1_{\bA^1})|_U} \cong \sigma^*\ubk_{(\Gr^1_{\bA^1} \times \Gr^1_{\bA^1})|_U}$ induces an isomorphism 
\begin{equation} \label{eqn:sigma'pu'}
u'_{!*}(\ubk_{(\Gr^1_{\bA^1} \times \Gr^1_{\bA^1})|_U}[4])\cong (\sigma')^*u'_{!*}(\ubk_{(\Gr^1_{\bA^1} \times \Gr^1_{\bA^1})|_U}[4]).
\end{equation}
Then the involution $\eta$ is the composition
\begin{multline*}
\cT_2\;\overset{\eqref{eqn:e'pu'}}{\cong}\; (e')^*u'_{!*}\bigl(\ubk_{(\Gr^1_{\bA^1} \times \Gr^1_{\bA^1})|_U}[4]\bigr)[-2] \;\overset{\eqref{eqn:sigma'pu'}}{\cong} \;(e')^*(\sigma')^*u'_{!*}\bigl(\ubk_{(\Gr^1_{\bA^1} \times \Gr^1_{\bA^1})|_U}[4]\bigr)[-2]\\
\overset{\Co}{\cong} \;(e')^*u'_{!*}\bigl(\ubk_{(\Gr^1_{\bA^1} \times \Gr^1_{\bA^1})|_U}[4]\bigr)[-2] \;
\overset{\eqref{eqn:e'pu'}}{\cong} \;\cT_2.
\end{multline*}

It is convenient to have an alternative description of $\eta$. By base change and using the fact that $\varpi'$ is a finite morphism, \eqref{eqn:e'pu'} can be rewritten as
\[
\cT_2 \cong e^*(\varpi')_!u'_{!*}\bigl(\ubk_{(\Gr^1_{\bA^1} \times \Gr^1_{\bA^1})|_U}[4]\bigr)[-2] \cong e^* u_{!*}\bigl(\varpi_!\ubk_{(\Gr^1_{\bA^1} \times \Gr^1_{\bA^1})|_U}[4]\bigr)[-2].
\]

\begin{lem}\label{lem:T2-W-action}
Consider the involution of $\varpi_!\ubk_{(\Gr^1_{\bA^1} \times \Gr^1_{\bA^1})|_U}$ resulting from the natural isomorphism $\ubk_{(\Gr^1_{\bA^1} \times \Gr^1_{\bA^1})|_U} \cong \sigma^*\ubk_{(\Gr^1_{\bA^1} \times \Gr^1_{\bA^1})|_U}$. The induced involution of $\cT_2$ is $\eta$.
\end{lem}
\begin{proof}
This follows by applying 
Lemmas~\ref{lem:basechangecomposition} and~\ref{lem:compbasechange} to the diagram
\[
\vc{\begin{tikzpicture}[xscale=1.9,smallcube]
\compuinit
% hidden 0-cell
\node (rmd) at (1,0,0) {$\Gr^\sm_{\bA^{(2)}}$};
\node (rrd) at (2,0,0) {$\Gr^\sm_{U'}$};
% outer 0-cells
\node (rlu) at (0,1,0) {$\Gr^\sm$};
\node (rmu) at (1,1,0) {$\Gr^\sm_{\bA^2}$};
\node (rru) at (2,1,0) {$(\Gr^1_{\bA^1} \times \Gr^1_{\bA^1})|_U$};

\node (rld) at (0,0,0) {$\Gr^\sm$};
\node (fru) at (2,1,1) {$(\Gr^1_{\bA^1} \times \Gr^1_{\bA^1})|_U$};

\node (fld) at (0,0,1) {$\Gr^\sm$};
\node (fmd) at (1,0,1) {$\Gr^\sm_{\bA^{(2)}}$};
\node (frd) at (2,0,1) {$\Gr^\sm_{U'}$};
% hidden 1-cells
\draw[->] (rld) -- node[arl,pos=.3] {$\al e$} (rmd);
\draw[->] (rmu) -- node[arl,pos=.7] {$\al \varpi'$} (rmd);
\draw[-,double distance=1.5pt] (rmd) -- node[arl] {} (fmd);
\draw[<-] (rmd) -- node[arr,pos=.8] {$\al u$} (rrd);
\draw[->] (rru) -- node[arl,pos=.7] {$\al \varpi$} (rrd);
\draw[-,double distance=1.5pt] (rrd) -- node[arl] {} (frd);
% outer 1-cells
\draw[->] (rlu) -- node[arl,pos=.3] {$\al e'$} (rmu);
\draw[<-] (rmu) -- node[arl] {$\al u'$} (rru);

\draw[->] (rru) -- node[arr] {$\al \sigma$} (fru);
\draw[->] (fru) -- node[arl] {$\al \varpi$} (frd);

\draw[<-] (fmd) -- node[arr] {$\al u$} (frd);
\draw[->] (fld) -- node[arr] {$\al e$} (fmd);
\draw[-,double distance=1.5pt] (rld) -- (fld);
\draw[-,double distance=1.5pt] (rlu) -- (rld);
% visible 0- and 1-cells
\node (flu) at (0,1,1) {$\Gr^\sm$};
\node (fmu) at (1,1,1) {$\Gr^\sm_{\bA^2}$};
\draw[-,double distance=1.5pt] (rlu) -- (flu);
\draw[linef,-,double distance=1.5pt] (flu) -- (fld);
\draw[linef,->] (rmu) -- node[arl] {$\al \sigma'$} (fmu);
\draw[linef,->] (fmu) -- node[arl,pos=.3] {$\al \varpi'$} (fmd);
\draw[linef,->] (flu) -- node[arl,pos=.3] {$\al e'$} (fmu);
\draw[linef,<-] (fmu) -- node[arl,pos=.3] {$\al u'$} (fru);
\end{tikzpicture}}
\]
in which every square is cartesian.
\end{proof}

%--------------------------------------------------------------------------
\subsection{Geometric properties of $\cT_2$}
\label{ss:r1-geometric}
%--------------------------------------------------------------------------

For $\PGL(2)$, the map $\pi: \cM \to \cN$ is an isomorphism of varieties.  In this section, we will identify $\cM$ with $\cN$ via this map.  Then we can extend the embedding $j: \cN \to \Gr^\sm$ to a `global' version.  Note that $\cN$ can also be identified with the nilpotent cone in the Lie algebra $\fgl(2)$ of the $\GL(2)$, and that $\PGL(2)$ acts on $\fgl(2)$. In the following lemma we denote by $\tglt$ the Grothendieck--Springer resolution (see \S\ref{ss:functor-Springer}) for the group $\GL(2)$.

\begin{lem}\label{lem:rankone-comm}
There is a commutative diagram of $\PGL(2)$-equivariant maps
\[
\vc{\begin{tikzpicture}[xscale=1.9,smallcube]
\compuinit
% hidden 0-cell
\node (rmd) at (1,0,0) {$\fgl(2)$};
\node (rrd) at (2,0,0) {$\fgl(2)^{\rs}$};
% outer 0-cells
\node (rlu) at (0,1,0) {$\tcN$};
\node (rmu) at (1,1,0) {$\tglt$};
\node (rru) at (2,1,0) {$\tglt^{\rs}$};

\node (rld) at (0,0,0) {$\cN$};
\node (fru) at (2,1,1) {$(\Gr^1_{\bA^1} \times \Gr^1_{\bA^1})|_U$};

\node (fld) at (0,0,1) {$\Gr^\sm$};
\node (fmd) at (1,0,1) {$\Gr^\sm_{\bA^{(2)}}$};
\node (frd) at (2,0,1) {$\Gr^\sm_{U'}$};
% hidden 1-cells
\draw[->] (rld) -- node[arl,pos=.3] {$\al i_{\fgl(2)}$} (rmd);
\draw[->] (rmu) -- node[arl,pos=.7] {$\al \mu_{\fgl(2)}$} (rmd);
\draw[->] (rmd) -- node[arl] {$\al j'$} (fmd);
\draw[<-] (rmd) -- node[arr,pos=.8] {$\al h$} (rrd);
\draw[->] (rru) -- node[arl,pos=.7] {$\al \mu_{\fgl(2)}^\rs$} (rrd);
\draw[->] (rrd) -- node[arl] {$\al j'$} (frd);
% outer 1-cells
\draw[->] (rlu) -- node[arl,pos=.3] {$\al i_{\tglt}$} (rmu);
\draw[<-] (rmu) -- node[arl] {$\al \tilde h$} (rru);

\draw[->] (rru) -- node[arl] {$\al \tilde\jmath$} (fru);
\draw[->] (fru) -- node[arl] {$\al \varpi$} (frd);

\draw[<-] (fmd) -- node[arr] {$\al u$} (frd);
\draw[->] (fld) -- node[arr] {$\al e$} (fmd);
\draw[->] (rld) -- node[arr] {$\al j$} (fld);
\draw[->] (rlu) -- node[arr] {$\al \mu_{\cN}$} (rld);
% visible 0- and 1-cells
\node (flu) at (0,1,1) {$\Gr^1 \wtimes \Gr^1$};
\node (fmu) at (1,1,1) {$\Gr^1_{\bA^1} \wtimes \Gr^1_{\bA^1}$};
\draw[->] (rlu) -- node[arr] {$\al \tilde\jmath$} (flu);
\draw[linef,->] (flu) -- node[arl,pos=.3] {$\al m$} (fld);
\draw[linef,->] (rmu) -- node[arl] {$\al \tilde\jmath$} (fmu);
\draw[linef,->] (fmu) -- node[arl,pos=.3] {$\al \varpi' m'$} (fmd);
\draw[linef,->] (flu) -- node[arl,pos=.3] {$\al \tilde e$} (fmu);
\draw[linef,<-] (fmu) -- node[arl,pos=.3] {$\al \tilde u$} (fru);
\end{tikzpicture}}
\]
Every square in this diagram is cartesian.  Moreover, the isomorphism
\begin{equation}\label{eqn:spr-T2-isom}
\Spr \cong \Psi_G(\cT_2)
\end{equation}
defined using base change for the left-most square is $\W$-equivariant.
\end{lem}
\begin{proof}
We give only a brief sketch of the argument.  (A closely related result for $\GL(n)$ is proved in~\cite{mautner2} using earlier constructions in \cite{mvy}.)  We start by interpreting the various affine Grassmannians in terms of lattices.  Specifically, let $\cL_0 := \fO^2 \subset \fK^2$ be the standard  $\fO$-lattice in $\fK^2$ with natural basis $(e_1,e_2)$.  We have identifications
\begin{gather*}
\Gr^{\sm} = \overline{\Gr^2} = \{ \cL_2 \subset \fK^2 \mid \text{$\cL_2 \subset t^{-1} \cL_0$ and $\dim(t^{-1} \cL_0 / \cL_2)=2$} \}, \\
\Gr^1 \wtimes \Gr^1 = \{(\cL_1,\cL_2) \mid \text{$\cL_2 \subset \cL_1 \subset t^{-1} \cL_0$, $\dim(\cL_1/\cL_2)=\dim(t^{-1} \cL_0 / \cL_1)=1$} \}
\end{gather*}
(where the $\cL_i$'s are implicitly required to be $\fO$-lattices). The image of the embedding $j: \cN \to \Gr^\sm$ is given by
\[
\cN \cong \{ \cL_2 \in \Gr^{\sm} \mid \text{the images of $t^{-1} e_1$ and $t^{-1} e_2$ form a basis of $t^{-1} \cL_0 / \cL_2$} \}.
\]
The global versions can be described using $\C[t]$-lattices in $\C(t)^{2}$.  Let $\mathfrak{L}_0:=\C[t]^2$ be the standard lattice.  We have:
\begin{gather*}
\Gr^{\sm}_{\bA^{(2)}} = \{ \mathfrak{L}_2 \subset \C(t)^2 \mid \text{$\mathfrak{L}_2 \subset t^{-1} \mathfrak{L}_0$ and $\dim(t^{-1} \mathfrak{L}_0 / \mathfrak{L}_2)=2$} \}, \\
\Gr^1_{\bA^1} \wtimes \Gr^1_{\bA^1} = \{(\mathfrak{L}_1, \mathfrak{L}_2) \mid \text{$\mathfrak{L}_2 \subset \mathfrak{L}_1 \subset t^{-1} \mathfrak{L}_0$, $\dim(\mathfrak{L}_1 / \mathfrak{L}_2)=\dim(t^{-1} \mathfrak{L}_0 / \mathfrak{L}_1)=1$} \}
\end{gather*}
(where $\mathfrak{L}_i$'s are required to be $\C[t]$-lattices). It is left to the reader to supply explicit descriptions for the images of $j'$ and $\tilde\jmath$ and for the maps $e$ and $\tilde e$.  It follows from those descriptions that the left-hand cube is commutative and that each square in it is cartesian.  The same holds for the right-hand cube because it is obtained by forming pullbacks with respect to the inclusion $U' \to \bA^{(2)}$.

Finally, recall that the $\W$-action on $\Spr$ is defined using its action on $\tglt^{\rs}$.  Since this is just the restriction of the $\W$-action on $(\Gr^1_{\bA^1} \times \Gr^1_{\bA^1})|_U$, it can be seen from Lemma~\ref{lem:T2-W-action} and several applications of Lemmas~\ref{lem:basechangecomposition} and~\ref{lem:compbasechange} that the isomorphism~\eqref{eqn:spr-T2-isom} is $\W$-equivariant.
\end{proof}

\begin{lem}\label{lem:psi-ffaithful}
The functor $\Psi_G: \Perv_{\GO}(\Gr^\sm,\bk) \to \Perv_G(\cN,\bk)$ is fully faithful.
\end{lem}
\begin{proof}
Let $Z \subset \Gr^\sm$ be the complement of the open set $j(\cN) \subset \Gr^\sm$.  This is a closed, $G$-stable (but not $\GO$-stable) subset of $\Gr^2$.  It is well known that $j_{!*}: \Perv_G(\cN,\bk) \to \Perv_G(\Gr^\sm,\bk)$ is fully faithful, and that its essential image is the full subcategory $\sP^Z \subset \Perv_G(\Gr^\sm,\bk)$ of perverse sheaves with no quotient or subobject supported on $Z$.  Moreover, $j^!$ is left inverse to $j_{!*}$.  In particular, $j^!|_{\sP^Z}$ is fully faithful.  It is clear that $\Perv_{\GO}(\Gr^\sm,\bk) \subset \sP^Z$, so the result follows.
\end{proof}

In fact, $\Psi_G$ is an equivalence (see~\cite[Theorem~4.1]{mautner2}), but we will not need this stronger result.  Lemmas~\ref{lem:rankone-comm} and \ref{lem:psi-ffaithful} have the following immediate consequence.

\begin{cor}\label{cor:T-isom}
There is a natural isomorphism of functors $\bT \natisom \Springer{G} \circ \Psi_G$.
\end{cor}

%--------------------------------------------------------------------------
\subsection{Algebraic properties of $\cT_2$}
\label{ss:r1-algebraic}
%--------------------------------------------------------------------------

It is well known that $\cT_2$ is a tilting object. In particular, we have two exact sequences of perverse sheaves
\begin{equation}
\label{eqn:exact-sequence-T2}
\Delta_2(\bk) \hookrightarrow \cT_2 \twoheadrightarrow \Delta_0(\bk)
\qquad\text{and}\qquad
\nabla_0(\bk) \hookrightarrow \cT_2 \twoheadrightarrow \nabla_2(\bk).
\end{equation}
The representations corresponding to these perverse sheaves under the Satake equivalence are described as follows. We have $\Satake{G}(\Delta_0(\bk))=\Satake{G}(\nabla_0(\bk))\cong\bk$ (the trivial representation), and $\Satake{G}(\cT_2)\cong V\otimes V$. The sub-representation $\Satake{G}(\Delta_2(\bk))$ of $\Satake{G}(\cT_2)$ consists of the symmetric tensors in $V\otimes V$, i.e.~the invariant submodule of $\Satake{G}(\eta)$. The quotient $\Satake{G}(\nabla_2(\bk))$ of $\Satake{G}(\cT_2)$ is the symmetric square $S^2(V)$.

\begin{lem}
\label{lem:proj-generator-SL2}
The object $\cT_2 \oplus \Delta_2(\bk)$ is a projective generator of $\Perv_{\GO}(\Gr^{\sm},\bk)$.
\end{lem}
\begin{proof}
By Proposition~\ref{prop:springer-exactness}, $\Springer{G} = \Hom(\Spr,{-})$ is exact, so $\Spr$ is a projective object in $\Perv_G(\cN,\bk)$.  It follows from Lemmas~\ref{lem:rankone-comm} and~\ref{lem:psi-ffaithful} that $\cT_2$ is a projective object in $\Perv_{\GO}(\Gr^{\sm},\bk)$. Next, for any object $M$ in $\Perv_{\GO}(\Gr^{\sm},\bk)$, we have
\[
\Hom(\Delta_2(\bk),M) \cong \Hom(\ubk_{\Gr^2}[2], \p (j_2)^! M)
\]
by adjunction.  
It easily follows that $\Delta_2(\bk)$ is projective.

Now by \cite[Lemma 2.1.4]{rsw}, any object of $\Perv_{\GO}(\Gr^{\sm},\bk)$ is a successive extension of objects of the form $\IC_i(E)$ for $i \in \{0,2\}$ and $E$ a finitely-generated $\bk$-module. Hence to finish the proof, it suffices
to prove the following claim: For any finitely-generated $\bk$-module $E$ and any $i \in \{0,2\}$, there exists $n \in \Z_{\geq 0}$ and a surjection $\bigl( \cT_2 \oplus \Delta_2(\bk) \bigr)^{\oplus n} \twoheadrightarrow \IC_i(E)$. As the functor $\IC_i(-)$ preserves surjections, it is enough to prove this when $E=\bk$. However, by definition we have a surjection $\Delta_2(\bk) \twoheadrightarrow \IC_2(\bk)$, and by \eqref{eqn:exact-sequence-T2} there is a surjection $\cT_2 \twoheadrightarrow \IC_0(\bk)$.
\end{proof}

\begin{lem}\label{lem:end-T2}
\begin{enumerate}
\item The action map $\bk \W \to \End(\cT_2)$ is an isomorphism.\label{it:springer}
\item The object $\bT(\Delta_2(\bk)) \in \Rep(\W,\bk)$ is a free $\bk$-module of rank one with trivial $\W$-action.\label{it:T2-delta}
\item The object $\bT(\cT_2) \in \Rep(\W,\bk)$ is a free $\bk$-module of rank two on which $s\in\W$ acts as $\bT(\eta)$.\label{it:T2-s-action}
\end{enumerate}
\end{lem}
\begin{proof}
\eqref{it:springer}~Using the two exact sequences~\eqref{eqn:exact-sequence-T2} together with adjunction and the fact that $\cT_2$ is projective, we find an exact sequence
\[
0 \to \Hom(\Delta_0(\bk), \nabla_0(\bk)) \to \End(\cT_2) \to \Hom(\Delta_2(\bk), \nabla_2(\bk)) \to 0.
\]
We also have $\Hom(\Delta_0(\bk), \nabla_0(\bk)) \cong \Hom(\Delta_2(\bk), \nabla_2(\bk)) \cong \bk$ by adjunction, so it follows that $\End(\cT_2)$ is a free $\bk$-module of rank two.  It is spanned by the identity map together with the composition $c: \cT_2 \to \cT_2$ given by
\[
\cT_2 \twoheadrightarrow \Delta_0(\bk) = \nabla_0(\bk) \hookrightarrow \cT_2.
\]
It is easy to see from the description of the corresponding representations that $c$ is (up to multiplication by a unit) the action of $1-s \in \bk \W$. The result follows.

\eqref{it:T2-delta}~By adjunction, we have $\Hom(\Delta_2(\bk), \Delta_0(\bk)) = 0$.  It then follows from the first short exact sequence in~\eqref{eqn:exact-sequence-T2} that we have an isomorphism $\Hom(\Delta_0(\bk),\Delta_0(\bk)) \xrightarrow{\sim} \Hom(\cT_2, \Delta_0(\bk))$.  In particular, the last term in the following short exact sequence is a free $\bk$-module of rank one:
\[
0 \to \Hom(\cT_2,\Delta_2(\bk)) \to \End(\cT_2) \overset{p}{\to} \Hom(\cT_2,\Delta_0(\bk)) \to 0.
\]
Thus, $\Hom(\cT_2,\Delta_2(\bk))$ is identified with $\ker p$, or, equivalently, with $\ker i \circ p$, where $i$ is the injective map $\Hom(\cT_2,\Delta_0(\bk)) \to \Hom(\cT_2,\cT_2)$ induced by the inclusion $\Delta_0(\bk) = \nabla_0(\bk) \hookrightarrow \cT_2$.  Now, $i \circ p: \End(\cT_2) \to \End(\cT_2)$ is induced by composition with the map $c$ defined above.  It follows that
\[
\Hom(\cT_2,\Delta_2(\bk)) \cong \{ a \in \bk \W \mid (1-s)a = 0 \} = \bk \cdot (1+s) \subset \bk \W.
\]
Thus, $\Hom(\cT_2,\Delta_2(\bk))$ is free of rank one over $\bk$, and $\W$ acts on it trivially.

\eqref{it:T2-s-action}~By definition, $\bT(\cT_2) = \Hom(\cT_2,\cT_2)$ is isomorphic to $\bk \W$ by~\eqref{it:springer}. The action of $s$ on $\bT(\cT_2)$ comes from applying $\eta$ to the first copy of $\cT_2$ in $\Hom(\cT_2,\cT_2)$, so it corresponds to right multiplication by $s$ on $\bk \W$. The action of $\bT(\eta)$ on $\bT(\cT_2)$ comes from applying $\eta$ to the second copy of $\cT_2$ in $\Hom(\cT_2,\cT_2)$, so it corresponds to left multiplication by $s$ on $\bk \W$. Since $\bk\W$ is commutative, these are the same.
\end{proof}

An easy calculation yields the following fact.

\begin{lem}
\label{lem:action-s-SL2}
The restriction of $\Satakesm{G}(\eta):V \otimes V \to V \otimes V$ to $(V \otimes V)^{\Tv}$ is the action of $s\in\W$ on $\Phi_{\Gv}(V \otimes V)$.
\end{lem}

%--------------------------------------------------------------------------
\subsection{Proof of Theorem~\ref{thm:main-result-again} for $G=\PGL(2)$}
%--------------------------------------------------------------------------

As in \S\ref{subsect:main-proof}, we have an isomorphism
\[
\phi : \For^{\W} \circ \Phi_{\Gv} \circ \Satakesm{G} \directednatisom \For^{\W} \circ \Springer{G} \circ \Psi_G.
\]
All we need to show is that for each object $M \in \Perv_{\GO}(\Gr^\sm_G,\bk)$, the map of $\bk$-modules $\phi_M$ is actually $\W$-equivariant. Let $\psi_M: \Springer{G}(\Psi_G(M)) \to \bT(M)$ be the isomorphism deduced from Corollary~\ref{cor:T-isom}.  By definition $\psi_M$ is $\W$-equivariant, so it suffices to show that the composition 
\[
\phi'_M = \For^{\W}(\psi_M) \circ \phi_M: \For^{\W}(\Phi_{\Gv}(\Satakesm{G}(M))) \to \For^{\W}(\bT(M))
\]
is $\W$-equivariant.  The functors $\Phi_{\Gv} \circ \Satakesm{G}$ and $\bT$ are exact, so by Lemma~\ref{lem:proj-generator-SL2}, it is enough to prove this for $M = \cT_2$ and $M = \Delta_2(\bk)$.

Suppose first that $M = \Delta_2(\bk)$.  One can easily check that $\Phi_{\Gv}(\Satakesm{G}(\Delta_2(\bk)))$ is the trivial $\W$-module (free of rank one over $\bk$). The same description applies to $\bT(\Delta_2(\bk))$ by Lemma~\ref{lem:end-T2}\eqref{it:T2-delta}, so any morphism of $\bk$-modules $\Phi_{\Gv}(\Satakesm{G}(\Delta_2(\bk))) \to \bT(\Delta_2(\bk))$ is $\W$-equivariant.

Now suppose that $M = \cT_2$.  Since $\phi'$ is a morphism of functors, we have
\begin{equation}\label{eqn:phip-intertwine-s}
\phi'_{\cT_2} \circ \For^\W(\Phi_{\Gv}(\Satakesm{G}(\eta))) = \For^\W(\bT(\eta)) \circ \phi'_{\cT_2}.
\end{equation}
By Lemmas~\ref{lem:action-s-SL2} and~\ref{lem:end-T2}\eqref{it:T2-s-action}, the maps $\Phi_{\Gv}(\Satakesm{G}(\eta))$ and $\bT(\eta)$ each coincide with the action of $s$ on the appropriate object.  Thus,~\eqref{eqn:phip-intertwine-s} says that $\phi'_{\cT_2}$ commutes with the action of $s$, as desired. 

%%%%%%%%%%%%%%%%%%%%%%%%%%%%%%%%%%%%%%%%%%%%%%%%%%%%%%%%%%%%%%%%%%%%%%%%%%%%%%%%%%%%%%%%%%%

\appendix

%%%%%%%%%%%%%%%%%%%%%%%%%%%%%%%%%%%%%%%%%%%%%%%%%%%%%%%%%%%%%%%%%%%%%%%%%%%%%%%%%

\section{Commutative diagrams in $2$-categories}
\label{sect:cubes}

Many of the arguments in this paper require us to keep track of equalities of natural isomorphisms of functors, which means that we are effectively working in the $2$-category $\sCat$ (see \cite[\S XII.3]{maclane}, \cite{kellystreet}). To carry out computations in this setting, we need some basic facts about commutative diagrams in $2$-categories.  

We apologize to category theorists for the informality and narrowness of our exposition. The `correct' level of generality is that of Power's $n$-categorical pasting theorem~\cite{powern}, but the cases of that result that we need are so special that explaining them in their own right is easier than explaining how to see them as special cases.

\subsection{The definition of commutativity}
\label{subsect:commutativity}

Let us first review the definition of a commutative diagram in ordinary category theory. A diagram in a category $\cA$ 
is a pair $(\Gamma,f)$, where $\Gamma$ is a finite directed graph and $f$ is a \emph{labelling} of $\Gamma$ in $\cA$: to every vertex $v$ of $\Gamma$ we assign an object $f(v)$ of $\cA$, and to every arc $e$ with source $v$ and target $v'$ we assign a morphism $f(e):f(v)\to f(v')$.
If $\gamma$ is a directed path in $\Gamma$ with initial vertex $v_1$ and final vertex $v_2$, then the labelling $f$ 
defines a morphism $f(\gamma):f(v_1)\to f(v_2)$, namely the composite of the labels of all the arcs in the path. One says that the diagram $(\Gamma,f)$ is \emph{commutative} if, for any two directed paths $\gamma,\gamma'$ in $\Gamma$ with the same initial and final vertices, we have $f(\gamma)=f(\gamma')$.

The $2$-categorical analogues of these concepts are as follows. A diagram in a $2$-category $\cA$ is a triple $(\Gamma,\Delta,f)$, where $(\Gamma,\Delta)$ is a \emph{$2$-computad} and $f$ is a labelling of $(\Gamma,\Delta)$ in $\cA$. Here, following~\cite{powern}, a $2$-computad $(\Gamma,\Delta)$ is a pair of finite directed graphs where the vertex set of $\Delta$ is a subset of the set of directed paths of $\Gamma$, and every arc of $\Delta$ joins two directed paths with the same initial and final vertices. A labelling $f$ of $(\Gamma,\Delta)$ in $\cA$ is a labelling of $\Gamma$ in the underlying $1$-category of $\cA$,
together with a $2$-cell $f(\eta):f(\gamma)\Rightarrow f(\gamma')$ for every arc $\eta$ of $\Delta$ whose source is the directed path $\gamma$ of $\Gamma$ and whose target is $\gamma'$.

Among all $2$-computads, the \emph{$2$-pasting schemes} play the role that directed paths play among directed graphs, in that they describe the valid ways to define a composite of $2$-cells, allowing a mix of `horizontal' and `vertical' composition; see~\cite[Definition 2.2]{powern} for the precise definition. Up to isomorphism, any $2$-pasting scheme $(\Gamma,\Delta)$ arises from a polygonal decomposition of a convex polygon in $\mathbb{R}^2$, as follows:
\begin{itemize}
\item $\Gamma$ consists of the vertices and edges of the polygons, with every edge oriented in the direction of increasing $x$-coordinate (assume that no two vertices have the same $x$-coordinate);
\item there is one arc of $\Delta$ for every interior polygon, joining the two directed paths that make up the boundary of that polygon, and oriented in the direction of decreasing $y$-coordinate.
\end{itemize}
The boundary of the exterior polygon is the union of two directed paths with the same initial and final vertices. We call these paths the domain and codomain of $(\Gamma,\Delta)$, where the domain is the one with higher $y$-coordinates. (We are using $x$- and $y$-coordinates just to establish consistent orientations, and they do not always correlate with the horizontal and vertical directions in our pictures.)

\begin{ex} \label{ex:pasting-scheme}
The following is an example of a $2$-pasting scheme, where dots and single arrows represent $\Gamma$, and double arrows represent the arcs of $\Delta$:
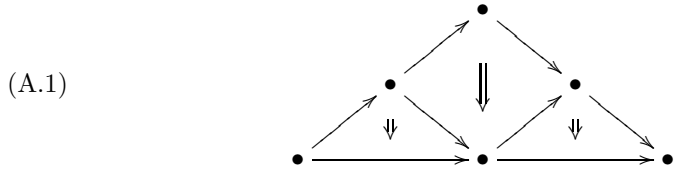
\begin{equation}
\vcenter{\xymatrix@R=18pt{
&&\bullet\ar[dr]\ar@{|o|=>|o|}[dd]\\
&\bullet\ar[ur]\ar[dr]\ar@{| |=>| |}[d]&&\bullet\ar[dr]\ar@{| |=>| |}[d]\\
\bullet\ar[ur]\ar[rr]&&\bullet\ar[ur]\ar[rr]&&\bullet
}}
\end{equation}
\end{ex}

It is shown in~\cite[Theorem 3.3]{power} (see also~\cite[Theorem 2.7]{powern}) that any labelling $f$ of a $2$-pasting scheme $(\Gamma,\Delta)$
defines a unique composite $2$-cell $f(\Gamma,\Delta):f(\alpha)\Rightarrow f(\beta)$ where $\alpha$ and $\beta$ are the domain and codomain of $(\Gamma,\Delta)$. We refer to a diagram $(\Gamma,\Delta,f)$ where $(\Gamma,\Delta)$ is a $2$-pasting scheme simply as a \emph{pasting diagram}. 

In displaying pasting diagrams, we often indicate the arcs of $\Delta$ not by double arrows but by shaded polygons on which a label (or reference number) can be displayed more conveniently. This creates ambiguity about which is the domain and which is the codomain of the $2$-pasting scheme, but it does not matter since we use this method of display only when the $2$-cells under consideration are invertible. 
\begin{ex} \label{ex:pasting-diagram}
A labelling of the $2$-pasting scheme of Example~\ref{ex:pasting-scheme} in a $2$-category $\cA$ might be depicted as:
\begin{equation}
\vc{\begin{tikzpicture}[xscale=2,yscale=1.6]
\compuinit
% 2-cells
\node[lttricelld] at (0.5+0.5*\tric, 1.5+0.5*\tric) {$\chi$};
\node[lttricelld] at (1.5+0.5*\tric, 0.5+0.5*\tric) {$\psi$};
\node[cubef] at (1.5,1.5) {$\omega$};
% 0-cells
\node (lu) at (0,2) {$A$};
\node (mu) at (1,2) {$B$};
\node (ru) at (2,2) {$C$};
\node (mm) at (1,1) {$D$};
\node (rm) at (2,1) {$E$};
\node (rd) at (2,0) {$F$};
% 1-cells
\draw[->] (ru) -- node[arl] {$\epsilon$} (rm);
\draw[->] (lu) -- node[arl] {$\alpha$} (mu);
\draw[->] (lu) -- node[arr] {$\gamma$} (mm);
\draw[->] (mu) -- node[arl] {$\beta$} (ru);
\draw[->] (mu) -- node[arr] {$\delta$} (mm);
\draw[->] (rm) -- node[arl] {$\theta$} (rd);
\draw[->] (mm) -- node[arl] {$\zeta$} (rm);
\draw[->] (mm) -- node[arr] {$\eta$} (rd);
\end{tikzpicture}}
\end{equation}
Here, 
$A, \dots, F$ denote $0$-cells of $\cA$, and 
$\alpha, \dots, \theta$ denote $1$-cells of $\cA$ with domains and codomains as indicated. In one of the two possible interpretations of the picture, the named $2$-cells are 
\[ 
\chi:\delta\circ\alpha\Rightarrow\gamma,\qquad 
\psi:\theta\circ\zeta\Rightarrow\eta\qquad\text{and}\qquad
\omega:\epsilon\circ\beta\Rightarrow\zeta\circ\delta,
\]
and the composite $2$-cell defined by the pasting diagram has domain $\theta\circ\epsilon\circ\beta\circ\alpha$ and codomain $\eta\circ\gamma$.  In the other interpretation, the domains and codomains of all $2$-cells are switched. If we replace each of $\chi$, $\psi$ and $\omega$ by a symbol indicating an inverse pair of $2$-cells, the two interpretations of the picture define an inverse pair of $2$-cells $\theta\circ\epsilon\circ\beta\circ\alpha\natisom\eta\circ\gamma$.
\end{ex}

We say that a diagram $(\Gamma,\Delta,f)$ in a $2$-category is \emph{commutative} if, for any two sub-$2$-computads $(\gamma,\delta)$ and $(\gamma',\delta')$ of $(\Gamma,\Delta)$, which are both $2$-pasting schemes and have the same domain and codomain, we have $f(\gamma,\delta)=f(\gamma',\delta')$. (The definition of sub-$2$-computad is the obvious one.) 

\subsection{Polyhedral $2$-computads}
\label{subsect:polyhedral}

Apart from $2$-pasting schemes, almost all the $2$-computads encountered in this paper are of a special polyhedral kind, for which the definition of commutativity can be rephrased in simpler terms. 

A convex polyhedron in $\mathbb{R}^3$ (or rather, its boundary) gives rise to a $2$-computad $(\Gamma,\Delta)$ as follows:
\begin{itemize}
\item $\Gamma$ consists of the vertices and edges, with every edge oriented in the direction of increasing $x$-coordinate; 
\item there are two arcs of $\Delta$ for every face, joining the two directed paths that make up the boundary of that face, one arc each way.
\end{itemize}
When considering labellings of this $2$-computad in a $2$-category $\cA$, we always impose the extra condition that, for each face of the polyhedron, the $2$-cells assigned to the two arcs on that face are inverse to each other, so that each determines the other. (Thus, we really have a `$2$-computad with relations'.) For instance, when $\cA=\sCat$, such a labelling assigns a category to each vertex, a functor to each edge, and a natural isomorphism of functors to each face. We simply refer to a \emph{cube}, \emph{tetrahedron}, etc., meaning a diagram in a $2$-category (specifically, $\sCat$) obtained by labelling the $2$-computad associated with a cube, tetrahedron, etc.~in $\mathbb{R}^3$.
 
\begin{ex} \label{ex:cube1}
Consider the case of a cube in a $2$-category $\cA$. The $1$-skeleton of this cube, obtained by forgetting $\Delta$, is a diagram in the underlying $1$-category of $\cA$, of the kind that one would ordinarily mean by a `cube':
\begin{equation} \label{eqn:cube}
\vcenter{
\xymatrix@R=8pt@C=6pt{
A\ar[rrrr]^\alpha\ar[dd]_(.3){\epsilon}\ar[drr]^(.6){\beta}&&&&B\ar'[d]^(.6){\zeta}[dd]\ar[drr]^\gamma\\
&&C\ar[rrrr]^(.4){\delta}\ar[dd]_(.3){\eta}&&&&D\ar[dd]^(.3){\theta}\\
E\ar'[rr][rrrr]^(.3){\iota}\ar[drr]_\kappa&&&&F\ar[drr]^\lambda\\
&&G\ar[rrrr]_\mu&&&&H
}
}
\end{equation}
That is, 
$A, \dots, H$ denote $0$-cells of $\cA$, and $\alpha, \dots, \mu$ denote $1$-cells of $\cA$ with domains and codomains as indicated. To specify the full cube 
we must also specify, for each face, an inverse pair of $2$-cells between the two compositions of $1$-cells around the edges of that face. For example, the face $ABCD$ in the above picture should be labelled by an inverse pair of $2$-cells $\delta\circ\beta\natisom\gamma\circ\alpha$. When we want to display the names of these $2$-cells 
we use a picture such as
\begin{equation}
\vc{\begin{tikzpicture}[smallcube]
\compuinit
% hidden 0-cell
\node (rrd) at (1,0,0) {$F$};
% hidden 2-cells
\node[cuber] at (0.5,0.5,0) {$\tau$};
\node[cubed] at (1,0.5,0.5) {$\upsilon$};
\node[cubeb] at (0.5,0,0.5) {$\phi$};
% outer 0-cells
\node (rlu) at (0,1,0) {$A$};
\node (rru) at (1,1,0) {$B$};
\node (fru) at (1,1,1) {$D$};
\node (frd) at (1,0,1) {$H$};
\node (fld) at (0,0,1) {$G$};
\node (rld) at (0,0,0) {$E$};
% hidden 1-cells
\draw[liner,->] (rru) -- node[arl] {$\al \zeta$} (rrd);
\draw[liner,->] (rld) -- node[arl,pos=.4] {$\al \iota$} (rrd);
\draw[liner,->] (rrd) -- node[arl] {$\al \lambda$} (frd);
% outer 1-cells
\draw[->] (rlu) -- node[arl] {$\al \alpha$} (rru);
\draw[->] (rru) -- node[arl] {$\al \gamma$} (fru);
\draw[->] (fru) -- node[arl] {$\al \theta$}(frd);
\draw[->] (fld) -- node[arr] {$\al \mu$} (frd);
\draw[->] (rld) -- node[arr] {$\al \kappa$} (fld);
\draw[->] (rlu) -- node[arr] {$\al \epsilon$} (rld);
% visible 2-cells
\node[cubel] at (0,0.5,0.5) {$\chi$};
\node[cubet] at (0.5,1,0.5) {$\psi$};
\node[cubef] at (0.5,0.5,1) {$\omega$};
% visible 0- and 1-cells
\node (flu) at (0,1,1) {$C$};
\draw[->] (rlu) -- node[arl,pos=.7] {$\al \beta$} (flu);
\draw[->] (flu) -- node[arr,pos=.3] {$\al \delta$} (fru);
\draw[->] (flu) -- node[arl,pos=.3] {$\al \eta$} (fld);
\end{tikzpicture}}
\end{equation}
To avoid clutter, we sometimes display just the $1$-skeleton, when the context makes clear which $2$-cells are meant.
\end{ex}

Many of our results
assert that a particular cube (or tetrahedron, etc.) is commutative. According to the definition of commutativity given in \S\ref{subsect:commutativity}, this appears to require a number of different equalities of $2$-cells, but in fact the equalities are all equivalent because of our assumption that the $2$-cells assigned to each face are inverse to each other.
\begin{ex} \label{ex:cube2}
Continue with the cube of Example~\ref{ex:cube1}. One of the equalities of $2$-cells entailed by saying that this cube is commutative is
\begin{equation} \label{eqn:hexagonaspasting}
\vcenter{
\xymatrix@R=7pt@C=7pt{
A\ar[rrrr]^\alpha\ar[dd]_\epsilon\ar[drr]^\beta&&&&B\ar[drr]^\gamma\\
&&C\ar[rrrr]^\delta\ar[dd]_\eta\ar@{| |=>| |}[urr]&&&&D\ar[dd]^\theta\\
E\ar[drr]_\kappa\ar@{| |=>| |}[urr]&&&&\\
&&G\ar[rrrr]_\mu\ar@{|o|=>|o|}[uurrrr]&&&&H
}
}
\quad=\quad
\vcenter{
\xymatrix@R=7pt@C=7pt{
A\ar[rrrr]^\alpha\ar[dd]_\epsilon&&&&B\ar[dd]^\zeta\ar[drr]^\gamma\\
&&&&&&D\ar[dd]^\theta\\
E\ar[rrrr]^\iota\ar[drr]_\kappa\ar@{|o|=>|o|}[uurrrr]&&&&F\ar[drr]^\lambda\ar@{| |=>| |}[urr]\\
&&G\ar[rrrr]_\mu\ar@{| |=>| |}[urr]&&&&H
}
}
\end{equation}
Here by abuse we let these pasting diagrams stand for their composite $2$-cells. These pasting diagrams appear on the `front' and `back' of the cube when viewed from the angle suggested in \eqref{eqn:cube}, with a particular choice of which of the two directed paths in the visual boundary is the domain and which is the codomain. Other equations could be obtained by making different choices of angles and orientations.  However, all of these equations are equivalent to the statement that the following hexagon 
commutes, in the 
sense of diagrams in the category of $1$-cells from $A$ to $H$:
\begin{equation} \label{eqn:hexagon}
\vcenter{
\xymatrix@C=8pt@R=16pt{
&&\theta\circ\gamma\circ\alpha\ar@{<=>}[dll]\ar@{<=>}[drr]\\
\theta\circ\delta\circ\beta\ar@{<=>}[d]&&&&\lambda\circ\zeta\circ\alpha\ar@{<=>}[d]\\
\mu\circ\eta\circ\beta\ar@{<=>}[drr]&&&&\lambda\circ\iota\circ\epsilon\ar@{<=>}[dll]\\
&&\mu\circ\kappa\circ\epsilon
}
}
\end{equation}
Here the vertices of the hexagon are the six $1$-cells $A\to H$ obtained by composing $1$-cells labelling the edges of the cube, and the edges of the hexagon correspond to the faces of the cube.
The particular equation \eqref{eqn:hexagonaspasting} is obtained by breaking the hexagon \eqref{eqn:hexagon} into its left and right halves.

This characterization of commutativity immediately implies statements of the following kind: if the $1$-skeleton of the cube has been specified, along with the $2$-cells labelling all faces other than the face $ABCD$, and if the $1$-cell $\theta$ is such that every $2$-cell $\theta\circ\varphi\Rightarrow\theta\circ\psi$ is induced by a unique $2$-cell $\varphi\Rightarrow\psi$ (for example, if $\theta$ is a full and faithful functor in $\sCat$), then there is a unique way to label the face $ABCD$ so that the cube is commutative. 

Similarly, if the missing labels are those of the face $EFGH$, and if the $1$-cell $\epsilon$ is such that every $2$-cell $\varphi\circ\epsilon\Rightarrow\psi\circ\epsilon$ is induced by a unique $2$-cell $\varphi\Rightarrow\psi$ (for example, if $\epsilon$ is a full and essentially surjective functor in $\sCat$), then there is a unique way to label the face $EFGH$ so that the cube is commutative.
\end{ex}

\begin{ex} \label{ex:prism}
Because it plays an important role in the proof of Theorem~\ref{thm:main-result}, let us examine also the case where the polyhedron is a triangular prism; we refer to a $2$-category diagram of this shape simply as a \emph{prism}. The $1$-skeleton of a prism has the form
\begin{equation} \label{eqn:prism}
\vcenter{
\xymatrix@R=8pt@C=6pt{
A\ar[rrrr]^\alpha\ar[dd]_{\epsilon}\ar[drr]^(.6){\beta}&&&&B\ar'[d]^(.6){\zeta}[dd]\ar[drr]^\gamma\\
&&C\ar[rrrr]^(.4){\delta}\ar[dll]^{\eta}&&&&D\ar[dll]^{\theta}\\
E\ar[rrrr]_{\iota}&&&&F
}
}
\end{equation}
The prism is commutative if and only if the following diagram commutes:
\begin{equation} \label{eqn:pentagon}
\vcenter{
\xymatrix@C=12pt@R=15pt{
\theta\circ\delta\circ\beta\ar@{<=>}[d]&&\theta\circ\gamma\circ\alpha\ar@{<=>}[ll]\ar@{<=>}[rr] &&\zeta\circ\alpha\ar@{<=>}[d]\\
\iota\circ\eta\circ\beta\ar@{<=>}[rrrr]&&&&\iota\circ\epsilon
}
}
\end{equation}
Notice that this condition uniquely determines the inverse $2$-cells labelling the face $ABEF$ in terms of the rest of the data.
\end{ex}

\subsection{The gluing principle}
\label{subsect:gluing}

An obvious yet important fact in ordinary category theory is that a diagram composed of commutative triangles and squares (say) joined together along their edges, in such a way that the result can be drawn in $\mathbb{R}^2$, is commutative as a whole. We now want to explain a $2$-categorical version of this fact, which we call the \emph{gluing principle}. We use this principle throughout the paper to construct new commutative cubes, prisms, etc.~from known ones.

\begin{ex} \label{ex:gluing-cubes}
Let us examine in detail the case of gluing two cubes along a common face. We suppose we have two consistently oriented cubes in our $2$-category $\cA$, 
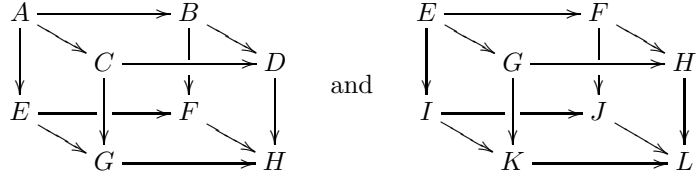
\[
\vcenter{
\xymatrix@R=6pt@C=6pt{
A\ar[rrrr]\ar[dd]\ar[drr]&&&&B\ar'[d][dd]\ar[drr]\\
&&C\ar[rrrr]\ar[dd]&&&&D\ar[dd]\\
E\ar'[rr][rrrr]\ar[drr]&&&&F\ar[drr]\\
&&G\ar[rrrr]&&&&H
}
}
\quad\text{ and }\quad
\vcenter{
\xymatrix@R=6pt@C=6pt{
E\ar[rrrr]\ar[dd]\ar[drr]&&&&F\ar'[d][dd]\ar[drr]\\
&&G\ar[rrrr]\ar[dd]&&&&H\ar[dd]\\
I\ar'[rr][rrrr]\ar[drr]&&&&J\ar[drr]\\
&&K\ar[rrrr]&&&&L
}
}
\]
where the $1$-cells and $2$-cells labelling the face $EFGH$ are the same in both cubes. Then we can glue these together to obtain a cube
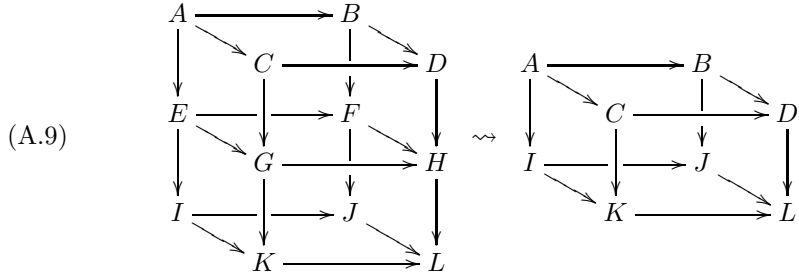
\begin{equation} \label{eqn:cube-gluing}
\vcenter{
\xymatrix@R=6pt@C=6pt{
A\ar[rrrr]\ar[dd]\ar[drr]&&&&B\ar'[d][dd]\ar[drr]\\
&&C\ar[rrrr]\ar[dd]&&&&D\ar[dd]\\
E\ar'[rr][rrrr]\ar[dd]\ar[drr]&&&&F\ar'[d][dd]\ar[drr]\\
&&G\ar[rrrr]\ar[dd]&&&&H\ar[dd]\\
I\ar'[rr][rrrr]\ar[drr]&&&&J\ar[drr]\\
&&K\ar[rrrr]&&&&L
}
}
\rightsquigarrow
\vcenter{
\xymatrix@R=6pt@C=6pt{
A\ar[rrrr]\ar[dd]\ar[drr]&&&&B\ar'[d][dd]\ar[drr]\\
&&C\ar[rrrr]\ar[dd]&&&&D\ar[dd]\\
I\ar'[rr][rrrr]\ar[drr]&&&&J\ar[drr]\\
&&K\ar[rrrr]&&&&L
}
}
\end{equation}
by appropriate compositions of $1$-cells and $2$-cells as suggested by the picture. Our claim is that if the original two cubes are commutative, so is the resulting cube.

One way to prove this (see~\cite[\S 4]{hkk}) is to write down the hexagon \eqref{eqn:hexagon} for the resulting cube, and show that it can be obtained by joining together two hexagons induced by those for the original two cubes, and two squares whose commutativity follows from the $2$-category axioms. A similar proof could be given for every case of the gluing principle that we need, but this would be tedious.

A better way to prove the claim is to use pasting diagrams: 
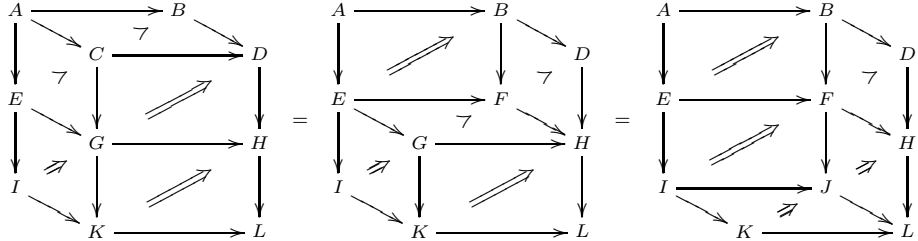
\[
\hbox{\scriptsize$
\vcenter{
\xymatrix@R=6pt@C=6pt{
A\ar[rrrr]\ar[dd]\ar[drr]&&&&B\ar[drr]\\
&&C\ar[rrrr]\ar[dd]\ar@{| |=>| |}[urr]&&&&D\ar[dd]\\
E\ar[drr]\ar[dd]\ar@{| |=>| |}[urr]&&&&\\
&&G\ar[rrrr]\ar[dd]\ar@{|o|=>|o|}[uurrrr]&&&&H\ar[dd]\\
I\ar[drr]\ar@{| |=>| |}[urr]&&&&\\
&&K\ar[rrrr]\ar@{|o|=>|o|}[uurrrr]&&&&L
}
}
=
\vcenter{
\xymatrix@R=6pt@C=6pt{
A\ar[rrrr]\ar[dd]&&&&B\ar[dd]\ar[drr]\\
&&&&&&D\ar[dd]\\
E\ar[rrrr]\ar[dd]\ar[drr]\ar@{|o|=>|o|}[uurrrr]&&&&F\ar[drr]\ar@{| |=>| |}[urr]\\
&&G\ar[rrrr]\ar[dd]\ar@{| |=>| |}[urr]&&&&H\ar[dd]\\
I\ar[drr]\ar@{| |=>| |}[urr]&&&&\\
&&K\ar[rrrr]\ar@{|o|=>|o|}[uurrrr]&&&&L
}
}
=
\vcenter{
\xymatrix@R=6pt@C=6pt{
A\ar[rrrr]\ar[dd]&&&&B\ar[dd]\ar[drr]\\
&&&&&&D\ar[dd]\\
E\ar[rrrr]\ar[dd]\ar@{|o|=>|o|}[uurrrr]&&&&F\ar[drr]\ar[dd]\ar@{| |=>| |}[urr]\\
&&&&&&H\ar[dd]\\
I\ar[rrrr]\ar[drr]\ar@{|o|=>|o|}[uurrrr]&&&&J\ar[drr]\ar@{| |=>| |}[urr]\\
&&K\ar[rrrr]\ar@{| |=>| |}[urr]&&&&L
}
}
$}
\]
\noindent
Here, each step uses the commutativity of one of the two cubes, expressed in the form \eqref{eqn:hexagonaspasting}. The conclusion that the composite $2$-cell of the first pasting diagram equals that of the third is equivalent to the commutativity of the resulting cube. Notice how this argument works visually: the first pasting diagram is what appears on the `front' of the gluing picture \eqref{eqn:cube-gluing}, and the third is what appears on the `back'. The intermediate stage is obtained by `passing through' one of the two original cubes but not the other.

This observation suggests a more sophisticated way to express the proof, using the formalism of $3$-categorical pasting~\cite{powern}. We can think of $\cA$ as a $3$-category where the only $3$-cells are identities. Then a commutative cube can be regarded as a $3$-computad labelled in $\cA$, where the $3$-arrow joins the two $2$-pasting schemes whose labellings are the two sides of \eqref{eqn:hexagonaspasting}, and is labelled by the $3$-cell that asserts the equality of those two sides. The gluing picture \eqref{eqn:cube-gluing} is a valid $3$-pasting diagram, so it does define a composite $3$-cell, 
and that $3$-cell asserts the commutativity of the glued cube.  
\end{ex}

The gluing principle we need is not much more general than Example \ref{ex:gluing-cubes}. An informal statement is: \textit{if we take a collection of commutative labelled $2$-computads of the polyhedral kind, and glue them along matching faces in such a way that the gluing can be depicted in $\mathbb{R}^3$, then the resulting labelled $2$-computad is commutative}.  We will not state the gluing principle more precisely, because we do not need to give a general proof. For every case of the principle that appears in this paper, it is evident that one could give a proof consisting of a chain of equalities of (the composite $2$-cells of) pasting diagrams along the above lines,
starting with the `front' of the picture and working through to the `back' by `passing through' one constituent polyhedron at a time. 
Representative examples of gluing pictures are Figure~\ref{fig:main-picture} and \eqref{eqn:j-prism}. 

On a handful of occasions, we use a sort of converse to the gluing principle, which allows us, under certain circumstances, to deduce the commutativity of one of the constituent polyhedra in the gluing. Again, we content ourselves here with the example of gluing two cubes.

\begin{ex} \label{ex:gluing-converse}
Continue with the notation of Example \ref{ex:gluing-cubes}. Suppose we know that the cubes $ABCDEFGH$ and $ABCDIJKL$ are commutative.
Under these assumptions we have the first equality of pasting diagrams, and the composition of the two equalities, so we can deduce the second equality. If the $1$-cell $\epsilon:A\to E$ has the property that a $2$-cell $\varphi\Rightarrow\psi$ is determined by the $2$-cell $\varphi\circ\epsilon\Rightarrow\psi\circ\epsilon$ it induces (when 
defined), then we can conclude that the cube $EFGHIJKL$ is commutative. (For example, an essentially surjective functor $\epsilon$ has this property in $\sCat$.) 

Similarly, if we know that the cubes $EFGHIJKL$ and $ABCDIJKL$ are commutative, and that the $1$-cell $\theta:H\to L$ has the property that a $2$-cell $\varphi\Rightarrow\psi$ is determined by the $2$-cell $\theta\circ\varphi\Rightarrow\theta\circ\psi$ it induces (when 
defined), then we can conclude that the cube $ABCDEFGH$ is commutative. (For example, a faithful functor $\theta$ has this property in $\sCat$.)
\end{ex}

%%%%%%%%%%%%%%%%%%%%%%%%%%%%%%%%%%%%%%%%%%%%%%%%%%%%%%%%%%%%%%%%%%%%%%%%%%%%%%%%%%%%%%%%%%%

\section{Commutativity lemmas for sheaf functors} 
\label{sect:lemmas}

This appendix contains a collection of results asserting the commutativity of various $2$-categorical diagrams. These diagrams are all labelled $2$-computads of the polyhedral kind described in \S\ref{subsect:polyhedral}, where the $2$-category is $\sCat$ and the categories involved are derived categories of sheaves on varieties. Thus, the results concern equalities of natural isomorphisms between sheaf functors. We use the same conventions as in \S\ref{ss:not}.
A few of the analogous statements in the context of {\'e}tale sheaves are proved in \cite[\S\S5.1--5.2]{de} (see also \cite[\S12]{rouquier}).

Some explanation on the use of this appendix is needed.  Because the results are so numerous, they are not stated in the usual `Lemma---Proof' format; instead, references such as `Lemma~\ref{lem:cocycle^!}', here and in the main body of the paper, should be understood as directing the reader to consult part~\subref{lem:cocycle^!} of \emph{Figure}~\ref{lem:cocycle}.  (The sole exception is Lemma~\ref{lem:_!action}.)  Each figure in the appendix mentions a `Setting', which is a certain commutative diagram of varieties and morphisms of varieties, giving context and notation for the accompanying polyhedral diagrams.  The proof that the diagrams in a given figure are commutative appears in the \emph{subsection} with the same number.
We will frequently use the gluing principle of \S\ref{subsect:gluing}. 

Some lemmas in this appendix show only ordinary (nonequivariant) derived categories, but are invoked in situations involving equivariant derived categories.  For a justification of this, see \S\ref{subsect:equivariant} below.

\begin{rmk}
\label{rmk:lz}
As explained in the introduction, at least some of the properties proved in this appendix are implicitly contained in \cite{lz}. However, making those implicit facts explicit takes some work.  An $\infty$-category is a special kind of simplicial set.  By examining certain $3$-simplices in the Liu--Zheng construction, one can see that the commutativity of most of the tetrahedral diagrams we consider are truly trivial consequences of~\cite{lz}.  But diagrams shaped like cubes or prisms must still be assembled from tetrahedral ones by the `pasting' operation explained in Appendix~\ref{sect:cubes}, so working in an $\infty$-categorical framework would not seem to significantly shorten our arguments.  
\end{rmk}

%---------------------------------------------------------------------
\subsection{Notation}
%---------------------------------------------------------------------

%Before beginning the proofs of commutativity results, we first explain and fix notation for the basic isomorphisms of functors we will encounter. 

%.....................................................................
\subsubsection{Composition}
%.....................................................................

\label{subsect:composition-definitions}

Suppose we have variety morphisms
\[
\xymatrix{
X \ar[r]^-{f_1} & Y \ar[r]^-{f_2} & Z
}
\]
and set $f=f_2f_1$. Then we obtain 
isomorphisms 
which will be denoted as follows:
\[
\vc{\begin{tikzpicture}[stdtriangle]
\compuinit
% 0-cells
\node (lu) at (0,1) {$\al \cDb(X)$};
\node (r) at (1,0.5) {$\al \cDb(Y)$};
\node (ld) at (0,0) {$\al \cDb(Z)$};
% 1-cells
\draw[->] (lu) -- node[arl] {$\al (f_1)_*$} (r);
\draw[->] (lu) -- node[arr] {$\al f_*$} (ld);
\draw[->] (r) -- node[arl] {$\al (f_2)_*$} (ld);
% 2-cell
\node[tricell] at (\tric,0.5) {\tiny$\Comp$};
\end{tikzpicture}}
\vc{\begin{tikzpicture}[stdtriangle]
\compuinit
% 0-cells
\node (lu) at (0,1) {$\al \cDb(X)$};
\node (r) at (1,0.5) {$\al \cDb(Y)$};
\node (ld) at (0,0) {$\al \cDb(Z)$};
% 1-cells
\draw[->] (lu) -- node[arl] {$\al (f_1)_!$} (r);
\draw[->] (lu) -- node[arr] {$\al f_!$} (ld);
\draw[->] (r) -- node[arl] {$\al (f_2)_!$} (ld);
% 2-cell
\node[tricell] at (\tric,0.5) {\tiny$\Comp$};
\end{tikzpicture}}
\vc{\begin{tikzpicture}[stdtriangle]
\compuinit
% 0-cells
\node (lu) at (0,1) {$\al \cDb(X)$};
\node (r) at (1,0.5) {$\al \cDb(Y)$};
\node (ld) at (0,0) {$\al \cDb(Z)$};
% 1-cells
\draw[<-] (lu) -- node[arl] {$\al (f_1)^*$} (r);
\draw[<-] (lu) -- node[arr] {$\al f^*$} (ld);
\draw[<-] (r) -- node[arl] {$\al (f_2)^*$} (ld);
% 2-cell
\node[tricell] at (\tric,0.5) {\tiny$\al \Comp$};
\end{tikzpicture}}
\vc{\begin{tikzpicture}[stdtriangle]
\compuinit
% 0-cells
\node (lu) at (0,1) {$\al \cDb(X)$};
\node (r) at (1,0.5) {$\al \cDb(Y)$};
\node (ld) at (0,0) {$\al \cDb(Z)$};
% 1-cells
\draw[<-] (lu) -- node[arl] {$\al (f_1)^!$} (r);
\draw[<-] (lu) -- node[arr] {$\al f^!$} (ld);
\draw[<-] (r) -- node[arl] {$\al (f_2)^!$} (ld);
% 2-cell
\node[tricell] at (\tric,0.5) {\tiny$\al \Comp$};
\end{tikzpicture}}
\]
The first isomorphism is defined in \cite[Equation (2.6.5)]{kas}: to construct it, one uses the fact that, if $f_*^0$, $(f_1)^0_*$ and $(f_2)^0_*$ denote the non-derived direct image functors (between abelian categories of $\bk$-sheaves), the natural morphism of functors
\[
f_* = R(f^0_*) \directednatisom R\bigl( (f_2)^0_* \circ (f_1)^0_* \bigr) \Rightarrow R\bigl( (f_2)^0_* \bigr)\circ R\bigl( (f_1)^0_* \bigr) = (f_2)_* \circ (f_1)_*
\]
is an isomorphism. The second and third isomorphisms are defined similarly (see \cite[Equations (2.6.6) and (2.3.9)]{kas}. Finally, the fourth isomorphism is proved in \cite[Proposition 3.1.8]{kas}. Note that this fourth isomorphism is deduced from the second one by adjunction, in a sense that will be made precise in Lemma \ref{lem:_!^!compositionadjunction} below.

Consequently, given a commutative square of variety morphisms
\[
\vc{\begin{tikzpicture}[vsmallcube]
\compuinit
% 2-cells
%\node[cart] at (0.5,0.5) {};
% 0-cells
\node (lu) at (0,1) {$W$};
\node (ru) at (1,1) {$X$};
\node (ld) at (0,0) {$Y$};
\node (rd) at (1,0) {$Z$};
% 1-cells
\draw[->] (lu) -- node[arl] {$\al f_1$} (ru);
\draw[->] (lu) -- node[arr] {$\al f_3$} (ld);
\draw[->] (ru) -- node[arl] {$\al f_2$} (rd);
\draw[->] (ld) -- node[arr] {$\al f_4$} (rd);
\end{tikzpicture}}
\]
we obtain natural isomorphisms $(f_2)_*\circ(f_1)_*\natisom(f_4)_*\circ(f_3)_*$ etc., by composing the composition isomorphisms $(f_2)_*\circ(f_1)_*\natisom f_*$ and $f_*\natisom (f_4)_*\circ(f_3)_*$ where $f=f_2f_1=f_4f_3$.  These isomorphisms will be labelled `$\Co$' as well.

%.....................................................................
\subsubsection{Base change}
%.....................................................................

Suppose we have a cartesian square of variety morphisms
\[
\vc{\begin{tikzpicture}[vsmallcube]
\compuinit
% 2-cells
\node[cart] at (0.5,0.5) {};
% 0-cells
\node (lu) at (0,1) {$W$};
\node (ru) at (1,1) {$X$};
\node (ld) at (0,0) {$Y$};
\node (rd) at (1,0) {$Z$};
% 1-cells
\draw[->] (lu) -- node[arl] {$\al g'$} (ru);
\draw[->] (lu) -- node[arr] {$\al f'$} (ld);
\draw[->] (ru) -- node[arl] {$\al f$} (rd);
\draw[->] (ld) -- node[arr] {$\al g$} (rd);
\end{tikzpicture}}
\]
Then we obtain base change isomorphisms $g^*\circ f_!\natisom (f')_!\circ (g')^*$ and $g^!\circ f_*\natisom (f')_*\circ (g')^!$ which will be denoted as follows:
\[
\vc{\begin{tikzpicture}[smallcube]
\compuinit
% 2-cells
\node[cubef] at (0.5,0.5) {$\BC$};
% 0-cells
\node (lu) at (0,1) {$\cDb(X)$};
\node (ru) at (1,1) {$\cDb(W)$};
\node (ld) at (0,0) {$\cDb(Z)$};
\node (rd) at (1,0) {$\cDb(Y)$};
% 1-cells
\draw[->] (lu) -- node[arl] {$\al (g')^*$} (ru);
\draw[->] (lu) -- node[arr] {$\al f_!$} (ld);
\draw[->] (ru) -- node[arl] {$\al (f')_!$} (rd);
\draw[->] (ld) -- node[arr] {$\al g^*$} (rd);
\end{tikzpicture}}
\qquad
\vc{\begin{tikzpicture}[smallcube]
\compuinit
% 2-cells
\node[cubef] at (0.5,0.5) {$\BC$};
% 0-cells
\node (lu) at (0,1) {$\cDb(X)$};
\node (ru) at (1,1) {$\cDb(W)$};
\node (ld) at (0,0) {$\cDb(Z)$};
\node (rd) at (1,0) {$\cDb(Y)$};
% 1-cells
\draw[->] (lu) -- node[arl] {$\al (g')^!$} (ru);
\draw[->] (lu) -- node[arr] {$\al f_*$} (ld);
\draw[->] (ru) -- node[arl] {$\al (f')_*$} (rd);
\draw[->] (ld) -- node[arr] {$\al g^!$} (rd);
\end{tikzpicture}}
\]
The first isomorphism is proved in \cite[Proposition 2.6.7]{kas}. The second isomorphism is proved in \cite[Proposition 3.1.9]{kas}; in fact it is deduced from the first one by adjunction, in a sense that will be made precise in Lemma \ref{lem:basechangeadjunction} below.

%.....................................................................
\subsubsection{Adjunction}
\label{sss:adjunction}
%.....................................................................

For any morphism $f:X\to Y$, the adjunctions $f^*\dashv f_*$ and $f_!\dashv f^!$ give rise to (indeed, are equivalent to) adjunction isomorphisms
\[
\vc{\begin{tikzpicture}[smallcube]
\compuinit
% 2-cells
\node[cubef] at (0.5,0.5) {$\Adj$};
% 0-cells
\node (lu) at (0,1) {$\cDb(X)$};
\node (ru) at (1,1) {$\Vect^{\cDb(X)^\op}$};
\node (ld) at (0,0) {$\cDb(Y)$};
\node (rd) at (1,0) {$\Vect^{\cDb(Y)^\op}$};
% 1-cells
\draw[->] (lu) -- node[arl] {$\al \Yon$} (ru);
\draw[->] (lu) -- node[arr] {$\al f_*$} (ld);
\draw[->] (ru) -- node[arl] {$\al -\circ f^{*,\op}$} (rd);
\draw[->] (ld) -- node[arr] {$\al \Yon$} (rd);
\end{tikzpicture}}
\qquad
\vc{\begin{tikzpicture}[smallcube]
\compuinit
% 2-cells
\node[cubef] at (0.5,0.5) {$\Adj$};
% 0-cells
\node (lu) at (0,1) {$\cDb(X)$};
\node (ru) at (1,1) {$\Vect^{\cDb(X)^\op}$};
\node (ld) at (0,0) {$\cDb(Y)$};
\node (rd) at (1,0) {$\Vect^{\cDb(Y)^\op}$};
% 1-cells
\draw[->] (lu) -- node[arl] {$\al \Yon$} (ru);
\draw[<-] (lu) -- node[arr] {$\al f^!$} (ld);
\draw[<-] (ru) -- node[arl] {$\al -\circ(f_!)^\op$} (rd);
\draw[->] (ld) -- node[arr] {$\al \Yon$} (rd);
\end{tikzpicture}}
\]
Here $\Vect$ is short for $\Vect(\bk)$ where $\bk$ is the coefficient ring of the derived categories, and $\Yon:\mathsf{C}\to\Vect^{\mathsf{C}^\op}$ denotes the Yoneda embedding \cite[III.2(7)]{maclane}, defined on objects by $\Yon(c)=\Hom_{\mathsf{C}}(-,c)$. The second isomorphism is essentially the definition of the functor $f^!$: see \cite[Theorem 3.1.5]{kas}. The first isomorphism is proved in \cite[Proposition 2.6.4]{kas}. It is deduced from the following observation: if we denote by $f_*^0$ and $f^*_0$ the non-derived direct and inverse image functors (between abelian categories of $\bk$-sheaves), then for any complex $M$ of sheaves on $Y$, the natural morphism of functors
\[
R \bigl( \Hom(f^*_0 M, -) \bigr) \directednatisom R \bigl( \Hom(M, -) \circ f_*^0 \bigr) \Rightarrow R\Hom(M, -) \circ f_*
\]
is an isomorphism.

%.....................................................................
\subsubsection{Constant sheaf under inverse image}
\label{sss:constantsheaf-res}
%.....................................................................

Let $\bb$ denote the trivial group, regarded as a one-object category. The datum of the constant sheaf $\ubk_X$ on a variety $X$ defines a functor
\[
\ubk_X : \bb \to \cDb(X).
\]
We have a canonical isomorphism $\ubk_X\cong a_X^*\ubk_{\pt}$ where $a_X$ is the morphism $X\to\pt$. Hence for any morphism $f: X \to Y$ we obtain an isomorphism 
\[ f^*(\ubk_Y) \cong f^* \bigl( (a_Y)^*(\ubk_{\pt}) \bigr) \overset{\Co}{\cong} (a_X)^*(\ubk_{\pt})\cong \ubk_X. \] 
We can regard this as an isomorphism of functors:
\[
\vc{\begin{tikzpicture}[stdtriangle]
\compuinit
% 0-cells
\node (lu) at (0,1) {$\bb$};
\node (r) at (1,0.5) {$\cDb(Y)$};
\node (ld) at (0,0) {$\cDb(X)$};
% 1-cells
\draw[->] (lu) -- node[arl] {$\al \ubk_Y$} (r);
\draw[->] (lu) -- node[arr] {$\al \ubk_X$} (ld);
\draw[->] (r) -- node[arl] {$\al f^*$} (ld);
% 2-cell
\node[tricell] at (\tric,0.5) {\tiny$\Cnst$};
\end{tikzpicture}}
\]

%---------------------------------------------------------------------
% Start of non-equivariant commutativity lemmas
%---------------------------------------------------------------------

\addtocounter{figure}{1}

%---------------------------------------------------------------------
\subsection{Composition and adjunction}
\label{subsect:compositionadjunction}
%---------------------------------------------------------------------

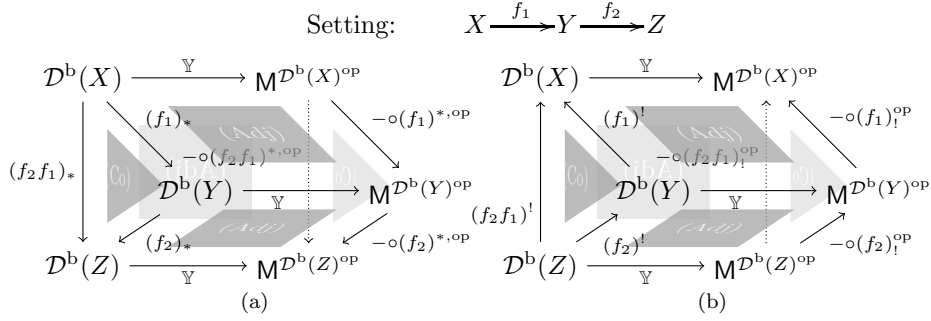
\begin{figure}
\begin{center}
Setting: \qquad $\xymatrix@1{X\ar[r]^{f_1}&Y\ar[r]^{f_2}&Z}$

\subfigure[][]{\qc{\begin{tikzpicture}[stdprism]
\compuinit
% hidden 2-cells
\node[cuber] at (0.5,0.5,0) {$\Adj$};
\node[prismdf] at (1,0.5,\tric) {$\Comp$};
% outer 0-cells
\node (rlu) at (0,1,0) {$\cDb(X)$};
\node (rru) at (1,1,0) {$\Vect^{\cDb(X)^\op}$};
\node (fr) at (	1,0.5,1) {$\Vect^{\cDb(Y)^\op}$};
\node (rld) at (0,0,0) {$\cDb(Z)$};
\node (rrd) at (1,0,0) {$\Vect^{\cDb(Z)^\op}$};
% hidden 1-cells
\draw[liner,->] (rru) -- node[arr,pos=.4] {$\al -\circ(f_2f_1)^{*,\op}$} (rrd);
% outer 1-cells
\draw[->] (rlu) -- node[arl] {$\al \Yon$} (rru);
\draw[->] (rlu) -- node[arr] {\llap{$\al (f_2f_1)_*$}} (rld);
\draw[->] (rru) -- node[arl] {$\al -\circ(f_1)^{*,\op}$} (fr);
\draw[->] (rld) -- node[arr] {$\al \Yon$} (rrd);
\draw[->] (fr) -- node[arl] {$\al -\circ(f_2)^{*,\op}$} (rrd);
% visible 2-cells
\node[prismtf] at (0.5,0.75,0.5) {$\Adj$};
\node[prismbf] at (0.5,0.25,0.5) {$\Adj$};
\node[prismlf] at (0,0.5,\tric) {$\Comp$};
% visible 0- and 1-cells
\node (fl) at (0,0.5,1) {$\cDb(Y)$};
\draw[->] (rlu) -- node[arl] {$\al (f_1)_*$} (fl);
\draw[->] (fl) -- node[arl] {$\al (f_2)_*$} (rld);
\draw[->] (fl)  -- node[arr,pos=.3] {$\al \Yon$} (fr);
\end{tikzpicture}}\label{lem:^*_*compositionadjunction}}~
\subfigure[][]{\qc{\begin{tikzpicture}[stdprism]
\compuinit
% hidden 2-cells
\node[cuber] at (0.5,0.5,0) {$\Adj$};
\node[prismdf] at (1,0.5,\tric) {$\Comp$};
% outer 0-cells
\node (rlu) at (0,1,0) {$\cDb(X)$};
\node (rru) at (1,1,0) {$\Vect^{\cDb(X)^\op}$};
\node (fr) at (	1,0.5,1) {$\Vect^{\cDb(Y)^\op}$};
\node (rld) at (0,0,0) {$\cDb(Z)$};
\node (rrd) at (1,0,0) {$\Vect^{\cDb(Z)^\op}$};
% hidden 1-cells
\draw[liner,->] (rrd) -- node[arl,pos=.6] {$\al -\circ(f_2f_1)_!^\op$} (rru);
% outer 1-cells
\draw[->] (rlu) -- node[arl] {$\al \Yon$} (rru);
\draw[->] (rld) -- node[arl,pos=.2] {\llap{$\al (f_2f_1)^!$}} (rlu);
\draw[->] (fr) -- node[arr] {$\al -\circ(f_1)_!^\op$} (rru);
\draw[->] (rld) -- node[arr] {$\al \Yon$} (rrd);
\draw[->] (rrd) -- node[arr] {$\al -\circ(f_2)_!^\op$} (fr);
% visible 2-cells
\node[prismtf] at (0.5,0.75,0.5) {$\Adj$};
\node[prismbf] at (0.5,0.25,0.5) {$\Adj$};
\node[prismlf] at (0,0.5,\tric) {$\Comp$};
% visible 0- and 1-cells
\node (fl) at (0,0.5,1) {$\cDb(Y)$};
\draw[->] (fl) -- node[arr] {$\al (f_1)^!$} (rlu);
\draw[->] (rld) -- node[arr] {$\al (f_2)^!$} (fl);
\draw[->] (fl)  -- node[arr,pos=.3] {$\al \Yon$} (fr);
\end{tikzpicture}}\label{lem:_!^!compositionadjunction}}
\end{center}
\caption{Composition and adjunction}\label{lem:compositionadjunction}
\end{figure}

For Part~\subref{lem:^*_*compositionadjunction}, one can easily check that the similar statement where derived categories are replaced by abelian categories of sheaves, and the derived functors by their non-derived variants, holds. Then our claim follows, by construction of the adjunction $f^*\dashv f_*$ (see \S\ref{sss:adjunction}), using usual properties of derived functors and morphisms between them.
%\begin{itemize}
%\item If $F,G,H$ are three composable functors which admit derived functors (as well as their compositions), then the diagram of natural morphisms
%\[
%\xymatrix@R=10pt{
%R(F \circ G \circ H) \ar@{=>}[r] \ar@{=>}[d] & R(F \circ G) \circ RH \ar@{=>}[d] \\
%RF \circ R(G \circ H) \ar@{=>}[r] & RF \circ RG \circ RH
%}
%\]
%commutes;
%\item If $\varphi : F \Rightarrow G$ and $\varphi' : G \Rightarrow H$ are morphisms of functors which admit derived functors, the induced morphisms between derived functors satisfy $R(\varphi' \circ \varphi) = R(\varphi') \circ R\varphi$;
%\item If $F,G,G'$, respectively $F,F',G$, are functors which admit derived functors and $\varphi : G \Rightarrow G'$, respectively $\varphi : F \Rightarrow F'$, are morphisms of functors, then the diagrams of natural morphisms
%\[
%\xymatrix@R=10pt{
%R(F \circ G) \ar@{=>}[r] \ar@{=>}[d] & R(F \circ G') \ar@{=>}[d] \\
%RF \circ RG \ar@{=>}[r] & RF \circ RG'
%}
%\qquad \qquad
%\xymatrix@R=10pt{
%R(F \circ G) \ar@{=>}[r] \ar@{=>}[d] & R(F' \circ G) \ar@{=>}[d] \\
%RF \circ RG \ar@{=>}[r] & RF' \circ RG
%}
%\]
%commute.
%\end{itemize}

For Figure~\ref{lem:_!^!compositionadjunction}, recall that in \cite[Proposition 3.1.8]{kas}, the isomorphism $(f_2 f_1)^! \natisom (f_1)^!\circ(f_2)^!$ is deduced from the isomorphism $(f_2 f_1)_! \natisom (f_2)_!\circ (f_1)_!$ by adjunction. In other words, it is defined precisely so as to make this prism commutative. (This makes sense because $\Yon:\cDb(X)\to\Vect^{\cDb(X)^\op}$ is full and faithful; see Example \ref{ex:cube2}.)

%---------------------------------------------------------------------
\subsection{Base change and adjunction}
%---------------------------------------------------------------------

\begin{figure}
\begin{center}
\begin{tabular}{@{}c@{}}
Setting: \\
\qc{\begin{tikzpicture}[vsmallcube]
\compuinit
% 2-cells
\node[cart] at (0.5,0.5) {};
% 0-cells
\node (lu) at (0,1) {$W$};
\node (ru) at (1,1) {$X$};
\node (ld) at (0,0) {$Y$};
\node (rd) at (1,0) {$Z$};
% 1-cells
\draw[->] (lu) -- node[arl] {$\al g'$} (ru);
\draw[->] (lu) -- node[arr] {$\al f'$} (ld);
\draw[->] (ru) -- node[arl] {$\al f$} (rd);
\draw[->] (ld) -- node[arr] {$\al g$} (rd);
\end{tikzpicture}}
\end{tabular}
\qquad
\qc{\begin{tikzpicture}[stdcube]
\compuinit
% hidden 0-cell
\node (rrd) at (1,0,0) {$\Vect^{\cDb(Z)^\op}$};
% hidden 2-cells
\node[cuber] at (0.5,0.5,0) {$\Adj$};
\node[cubed] at (1,0.5,0.5) {$\BC$};
\node[cubeb] at (0.5,0,0.5) {$\Adj$};
% outer 0-cells
\node (rlu) at (0,1,0) {$\cDb(X)$};
\node (rru) at (1,1,0) {$\Vect^{\cDb(X)^\op}$};
\node (fru) at (1,1,1) {$\Vect^{\cDb(W)^\op}$};
\node (frd) at (1,0,1) {$\Vect^{\cDb(Y)^\op}$};
\node (fld) at (0,0,1) {$\cDb(Y)$};
\node (rld) at (0,0,0) {$\cDb(Z)$};
% hidden 1-cells
\draw[liner,->] (rru) -- node[arl,pos=.7] {$\al -\circ f^{*,\op}$} (rrd);
\draw[liner,->] (rld) -- node[arl,pos=.3] {$\al \Yon$} (rrd);
\draw[liner,->] (rrd) -- node[arl] {$\al -\circ g_!^\op$} (frd);
% outer 1-cells
\draw[->] (rlu) -- node[arl] {$\al \Yon$} (rru);
\draw[->] (rru) -- node[arl] {$\al -\circ (g')_!^\op$} (fru);
\draw[->] (fru) -- node[arl] {$\al -\circ(f')^{*,\op}$}(frd);
\draw[->] (fld) -- node[arr] {$\al \Yon$} (frd);
\draw[->] (rld) -- node[arr] {$\al g^!$} (fld);
\draw[->] (rlu) -- node[arr] {$\al f_*$} (rld);
% visible 2-cells
\node[cubel] at (0,0.5,0.5) {$\BC$};
\node[cubet] at (0.5,1,0.5) {$\Adj$};
\node[cubef] at (0.5,0.5,1) {$\Adj$};
% visible 0- and 1-cells
\node (flu) at (0,1,1) {$\cDb(W)$};
\draw[->] (rlu) -- node[arr,pos=.7] {$\al (g')^!$} (flu);
\draw[->] (flu) -- node[arr,pos=.3] {$\al \Yon$} (fru);
\draw[->] (flu) -- node[arl,pos=.3] {$\al (f')_*$} (fld);
\end{tikzpicture}}
\end{center}
\caption{Base change and adjunction}\label{lem:basechangeadjunction}
\end{figure}
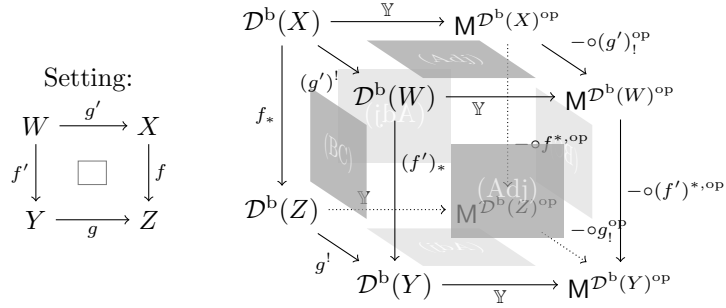

In \cite[Proposition 3.1.9]{kas}, the isomorphism $g^!\circ f_*\natisom (f')_*\circ(g')^!$ is deduced from the isomorphism $f^*\circ g_! \natisom (g')_!\circ(f')^*$ by adjunction. In other words, it is defined precisely so as to make this cube commutative.

%---------------------------------------------------------------------
\subsection{Cocycle property of composition}
%---------------------------------------------------------------------

\begin{figure}
\begin{center}
Setting: \qquad $\xymatrix@1{W\ar[r]^{f_1}&X\ar[r]^{f_2}&Y\ar[r]^{f_3}&Z}$ \quad and \quad $f=f_3f_2f_1$

\subfigure[][]{\qc{\begin{tikzpicture}[stdtetr]
\compuinit
% hidden 2-cells
\node[tetrlr] at ({-\tric/2},0.5,\tric) {$\Comp$};
\node[tetrdr] at ({\tric/2},0.5,\tric) {$\Comp$};
% outer 0-cells
\node (ru) at (0,1,0) {$\cDb(X)$};
\node (fl) at (-0.5,0.5,1) {$\cDb(W)$};
\node (fr) at (0.5,0.5,1) {$\cDb(Z)$};
\node (rd) at (0,0,0) {$\cDb(Y)$};
% hidden 1-cell
\draw[liner,->] (ru) -- node[arl,pos=.6] {$\al (f_2)_*$} (rd);
% visible 2-cells
\node[tetrtf] at (0, {(1+\tric)/2}, {1-\tric}) {$\Comp$};
\node[tetrbf] at (0, {(1-\tric)/2}, {1-\tric}) {$\Comp$};
% visible 1-cells
\draw[->] (fl) -- node[arl] {$\al (f_1)_*$} (ru);
\draw[->] (fl) -- node[arr] {$\al (f_2f_1)_*$} (rd);
\draw[->] (ru) -- node[arl] {$\al (f_3f_2)_*$} (fr);
\draw[->] (rd) -- node[arr] {$\al (f_3)_*$} (fr);
\draw[->] (fl) -- node[arl,pos=.3] {$\al f_*$} (fr);
\end{tikzpicture}}\label{lem:cocycle_*}}
\quad
\subfigure[][]{\qc{\begin{tikzpicture}[stdtetr]
\compuinit
% hidden 2-cells
\node[tetrlr] at ({-\tric/2},0.5,\tric) {$\Comp$};
\node[tetrdr] at ({\tric/2},0.5,\tric) {$\Comp$};
% outer 0-cells
\node (ru) at (0,1,0) {$\cDb(X)$};
\node (fl) at (-0.5,0.5,1) {$\cDb(W)$};
\node (fr) at (0.5,0.5,1) {$\cDb(Z)$};
\node (rd) at (0,0,0) {$\cDb(Y)$};
% hidden 1-cell
\draw[liner,->] (ru) -- node[arl,pos=.6] {$\al (f_2)_!$} (rd);
% visible 2-cells
\node[tetrtf] at (0, {(1+\tric)/2}, {1-\tric}) {$\Comp$};
\node[tetrbf] at (0, {(1-\tric)/2}, {1-\tric}) {$\Comp$};
% visible 1-cells
\draw[->] (fl) -- node[arl] {$\al (f_1)_!$} (ru);
\draw[->] (fl) -- node[arr] {$\al (f_2f_1)_!$} (rd);
\draw[->] (ru) -- node[arl] {$\al (f_3f_2)_!$} (fr);
\draw[->] (rd) -- node[arr] {$\al (f_3)_!$} (fr);
\draw[->] (fl) -- node[arl,pos=.3] {$\al f_!$} (fr);
\end{tikzpicture}}\label{lem:cocycle_!}}

\subfigure[][]{\qc{\begin{tikzpicture}[stdtetr]
\compuinit
% hidden 2-cells
\node[tetrlr] at ({-\tric/2},0.5,\tric) {$\Comp$};
\node[tetrdr] at ({\tric/2},0.5,\tric) {$\Comp$};
% outer 0-cells
\node (ru) at (0,1,0) {$\cDb(X)$};
\node (fl) at (-0.5,0.5,1) {$\cDb(W)$};
\node (fr) at (0.5,0.5,1) {$\cDb(Z)$};
\node (rd) at (0,0,0) {$\cDb(Y)$};
% hidden 1-cell
\draw[liner,<-] (ru) -- node[arl,pos=.6] {$\al (f_2)^*$} (rd);
% visible 2-cells
\node[tetrtf] at (0, {(1+\tric)/2}, {1-\tric}) {$\Comp$};
\node[tetrbf] at (0, {(1-\tric)/2}, {1-\tric}) {$\Comp$};
% visible 1-cells
\draw[<-] (fl) -- node[arl] {$\al (f_1)^*$} (ru);
\draw[<-] (fl) -- node[arr] {$\al (f_2f_1)^*$} (rd);
\draw[<-] (ru) -- node[arl] {$\al (f_3f_2)^*$} (fr);
\draw[<-] (rd) -- node[arr] {$\al (f_3)^*$} (fr);
\draw[<-] (fl) -- node[arl,pos=.3] {$\al f^*$} (fr);
\end{tikzpicture}}\label{lem:cocycle^*}}
\quad
\subfigure[][]{\qc{\begin{tikzpicture}[stdtetr]
\compuinit
% hidden 2-cells
\node[tetrlr] at ({-\tric/2},0.5,\tric) {$\Comp$};
\node[tetrdr] at ({\tric/2},0.5,\tric) {$\Comp$};
% outer 0-cells
\node (ru) at (0,1,0) {$\cDb(X)$};
\node (fl) at (-0.5,0.5,1) {$\cDb(W)$};
\node (fr) at (0.5,0.5,1) {$\cDb(Z)$};
\node (rd) at (0,0,0) {$\cDb(Y)$};
% hidden 1-cell
\draw[liner,<-] (ru) -- node[arl,pos=.6] {$\al (f_2)^!$} (rd);
% visible 2-cells
\node[tetrtf] at (0, {(1+\tric)/2}, {1-\tric}) {$\Comp$};
\node[tetrbf] at (0, {(1-\tric)/2}, {1-\tric}) {$\Comp$};
% visible 1-cells
\draw[<-] (fl) -- node[arl] {$\al (f_1)^!$} (ru);
\draw[<-] (fl) -- node[arr] {$\al (f_2f_1)^!$} (rd);
\draw[<-] (ru) -- node[arl] {$\al (f_3f_2)^!$} (fr);
\draw[<-] (rd) -- node[arr] {$\al (f_3)^!$} (fr);
\draw[<-] (fl) -- node[arl,pos=.3] {$\al f^!$} (fr);
\end{tikzpicture}}\label{lem:cocycle^!}}
\end{center}
\caption{Cocycle property of composition}\label{lem:cocycle}
\end{figure}
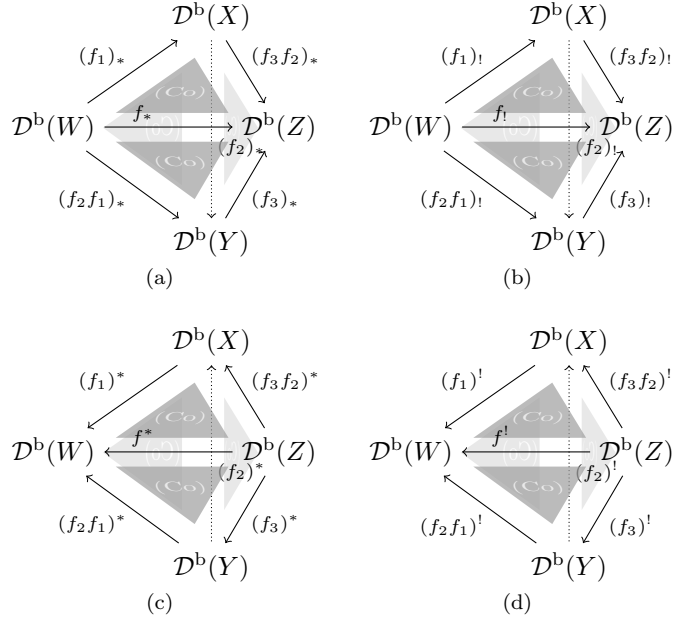

For part~\subref{lem:cocycle_*}, one easily checks the similar claim where derived categories are replaced by abelian categories of sheaves, and derived functors by their non-derived counterparts. Our claim follows, as in the proof of Lemma~\ref{lem:^*_*compositionadjunction}. The proofs of~\subref{lem:cocycle_!} and~\subref{lem:cocycle^*} are similar.

Finally, part~\subref{lem:cocycle^!} follows from part~\subref{lem:cocycle_!} by adjunction, using Lemma \ref{lem:_!^!compositionadjunction}.  To be more precise, what follows from part~\subref{lem:cocycle_!} is the commutativity of the following tetrahedron:
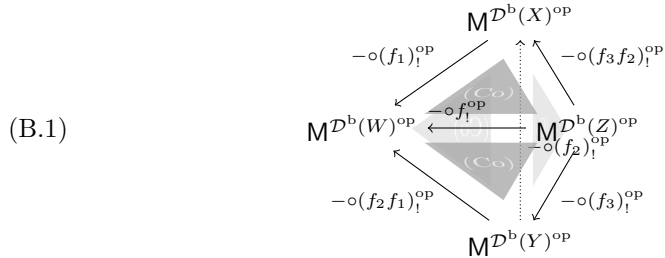
\begin{equation}\label{eqn:prismforcocycles}
\vc{\begin{tikzpicture}[stdtetr]
\compuinit
% hidden 2-cells
\node[tetrlr] at ({-\tric/2},0.5,\tric) {$\Comp$};
\node[tetrdr] at ({\tric/2},0.5,\tric) {$\Comp$};
% outer 0-cells
\node (ru) at (0,1,0) {$\Vect^{\cDb(X)^\op}$};
\node (fl) at (-0.5,0.5,1) {$\Vect^{\cDb(W)^\op}$};
\node (fr) at (0.5,0.5,1) {$\Vect^{\cDb(Z)^\op}$};
\node (rd) at (0,0,0) {$\Vect^{\cDb(Y)^\op}$};
% hidden 1-cell
\draw[liner,<-] (ru) -- node[arl,pos=.6] {$\al -\circ(f_2)_!^\op$} (rd);
% visible 2-cells
\node[tetrtf] at (0, {(1+\tric)/2}, {1-\tric}) {$\Comp$};
\node[tetrbf] at (0, {(1-\tric)/2}, {1-\tric}) {$\Comp$};
% visible 1-cells
\draw[<-] (fl) -- node[arl] {$\al -\circ(f_1)_!^\op$} (ru);
\draw[<-] (fl) -- node[arr] {$\al -\circ(f_2f_1)_!^\op$} (rd);
\draw[<-] (ru) -- node[arl] {$\al -\circ(f_3f_2)_!^\op$} (fr);
\draw[<-] (rd) -- node[arr] {$\al -\circ(f_3)_!^\op$} (fr);
\draw[<-] (fl) -- node[arl,pos=.3] {$\al -\circ f_!^\op$} (fr);
\end{tikzpicture}}
\end{equation}
Another description of this tetrahedron is as follows: it is obtained from the (not yet known to be commutative) tetrahedron in part~\subref{lem:cocycle^!} by gluing on four instances of Lemma~\ref{lem:_!^!compositionadjunction}, one to each face.  Because the Yoneda embedding is faithful, this implies that Figure~\ref{lem:cocycle^!} commutes (see Example \ref{ex:gluing-converse}).

%---------------------------------------------------------------------
\subsection{Constant sheaf and composition}
%---------------------------------------------------------------------

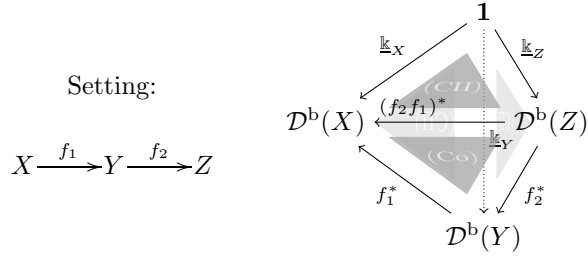
\begin{figure}
\begin{center}
\begin{tabular}{@{}c@{}}
Setting: \\
\ \\
$\xymatrix@1{X\ar[r]^{f_1}&Y\ar[r]^{f_2}&Z}$
\end{tabular}
\qquad
\qc{\begin{tikzpicture}[stdtetr]
\compuinit
% hidden 2-cells
\node[tetrlr] at ({-\tric/2},0.5,\tric) {$\Cnst$};
\node[tetrdr] at ({\tric/2},0.5,\tric) {$\Cnst$};
% outer 0-cells
\node (ru) at (0,1,0) {$\bb$};
\node (fl) at (-0.5,0.5,1) {$\cDb(X)$};
\node (fr) at (0.5,0.5,1) {$\cDb(Z)$};
\node (rd) at (0,0,0) {$\cDb(Y)$};
% hidden 1-cell
\draw[liner,->] (ru) -- node[arl,pos=.6] {$\al \ubk_Y$} (rd);
% visible 2-cells
\node[tetrtf] at (0, {(1+\tric)/2}, {1-\tric}) {$\Cnst$};
\node[tetrbf] at (0, {(1-\tric)/2}, {1-\tric}) {$\Comp$};
% visible 1-cells
\draw[<-] (fl) -- node[arl] {$\al \ubk_X$} (ru);
\draw[<-] (fl) -- node[arr] {$\al f_1^*$} (rd);
\draw[->] (ru) -- node[arl] {$\al \ubk_Z$} (fr);
\draw[<-] (rd) -- node[arr] {$\al f_2^*$} (fr);
\draw[<-] (fl) -- node[arl,pos=.3] {$\al (f_2f_1)^*$} (fr);
\end{tikzpicture}}
\end{center}
\caption{Constant sheaf and composition}\label{lem:^*compositionconstant}
\end{figure}

Since the isomorphism $\CnstRes$ was defined using the isomorphism $\Co$ for $(\cdot)^*$, this follows easily from Lemma~\ref{lem:cocycle^*}.

%---------------------------------------------------------------------
\subsection{Iterated composition}
%---------------------------------------------------------------------

\begin{figure}
\begin{center}
Setting:\qquad
\qc{\begin{tikzpicture}[smallprism]
\compuinit
% hidden 2-cells
%\node[cuber] at (0.5,0.5,0) {$\Comp$};
%\node[prismdf] at (1,0.5,\tric) {$\Comp$};
% outer 0-cells
\node (rlu) at (0,1,0) {$X$};
\node (rru) at (1,1,0) {$X'$};
\node (fr) at (	1,0.5,1) {$Y'$};
\node (rld) at (0,0,0) {$Z$};
\node (rrd) at (1,0,0) {$Z'$};
% hidden 1-cells
\draw[->] (rru) -- node[arr,pos=.3] {$\al f'$} (rrd);
% outer 1-cells
\draw[->] (rlu) -- node[arl] {$\al g_X$} (rru);
\draw[->] (rlu) -- node[arr] {$\al f$} (rld);
\draw[->] (rru) -- node[arl] {$\al f'_1$} (fr);
\draw[->] (rld) -- node[arr] {$\al g_Z$} (rrd);
\draw[->] (fr) -- node[arl] {$\al f'_2$} (rrd);
% visible 2-cells
%\node[prismtf] at (0.5,0.75,0.5) {$\Comp$};
%\node[prismbf] at (0.5,0.25,0.5) {$\Comp$};
%\node[prismlf] at (0,0.5,\tric) {$\Comp$};
% visible 0- and 1-cells
\node (fl) at (0,0.5,1) {$Y$};
\draw[->] (rlu) -- node[arl] {$\al f_1$} (fl);
\draw[->] (fl) -- node[arl,pos=.3] {$\al f_2$} (rld);
\draw[linef,->] (fl)  -- node[arr,pos=.3] {$\al g_Y$} (fr);
\end{tikzpicture}}

\subfigure[][]{\qc{\begin{tikzpicture}[stdprism]
\compuinit
% hidden 2-cells
\node[cuber] at (0.5,0.5,0) {$\Comp$};
\node[prismdf] at (1,0.5,\tric) {$\Comp$};
% outer 0-cells
\node (rlu) at (0,1,0) {$\cDb(X)$};
\node (rru) at (1,1,0) {$\cDb(X')$};
\node (fr) at (	1,0.5,1) {$\cDb(Y')$};
\node (rld) at (0,0,0) {$\cDb(Z)$};
\node (rrd) at (1,0,0) {$\cDb(Z')$};
% hidden 1-cells
\draw[liner,->] (rru) -- node[arr,pos=.3] {$\al f'_*$} (rrd);
% outer 1-cells
\draw[->] (rlu) -- node[arl] {$\al (g_X)_*$} (rru);
\draw[->] (rlu) -- node[arr] {$\al f_*$} (rld);
\draw[->] (rru) -- node[arl] {$\al (f'_1)_*$} (fr);
\draw[->] (rld) -- node[arr] {$\al (g_Z)_*$} (rrd);
\draw[->] (fr) -- node[arl] {$\al (f'_2)_*$} (rrd);
% visible 2-cells
\node[prismtf] at (0.5,0.75,0.5) {$\Comp$};
\node[prismbf] at (0.5,0.25,0.5) {$\Comp$};
\node[prismlf] at (0,0.5,\tric) {$\Comp$};
% visible 0- and 1-cells
\node (fl) at (0,0.5,1) {$\cDb(Y)$};
\draw[->] (rlu) -- node[arl] {$\al (f_1)_*$} (fl);
\draw[->] (fl) -- node[arl,pos=.3] {$\al (f_2)_*$} (rld);
\draw[->] (fl)  -- node[arr,pos=.3] {$\al (g_Y)_*$} (fr);
\end{tikzpicture}}\label{lem:_*composition_*}}~
\subfigure[][]{\qc{\begin{tikzpicture}[stdprism]
\compuinit
% hidden 2-cells
\node[cuber] at (0.5,0.5,0) {$\Comp$};
\node[prismdf] at (1,0.5,\tric) {$\Comp$};
% outer 0-cells
\node (rlu) at (0,1,0) {$\cDb(X)$};
\node (rru) at (1,1,0) {$\cDb(X')$};
\node (fr) at (	1,0.5,1) {$\cDb(Y')$};
\node (rld) at (0,0,0) {$\cDb(Z)$};
\node (rrd) at (1,0,0) {$\cDb(Z')$};
% hidden 1-cells
\draw[liner,->] (rru) -- node[arr,pos=.3] {$\al f'_!$} (rrd);
% outer 1-cells
\draw[->] (rlu) -- node[arl] {$\al (g_X)_!$} (rru);
\draw[->] (rlu) -- node[arr] {$\al f_!$} (rld);
\draw[->] (rru) -- node[arl] {$\al (f'_1)_!$} (fr);
\draw[->] (rld) -- node[arr] {$\al (g_Z)_!$} (rrd);
\draw[->] (fr) -- node[arl] {$\al (f'_2)_!$} (rrd);
% visible 2-cells
\node[prismtf] at (0.5,0.75,0.5) {$\Comp$};
\node[prismbf] at (0.5,0.25,0.5) {$\Comp$};
\node[prismlf] at (0,0.5,\tric) {$\Comp$};
% visible 0- and 1-cells
\node (fl) at (0,0.5,1) {$\cDb(Y)$};
\draw[->] (rlu) -- node[arl] {$\al (f_1)_!$} (fl);
\draw[->] (fl) -- node[arl,pos=.3] {$\al (f_2)_!$} (rld);
\draw[->] (fl)  -- node[arr,pos=.3] {$\al (g_Y)_!$} (fr);
\end{tikzpicture}}\label{lem:_!composition_!}}

\subfigure[][]{\qc{\begin{tikzpicture}[stdprism]
\compuinit
% hidden 2-cells
\node[cuber] at (0.5,0.5,0) {$\Comp$};
\node[prismdf] at (1,0.5,\tric) {$\Comp$};
% outer 0-cells
\node (rlu) at (0,1,0) {$\cDb(X)$};
\node (rru) at (1,1,0) {$\cDb(X')$};
\node (fr) at (	1,0.5,1) {$\cDb(Y')$};
\node (rld) at (0,0,0) {$\cDb(Z)$};
\node (rrd) at (1,0,0) {$\cDb(Z')$};
% hidden 1-cells
\draw[liner,<-] (rru) -- node[arr,pos=.3] {$\al (f')^*$} (rrd);
% outer 1-cells
\draw[<-] (rlu) -- node[arl] {$\al (g_X)^*$} (rru);
\draw[<-] (rlu) -- node[arr] {$\al f^*$} (rld);
\draw[<-] (rru) -- node[arl] {$\al (f'_1)^*$} (fr);
\draw[<-] (rld) -- node[arr] {$\al (g_Z)^*$} (rrd);
\draw[<-] (fr) -- node[arl] {$\al (f'_2)^*$} (rrd);
% visible 2-cells
\node[prismtf] at (0.5,0.75,0.5) {$\Comp$};
\node[prismbf] at (0.5,0.25,0.5) {$\Comp$};
\node[prismlf] at (0,0.5,\tric) {$\Comp$};
% visible 0- and 1-cells
\node (fl) at (0,0.5,1) {$\cDb(Y)$};
\draw[<-] (rlu) -- node[arl] {$\al (f_1)^*$} (fl);
\draw[<-] (fl) -- node[arl,pos=.3] {$\al (f_2)^*$} (rld);
\draw[<-] (fl)  -- node[arr,pos=.3] {$\al (g_Y)^*$} (fr);
\end{tikzpicture}}\label{lem:^*composition^*}}~
\subfigure[][]{\qc{\begin{tikzpicture}[stdprism]
\compuinit
% hidden 2-cells
\node[cuber] at (0.5,0.5,0) {$\Comp$};
\node[prismdf] at (1,0.5,\tric) {$\Comp$};
% outer 0-cells
\node (rlu) at (0,1,0) {$\cDb(X)$};
\node (rru) at (1,1,0) {$\cDb(X')$};
\node (fr) at (	1,0.5,1) {$\cDb(Y')$};
\node (rld) at (0,0,0) {$\cDb(Z)$};
\node (rrd) at (1,0,0) {$\cDb(Z')$};
% hidden 1-cells
\draw[liner,<-] (rru) -- node[arr,pos=.3] {$\al (f')^!$} (rrd);
% outer 1-cells
\draw[<-] (rlu) -- node[arl] {$\al (g_X)^!$} (rru);
\draw[<-] (rlu) -- node[arr] {$\al f^!$} (rld);
\draw[<-] (rru) -- node[arl] {$\al (f'_1)^!$} (fr);
\draw[<-] (rld) -- node[arr] {$\al (g_Z)^!$} (rrd);
\draw[<-] (fr) -- node[arl] {$\al (f'_2)^!$} (rrd);
% visible 2-cells
\node[prismtf] at (0.5,0.75,0.5) {$\Comp$};
\node[prismbf] at (0.5,0.25,0.5) {$\Comp$};
\node[prismlf] at (0,0.5,\tric) {$\Comp$};
% visible 0- and 1-cells
\node (fl) at (0,0.5,1) {$\cDb(Y)$};
\draw[<-] (rlu) -- node[arl] {$\al (f_1)^!$} (fl);
\draw[<-] (fl) -- node[arl,pos=.3] {$\al (f_2)^!$} (rld);
\draw[<-] (fl)  -- node[arr,pos=.3] {$\al (g_Y)^!$} (fr);
\end{tikzpicture}}\label{lem:^!composition^!}}
\end{center}
\caption{Iterated composition}\label{lem:composition}
\end{figure}

Part~\subref{lem:_*composition_*} follows from the gluing principle, since the prism can be obtained by gluing together three tetrahedra that are commutative by Lemma~\ref{lem:cocycle_*}, namely:
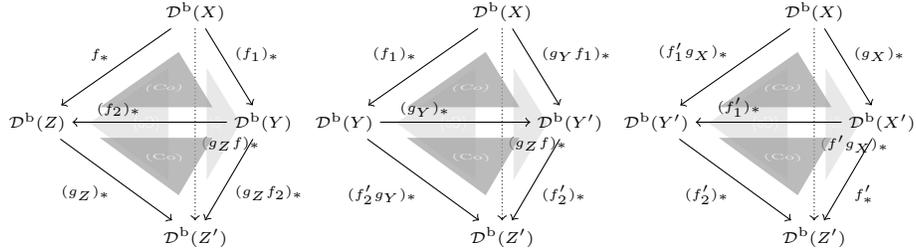
\[
\vc{\tiny\begin{tikzpicture}[stdtetr]
\compuinit
% hidden 2-cells
\node[tetrlr] at ({-\tric/2},0.5,\tric) {$\Comp$};
\node[tetrdr] at ({\tric/2},0.5,\tric) {$\Comp$};
% outer 0-cells
\node (ru) at (0,1,0) {$\cDb(X)$};
\node (fl) at (-0.5,0.5,1) {$\cDb(Z)$};
\node (fr) at (0.5,0.5,1) {$\cDb(Y)$};
\node (rd) at (0,0,0) {$\cDb(Z')$};
% hidden 1-cell
\draw[liner,->] (ru) -- node[arl,pos=.6] {$\al (g_Zf)_*$} (rd);
% visible 2-cells
\node[tetrtf] at (0, {(1+\tric)/2}, {1-\tric}) {$\Comp$};
\node[tetrbf] at (0, {(1-\tric)/2}, {1-\tric}) {$\Comp$};
% visible 1-cells
\draw[<-] (fl) -- node[arl] {$\al f_*$} (ru);
\draw[->] (fl) -- node[arr] {$\al (g_Z)_*$} (rd);
\draw[->] (ru) -- node[arl] {$\al (f_1)_*$} (fr);
\draw[<-] (rd) -- node[arr] {$\al (g_Zf_2)_*$} (fr);
\draw[<-] (fl) -- node[arl,pos=.3] {$\al (f_2)_*$} (fr);
\end{tikzpicture}}
\vc{\tiny\begin{tikzpicture}[stdtetr]
\compuinit
% hidden 2-cells
\node[tetrlr] at ({-\tric/2},0.5,\tric) {$\Comp$};
\node[tetrdr] at ({\tric/2},0.5,\tric) {$\Comp$};
% outer 0-cells
\node (ru) at (0,1,0) {$\cDb(X)$};
\node (fl) at (-0.5,0.5,1) {$\cDb(Y)$};
\node (fr) at (0.5,0.5,1) {$\cDb(Y')$};
\node (rd) at (0,0,0) {$\cDb(Z')$};
% hidden 1-cell
\draw[liner,->] (ru) -- node[arl,pos=.6] {$\al (g_Zf)_*$} (rd);
% visible 2-cells
\node[tetrtf] at (0, {(1+\tric)/2}, {1-\tric}) {$\Comp$};
\node[tetrbf] at (0, {(1-\tric)/2}, {1-\tric}) {$\Comp$};
% visible 1-cells
\draw[<-] (fl) -- node[arl] {$\al (f_1)_*$} (ru);
\draw[->] (fl) -- node[arr] {$\al (f'_2g_Y)_*$} (rd);
\draw[->] (ru) -- node[arl] {$\al (g_Yf_1)_*$} (fr);
\draw[<-] (rd) -- node[arr] {$\al (f'_2)_*$} (fr);
\draw[->] (fl) -- node[arl,pos=.3] {$\al (g_Y)_*$} (fr);
\end{tikzpicture}}
\vc{\tiny\begin{tikzpicture}[stdtetr]
\compuinit
% hidden 2-cells
\node[tetrlr] at ({-\tric/2},0.5,\tric) {$\Comp$};
\node[tetrdr] at ({\tric/2},0.5,\tric) {$\Comp$};
% outer 0-cells
\node (ru) at (0,1,0) {$\cDb(X)$};
\node (fl) at (-0.5,0.5,1) {$\cDb(Y')$};
\node (fr) at (0.5,0.5,1) {$\cDb(X')$};
\node (rd) at (0,0,0) {$\cDb(Z')$};
% hidden 1-cell
\draw[liner,->] (ru) -- node[arl,pos=.6] {$\al (f'g_X)_*$} (rd);
% visible 2-cells
\node[tetrtf] at (0, {(1+\tric)/2}, {1-\tric}) {$\Comp$};
\node[tetrbf] at (0, {(1-\tric)/2}, {1-\tric}) {$\Comp$};
% visible 1-cells
\draw[<-] (fl) -- node[arl] {$\al (f'_1g_X)_*$} (ru);
\draw[->] (fl) -- node[arr] {$\al (f'_2)_*$} (rd);
\draw[->] (ru) -- node[arl] {$\al (g_X)_*$} (fr);
\draw[<-] (rd) -- node[arr] {$\al f'_*$} (fr);
\draw[<-] (fl) -- node[arl,pos=.3] {$\al (f'_1)_*$} (fr);
\end{tikzpicture}}
\]
The proofs of parts~\subref{lem:_!composition_!}--\subref{lem:^!composition^!} are similar, using the other parts of Lemma~\ref{lem:cocycle}.

%---------------------------------------------------------------------
\subsection{Base change and composition}
%---------------------------------------------------------------------

\begin{figure}
\begin{center}
Setting:\qquad
\qc{\begin{tikzpicture}[smallprism]
\compuinit
% hidden 2-cells
\node[cart] at (0.5,0.5,0) {};
%\node[prismdf] at (1,0.5,\tric) {$\Comp$};
% outer 0-cells
\node (rlu) at (0,1,0) {$X$};
\node (rru) at (1,1,0) {$X'$};
\node (fr) at (	1,0.5,1) {$Y'$};
\node (rld) at (0,0,0) {$Z$};
\node (rrd) at (1,0,0) {$Z'$};
% hidden 1-cells
\draw[->] (rru) -- node[arr,pos=.3] {$\al f'$} (rrd);
% outer 1-cells
\draw[->] (rlu) -- node[arl] {$\al g_X$} (rru);
\draw[->] (rlu) -- node[arr] {$\al f$} (rld);
\draw[->] (rru) -- node[arl] {$\al f'_1$} (fr);
\draw[->] (rld) -- node[arr] {$\al g_Z$} (rrd);
\draw[->] (fr) -- node[arl] {$\al f'_2$} (rrd);
% visible 2-cells
\node[cart,diagup] at (0.5,0.75,0.5) {};
\node[cart,diagdp] at (0.5,0.25,0.5) {};
%\node[prismlf] at (0,0.5,\tric) {$\Comp$};
% visible 0- and 1-cells
\node (fl) at (0,0.5,1) {$Y$};
\draw[->] (rlu) -- node[arl] {$\al f_1$} (fl);
\draw[->] (fl) -- node[arl,pos=.3] {$\al f_2$} (rld);
\draw[linef,->] (fl)  -- node[arr,pos=.3] {$\al g_Y$} (fr);
\end{tikzpicture}}

\subfigure[][]{\qc{\begin{tikzpicture}[stdprism]
\compuinit
% hidden 2-cells
\node[cuber] at (0.5,0.5,0) {$\BC$};
\node[prismdf] at (1,0.5,\tric) {$\Comp$};
% outer 0-cells
\node (rlu) at (0,1,0) {$\cDb(X)$};
\node (rru) at (1,1,0) {$\cDb(X')$};
\node (fr) at (	1,0.5,1) {$\cDb(Y')$};
\node (rld) at (0,0,0) {$\cDb(Z)$};
\node (rrd) at (1,0,0) {$\cDb(Z')$};
% hidden 1-cells
\draw[liner,<-] (rru) -- node[arr,pos=.3] {$\al (f')^*$} (rrd);
% outer 1-cells
\draw[->] (rlu) -- node[arl] {$\al (g_X)_!$} (rru);
\draw[<-] (rlu) -- node[arr] {$\al f^*$} (rld);
\draw[<-] (rru) -- node[arl] {$\al (f'_1)^*$} (fr);
\draw[->] (rld) -- node[arr] {$\al (g_Z)_!$} (rrd);
\draw[<-] (fr) -- node[arl] {$\al (f'_2)^*$} (rrd);
% visible 2-cells
\node[prismtf] at (0.5,0.75,0.5) {$\BC$};
\node[prismbf] at (0.5,0.25,0.5) {$\BC$};
\node[prismlf] at (0,0.5,\tric) {$\Comp$};
% visible 0- and 1-cells
\node (fl) at (0,0.5,1) {$\cDb(Y)$};
\draw[<-] (rlu) -- node[arl] {$\al (f_1)^*$} (fl);
\draw[<-] (fl) -- node[arl,pos=.3] {$\al (f_2)^*$} (rld);
\draw[->] (fl)  -- node[arr,pos=.3] {$\al (g_Y)_!$} (fr);
\end{tikzpicture}}\label{lem:^*composition_!}}~
\subfigure[][]{\qc{\begin{tikzpicture}[stdprism]
\compuinit
% hidden 2-cells
\node[cuber] at (0.5,0.5,0) {$\BC$};
\node[prismdf] at (1,0.5,\tric) {$\Comp$};
% outer 0-cells
\node (rlu) at (0,1,0) {$\cDb(X)$};
\node (rru) at (1,1,0) {$\cDb(X')$};
\node (fr) at (	1,0.5,1) {$\cDb(Y')$};
\node (rld) at (0,0,0) {$\cDb(Z)$};
\node (rrd) at (1,0,0) {$\cDb(Z')$};
% hidden 1-cells
\draw[liner,->] (rru) -- node[arr,pos=.3] {$\al (f')_!$} (rrd);
% outer 1-cells
\draw[<-] (rlu) -- node[arl] {$\al (g_X)^*$} (rru);
\draw[->] (rlu) -- node[arr] {$\al f_!$} (rld);
\draw[->] (rru) -- node[arl] {$\al (f'_1)_!$} (fr);
\draw[<-] (rld) -- node[arr] {$\al (g_Z)^*$} (rrd);
\draw[->] (fr) -- node[arl] {$\al (f'_2)_!$} (rrd);
% visible 2-cells
\node[prismtf] at (0.5,0.75,0.5) {$\BC$};
\node[prismbf] at (0.5,0.25,0.5) {$\BC$};
\node[prismlf] at (0,0.5,\tric) {$\Comp$};
% visible 0- and 1-cells
\node (fl) at (0,0.5,1) {$\cDb(Y)$};
\draw[->] (rlu) -- node[arl] {$\al (f_1)_!$} (fl);
\draw[->] (fl) -- node[arl,pos=.3] {$\al (f_2)_!$} (rld);
\draw[<-] (fl)  -- node[arr,pos=.3] {$\al (g_Y)^*$} (fr);
\end{tikzpicture}}\label{lem:_!composition^*}}

\subfigure[][]{\qc{\begin{tikzpicture}[stdprism]
\compuinit
% hidden 2-cells
\node[cuber] at (0.5,0.5,0) {$\BC$};
\node[prismdf] at (1,0.5,\tric) {$\Comp$};
% outer 0-cells
\node (rlu) at (0,1,0) {$\cDb(X)$};
\node (rru) at (1,1,0) {$\cDb(X')$};
\node (fr) at (	1,0.5,1) {$\cDb(Y')$};
\node (rld) at (0,0,0) {$\cDb(Z)$};
\node (rrd) at (1,0,0) {$\cDb(Z')$};
% hidden 1-cells
\draw[liner,<-] (rru) -- node[arr,pos=.3] {$\al (f')^!$} (rrd);
% outer 1-cells
\draw[->] (rlu) -- node[arl] {$\al (g_X)_*$} (rru);
\draw[<-] (rlu) -- node[arr] {$\al f^!$} (rld);
\draw[<-] (rru) -- node[arl] {$\al (f'_1)^!$} (fr);
\draw[->] (rld) -- node[arr] {$\al (g_Z)_*$} (rrd);
\draw[<-] (fr) -- node[arl] {$\al (f'_2)^!$} (rrd);
% visible 2-cells
\node[prismtf] at (0.5,0.75,0.5) {$\BC$};
\node[prismbf] at (0.5,0.25,0.5) {$\BC$};
\node[prismlf] at (0,0.5,\tric) {$\Comp$};
% visible 0- and 1-cells
\node (fl) at (0,0.5,1) {$\cDb(Y)$};
\draw[<-] (rlu) -- node[arl] {$\al (f_1)^!$} (fl);
\draw[<-] (fl) -- node[arl,pos=.3] {$\al (f_2)^!$} (rld);
\draw[->] (fl)  -- node[arr,pos=.3] {$\al (g_Y)_*$} (fr);
\end{tikzpicture}}\label{lem:^!composition_*}}~
\subfigure[][]{\qc{\begin{tikzpicture}[stdprism]
\compuinit
% hidden 2-cells
\node[cuber] at (0.5,0.5,0) {$\BC$};
\node[prismdf] at (1,0.5,\tric) {$\Comp$};
% outer 0-cells
\node (rlu) at (0,1,0) {$\cDb(X)$};
\node (rru) at (1,1,0) {$\cDb(X')$};
\node (fr) at (	1,0.5,1) {$\cDb(Y')$};
\node (rld) at (0,0,0) {$\cDb(Z)$};
\node (rrd) at (1,0,0) {$\cDb(Z')$};
% hidden 1-cells
\draw[liner,->] (rru) -- node[arr,pos=.3] {$\al (f')_*$} (rrd);
% outer 1-cells
\draw[<-] (rlu) -- node[arl] {$\al (g_X)^!$} (rru);
\draw[->] (rlu) -- node[arr] {$\al f_*$} (rld);
\draw[->] (rru) -- node[arl] {$\al (f'_1)_*$} (fr);
\draw[<-] (rld) -- node[arr] {$\al (g_Z)^!$} (rrd);
\draw[->] (fr) -- node[arl] {$\al (f'_2)_*$} (rrd);
% visible 2-cells
\node[prismtf] at (0.5,0.75,0.5) {$\BC$};
\node[prismbf] at (0.5,0.25,0.5) {$\BC$};
\node[prismlf] at (0,0.5,\tric) {$\Comp$};
% visible 0- and 1-cells
\node (fl) at (0,0.5,1) {$\cDb(Y)$};
\draw[->] (rlu) -- node[arl] {$\al (f_1)_*$} (fl);
\draw[->] (fl) -- node[arl,pos=.3] {$\al (f_2)_*$} (rld);
\draw[<-] (fl)  -- node[arr,pos=.3] {$\al (g_Y)^!$} (fr);
\end{tikzpicture}}\label{lem:_*composition^!}}
\end{center}
\caption{Base change and composition}\label{lem:basechangecomposition}
\end{figure}
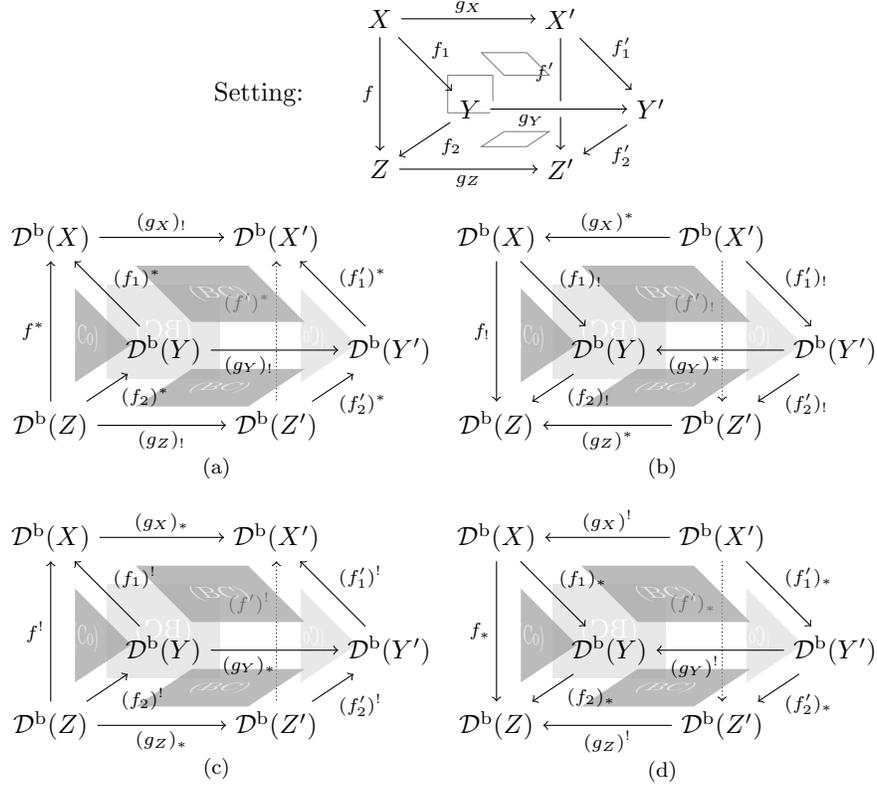

We begin with part~\subref{lem:^*composition_!}.  By construction, the base change isomorphism is deduced from a similar isomorphism between non-derived functors (which we denote with a sub- or superscript ``$0$''). As for Lemma \ref{lem:^*_*compositionadjunction}, one can check that it is enough to prove the corresponding statement for the non-derived functors. In concrete terms, to prove the latter statement we have to prove that the following diagram of isomorphisms of functors commutes:
\begin{equation}
\label{eqn:diag-^*composition_!}
\vcenter{
\xymatrix@R=12pt{
(f')^*_0 (g_Z)^0_! \ar@{<=>}[r] \ar@{<=>}[d] & (f'_1)^*_0 (f'_2)^*_0 (g_Z)_!^0 \ar@{<=>}[r] & (f'_1)^*_0 (g_Y)_!^0 (f_2)^*_0 \ar@{<=>}[d] \\
(g_X)_!^0 f^*_0 \ar@{<=>}[rr] & & (g_X)_!^0 (f_1)^*_0 (f_2)^*_0
}
}
\end{equation}
Now recall that the isomorphism $(f')^*_0 (g_Z)^0_! \natisom (g_X)_!^0 f^*_0$ is obtained by adjunction from the morphism of functors $(g_Z)^0_! f_*^0 \Rightarrow (f')_*^0 (g_X)_!^0$ induced by the composition isomorphism $(g_Z)^0_* f_*^0 \natisom (f')_*^0 (g_X)_*^0$, and similarly for the other base change isomorphisms (see \cite[Proposition 2.5.11]{kas}). One can check (using in particular the non-derived version of Lemma \ref{lem:^*_*compositionadjunction}) that the commutativity of diagram \eqref{eqn:diag-^*composition_!} follows from the commutativity of the following diagram:
\[
\xymatrix@R=12pt{
(g_Z)_!^0 f_*^0 \ar@{<=>}[r] \ar@{=>}[d] & (g_Z)_!^0 (f_2)_*^0 (f_1)_*^0 \ar@{=>}[r] & (f_2')_*^0 (g_Y)_!^0 (f_1)_*^0 \ar@{=>}[d] \\
(f')_*^0 (g_X)_!^0 \ar@{<=>}[rr] & & (f'_2)_*^0 (f'_1)_*^0 (g_X)_!^0
}
\]
which itself follows easily from the non-derived version of Lemma~\ref{lem:cocycle_*}. The proof of part~\subref{lem:_!composition^*} is similar.

The proof of part~\subref{lem:^!composition_*} is similar to that of Lemma~\ref{lem:cocycle^!}: the claim follows from part~\subref{lem:_!composition^*}, using Lemma~\ref{lem:basechangeadjunction} and Lemma~\ref{lem:_!^!compositionadjunction}. Similarly, part~\subref{lem:_*composition^!} follows from part~\subref{lem:^*composition_!}, using Lemma \ref{lem:basechangeadjunction} and Lemma \ref{lem:^*_*compositionadjunction}.

%---------------------------------------------------------------------
\subsection{Base change and iterated composition}
%---------------------------------------------------------------------

\begin{figure}
\begin{center}
Setting:\qquad
\qc{\begin{tikzpicture}[smallcube]
\compuinit
% hidden 0-cell
\node (rrd) at (1,0,0) {$X'$};
% hidden 2-cells
\node[cart] at (0.5,0.5,0) {};
\node[cart,horizc] at (0.5,0,0.5) {};
% outer 0-cells
\node (rlu) at (0,1,0) {$W$};
\node (rru) at (1,1,0) {$X$};
\node (fru) at (1,1,1) {$Z$};
\node (frd) at (1,0,1) {$Z'$};
\node (fld) at (0,0,1) {$Y'$};
\node (rld) at (0,0,0) {$W'$};
% hidden 1-cells
\draw[->] (rru) -- node[arl,pos=.7] {$\al h_X$} (rrd);
\draw[->] (rld) -- node[arl,pos=.3] {$\al g_1'$} (rrd);
\draw[->] (rrd) -- node[arl] {$\al f_1$} (frd);
% outer 1-cells
\draw[->] (rlu) -- node[arl] {$\al g'$} (rru);
\draw[->] (rru) -- node[arl] {$\al f$} (fru);
\draw[->] (fru) -- node[arl] {$\al h_Z$}(frd);
\draw[->] (fld) -- node[arr] {$\al g_1$} (frd);
\draw[->] (rld) -- node[arr] {$\al f'_1$} (fld);
\draw[->] (rlu) -- node[arr] {$\al h_W$} (rld);
% visible 2-cells
\node[cart,horizc] at (0.5,1,0.5) {};
\node[cart] at (0.5,0.5,1) {};
% visible 0- and 1-cells
\node (flu) at (0,1,1) {$Y$};
\draw[linef,->] (rlu) -- node[arr,pos=.7] {$\al f'$} (flu);
\draw[linef,->] (flu) -- node[arr,pos=.3] {$\al g$} (fru);
\draw[linef,->] (flu) -- node[arl,pos=.3] {$\al h_Y$} (fld);
\end{tikzpicture}}

\subfigure[][]{\qc{\begin{tikzpicture}[stdcube]
\compuinit
% hidden 0-cell
\node (rrd) at (1,0,0) {$\cDb(X')$};
% hidden 2-cells
\node[cuber] at (0.5,0.5,0) {$\BC$};
\node[cubed] at (1,0.5,0.5) {$\Comp$};
\node[cubeb] at (0.5,0,0.5) {$\BC$};
% outer 0-cells
\node (rlu) at (0,1,0) {$\cDb(W)$};
\node (rru) at (1,1,0) {$\cDb(X)$};
\node (fru) at (1,1,1) {$\cDb(Z)$};
\node (frd) at (1,0,1) {$\cDb(Z')$};
\node (fld) at (0,0,1) {$\cDb(Y')$};
\node (rld) at (0,0,0) {$\cDb(W')$};
% hidden 1-cells
\draw[liner,->] (rru) -- node[arl,pos=.7] {$\al (h_X)_!$} (rrd);
\draw[liner,<-] (rld) -- node[arl,pos=.3] {$\al (g_1')^*$} (rrd);
\draw[liner,->] (rrd) -- node[arl] {$\al (f_1)_!$} (frd);
% outer 1-cells
\draw[<-] (rlu) -- node[arl] {$\al (g')^*$} (rru);
\draw[->] (rru) -- node[arl] {$\al f_!$} (fru);
\draw[->] (fru) -- node[arl] {$\al (h_Z)_!$}(frd);
\draw[<-] (fld) -- node[arr] {$\al (g_1)^*$} (frd);
\draw[->] (rld) -- node[arr] {$\al (f_1')_!$} (fld);
\draw[->] (rlu) -- node[arr] {$\al (h_W)_!$} (rld);
% visible 2-cells
\node[cubel] at (0,0.5,0.5) {$\Comp$};
\node[cubet] at (0.5,1,0.5) {$\BC$};
\node[cubef] at (0.5,0.5,1) {$\BC$};
% visible 0- and 1-cells
\node (flu) at (0,1,1) {$\cDb(Y)$};
\draw[->] (rlu) -- node[arr,pos=.7] {$\al (f')_!$} (flu);
\draw[<-] (flu) -- node[arr,pos=.3] {$\al g^*$} (fru);
\draw[->] (flu) -- node[arl,pos=.3] {$\al (h_Y)_!$} (fld);
\end{tikzpicture}}\label{lem:_!^*basechange_!}}~
\subfigure[][]{\qc{\begin{tikzpicture}[stdcube]
\compuinit
% hidden 0-cell
\node (rrd) at (1,0,0) {$\cDb(X')$};
% hidden 2-cells
\node[cuber] at (0.5,0.5,0) {$\BC$};
\node[cubed] at (1,0.5,0.5) {$\Comp$};
\node[cubeb] at (0.5,0,0.5) {$\BC$};
% outer 0-cells
\node (rlu) at (0,1,0) {$\cDb(W)$};
\node (rru) at (1,1,0) {$\cDb(X)$};
\node (fru) at (1,1,1) {$\cDb(Z)$};
\node (frd) at (1,0,1) {$\cDb(Z')$};
\node (fld) at (0,0,1) {$\cDb(Y')$};
\node (rld) at (0,0,0) {$\cDb(W')$};
% hidden 1-cells
\draw[liner,->] (rru) -- node[arl,pos=.7] {$\al (h_X)_*$} (rrd);
\draw[liner,<-] (rld) -- node[arl,pos=.3] {$\al (g_1')^!$} (rrd);
\draw[liner,->] (rrd) -- node[arl] {$\al (f_1)_*$} (frd);
% outer 1-cells
\draw[<-] (rlu) -- node[arl] {$\al (g')^!$} (rru);
\draw[->] (rru) -- node[arl] {$\al f_*$} (fru);
\draw[->] (fru) -- node[arl] {$\al (h_Z)_*$}(frd);
\draw[<-] (fld) -- node[arr] {$\al (g_1)^!$} (frd);
\draw[->] (rld) -- node[arr] {$\al (f_1')_*$} (fld);
\draw[->] (rlu) -- node[arr] {$\al (h_W)_*$} (rld);
% visible 2-cells
\node[cubel] at (0,0.5,0.5) {$\Comp$};
\node[cubet] at (0.5,1,0.5) {$\BC$};
\node[cubef] at (0.5,0.5,1) {$\BC$};
% visible 0- and 1-cells
\node (flu) at (0,1,1) {$\cDb(Y)$};
\draw[->] (rlu) -- node[arr,pos=.7] {$\al (f')_*$} (flu);
\draw[<-] (flu) -- node[arr,pos=.3] {$\al g^!$} (fru);
\draw[->] (flu) -- node[arl,pos=.3] {$\al (h_Y)_*$} (fld);
\end{tikzpicture}}\label{lem:_*^!basechange_*}}

\subfigure[][]{\qc{\begin{tikzpicture}[stdcube]
\compuinit
% hidden 0-cell
\node (rrd) at (1,0,0) {$\cDb(X')$};
% hidden 2-cells
\node[cuber] at (0.5,0.5,0) {$\BC$};
\node[cubed] at (1,0.5,0.5) {$\Comp$};
\node[cubeb] at (0.5,0,0.5) {$\BC$};
% outer 0-cells
\node (rlu) at (0,1,0) {$\cDb(W)$};
\node (rru) at (1,1,0) {$\cDb(X)$};
\node (fru) at (1,1,1) {$\cDb(Z)$};
\node (frd) at (1,0,1) {$\cDb(Z')$};
\node (fld) at (0,0,1) {$\cDb(Y')$};
\node (rld) at (0,0,0) {$\cDb(W')$};
% hidden 1-cells
\draw[liner,<-] (rru) -- node[arl,pos=.7] {$\al (h_X)^*$} (rrd);
\draw[liner,->] (rld) -- node[arl,pos=.3] {$\al (g_1')_!$} (rrd);
\draw[liner,<-] (rrd) -- node[arl] {$\al (f_1)^*$} (frd);
% outer 1-cells
\draw[->] (rlu) -- node[arl] {$\al (g')_!$} (rru);
\draw[<-] (rru) -- node[arl] {$\al f^*$} (fru);
\draw[<-] (fru) -- node[arl] {$\al (h_Z)^*$}(frd);
\draw[->] (fld) -- node[arr] {$\al (g_1)_!$} (frd);
\draw[<-] (rld) -- node[arr] {$\al (f_1')^*$} (fld);
\draw[<-] (rlu) -- node[arr] {$\al (h_W)^*$} (rld);
% visible 2-cells
\node[cubel] at (0,0.5,0.5) {$\Comp$};
\node[cubet] at (0.5,1,0.5) {$\BC$};
\node[cubef] at (0.5,0.5,1) {$\BC$};
% visible 0- and 1-cells
\node (flu) at (0,1,1) {$\cDb(Y)$};
\draw[<-] (rlu) -- node[arr,pos=.7] {$\al (f')^*$} (flu);
\draw[->] (flu) -- node[arr,pos=.3] {$\al g_!$} (fru);
\draw[<-] (flu) -- node[arl,pos=.3] {$\al (h_Y)^*$} (fld);
\end{tikzpicture}}\label{lem:_!^*basechange^*}}~
\subfigure[][]{\qc{\begin{tikzpicture}[stdcube]
\compuinit
% hidden 0-cell
\node (rrd) at (1,0,0) {$\cDb(X')$};
% hidden 2-cells
\node[cuber] at (0.5,0.5,0) {$\BC$};
\node[cubed] at (1,0.5,0.5) {$\Comp$};
\node[cubeb] at (0.5,0,0.5) {$\BC$};
% outer 0-cells
\node (rlu) at (0,1,0) {$\cDb(W)$};
\node (rru) at (1,1,0) {$\cDb(X)$};
\node (fru) at (1,1,1) {$\cDb(Z)$};
\node (frd) at (1,0,1) {$\cDb(Z')$};
\node (fld) at (0,0,1) {$\cDb(Y')$};
\node (rld) at (0,0,0) {$\cDb(W')$};
% hidden 1-cells
\draw[liner,<-] (rru) -- node[arl,pos=.7] {$\al (h_X)^!$} (rrd);
\draw[liner,->] (rld) -- node[arl,pos=.3] {$\al (g_1')_*$} (rrd);
\draw[liner,<-] (rrd) -- node[arl] {$\al (f_1)^!$} (frd);
% outer 1-cells
\draw[->] (rlu) -- node[arl] {$\al (g')_*$} (rru);
\draw[<-] (rru) -- node[arl] {$\al f^!$} (fru);
\draw[<-] (fru) -- node[arl] {$\al (h_Z)^!$}(frd);
\draw[->] (fld) -- node[arr] {$\al (g_1)_*$} (frd);
\draw[<-] (rld) -- node[arr] {$\al (f_1')^!$} (fld);
\draw[<-] (rlu) -- node[arr] {$\al (h_W)^!$} (rld);
% visible 2-cells
\node[cubel] at (0,0.5,0.5) {$\Comp$};
\node[cubet] at (0.5,1,0.5) {$\BC$};
\node[cubef] at (0.5,0.5,1) {$\BC$};
% visible 0- and 1-cells
\node (flu) at (0,1,1) {$\cDb(Y)$};
\draw[<-] (rlu) -- node[arr,pos=.7] {$\al (f')^!$} (flu);
\draw[->] (flu) -- node[arr,pos=.3] {$\al g_*$} (fru);
\draw[<-] (flu) -- node[arl,pos=.3] {$\al (h_Y)^!$} (fld);
\end{tikzpicture}}\label{lem:_*^!basechange^!}}
\end{center}
\caption{Base change and iterated composition}\label{lem:compbasechange}
\end{figure}
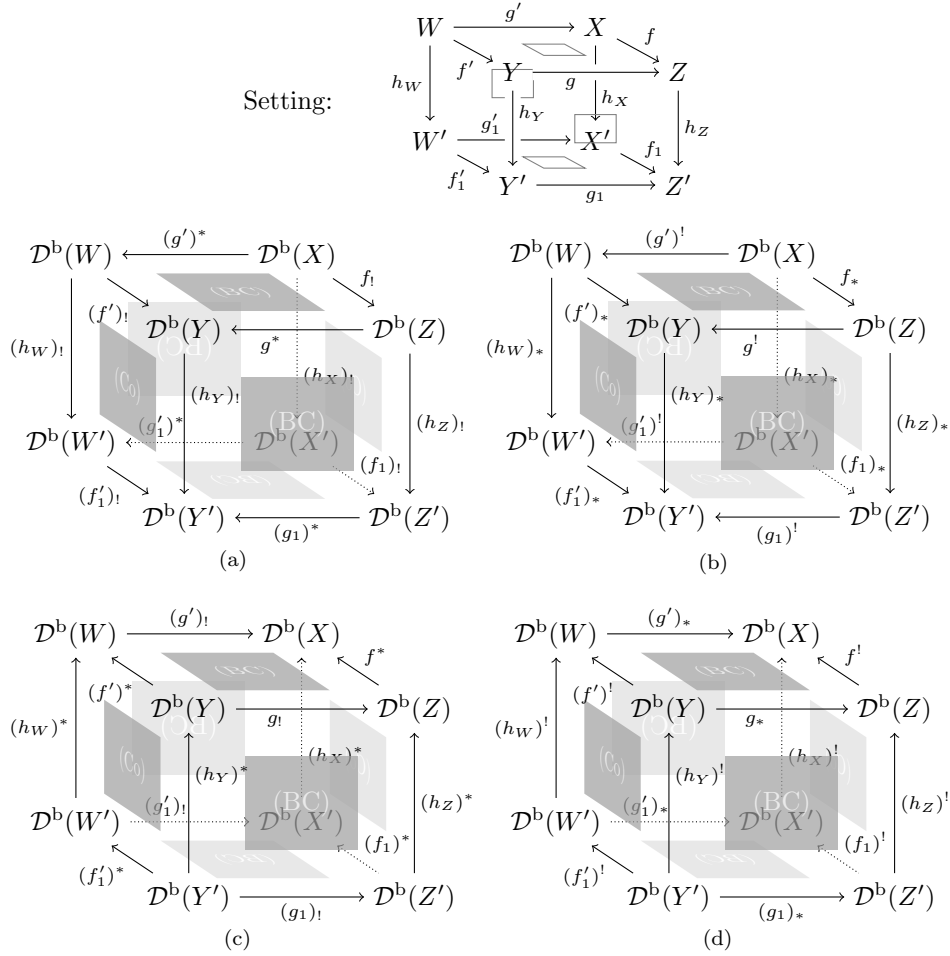

Part~\subref{lem:_!^*basechange_!} follows from the gluing principle, since the cube can be obtained by gluing together the following prisms, which are commutative by Lemma \ref{lem:_!composition^*}:
\[
\vc{\begin{tikzpicture}[stdprism]
\compuinit
% hidden 0-cell
\node (rr) at (	1,0.5,0) {$\cDb(X')$};
% hidden 2-cells
\node[prismtr] at (0.5,0.75,0.5) {$\BC$};
\node[prismbr] at (0.5,0.25,0.5) {$\BC$};
\node[prismdr] at (1,0.5,1-\tric) {$\Comp$};
% outer 0-cells
\node (flu) at (0,1,1) {$\cDb(W)$};
\node (fru) at (1,1,1) {$\cDb(X)$};
\node (rl) at (0,0.5,0) {$\cDb(W')$};
\node (fld) at (0,0,1) {$\cDb(Y')$};
\node (frd) at (1,0,1) {$\cDb(Z')$};
% hidden 1-cells
\draw[liner,<-] (rl) -- node[arl] {$\al (g'_1)^*$} (rr);
\draw[liner,->] (fru) -- node[arl] {$\al (h_X)_!$} (rr);
\draw[liner,->] (rr) -- node[arl] {$\al (f_1)_!$} (frd);
% outer 1-cells
\draw[<-] (flu) -- node[arl] {$\al (g')^*$} (fru);
\draw[->] (fru) -- node[arl] {$\al (f_1h_X)_!$} (frd);
\draw[<-] (fld) -- node[arr] {$\al (g_1)^*$} (frd);
\draw[->] (flu) -- node[arr] {$\al (h_W)_!$} (rl);
\draw[->] (rl) -- node[arr] {$\al (f'_1)_!$} (fld);
% visible 2-cells
\node[prismlr] at (0,0.5,1-\tric) {$\Comp$};
\node[cubef] at (0.5,0.5,1) {$\BC$};
% visible 1-cell
\draw[->] (flu) -- node[arl,pos=.7] {$\al (h_Yf')_!$} (fld);
\end{tikzpicture}}~
\vc{\begin{tikzpicture}[stdprism]
\compuinit
% hidden 2-cells
\node[cuber] at (0.5,0.5,0) {$\BC$};
\node[prismdf] at (1,0.5,\tric) {$\Comp$};
% outer 0-cells
\node (rlu) at (0,1,0) {$\cDb(W)$};
\node (rru) at (1,1,0) {$\cDb(X)$};
\node (fr) at (	1,0.5,1) {$\cDb(Z)$};
\node (rld) at (0,0,0) {$\cDb(Y')$};
\node (rrd) at (1,0,0) {$\cDb(Z')$};
% hidden 1-cells
\draw[liner,->] (rru) -- node[arr,pos=.3] {$\al (h_Zf)_!$} (rrd);
% outer 1-cells
\draw[<-] (rlu) -- node[arl] {$\al (g')^*$} (rru);
\draw[->] (rlu) -- node[arr] {$\al (h_Yf')_!$} (rld);
\draw[->] (rru) -- node[arl] {$\al f_!$} (fr);
\draw[<-] (rld) -- node[arr] {$\al (g_1)^*$} (rrd);
\draw[->] (fr) -- node[arl] {$\al (h_Z)_!$} (rrd);
% visible 2-cells
\node[prismtf] at (0.5,0.75,0.5) {$\BC$};
\node[prismbf] at (0.5,0.25,0.5) {$\BC$};
\node[prismlf] at (0,0.5,\tric) {$\Comp$};
% visible 0- and 1-cells
\node (fl) at (0,0.5,1) {$\cDb(Y)$};
\draw[->] (rlu) -- node[arl] {$\al (f')_!$} (fl);
\draw[->] (fl) -- node[arl,pos=.3] {$\al (h_Y)_!$} (rld);
\draw[<-] (fl)  -- node[arr,pos=.3] {$\al g^*$} (fr);
\end{tikzpicture}}
\]
The proofs of the other parts are similar.

%.....................................................................
\subsection{Equivariant versions of the above isomorphisms}
\label{subsect:equivariant}
%.....................................................................

Every isomorphism of functors described above has an equivariant version, where all varieties are assumed to have an action of an algebraic group $H$, every morphism is assumed to be $H$-equivariant, each derived category $\cDb(X)$ is replaced by the equivariant derived category $\cDb_H(X)$ of \cite{bl}, the constant sheaf $\ubk_X$ is replaced by the equivariant constant sheaf $\ubk^H_X$ of \cite[\S3.4.2]{bl}, and $f_*,f_!,f^*,f^!$ are defined as in \cite[\S3.3]{bl}. The equivariant versions of the isomorphisms are constructed from the ordinary isomorphisms, as explained in \cite[\S3.4]{bl}. We continue to use the notation `$\Co$', `$\BC$', and so on for the equivariant versions.

As mentioned before, we will cite any of Lemmas~\ref{lem:compositionadjunction}--\ref{lem:compbasechange} when we actually require the statement for the equivariant versions. To justify this, and for future reference, we briefly recall how the equivariant categories, functors and isomorphisms are defined.

For any $H$-variety $X$, an $H$-resolution $P$ of $X$ means a variety $P$ endowed with a free $H$-action and a smooth $H$-equivariant morphism $P\to X$. By definition, to specify an object $M$ of $\cDb_H(X)$ is to specify a compatible collection of objects of the categories $\cDb(H\backslash P)$ for various $H$-resolutions $P$ of $X$. More precisely, for each $P$ in a `sufficiently rich' class of $H$-resolutions of $X$ we must specify an object $M(P)$ of $\cDb(H\backslash P)$, and for any smooth morphism $g: P \to Q$ between such resolutions we must specify an isomorphism $\overline{g}^* (M(Q)) \cong M(P)$, where $\overline{g}: H\backslash P \to H\backslash Q$ is the morphism induced by $g$, such that a natural compatibility condition holds when we consider the composition of two smooth morphisms. See~\cite[\S\S2.4.4--2.4.5]{bl} for the details. 

The functors $f_*,f_!,f^*,f^!$ between equivariant derived categories are defined by means of the corresponding functors for the ordinary derived categories $\cDb(H\backslash P)$. Explicitly, if $f:X\to Y$ is an $H$-equivariant morphism and $M\in\cDb_H(X)$, then $f_*M\in\cDb_H(Y)$ is defined by $(f_*M)(P)=(\widetilde{f}^H_P)_* \bigl( M(P\times_Y X) \bigr)$, where the fibre product $P\times_Y X$ is defined using $f:X\to Y$, and $\widetilde{f}^H_P:H\backslash(P\times_Y X)\to H\backslash P$ is the map induced by the projection $P\times_Y X\to P$. The definition of $f_!$ is the same but with $(\widetilde{f}^H_P)_!$ instead of $(\widetilde{f}^H_P)_*$. If $N\in\cDb_H(Y)$, then $f^*N\in\cDb_H(X)$ is defined by $(f^*N)(P\times_Y X)=(\widetilde{f}^H_P)^* \bigl( N(P) \bigr)$. (The class of $H$-resolutions of $X$ of the form $P\times_Y X$ where $P$ is an $H$-resolution of $Y$ is `sufficiently rich'.) The definition of $f^!$ is the same but with $(\widetilde{f}^H_P)^!$ instead of $(\widetilde{f}^H_P)^*$.

As an example of an isomorphism of equivariant functors, consider the composition isomorphism for $(\cdot)_*$. Suppose we have $H$-equivariant morphisms $f:X\to Y$ and $g:Y\to Z$. To define an isomorphism between the two functors $(gf)_*:\cDb_H(X)\to\cDb_H(Z)$ and $g_*\circ f_*:\cDb_H(X)\to\cDb_H(Z)$, it suffices to define, for each object $M$ of $\cDb_H(X)$ and each $H$-resolution $P$ of $Z$, an isomorphism between $\bigl( (gf)_* M \bigr)(P)$ and $(g_*(f_*M))(P)$ that is suitably natural in $P$. But by definition, 
\[ (g_*(f_*M))(P) = (\widetilde{g}^H_P)_*\bigl((f_*M)(P\times_Z Y)\bigr) =
(\widetilde{g}^H_P)_*(\widetilde{f}^H_{P\times_Z Y})_*(M(P\times_Z X)), \]
where we have identified $(P\times_Z Y)\times_Y X$ with $P\times_Z X$. Since the composition $\widetilde{g}^H_P\widetilde{f}^H_{P\times_Z Y}:H\backslash(P\times_Z X)\to H\backslash P$ is exactly $\widetilde{(gf)}^H_P$, the ordinary $\Co$ isomorphism $(\widetilde{g}^H_P)_*\circ(\widetilde{f}^H_{P\times_Z Y})_*\natisom(\widetilde{(gf)}^H_P)_*$ provides the required isomorphism.

To show the equivariant version of Lemma~\ref{lem:cocycle_*}, we can restrict attention to a single object $M$ of $\cDb_H(W)$, and evaluate all the resulting objects of $\cDb_H(Z)$ at a single $H$-resolution $P$ of $Z$. Unravelling the definitions, the commutativity statement we have to prove becomes a special case of the ordinary Lemma~\ref{lem:cocycle_*}.

By similar arguments, every part of Lemmas~\ref{lem:compositionadjunction}--\ref{lem:compbasechange} implies the corresponding equivariant statement.

%.....................................................................
\subsection{Notation for isomorphisms of equivariant functors}
%.....................................................................

As well as the equivariant versions of $\Co$, $\BC$, etc., we need to consider some isomorphisms of functors specific to the equivariant setting.

%.....................................................................
\subsubsection{Forgetting and integration}
\label{sss:forgetting}
%.....................................................................

Let $K$ be a closed subgroup of $H$, and $X$ an $H$-variety. There is a `forgetful' functor $\For^H_K:\cDb_H(X)\to\cDb_K(X)$, denoted $\mathit{Res}_{K,H}$ in \cite[\S2.6.1]{bl}, which is defined so that for $M$ an object of $\cDb_H(X)$ and $P$ a $K$-resolution of $X$, we have 
\[ \bigl( \For^H_K M\bigr)(P)=M(H\times^K P). \] 
Here and subsequently, we use the obvious identification of $H\backslash(H\times^K P)$ with $K\backslash P$. When $K$ is the trivial group, $\For^H_K$ becomes the forgetful functor $\For:\cDb_H(X)\to\cDb(X)$ under the obvious identification of $\cDb_K(X)$ with $\cDb(X)$. 

We also have an `integration' functor $\hamma^H_K:\cDb_K(X)\to\cDb_H(X)$ defined as follows: for $M$ an object of $\cDb_K(X)$ and $P$ an $H$-resolution of $X$, we have 
\[ \bigl( \hamma_K^H M \bigr) (P)=(q_P)_! M(P)[2\dim(H/K)], \] 
where $q_P : K\backslash P \to H\backslash P$ is the quotient morphism and $M(P)$ is defined by regarding $P$ as a $K$-resolution of $X$. It is easy to see that $\hamma^H_K$ is isomorphic to the functor denoted $\mathit{Ind}_!$ in \cite[\S3.7.1]{bl}, and therefore it is left adjoint to $\For^H_K$. In fact, we can see this adjunction explicitly: for any $H$-resolution $P$ of $X$ and objects $M$ of $\cDb_K(X)$ and $N$ of $\cDb_H(X)$, we have natural isomorphisms
\begin{multline*}
\Hom_{\cDb(H\backslash P)} \bigl( (q_P)_! M(P)[2 \dim(H/K)], N(P) \bigr) \\ 
\cong \Hom_{\cDb(K\backslash P)} \bigl( M(P), (q_P)^! N(P) [-2\dim(H/K)] \bigr) \\
\cong \Hom_{\cDb(K\backslash P)} \bigl( M(P), (q_P)^* N(P) \bigr) \\
\cong \Hom_{\cDb(K\backslash P)} \bigl( M(P), N(H \times^K P) \bigr)
\end{multline*}
where the second isomorphism uses the isomorphism $(q_P)^! \natisom (q_P)^*[2\dim(H/K)]$ which holds since $q_P$ is smooth, and the third isomorphism uses the isomorphism $(q_P)^* N(P) \cong N(H \times^K P)$ which is part of the structure of $N$ as an object of $\cDb_H(X)$. We thus obtain an adjunction isomorphism
\[
\vc{\begin{tikzpicture}[smallcube]
\compuinit
% 2-cells
\node[cubef] at (0.5,0.5) {$\Adj$};
% 0-cells
\node (lu) at (0,1) {$\cDb_H(X)$};
\node (ru) at (1,1) {$\Vect^{\cDb_H(X)^\op}$};
\node (ld) at (0,0) {$\cDb_K(X)$};
\node (rd) at (1,0) {$\Vect^{\cDb_K(X)^\op}$};
% 1-cells
\draw[->] (lu) -- node[arl] {$\al \Yon$} (ru);
\draw[->] (lu) -- node[arr] {$\al \For^H_K$} (ld);
\draw[->] (ru) -- node[arl] {$\al -\circ(\hamma^H_K)^\op$} (rd);
\draw[->] (ld) -- node[arr] {$\al \Yon$} (rd);
\end{tikzpicture}}
\]

As stated in \cite[Theorem 3.4.1]{bl}, there are isomorphisms 
\[
\vc{\begin{tikzpicture}[smallercube]
\compuinit
% 2-cells
\node[cubef] at (0.5,0.5) {$\mFor$};
% 0-cells
\node (lu) at (0,1) {$\al\cDb_H(X)$};
\node (ru) at (1,1) {$\al\cDb_K(X)$};
\node (ld) at (0,0) {$\al\cDb_H(Y)$};
\node (rd) at (1,0) {$\al\cDb_K(Y)$};
% 1-cells
\draw[->] (lu) -- node[arl] {$\al \For^H_K$} (ru);
\draw[->] (lu) -- node[arr] {$\al f_*$} (ld);
\draw[->] (ru) -- node[arl] {$\al f_*$} (rd);
\draw[->] (ld) -- node[arr] {$\al \For^H_K$} (rd);
\end{tikzpicture}}
\;
\vc{\begin{tikzpicture}[smallercube]
\compuinit
% 2-cells
\node[cubef] at (0.5,0.5) {$\mFor$};
% 0-cells
\node (lu) at (0,1) {$\al\cDb_H(X)$};
\node (ru) at (1,1) {$\al\cDb_K(X)$};
\node (ld) at (0,0) {$\al\cDb_H(Y)$};
\node (rd) at (1,0) {$\al\cDb_K(Y)$};
% 1-cells
\draw[->] (lu) -- node[arl] {$\al \For^H_K$} (ru);
\draw[->] (lu) -- node[arr] {$\al f_!$} (ld);
\draw[->] (ru) -- node[arl] {$\al f_!$} (rd);
\draw[->] (ld) -- node[arr] {$\al \For^H_K$} (rd);
\end{tikzpicture}}
\;
\vc{\begin{tikzpicture}[smallercube]
\compuinit
% 2-cells
\node[cubef] at (0.5,0.5) {$\mFor$};
% 0-cells
\node (lu) at (0,1) {$\al\cDb_H(X)$};
\node (ru) at (1,1) {$\al\cDb_K(X)$};
\node (ld) at (0,0) {$\al\cDb_H(Y)$};
\node (rd) at (1,0) {$\al\cDb_K(Y)$};
% 1-cells
\draw[->] (lu) -- node[arl] {$\al \For^H_K$} (ru);
\draw[<-] (lu) -- node[arr] {$\al f^*$} (ld);
\draw[<-] (ru) -- node[arl] {$\al f^*$} (rd);
\draw[->] (ld) -- node[arr] {$\al \For^H_K$} (rd);
\end{tikzpicture}}
\;
\vc{\begin{tikzpicture}[smallercube]
\compuinit
% 2-cells
\node[cubef] at (0.5,0.5) {$\mFor$};
% 0-cells
\node (lu) at (0,1) {$\al\cDb_H(X)$};
\node (ru) at (1,1) {$\al\cDb_K(X)$};
\node (ld) at (0,0) {$\al\cDb_H(Y)$};
\node (rd) at (1,0) {$\al\cDb_K(Y)$};
% 1-cells
\draw[->] (lu) -- node[arl] {$\al \For^H_K$} (ru);
\draw[<-] (lu) -- node[arr] {$\al f^!$} (ld);
\draw[<-] (ru) -- node[arl] {$\al f^!$} (rd);
\draw[->] (ld) -- node[arr] {$\al \For^H_K$} (rd);
\end{tikzpicture}}
\]
for any $H$-morphism $f:X\to Y$. To illustrate, we explain the first of these isomorphisms. It suffices to define, for any object $M$ of $\cDb_H(X)$ and any $K$-resolution $P$ of $Y$, an isomorphism between $(\For^H_K f_*M)(P)$ and $(f_*\For^H_K M)(P)$ that is suitably natural in $P$. But by definition, 
\[ 
\begin{split}
(\For^H_K f_*M)(P)&=(\widetilde{f}^H_{H\times^K P})_* M((H\times^K P)\times_Y X),\ \text{and}\\
(f_*\For^H_K M)(P)&=(\widetilde{f}^K_{P})_* M(H\times^K(P\times_Y X)). 
\end{split}
\]
Thus, the required isomorphism is supplied by the obvious $H$-variety isomorphism $H\times^K(P\times_Y X)\simto(H\times^K P)\times_Y X$. 

As stated in \cite[Proposition 3.7.2]{bl}, there are isomorphisms
\[
\vc{\begin{tikzpicture}[smallcube]
\compuinit
% 2-cells
\node[cubef] at (0.5,0.5) {$\mInt$};
% 0-cells
\node (lu) at (0,1) {$\cDb_H(X)$};
\node (ru) at (1,1) {$\cDb_K(X)$};
\node (ld) at (0,0) {$\cDb_H(Y)$};
\node (rd) at (1,0) {$\cDb_K(Y)$};
% 1-cells
\draw[<-] (lu) -- node[arl] {$\al \hamma^H_K$} (ru);
\draw[<-] (lu) -- node[arr] {$\al f^*$} (ld);
\draw[<-] (ru) -- node[arl] {$\al f^*$} (rd);
\draw[<-] (ld) -- node[arr] {$\al \hamma^H_K$} (rd);
\end{tikzpicture}}
\qquad
\vc{\begin{tikzpicture}[smallcube]
\compuinit
% 2-cells
\node[cubef] at (0.5,0.5) {$\mInt$};
% 0-cells
\node (lu) at (0,1) {$\cDb_H(X)$};
\node (ru) at (1,1) {$\cDb_K(X)$};
\node (ld) at (0,0) {$\cDb_H(Y)$};
\node (rd) at (1,0) {$\cDb_K(Y)$};
% 1-cells
\draw[<-] (lu) -- node[arl] {$\al \hamma^H_K$} (ru);
\draw[->] (lu) -- node[arr] {$\al f_!$} (ld);
\draw[->] (ru) -- node[arl] {$\al f_!$} (rd);
\draw[<-] (ld) -- node[arr] {$\al \hamma^H_K$} (rd);
\end{tikzpicture}}
\]
for any $H$-morphism $f:X\to Y$. To define the first of these, it suffices to define, for any object $M$ of $\cDb_K(Y)$ and any $H$-resolution $P$ of $Y$, an isomorphism between $(\hamma^H_K f^* M)(P\times_Y X)$ and $(f^*\hamma^H_K M)(P\times_Y X)$ that is suitably natural in $P$. But by definition,
\[
\begin{split}
(\hamma^H_K f^* M)(P\times_Y X)&=(q_{P\times_Y X})_! (\widetilde{f}^K_P)^* M(P)[2\dim(H/K)],\ \text{and}\\
(f^*\hamma^H_K M)(P\times_Y X)&=(\widetilde{f}^H_P)^* (q_P)_! M(P)[2\dim(H/K)].
\end{split}
\]
Thus, the required isomorphism is supplied by the base change isomorphism for the following cartesian square:
\begin{equation} \label{eqn:quotient-square}
\vc{\begin{tikzpicture}[vsmallcube]
\compuinit
% 2-cells
\node[cart] at (1,0.5) {};
% 0-cells
\node (lu) at (0,1) {$K\backslash(P\times_Y X)$};
\node (ru) at (2,1) {$K\backslash P$};
\node (ld) at (0,0) {$H\backslash(P\times_Y X)$};
\node (rd) at (2,0) {$H\backslash P$};
% 1-cells
\draw[->] (lu) -- node[arl] {$\al \widetilde{f}^K_P$} (ru);
\draw[->] (lu) -- node[arr] {$\al q_{P\times_Y X}$} (ld);
\draw[->] (ru) -- node[arl] {$\al q_P$} (rd);
\draw[->] (ld) -- node[arr] {$\al \widetilde{f}^H_P$} (rd);
\end{tikzpicture}}
\end{equation}
The other $\mInt$ isomorphism is defined similarly, but using the composition isomorphism for $(\cdot)_!$ instead of base change.

%.....................................................................
\subsubsection{Transitivity of forgetting and integration}
\label{sss:transitivity}
%.....................................................................

If we have a chain of closed subgroups $K\subset J\subset H$, we have transitivity isomorphisms
\[
\vc{\begin{tikzpicture}[stdtriangle]
\compuinit
% 0-cells
\node (lu) at (0,1) {$\al \cDb_H(X)$};
\node (r) at (1,0.5) {$\al \cDb_J(X)$};
\node (ld) at (0,0) {$\al \cDb_K(X)$};
% 1-cells
\draw[->] (lu) -- node[arl] {$\al \For^H_J$} (r);
\draw[->] (lu) -- node[arr] {$\al \For^H_K$} (ld);
\draw[->] (r) -- node[arl] {$\al \For^J_K$} (ld);
% 2-cell
\node[tricell] at (\tric,0.5) {\tiny$\mTr$};
\end{tikzpicture}}
\qquad
\vc{\begin{tikzpicture}[stdtriangle]
\compuinit
% 0-cells
\node (lu) at (0,1) {$\al \cDb_H(X)$};
\node (r) at (1,0.5) {$\al \cDb_J(X)$};
\node (ld) at (0,0) {$\al \cDb_K(X)$};
% 1-cells
\draw[<-] (lu) -- node[arl] {$\al \hamma^H_J$} (r);
\draw[<-] (lu) -- node[arr] {$\al \hamma^H_K$} (ld);
\draw[<-] (r) -- node[arl] {$\al \hamma^J_K$} (ld);
% 2-cell
\node[tricell] at (\tric,0.5) {\tiny$\mTr$};
\end{tikzpicture}}
\]
The definition of the former uses the obvious identification of $H\times^J(J\times^K P)$ with $H\times^K P$, and the definition of the latter uses the composition isomorphism $(q_P^{K\subset H})_!\natisom(q_P^{J\subset H})_!\circ(q_P^{K\subset J})_!$, where the superscripts on $q_P$ indicate the groups involved.

%.....................................................................
\subsubsection{Constant sheaf under forgetting and integration}
\label{sss:constantsheaf-fi}
%.....................................................................

Let $K \subset H$ be a closed subgroup, and $X$ an $H$-variety. By definition, the equivariant constant sheaf $\ubk_X^H$ assigns to every $H$-resolution $P$ of $X$ the constant sheaf on $H\backslash P$. Hence we have a canonical isomorphism $\ubk_X^K \cong \For^H_K(\ubk_X^H)$.

Assume now that $H/K$ is contractible (for instance, that $H$ is the semidirect product of $K$ and a normal unipotent subgroup). Then for any $H$-resolution $P$ of $X$ the natural morphism $(q_P)_! \ubk_{K\backslash P} [2\dim(H/K)] \simto (q_P)_! (q_P)^! \ubk_{H\backslash P} \to \ubk_{H\backslash P}$
induced by adjunction is an isomorphism. We deduce a canonical isomorphism $\hamma_K^H(\ubk_X^K)\cong\ubk_X^H$. (In fact, $\hamma_K^H$ is left inverse to $\For_K^H$ in this situation; see~\cite[Theorem 3.7.3]{bl}.)

We depict the resulting isomorphisms of functors as follows:
\[
\vc{\begin{tikzpicture}[stdtriangle]
\compuinit
% 0-cells
\node (lu) at (0,1) {$\bb$};
\node (r) at (1,0.5) {$\cDb_K(X)$};
\node (ld) at (0,0) {$\cDb_H(X)$};
% 1-cells
\draw[->] (lu) -- node[arl] {$\al \ubk_X^K$} (r);
\draw[->] (lu) -- node[arr] {$\al \ubk_X^H$} (ld);
\draw[<-] (r) -- node[arl] {$\al \For_K^H$} (ld);
% 2-cell
\node[tricell] at (\tric,0.5) {\tiny$\CFor$};
\end{tikzpicture}}
\qquad
\vc{\begin{tikzpicture}[stdtriangle]
\compuinit
% 0-cells
\node (lu) at (0,1) {$\bb$};
\node (r) at (1,0.5) {$\cDb_K(X)$};
\node (ld) at (0,0) {$\cDb_H(X)$};
% 1-cells
\draw[->] (lu) -- node[arl] {$\al \ubk_X^K$} (r);
\draw[->] (lu) -- node[arr] {$\al \ubk_X^H$} (ld);
\draw[->] (r) -- node[arl] {$\al \hamma_K^H$} (ld);
% 2-cell
\node[tricell] at (\tric,0.5) {\tiny$\CInt$};
\end{tikzpicture}}
\]

%---------------------------------------------------------------------
% Start of equivariant commutativity lemmas
%---------------------------------------------------------------------

\addtocounter{figure}{2}

%---------------------------------------------------------------------
\subsection{Forgetting, integration, and adjunction}
%---------------------------------------------------------------------

\begin{figure}
\begin{center}
Setting:  \qquad $\xymatrix@1{X\ar[r]^f&Y}$ \quad and \quad $K \subset J \subset H$

\subfigure[][]{\qc{\begin{tikzpicture}[stdcube]
\compuinit
% hidden 0-cell
\node (rrd) at (1,0,0) {$\Vect^{\cDb_H(Y)^\op}$};
% hidden 2-cells
\node[cuber] at (0.5,0.5,0) {$\Adj$};
\node[cubed] at (1,0.5,0.5) {$\mInt$};
\node[cubeb] at (0.5,0,0.5) {$\Adj$};
% outer 0-cells
\node (rlu) at (0,1,0) {$\cDb_H(X)$};
\node (rru) at (1,1,0) {$\Vect^{\cDb_H(X)^\op}$};
\node (fru) at (1,1,1) {$\Vect^{\cDb_K(X)^\op}$};
\node (frd) at (1,0,1) {$\Vect^{\cDb_K(Y)^\op}$};
\node (fld) at (0,0,1) {$\cDb_K(Y)$};
\node (rld) at (0,0,0) {$\cDb_H(Y)$};
% hidden 1-cells
\draw[liner,->] (rru) -- node[arl,pos=.7] {$\al -\circ f^{*,\op}$} (rrd);
\draw[liner,->] (rld) -- node[arl,pos=.3] {$\al \Yon$} (rrd);
\draw[liner,->] (rrd) -- node[arl] {$\al -\circ(\hamma^H_K)^\op$} (frd);
% outer 1-cells
\draw[->] (rlu) -- node[arl] {$\al \Yon$} (rru);
\draw[->] (rru) -- node[arl] {$\al -\circ(\hamma^H_K)^\op$} (fru);
\draw[->] (fru) -- node[arl,pos=.3] {\rlap{$\al -\circ f^{*,\op}$}} (frd);
\draw[->] (fld) -- node[arr] {$\al \Yon$} (frd);
\draw[->] (rld) -- node[arr] {$\al \For^H_K$} (fld);
\draw[->] (rlu) -- node[arr] {$\al f_*$} (rld);
% visible 2-cells
\node[cubel] at (0,0.5,0.5) {$\mFor$};
\node[cubet] at (0.5,1,0.5) {$\Adj$};
\node[cubef] at (0.5,0.5,1) {$\Adj$};
% visible 0- and 1-cells
\node (flu) at (0,1,1) {$\cDb_K(X)$};
\draw[->] (rlu) -- node[arr,pos=.7] {$\al \For^H_K$} (flu);
\draw[->] (flu) -- node[arr,pos=.3] {$\al \Yon$} (fru);
\draw[->] (flu) -- node[arl,pos=.3] {$\al f_*$} (fld);
\end{tikzpicture}}\label{lem:for_*gamma^*adjunction}}~
\subfigure[][]{\qc{\begin{tikzpicture}[stdcube]
\compuinit
% hidden 0-cell
\node (rrd) at (1,0,0) {$\Vect^{\cDb_H(Y)^\op}$};
% hidden 2-cells
\node[cuber] at (0.5,0.5,0) {$\Adj$};
\node[cubed] at (1,0.5,0.5) {$\mInt$};
\node[cubeb] at (0.5,0,0.5) {$\Adj$};
% outer 0-cells
\node (rlu) at (0,1,0) {$\cDb_H(X)$};
\node (rru) at (1,1,0) {$\Vect^{\cDb_H(X)^\op}$};
\node (fru) at (1,1,1) {$\Vect^{\cDb_K(X)^\op}$};
\node (frd) at (1,0,1) {$\Vect^{\cDb_K(Y)^\op}$};
\node (fld) at (0,0,1) {$\cDb_K(Y)$};
\node (rld) at (0,0,0) {$\cDb_H(Y)$};
% hidden 1-cells
\draw[liner,<-] (rru) -- node[arl,pos=.7] {$\al -\circ(f_!)^\op$} (rrd);
\draw[liner,->] (rld) -- node[arl,pos=.3] {$\al \Yon$} (rrd);
\draw[liner,->] (rrd) -- node[arl] {$\al -\circ(\hamma^H_K)^\op$} (frd);
% outer 1-cells
\draw[->] (rlu) -- node[arl] {$\al \Yon$} (rru);
\draw[->] (rru) -- node[arl] {$\al -\circ(\hamma^H_K)^\op$} (fru);
\draw[<-] (fru) -- node[arl] {\rlap{$\al -\circ(f_!)^\op$}} (frd);
\draw[->] (fld) -- node[arr] {$\al \Yon$} (frd);
\draw[->] (rld) -- node[arr] {$\al \For^H_K$} (fld);
\draw[<-] (rlu) -- node[arr] {$\al f^!$} (rld);
% visible 2-cells
\node[cubel] at (0,0.5,0.5) {$\mFor$};
\node[cubet] at (0.5,1,0.5) {$\Adj$};
\node[cubef] at (0.5,0.5,1) {$\Adj$};
% visible 0- and 1-cells
\node (flu) at (0,1,1) {$\cDb_K(X)$};
\draw[->] (rlu) -- node[arr,pos=.7] {$\al \For^H_K$} (flu);
\draw[->] (flu) -- node[arr,pos=.3] {$\al \Yon$} (fru);
\draw[<-] (flu) -- node[arl,pos=.3] {$\al f^!$} (fld);
\end{tikzpicture}}\label{lem:for^!gamma_!adjunction}}

\subfigure[][]{\qc{\begin{tikzpicture}[stdprism]
\compuinit
% hidden 2-cells
\node[cuber] at (0.5,0.5,0) {$\Adj$};
\node[prismdf] at (1,0.5,\tric) {$\mTr$};
% outer 0-cells
\node (rlu) at (0,1,0) {$\cDb_H(X)$};
\node (rru) at (1,1,0) {$\Vect^{\cDb_H(X)^\op}$};
\node (fr) at (	1,0.5,1) {$\Vect^{\cDb_J(X)^\op}$};
\node (rld) at (0,0,0) {$\cDb_K(X)$};
\node (rrd) at (1,0,0) {$\Vect^{\cDb_K(X)^\op}$};
% hidden 1-cells
\draw[liner,->] (rru) -- node[arr,pos=.3] {$\al -\circ(\hamma^H_K)^\op$} (rrd);
% outer 1-cells
\draw[->] (rlu) -- node[arl] {$\al \Yon$} (rru);
\draw[->] (rlu) -- node[arr] {$\al \For^H_K$} (rld);
\draw[->] (rru) -- node[arl] {$\al -\circ(\hamma^H_J)^\op$} (fr);
\draw[->] (rld) -- node[arr] {$\al \Yon$} (rrd);
\draw[->] (fr) -- node[arl] {$\al -\circ(\hamma^J_K)^\op$} (rrd);
% visible 2-cells
\node[prismtf] at (0.5,0.75,0.5) {$\Adj$};
\node[prismbf] at (0.5,0.25,0.5) {$\Adj$};
\node[prismlf] at (0,0.5,\tric) {$\mTr$};
% visible 0- and 1-cells
\node (fl) at (0,0.5,1) {$\cDb_J(X)$};
\draw[->] (rlu) -- node[arl] {$\al \For^H_J$} (fl);
\draw[->] (fl) -- node[arl,pos=.3] {$\al \For^J_K$} (rld);
\draw[->] (fl)  -- node[arr,pos=.3] {$\al \Yon$} (fr);
\end{tikzpicture}}\label{lem:transitivityadjunction}}
\end{center}
\caption{Forgetting, integration, and adjunction}\label{lem:forgammaadjunction}
\end{figure}
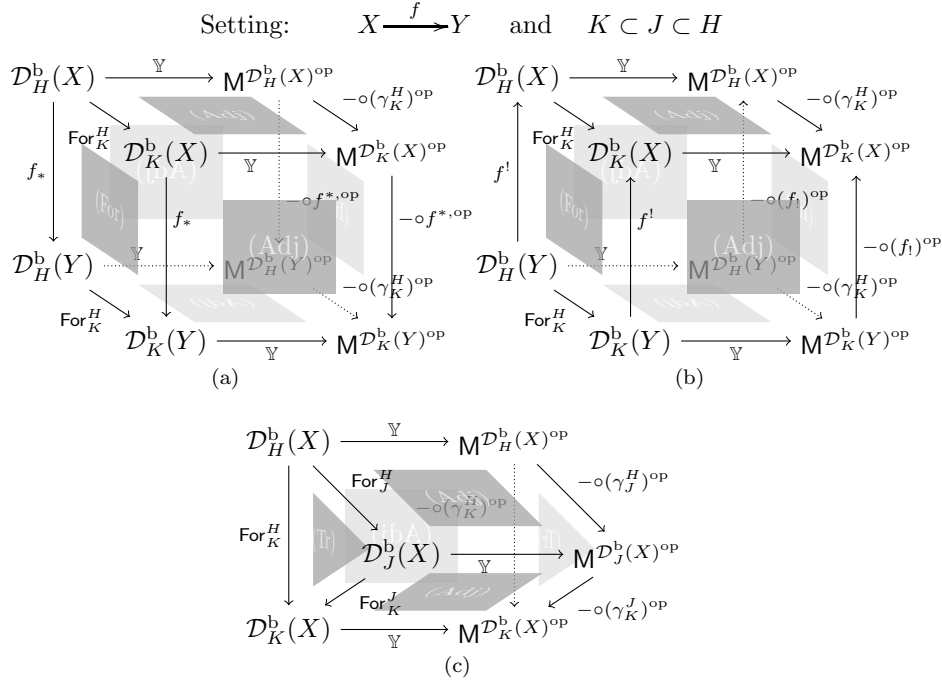

Unravelling the definitions, 
%one finds that 
part~\subref{lem:for_*gamma^*adjunction} is equivalent to the statement that for any $M$ in $\cDb_K(Y)$ and $N$ in $\cDb_H(X)$, and any $H$-resolution $P$ of $Y$, the following diagram of isomorphisms commutes:
\[
\vc{\begin{tikzpicture}[vsmallcube2]
\compuinit
% 0-cells
\node (lu) at (2,5) {$\Hom((\widetilde{f}^K_P)^*M(P),(q_{P\times_Y X})^!N(P\times_Y X))$};
%{$\Hom_{\cDb(K\backslash(P\times_Y X))}((f^*M)(P\times_Y X),(\For^H_K N)(P\times_Y X))$};
\node (ru) at (3.6,4) {$\Hom((q_{P\times_Y X})_!(\widetilde{f}^K_P)^*M(P),N(P\times_Y X))$};
%{$\Hom_{\cDb(H\backslash(P\times_Y X))}((\hamma^H_K f^*M)(P\times_Y X),N(P\times_Y X))$};
\node (lm) at (0,3) {$\Hom(M(P),(\widetilde{f}^K_P)_*(q_{P\times_Y X})^!N(P\times_Y X))$};
%{$\Hom_{\cDb(K\backslash P)}(M(P),(f_*\For^H_K N)(P))$};
\node (rm) at (3.6,2) {$\Hom((\widetilde{f}^H_P)^*(q_P)_!M(P),N(P\times_Y X))$};
%{$\Hom_{\cDb(H\backslash(P\times_Y X))}((f^*\hamma^H_K M)(P\times_Y X),N(P\times_Y X))$};
\node (ld) at (0,1) {$\Hom(M(P),(q_P)^!(\widetilde{f}^H_P)_*N(P\times_Y X))$};
%{$\Hom_{\cDb(K\backslash P)}(M(P),(\For^H_K f_*N)(P))$};
\node (rd) at (2,0) {$\Hom((q_P)_!M(P),(\widetilde{f}^H_P)_*N(P\times_Y X))$};
%{$\Hom_{\cDb(H\backslash P)}((q_P)_!M(P)[n],(f_* N)(P))$};
% 1-cells
\draw[<->] (lu) -- (ru);
\draw[<->] (lu) -- (lm);
\draw[<->] (lm) -- (ld);
\draw[<->] (ru) -- (rm);
\draw[<->] (rm) -- (rd);
\draw[<->] (ld) -- (rd);
\end{tikzpicture}}
\]
Here, to save space, we have omitted the subscripts 
indicating which derived categories we take $\Hom(\cdot,\cdot)$ in.
The isomorphisms are either adjunctions or base changes for the cartesian square \eqref{eqn:quotient-square}, 
so the commutativity of this diagram follows from Lemma~\ref{lem:basechangeadjunction}.
Similarly, parts~\subref{lem:for^!gamma_!adjunction} and~\subref{lem:transitivityadjunction} follow from Lemma~\ref{lem:_!^!compositionadjunction}. In proving \subref{lem:transitivityadjunction}, one also needs the fact that, when $P$ is an $H$-resolution of $X$, the composition
\[
(q_P^{K\subset J})^!\circ(q_P^{J\subset H})^!\natisom (q_P^{K\subset J})^*\circ(q_P^{J\subset H})^*[n]\overset{\Co}{\natisom}(q_P^{K\subset H})^*[n]\natisom (q_P^{K\subset H})^!
\]
(where $n=2\dim(H/K)$) coincides with $(q_P^{K\subset J})^!\circ(q_P^{J\subset H})^!\overset{\Co}{\natisom}(q_P^{K\subset H})^!$. 

%---------------------------------------------------------------------
\subsection{Forgetting, integration, and transitivity}
%---------------------------------------------------------------------

\begin{figure}
\begin{center}
Setting:  \qquad $\xymatrix@1{X\ar[r]^f&Y}$ \quad and \quad $I \subset K \subset J \subset H$

\subfigure[][]{\qc{\begin{tikzpicture}[stdtetr]
\compuinit
% hidden 2-cells
\node[tetrlr] at ({-\tric/2},0.5,\tric) {$\mTr$};
\node[tetrdr] at ({\tric/2},0.5,\tric) {$\mTr$};
% outer 0-cells
\node (ru) at (0,1,0) {$\cDb_K(X)$};
\node (fl) at (-0.5,0.5,1) {$\cDb_H(X)$};
\node (fr) at (0.5,0.5,1) {$\cDb_I(X)$};
\node (rd) at (0,0,0) {$\cDb_J(X)$};
% hidden 1-cell
\draw[liner,->] (ru) -- node[arl,pos=.6] {$\al \For^K_J$} (rd);
% visible 2-cells
\node[tetrtf] at (0, {(1+\tric)/2}, {1-\tric}) {$\mTr$};
\node[tetrbf] at (0, {(1-\tric)/2}, {1-\tric}) {$\mTr$};
% visible 1-cells
\draw[->] (fl) -- node[arl] {$\al \For^H_K$} (ru);
\draw[->] (fl) -- node[arr] {$\al \For^H_J$} (rd);
\draw[->] (ru) -- node[arl] {$\al \For^K_I$} (fr);
\draw[->] (rd) -- node[arr] {$\al \For^J_I$} (fr);
\draw[->] (fl) -- node[arl,pos=.3] {$\al \For^H_I$} (fr);
\end{tikzpicture}}\label{lem:cocycle-For}}
\qquad\qquad
\subfigure[][]{\qc{\begin{tikzpicture}[stdtetr]
\compuinit
% hidden 2-cells
\node[tetrlr] at ({-\tric/2},0.5,\tric) {$\mTr$};
\node[tetrdr] at ({\tric/2},0.5,\tric) {$\mTr$};
% outer 0-cells
\node (ru) at (0,1,0) {$\cDb_K(X)$};
\node (fl) at (-0.5,0.5,1) {$\cDb_H(X)$};
\node (fr) at (0.5,0.5,1) {$\cDb_I(X)$};
\node (rd) at (0,0,0) {$\cDb_J(X)$};
% hidden 1-cell
\draw[liner,<-] (ru) -- node[arl,pos=.6] {$\al \hamma^K_J$} (rd);
% visible 2-cells
\node[tetrtf] at (0, {(1+\tric)/2}, {1-\tric}) {$\mTr$};
\node[tetrbf] at (0, {(1-\tric)/2}, {1-\tric}) {$\mTr$};
% visible 1-cells
\draw[<-] (fl) -- node[arl] {$\al \hamma^H_K$} (ru);
\draw[<-] (fl) -- node[arr] {$\al \hamma^H_J$} (rd);
\draw[<-] (ru) -- node[arl] {$\al \hamma^K_I$} (fr);
\draw[<-] (rd) -- node[arr] {$\al \hamma^J_I$} (fr);
\draw[<-] (fl) -- node[arl,pos=.3] {$\al \hamma^H_I$} (fr);
\end{tikzpicture}}\label{lem:cocycle-Gamma}}

\subfigure[][]{\qc{\begin{tikzpicture}[stdprism]
\compuinit
% hidden 2-cells
\node[cuber] at (0.5,0.5,0) {$\mFor$};
\node[prismdf] at (1,0.5,\tric) {$\mTr$};
% outer 0-cells
\node (rlu) at (0,1,0) {$\cDb_H(X)$};
\node (rru) at (1,1,0) {$\cDb_H(Y)$};
\node (fr) at (	1,0.5,1) {$\cDb_J(Y)$};
\node (rld) at (0,0,0) {$\cDb_K(X)$};
\node (rrd) at (1,0,0) {$\cDb_K(Y)$};
% hidden 1-cells
\draw[liner,->] (rru) -- node[arr,pos=.3] {$\al \For^H_K$} (rrd);
% outer 1-cells
\draw[->] (rlu) -- node[arl] {$\al f_*$} (rru);
\draw[->] (rlu) -- node[arr] {$\al \For^H_K$} (rld);
\draw[->] (rru) -- node[arl] {$\al \For^H_J$} (fr);
\draw[->] (rld) -- node[arr] {$\al f_*$} (rrd);
\draw[->] (fr) -- node[arl] {$\al \For^J_K$} (rrd);
% visible 2-cells
\node[prismtf] at (0.5,0.75,0.5) {$\mFor$};
\node[prismbf] at (0.5,0.25,0.5) {$\mFor$};
\node[prismlf] at (0,0.5,\tric) {$\mTr$};
% visible 0- and 1-cells
\node (fl) at (0,0.5,1) {$\cDb_J(X)$};
\draw[->] (rlu) -- node[arl] {$\al \For^H_J$} (fl);
\draw[->] (fl) -- node[arl,pos=.3] {$\al \For^J_K$} (rld);
\draw[->] (fl)  -- node[arr,pos=.3] {$\al f_*$} (fr);
\end{tikzpicture}}\label{lem:for_*for}}~
\subfigure[][]{\qc{\begin{tikzpicture}[stdprism]
\compuinit
% hidden 2-cells
\node[cuber] at (0.5,0.5,0) {$\mFor$};
\node[prismdf] at (1,0.5,\tric) {$\mTr$};
% outer 0-cells
\node (rlu) at (0,1,0) {$\cDb_H(X)$};
\node (rru) at (1,1,0) {$\cDb_H(Y)$};
\node (fr) at (	1,0.5,1) {$\cDb_J(Y)$};
\node (rld) at (0,0,0) {$\cDb_K(X)$};
\node (rrd) at (1,0,0) {$\cDb_K(Y)$};
% hidden 1-cells
\draw[liner,->] (rru) -- node[arr,pos=.3] {$\al \For^H_K$} (rrd);
% outer 1-cells
\draw[<-] (rlu) -- node[arl] {$\al f^*$} (rru);
\draw[->] (rlu) -- node[arr] {$\al \For^H_K$} (rld);
\draw[->] (rru) -- node[arl] {$\al \For^H_J$} (fr);
\draw[<-] (rld) -- node[arr] {$\al f^*$} (rrd);
\draw[->] (fr) -- node[arl] {$\al \For^J_K$} (rrd);
% visible 2-cells
\node[prismtf] at (0.5,0.75,0.5) {$\mFor$};
\node[prismbf] at (0.5,0.25,0.5) {$\mFor$};
\node[prismlf] at (0,0.5,\tric) {$\mTr$};
% visible 0- and 1-cells
\node (fl) at (0,0.5,1) {$\cDb_J(X)$};
\draw[->] (rlu) -- node[arl] {$\al \For^H_J$} (fl);operator 
\draw[->] (fl) -- node[arl,pos=.3] {$\al \For^J_K$} (rld);
\draw[<-] (fl)  -- node[arr,pos=.3] {$\al f^*$} (fr);
\end{tikzpicture}}\label{lem:for^*for}}

\subfigure[][]{\qc{\begin{tikzpicture}[stdprism]
\compuinit
% hidden 2-cells
\node[cuber] at (0.5,0.5,0) {$\mFor$};
\node[prismdf] at (1,0.5,\tric) {$\mTr$};
% outer 0-cells
\node (rlu) at (0,1,0) {$\cDb_H(X)$};
\node (rru) at (1,1,0) {$\cDb_H(Y)$};
\node (fr) at (	1,0.5,1) {$\cDb_J(Y)$};
\node (rld) at (0,0,0) {$\cDb_K(X)$};
\node (rrd) at (1,0,0) {$\cDb_K(Y)$};
% hidden 1-cells
\draw[liner,->] (rru) -- node[arr,pos=.3] {$\al \For^H_K$} (rrd);
% outer 1-cells
\draw[->] (rlu) -- node[arl] {$\al f_!$} (rru);
\draw[->] (rlu) -- node[arr] {$\al \For^H_K$} (rld);
\draw[->] (rru) -- node[arl] {$\al \For^H_J$} (fr);
\draw[->] (rld) -- node[arr] {$\al f_!$} (rrd);
\draw[->] (fr) -- node[arl] {$\al \For^J_K$} (rrd);
% visible 2-cells
\node[prismtf] at (0.5,0.75,0.5) {$\mFor$};
\node[prismbf] at (0.5,0.25,0.5) {$\mFor$};
\node[prismlf] at (0,0.5,\tric) {$\mTr$};
% visible 0- and 1-cells
\node (fl) at (0,0.5,1) {$\cDb_J(X)$};
\draw[->] (rlu) -- node[arl] {$\al \For^H_J$} (fl);
\draw[->] (fl) -- node[arl,pos=.3] {$\al \For^J_K$} (rld);
\draw[->] (fl)  -- node[arr,pos=.3] {$\al f_!$} (fr);
\end{tikzpicture}}\label{lem:for_!for}}~
\subfigure[][]{\qc{\begin{tikzpicture}[stdprism]
\compuinit
% hidden 2-cells
\node[cuber] at (0.5,0.5,0) {$\mFor$};
\node[prismdf] at (1,0.5,\tric) {$\mTr$};
% outer 0-cells
\node (rlu) at (0,1,0) {$\cDb_H(X)$};
\node (rru) at (1,1,0) {$\cDb_H(Y)$};
\node (fr) at (	1,0.5,1) {$\cDb_J(Y)$};
\node (rld) at (0,0,0) {$\cDb_K(X)$};
\node (rrd) at (1,0,0) {$\cDb_K(Y)$};
% hidden 1-cells
\draw[liner,->] (rru) -- node[arr,pos=.3] {$\al \For^H_K$} (rrd);
% outer 1-cells
\draw[<-] (rlu) -- node[arl] {$\al f^!$} (rru);
\draw[->] (rlu) -- node[arr] {$\al \For^H_K$} (rld);
\draw[->] (rru) -- node[arl] {$\al \For^H_J$} (fr);
\draw[<-] (rld) -- node[arr] {$\al f^!$} (rrd);
\draw[->] (fr) -- node[arl] {$\al \For^J_K$} (rrd);
% visible 2-cells
\node[prismtf] at (0.5,0.75,0.5) {$\mFor$};
\node[prismbf] at (0.5,0.25,0.5) {$\mFor$};
\node[prismlf] at (0,0.5,\tric) {$\mTr$};
% visible 0- and 1-cells
\node (fl) at (0,0.5,1) {$\cDb_J(X)$};
\draw[->] (rlu) -- node[arl] {$\al \For^H_J$} (fl);
\draw[->] (fl) -- node[arl,pos=.3] {$\al \For^J_K$} (rld);
\draw[<-] (fl)  -- node[arr,pos=.3] {$\al f^!$} (fr);
\end{tikzpicture}}\label{lem:for^!for}}

\subfigure[][]{\qc{\begin{tikzpicture}[stdprism]
\compuinit
% hidden 2-cells
\node[cuber] at (0.5,0.5,0) {$\mInt$};
\node[prismdf] at (1,0.5,\tric) {$\mTr$};
% outer 0-cells
\node (rlu) at (0,1,0) {$\cDb_H(X)$};
\node (rru) at (1,1,0) {$\cDb_H(Y)$};
\node (fr) at (	1,0.5,1) {$\cDb_J(Y)$};
\node (rld) at (0,0,0) {$\cDb_K(X)$};
\node (rrd) at (1,0,0) {$\cDb_K(Y)$};
% hidden 1-cells
\draw[liner,<-] (rru) -- node[arr,pos=.3] {$\al \hamma^H_K$} (rrd);
% outer 1-cells
\draw[<-] (rlu) -- node[arl] {$\al f^*$} (rru);
\draw[<-] (rlu) -- node[arr] {$\al \hamma^H_K$} (rld);
\draw[<-] (rru) -- node[arl] {$\al \hamma^H_J$} (fr);
\draw[<-] (rld) -- node[arr] {$\al f^*$} (rrd);
\draw[<-] (fr) -- node[arl] {$\al \hamma^J_K$} (rrd);
% visible 2-cells
\node[prismtf] at (0.5,0.75,0.5) {$\mInt$};
\node[prismbf] at (0.5,0.25,0.5) {$\mInt$};
\node[prismlf] at (0,0.5,\tric) {$\mTr$};
% visible 0- and 1-cells
\node (fl) at (0,0.5,1) {$\cDb_J(X)$};
\draw[<-] (rlu) -- node[arl] {$\al \hamma^H_J$} (fl);
\draw[<-] (fl) -- node[arl,pos=.3] {$\al \hamma^J_K$} (rld);
\draw[<-] (fl)  -- node[arr,pos=.3] {$\al f^*$} (fr);
\end{tikzpicture}}\label{lem:Gamma^*Gamma}}~
\subfigure[][]{\qc{\begin{tikzpicture}[stdprism]
\compuinit
% hidden 2-cells
\node[cuber] at (0.5,0.5,0) {$\mInt$};
\node[prismdf] at (1,0.5,\tric) {$\mTr$};
% outer 0-cells
\node (rlu) at (0,1,0) {$\cDb_H(X)$};
\node (rru) at (1,1,0) {$\cDb_H(Y)$};
\node (fr) at (	1,0.5,1) {$\cDb_J(Y)$};
\node (rld) at (0,0,0) {$\cDb_K(X)$};
\node (rrd) at (1,0,0) {$\cDb_K(Y)$};
% hidden 1-cells
\draw[liner,<-] (rru) -- node[arr,pos=.3] {$\al \hamma^H_K$} (rrd);
% outer 1-cells
\draw[->] (rlu) -- node[arl] {$\al f_!$} (rru);
\draw[<-] (rlu) -- node[arr] {$\al \hamma^H_K$} (rld);
\draw[<-] (rru) -- node[arl] {$\al \hamma^H_J$} (fr);
\draw[->] (rld) -- node[arr] {$\al f_!$} (rrd);
\draw[<-] (fr) -- node[arl] {$\al \hamma^J_K$} (rrd);
% visible 2-cells
\node[prismtf] at (0.5,0.75,0.5) {$\mInt$};
\node[prismbf] at (0.5,0.25,0.5) {$\mInt$};
\node[prismlf] at (0,0.5,\tric) {$\mTr$};
% visible 0- and 1-cells
\node (fl) at (0,0.5,1) {$\cDb_J(X)$};
\draw[<-] (rlu) -- node[arl] {$\al \hamma^H_J$} (fl);
\draw[<-] (fl) -- node[arl,pos=.3] {$\al \hamma^J_K$} (rld);
\draw[->] (fl)  -- node[arr,pos=.3] {$\al f_!$} (fr);
\end{tikzpicture}}\label{lem:Gamma_!Gamma}}
\end{center}
\caption{Forgetting, integration, and transitivity}
\end{figure}
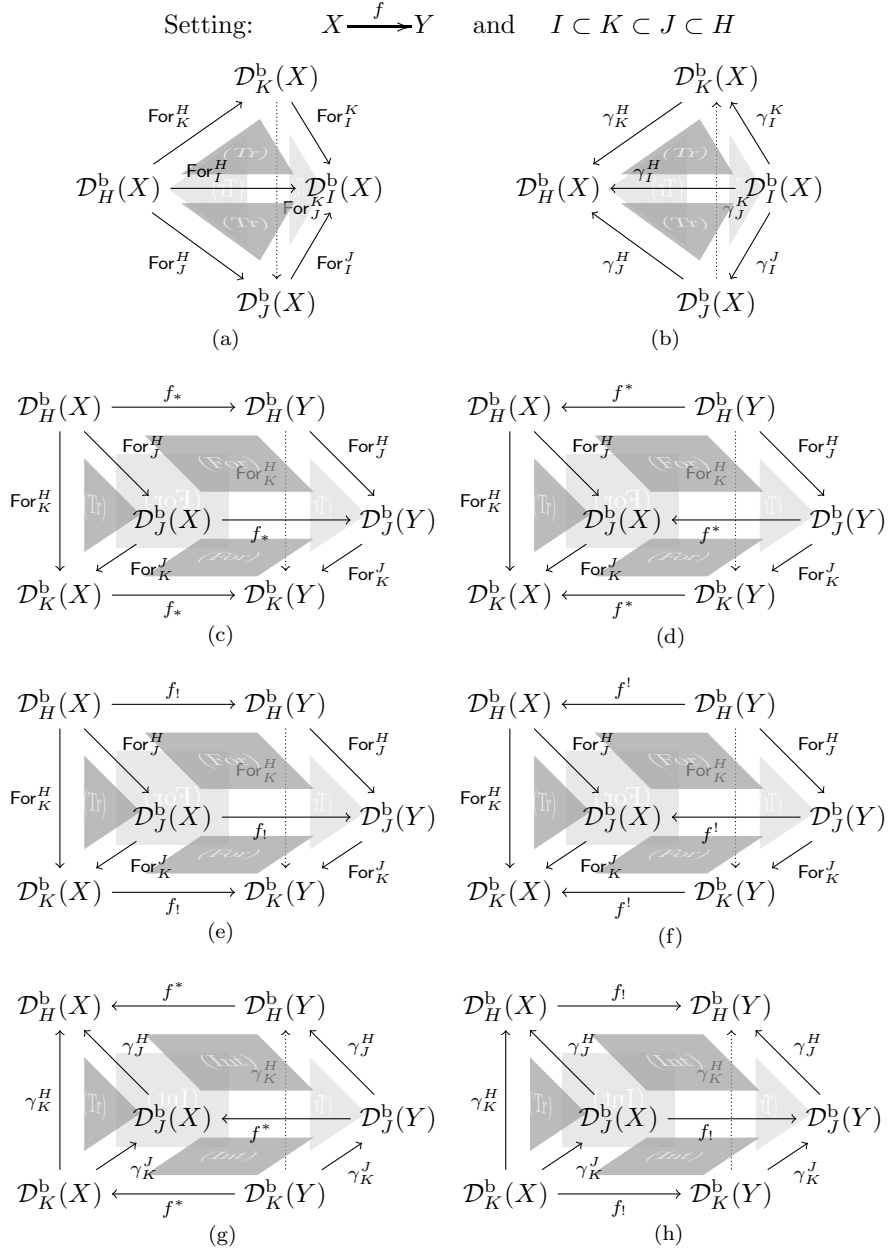

Parts~\subref{lem:cocycle-For}, \subref{lem:for_*for}, \subref{lem:for^*for}, \subref{lem:for_!for}, \subref{lem:for^!for} follow easily from the definitions. As, by Lemma~\ref{lem:transitivityadjunction}, the transitivity isomorphism for $\hamma$ can be obtained from that for $\For$ by adjunction, part~\subref{lem:cocycle-Gamma} follows from part~\subref{lem:cocycle-For} by the same argument we used to deduce Lemma~\ref{lem:cocycle^!} from Lemma~\ref{lem:cocycle_!}. Similarly, part~\subref{lem:Gamma^*Gamma} follows from part~\subref{lem:for_*for} and part~\subref{lem:Gamma_!Gamma} follows from part~\subref{lem:for^!for}.

%---------------------------------------------------------------------
\subsection{Forgetting, integration, and composition}
%---------------------------------------------------------------------

\begin{figure}
\begin{center}
Setting: \qquad $\xymatrix@1{X\ar[r]^{f_1}&Y\ar[r]^{f_2}&Z}$ \quad and \quad $K \subset H$

\subfigure[][]{\qc{\begin{tikzpicture}[stdprism]
\compuinit
% hidden 2-cells
\node[cuber] at (0.5,0.5,0) {$\mFor$};
\node[prismdf] at (1,0.5,\tric) {$\Comp$};
% outer 0-cells
\node (rlu) at (0,1,0) {$\cDb_H(X)$};
\node (rru) at (1,1,0) {$\cDb_K(X)$};
\node (fr) at (	1,0.5,1) {$\cDb_K(Y)$};
\node (rld) at (0,0,0) {$\cDb_H(Z)$};
\node (rrd) at (1,0,0) {$\cDb_K(Z)$};
% hidden 1-cells
\draw[liner,->] (rru) -- node[arr,pos=.3] {$\al (f_2f_1)_*$} (rrd);
% outer 1-cells
\draw[->] (rlu) -- node[arl] {$\al \For^H_K$} (rru);
\draw[->] (rlu) -- node[arr] {$\al (f_2f_1)_*$} (rld);
\draw[->] (rru) -- node[arl] {$\al (f_1)_*$} (fr);
\draw[->] (rld) -- node[arr] {$\al \For^H_K$} (rrd);
\draw[->] (fr) -- node[arl] {$\al (f_2)_*$} (rrd);
% visible 2-cells
\node[prismtf] at (0.5,0.75,0.5) {$\mFor$};
\node[prismbf] at (0.5,0.25,0.5) {$\mFor$};
\node[prismlf] at (0,0.5,\tric) {$\Comp$};
% visible 0- and 1-cells
\node (fl) at (0,0.5,1) {$\cDb_H(Y)$};
\draw[->] (rlu) -- node[arl] {$\al (f_1)_*$} (fl);
\draw[->] (fl) -- node[arl,pos=.3] {$\al (f_2)_*$} (rld);
\draw[->] (fl)  -- node[arr,pos=.3] {$\al \For^H_K$} (fr);
\end{tikzpicture}}\label{lem:_*compositionfor}}~
\subfigure[][]{\qc{\begin{tikzpicture}[stdprism]
\compuinit
% hidden 2-cells
\node[cuber] at (0.5,0.5,0) {$\mFor$};
\node[prismdf] at (1,0.5,\tric) {$\Comp$};
% outer 0-cells
\node (rlu) at (0,1,0) {$\cDb_H(X)$};
\node (rru) at (1,1,0) {$\cDb_K(X)$};
\node (fr) at (	1,0.5,1) {$\cDb_K(Y)$};
\node (rld) at (0,0,0) {$\cDb_H(Z)$};
\node (rrd) at (1,0,0) {$\cDb_K(Z)$};
% hidden 1-cells
\draw[liner,<-] (rru) -- node[arr,pos=.3] {$\al (f_2f_1)^*$} (rrd);
% outer 1-cells
\draw[->] (rlu) -- node[arl] {$\al \For^H_K$} (rru);
\draw[<-] (rlu) -- node[arr] {$\al (f_2f_1)^*$} (rld);
\draw[<-] (rru) -- node[arl] {$\al (f_1)^*$} (fr);
\draw[->] (rld) -- node[arr] {$\al \For^H_K$} (rrd);
\draw[<-] (fr) -- node[arl] {$\al (f_2)^*$} (rrd);
% visible 2-cells
\node[prismtf] at (0.5,0.75,0.5) {$\mFor$};
\node[prismbf] at (0.5,0.25,0.5) {$\mFor$};
\node[prismlf] at (0,0.5,\tric) {$\Comp$};
% visible 0- and 1-cells
\node (fl) at (0,0.5,1) {$\cDb_H(Y)$};
\draw[<-] (rlu) -- node[arl] {$\al (f_1)^*$} (fl);
\draw[<-] (fl) -- node[arl,pos=.3] {$\al (f_2)^*$} (rld);
\draw[->] (fl)  -- node[arr,pos=.3] {$\al \For^H_K$} (fr);
\end{tikzpicture}}\label{lem:^*compositionfor}}

\subfigure[][]{\qc{\begin{tikzpicture}[stdprism]
\compuinit
% hidden 2-cells
\node[cuber] at (0.5,0.5,0) {$\mFor$};
\node[prismdf] at (1,0.5,\tric) {$\Comp$};
% outer 0-cells
\node (rlu) at (0,1,0) {$\cDb_H(X)$};
\node (rru) at (1,1,0) {$\cDb_K(X)$};
\node (fr) at (	1,0.5,1) {$\cDb_K(Y)$};
\node (rld) at (0,0,0) {$\cDb_H(Z)$};
\node (rrd) at (1,0,0) {$\cDb_K(Z)$};
% hidden 1-cells
\draw[liner,->] (rru) -- node[arr,pos=.3] {$\al (f_2f_1)_!$} (rrd);
% outer 1-cells
\draw[->] (rlu) -- node[arl] {$\al \For^H_K$} (rru);
\draw[->] (rlu) -- node[arr] {$\al (f_2f_1)_!$} (rld);
\draw[->] (rru) -- node[arl] {$\al (f_1)_!$} (fr);
\draw[->] (rld) -- node[arr] {$\al \For^H_K$} (rrd);
\draw[->] (fr) -- node[arl] {$\al (f_2)_!$} (rrd);
% visible 2-cells
\node[prismtf] at (0.5,0.75,0.5) {$\mFor$};
\node[prismbf] at (0.5,0.25,0.5) {$\mFor$};
\node[prismlf] at (0,0.5,\tric) {$\Comp$};
% visible 0- and 1-cells
\node (fl) at (0,0.5,1) {$\cDb_H(Y)$};
\draw[->] (rlu) -- node[arl] {$\al (f_1)_!$} (fl);
\draw[->] (fl) -- node[arl,pos=.3] {$\al (f_2)_!$} (rld);
\draw[->] (fl)  -- node[arr,pos=.3] {$\al \For^H_K$} (fr);
\end{tikzpicture}}\label{lem:_!compositionfor}}~
\subfigure[][]{\qc{\begin{tikzpicture}[stdprism]
\compuinit
% hidden 2-cells
\node[cuber] at (0.5,0.5,0) {$\mFor$};
\node[prismdf] at (1,0.5,\tric) {$\Comp$};
% outer 0-cells
\node (rlu) at (0,1,0) {$\cDb_H(X)$};
\node (rru) at (1,1,0) {$\cDb_K(X)$};
\node (fr) at (	1,0.5,1) {$\cDb_K(Y)$};
\node (rld) at (0,0,0) {$\cDb_H(Z)$};
\node (rrd) at (1,0,0) {$\cDb_K(Z)$};
% hidden 1-cells
\draw[liner,<-] (rru) -- node[arr,pos=.3] {$\al (f_2f_1)^!$} (rrd);
% outer 1-cells
\draw[->] (rlu) -- node[arl] {$\al \For^H_K$} (rru);
\draw[<-] (rlu) -- node[arr] {$\al (f_2f_1)^!$} (rld);
\draw[<-] (rru) -- node[arl] {$\al (f_1)^!$} (fr);
\draw[->] (rld) -- node[arr] {$\al \For^H_K$} (rrd);
\draw[<-] (fr) -- node[arl] {$\al (f_2)^!$} (rrd);
% visible 2-cells
\node[prismtf] at (0.5,0.75,0.5) {$\mFor$};
\node[prismbf] at (0.5,0.25,0.5) {$\mFor$};
\node[prismlf] at (0,0.5,\tric) {$\Comp$};
% visible 0- and 1-cells
\node (fl) at (0,0.5,1) {$\cDb_H(Y)$};
\draw[<-] (rlu) -- node[arl] {$\al (f_1)^!$} (fl);
\draw[<-] (fl) -- node[arl,pos=.3] {$\al (f_2)^!$} (rld);
\draw[->] (fl)  -- node[arr,pos=.3] {$\al \For^H_K$} (fr);
\end{tikzpicture}}\label{lem:^!compositionfor}}

\subfigure[][]{\qc{\begin{tikzpicture}[stdprism]
\compuinit
% hidden 2-cells
\node[cuber] at (0.5,0.5,0) {$\mInt$};
\node[prismdf] at (1,0.5,\tric) {$\Comp$};
% outer 0-cells
\node (rlu) at (0,1,0) {$\cDb_H(X)$};
\node (rru) at (1,1,0) {$\cDb_K(X)$};
\node (fr) at (	1,0.5,1) {$\cDb_K(Y)$};
\node (rld) at (0,0,0) {$\cDb_H(Z)$};
\node (rrd) at (1,0,0) {$\cDb_K(Z)$};
% hidden 1-cells
\draw[liner,<-] (rru) -- node[arr,pos=.3] {$\al (f_2f_1)^*$} (rrd);
% outer 1-cells
\draw[<-] (rlu) -- node[arl] {$\al \hamma^H_K$} (rru);
\draw[<-] (rlu) -- node[arr] {$\al (f_2f_1)^*$} (rld);
\draw[<-] (rru) -- node[arl] {$\al (f_1)^*$} (fr);
\draw[<-] (rld) -- node[arr] {$\al \hamma^H_K$} (rrd);
\draw[<-] (fr) -- node[arl] {$\al (f_2)^*$} (rrd);
% visible 2-cells
\node[prismtf] at (0.5,0.75,0.5) {$\mInt$};
\node[prismbf] at (0.5,0.25,0.5) {$\mInt$};
\node[prismlf] at (0,0.5,\tric) {$\Comp$};
% visible 0- and 1-cells
\node (fl) at (0,0.5,1) {$\cDb_H(Y)$};
\draw[<-] (rlu) -- node[arl] {$\al (f_1)^*$} (fl);
\draw[<-] (fl) -- node[arl,pos=.3] {$\al (f_2)^*$} (rld);
\draw[<-] (fl)  -- node[arr,pos=.3] {$\al \hamma^H_K$} (fr);
\end{tikzpicture}}\label{lem:^*compositionGamma}}~
\subfigure[][]{\qc{\begin{tikzpicture}[stdprism]
\compuinit
% hidden 2-cells
\node[cuber] at (0.5,0.5,0) {$\mInt$};
\node[prismdf] at (1,0.5,\tric) {$\Comp$};
% outer 0-cells
\node (rlu) at (0,1,0) {$\cDb_H(X)$};
\node (rru) at (1,1,0) {$\cDb_K(X)$};
\node (fr) at (	1,0.5,1) {$\cDb_K(Y)$};
\node (rld) at (0,0,0) {$\cDb_H(Z)$};
\node (rrd) at (1,0,0) {$\cDb_K(Z)$};
% hidden 1-cells
\draw[liner,->] (rru) -- node[arr,pos=.3] {$\al (f_2f_1)_!$} (rrd);
% outer 1-cells
\draw[<-] (rlu) -- node[arl] {$\al \hamma^H_K$} (rru);
\draw[->] (rlu) -- node[arr] {$\al (f_2f_1)_!$} (rld);
\draw[->] (rru) -- node[arl] {$\al (f_1)_!$} (fr);
\draw[<-] (rld) -- node[arr] {$\al \hamma^H_K$} (rrd);
\draw[->] (fr) -- node[arl] {$\al (f_2)_!$} (rrd);
% visible 2-cells
\node[prismtf] at (0.5,0.75,0.5) {$\mInt$};
\node[prismbf] at (0.5,0.25,0.5) {$\mInt$};
\node[prismlf] at (0,0.5,\tric) {$\Comp$};
% visible 0- and 1-cells
\node (fl) at (0,0.5,1) {$\cDb_H(Y)$};
\draw[->] (rlu) -- node[arl] {$\al (f_1)_!$} (fl);
\draw[->] (fl) -- node[arl,pos=.3] {$\al (f_2)_!$} (rld);
\draw[<-] (fl)  -- node[arr,pos=.3] {$\al \hamma^H_K$} (fr);
\end{tikzpicture}}\label{lem:_!compositionGamma}}
\end{center}
\caption{Forgetting, integration, and composition}
\end{figure}
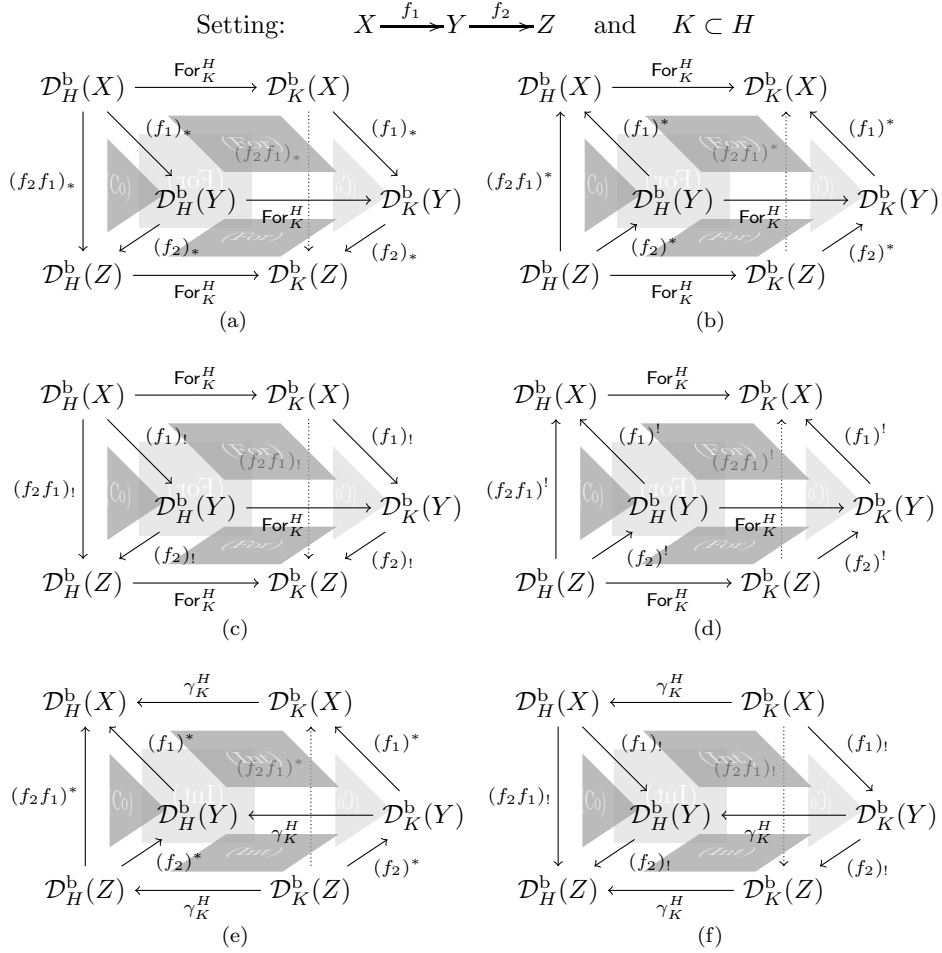

Parts~\subref{lem:_*compositionfor} -- ~\subref{lem:^!compositionfor} follow easily from the definitions. Since we know from Lemma~\ref{lem:for_*gamma^*adjunction} that the $(\cdot)^*$ version of isomorphism $\mInt$ can be obtained from the $(\cdot)_*$ version of isomorphism $\mFor$ by adjunction, part~\subref{lem:^*compositionGamma} follows from part~\subref{lem:_*compositionfor} and Lemma~\ref{lem:^*_*compositionadjunction}. Similarly, in view of Lemma~\ref{lem:for^!gamma_!adjunction}, part~\subref{lem:_!compositionGamma} follows from part~\subref{lem:^!compositionfor} and Lemma~\ref{lem:_!^!compositionadjunction}.

%---------------------------------------------------------------------
\subsection{Forgetting, integration, and base change}
%---------------------------------------------------------------------

\begin{figure}
\begin{center}
Setting: \qquad
\qc{\begin{tikzpicture}[vsmallcube]
\compuinit
% 2-cells
\node[cart] at (0.5,0.5) {};
% 0-cells
\node (lu) at (0,1) {$W$};
\node (ru) at (1,1) {$X$};
\node (ld) at (0,0) {$Y$};
\node (rd) at (1,0) {$Z$};
% 1-cells
\draw[->] (lu) -- node[arl] {$\al g'$} (ru);
\draw[->] (lu) -- node[arr] {$\al f'$} (ld);
\draw[->] (ru) -- node[arl] {$\al f$} (rd);
\draw[->] (ld) -- node[arr] {$\al g$} (rd);
\end{tikzpicture}}

\subfigure[][]{\qc{\begin{tikzpicture}[stdcube]
\compuinit
% hidden 0-cell
\node (rrd) at (1,0,0) {$\cDb_K(Z)$};
% hidden 2-cells
\node[cuber] at (0.5,0.5,0) {$\mFor$};
\node[cubed] at (1,0.5,0.5) {$\BC$};
\node[cubeb] at (0.5,0,0.5) {$\mFor$};
% outer 0-cells
\node (rlu) at (0,1,0) {$\cDb_H(X)$};
\node (rru) at (1,1,0) {$\cDb_K(X)$};
\node (fru) at (1,1,1) {$\cDb_K(W)$};
\node (frd) at (1,0,1) {$\cDb_K(Y)$};
\node (fld) at (0,0,1) {$\cDb_H(Y)$};
\node (rld) at (0,0,0) {$\cDb_H(Z)$};
% hidden 1-cells
\draw[liner,->] (rru) -- node[arl,pos=.7] {$\al f_*$} (rrd);
\draw[liner,->] (rld) -- node[arl,pos=.3] {$\al \For^H_K$} (rrd);
\draw[liner,->] (rrd) -- node[arl] {$\al g^!$} (frd);
% outer 1-cells
\draw[->] (rlu) -- node[arl] {$\al \For^H_K$} (rru);
\draw[->] (rru) -- node[arl] {$\al (g')^!$} (fru);
\draw[->] (fru) -- node[arl] {$\al (f')_*$}(frd);
\draw[->] (fld) -- node[arr] {$\al \For^H_K$} (frd);
\draw[->] (rld) -- node[arr] {$\al g^!$} (fld);
\draw[->] (rlu) -- node[arr] {$\al f_*$} (rld);
% visible 2-cells
\node[cubel] at (0,0.5,0.5) {$\BC$};
\node[cubet] at (0.5,1,0.5) {$\mFor$};
\node[cubef] at (0.5,0.5,1) {$\mFor$};
% visible 0- and 1-cells
\node (flu) at (0,1,1) {$\cDb_H(W)$};
\draw[->] (rlu) -- node[arr,pos=.7] {$\al (g')^!$} (flu);
\draw[->] (flu) -- node[arr,pos=.3] {$\al \For^H_K$} (fru);
\draw[->] (flu) -- node[arl,pos=.3] {$\al (f')_*$} (fld);
\end{tikzpicture}}\label{lem:basechangefor}}~
\subfigure[][]{\qc{\begin{tikzpicture}[stdcube]
\compuinit
% hidden 0-cell
\node (rrd) at (1,0,0) {$\cDb_K(Z)$};
% hidden 2-cells
\node[cuber] at (0.5,0.5,0) {$\mInt$};
\node[cubed] at (1,0.5,0.5) {$\BC$};
\node[cubeb] at (0.5,0,0.5) {$\mInt$};
% outer 0-cells
\node (rlu) at (0,1,0) {$\cDb_H(X)$};
\node (rru) at (1,1,0) {$\cDb_K(X)$};
\node (fru) at (1,1,1) {$\cDb_K(W)$};
\node (frd) at (1,0,1) {$\cDb_K(Y)$};
\node (fld) at (0,0,1) {$\cDb_H(Y)$};
\node (rld) at (0,0,0) {$\cDb_H(Z)$};
% hidden 1-cells
\draw[liner,->] (rru) -- node[arl,pos=.7] {$\al f_!$} (rrd);
\draw[liner,<-] (rld) -- node[arl,pos=.3] {$\al \hamma^H_K$} (rrd);
\draw[liner,->] (rrd) -- node[arl] {$\al g^*$} (frd);
% outer 1-cells
\draw[<-] (rlu) -- node[arl] {$\al \hamma^H_K$} (rru);
\draw[->] (rru) -- node[arl] {$\al (g')^*$} (fru);
\draw[->] (fru) -- node[arl] {$\al (f')_!$}(frd);
\draw[<-] (fld) -- node[arr] {$\al \hamma^H_K$} (frd);
\draw[->] (rld) -- node[arr] {$\al g^*$} (fld);
\draw[->] (rlu) -- node[arr] {$\al f_!$} (rld);
% visible 2-cells
\node[cubel] at (0,0.5,0.5) {$\BC$};
\node[cubet] at (0.5,1,0.5) {$\mInt$};
\node[cubef] at (0.5,0.5,1) {$\mInt$};
% visible 0- and 1-cells
\node (flu) at (0,1,1) {$\cDb_H(W)$};
\draw[->] (rlu) -- node[arr,pos=.7] {$\al (g')^*$} (flu);
\draw[<-] (flu) -- node[arr,pos=.3] {$\al \hamma^H_K$} (fru);
\draw[->] (flu) -- node[arl,pos=.3] {$\al (f')_!$} (fld);
\end{tikzpicture}}\label{lem:basechangeGamma}}

\end{center}
\caption{Forgetting, integration, and base change}
\end{figure}
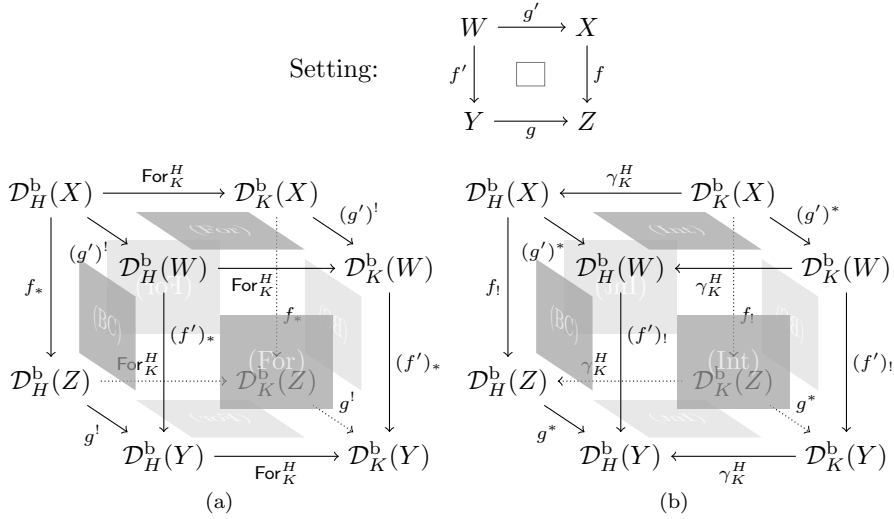

Part~\subref{lem:basechangefor} is easy. In view of Lemmas~\ref{lem:for_*gamma^*adjunction} and~\ref{lem:for^!gamma_!adjunction}, part~\subref{lem:basechangeGamma} follows from part~\subref{lem:basechangefor} using Lemma~\ref{lem:basechangeadjunction}.

%---------------------------------------------------------------------
\subsection{Constant sheaf and transitivity}
\label{ss:constant-sheaf-transitivity}
%---------------------------------------------------------------------

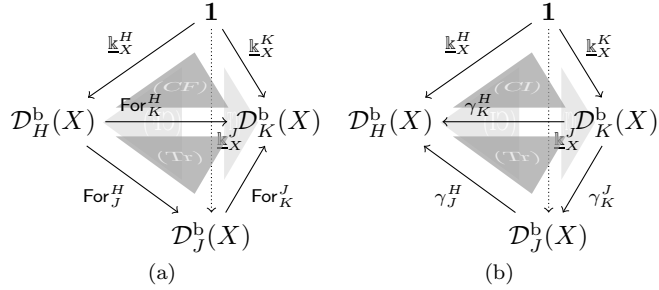
\begin{figure}
\begin{center}
Setting: \qquad $K\subset J\subset H$\quad and \quad (for~\subref{lem:Gammacompositionconstant})\quad $J/K$, $H/J$ contractible

\subfigure[][]{\qc{\begin{tikzpicture}[stdtetr]
\compuinit
% hidden 2-cells
\node[tetrlr] at ({-\tric/2},0.5,\tric) {$\CFor$};
\node[tetrdr] at ({\tric/2},0.5,\tric) {$\CFor$};
% outer 0-cells
\node (ru) at (0,1,0) {$\bb$};
\node (fl) at (-0.5,0.5,1) {$\cDb_H(X)$};
\node (fr) at (0.5,0.5,1) {$\cDb_K(X)$};
\node (rd) at (0,0,0) {$\cDb_J(X)$};
% hidden 1-cell
\draw[liner,->] (ru) -- node[arl,pos=.6] {$\al \ubk_X^J$} (rd);
% visible 2-cells
\node[tetrtf] at (0, {(1+\tric)/2}, {1-\tric}) {$\CFor$};
\node[tetrbf] at (0, {(1-\tric)/2}, {1-\tric}) {$\mTr$};
% visible 1-cells
\draw[<-] (fl) -- node[arl] {$\al \ubk_X^H$} (ru);
\draw[->] (fl) -- node[arr] {$\al \For_J^H$} (rd);
\draw[->] (ru) -- node[arl] {$\al \ubk_X^K$} (fr);
\draw[->] (rd) -- node[arr] {$\al \For_K^J$} (fr);
\draw[->] (fl) -- node[arl,pos=.3] {$\al \For_K^H$} (fr);
\end{tikzpicture}}\label{lem:Forcompositionconstant}}~
\subfigure[][]{\qc{\begin{tikzpicture}[stdtetr]
\compuinit
% hidden 2-cells
\node[tetrlr] at ({-\tric/2},0.5,\tric) {$\CInt$};
\node[tetrdr] at ({\tric/2},0.5,\tric) {$\CInt$};
% outer 0-cells
\node (ru) at (0,1,0) {$\bb$};
\node (fl) at (-0.5,0.5,1) {$\cDb_H(X)$};
\node (fr) at (0.5,0.5,1) {$\cDb_K(X)$};
\node (rd) at (0,0,0) {$\cDb_J(X)$};
% hidden 1-cell
\draw[liner,->] (ru) -- node[arl,pos=.6] {$\al \ubk_X^J$} (rd);
% visible 2-cells
\node[tetrtf] at (0, {(1+\tric)/2}, {1-\tric}) {$\CInt$};
\node[tetrbf] at (0, {(1-\tric)/2}, {1-\tric}) {$\mTr$};
% visible 1-cells
\draw[<-] (fl) -- node[arl] {$\al \ubk_X^H$} (ru);
\draw[<-] (fl) -- node[arr] {$\al \hamma_J^H$} (rd);
\draw[->] (ru) -- node[arl] {$\al \ubk_X^K$} (fr);
\draw[<-] (rd) -- node[arr] {$\al \hamma_K^J$} (fr);
\draw[<-] (fl) -- node[arl,pos=.3] {$\al \hamma_K^H$} (fr);
\end{tikzpicture}}\label{lem:Gammacompositionconstant}}
\end{center}
\caption{Constant sheaf and transitivity}
\end{figure}

Part~\subref{lem:Forcompositionconstant} is easy. By definition, part~\subref{lem:Gammacompositionconstant} is equivalent to the commutativity of a diagram of isomorphisms in $\cDb(H\backslash P)$ for a given $H$-resolution $P$ of $X$. This follows from Lemma \ref{lem:_!^!compositionadjunction}.

%---------------------------------------------------------------------
\subsection{Constant sheaf under inverse image, forgetting, and integration}
\label{ss:constant-inverse-forgetting-integration}
%---------------------------------------------------------------------

\begin{figure}
\begin{center}
Setting: \qquad $\xymatrix@1{X\ar[r]^{f}&Y}$, \quad $K \subset H$, \quad and \quad (for~\subref{lem:Gamma^*constant})\quad $H/K$ contractible

\subfigure[][]{\qc{\begin{tikzpicture}[stdpyra]
\compuinit
% hidden 0-cell
\node (rrd) at (1,0,0) {$\cDb_K(Y)$};
% hidden 2-cells
\node[pyrar] at (1-\tric,0.5,0.5*\tric) {$\CFor$};
\node[pyrab] at (1-\tric,0.5*\tric,0.5) {$\Cnst$};
% outer 0-cells
\node (l) at (0,0.5,0.5) {$\bb$};
\node (rru) at (1,1,0) {$\cDb_H(Y)$};
\node (frd) at (1,0,1) {$\cDb_K(X)$};
% hidden 1-cells
\draw[liner,->] (l) -- node[arl] {$\al \ubk_Y^K$} (rrd);
% outer 1-cells
\draw[->] (rru) -- node[arl,pos=.7] {$\al \For^H_K$} (rrd);
\draw[->] (rrd) -- node[arl] {$\al f^*$} (frd);
\draw[->] (l) -- node[arl] {$\al \ubk_{Y}^H$} (rru);
\draw[->] (l) -- node[arr] {$\al \ubk_X^K$} (frd);
% visible 2-cells
\node[cubel,xscale=-1] at (1,0.5,0.5) {$\mFor$};
\node[pyraf] at (1-\tric,0.5,1-0.5*\tric) {$\CFor$};
\node[pyrat] at (1-\tric,1-0.5*\tric,0.5) {$\Cnst$};
% visible 0- and 1-cells
\node (fru) at (1,1,1) {$\cDb_H(X)$};
\draw[->] (rru) -- node[arl] {$\al f^*$} (fru);
\draw[->] (fru) -- node[arl] {$\al \For^H_K$}(frd);
\draw[->] (l) -- node[arl,pos=.7] {$\al \ubk_{X}^H$} (fru);
\end{tikzpicture}}\label{lem:For^*constant}}~
\subfigure[][]{\qc{\begin{tikzpicture}[stdpyra]
\compuinit
% hidden 0-cell
\node (rrd) at (1,0,0) {$\cDb_K(Y)$};
% hidden 2-cells
\node[pyrar] at (1-\tric,0.5,0.5*\tric) {$\CInt$};
\node[pyrab] at (1-\tric,0.5*\tric,0.5) {$\Cnst$};
% outer 0-cells
\node (l) at (0,0.5,0.5) {$\bb$};
\node (rru) at (1,1,0) {$\cDb_H(Y)$};
\node (frd) at (1,0,1) {$\cDb_K(X)$};
% hidden 1-cells
\draw[liner,->] (l) -- node[arl] {$\al \ubk_Y^K$} (rrd);
% outer 1-cells
\draw[<-] (rru) -- node[arl,pos=.7] {$\al \hamma_K^H$} (rrd);
\draw[->] (rrd) -- node[arl] {$\al f^*$} (frd);
\draw[->] (l) -- node[arl] {$\al \ubk_{Y}^H$} (rru);
\draw[->] (l) -- node[arr] {$\al \ubk_X^K$} (frd);
% visible 2-cells
\node[cubel,xscale=-1] at (1,0.5,0.5) {$\mInt$};
\node[pyraf] at (1-\tric,0.5,1-0.5*\tric) {$\CInt$};
\node[pyrat] at (1-\tric,1-0.5*\tric,0.5) {$\Cnst$};
% visible 0- and 1-cells
\node (fru) at (1,1,1) {$\cDb_H(X)$};
\draw[->] (rru) -- node[arl] {$\al f^*$} (fru);
\draw[<-] (fru) -- node[arl] {$\al \hamma_K^H$}(frd);
\draw[->] (l) -- node[arl,pos=.7] {$\al \ubk_{X}^H$} (fru);
\end{tikzpicture}}\label{lem:Gamma^*constant}}
\end{center}
\caption{Constant sheaf under inverse image, forgetting, and integration}
\end{figure}
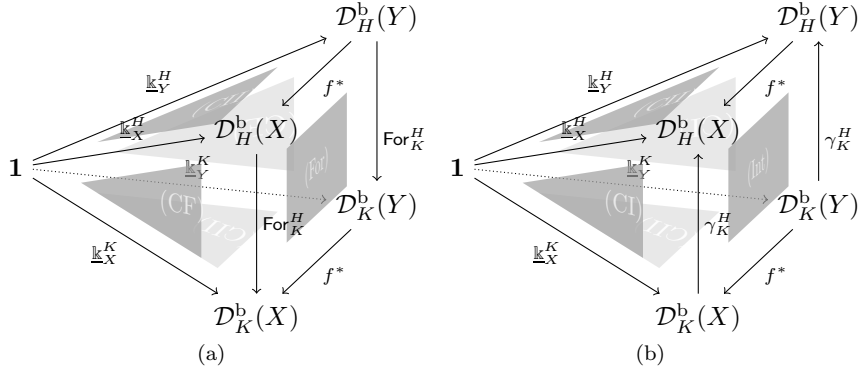

Part~\subref{lem:For^*constant} is easy. Unravelling the definitions, part~\subref{lem:Gamma^*constant} is equivalent to the commutativity of a diagram of isomorphisms in $\cDb(H\backslash(P\times_Y X))$ for a given $H$-resolution $P$ of $Y$. This follows from Lemma~\ref{lem:basechangeadjunction}.

%.....................................................................
\subsection{Induction equivalence}
\label{subsect:induction-equiv}
%.....................................................................

Let $K \subset H$ be a closed subgroup, and $X$ a $K$-variety. Form the induced $H$-variety $\widetilde{X}=H\times^K X$, and let $i: X \to \widetilde{X}$ be the inclusion. The category of $K$-resolutions of $X$ and smooth $K$-morphisms over $X$ is equivalent to the category of $H$-resolutions of $\widetilde{X}$ and smooth $H$-morphisms over $\widetilde{X}$ via the functor $P \mapsto H \times^K P$, whose inverse is $Q\mapsto Q\times_{\widetilde{X}} X$. This equivalence induces an equivalence of categories $\sInd_K^H : \cDb_K(X) \simto \cDb_H(\widetilde{X})$. Namely, if $M$ is an object of $\cDb_K(X)$ and $P$ is a $K$-resolution of $X$, we set
\[
(\sInd_K^H M)(H\times^K P)=M(P),
\]
where as usual we identify $H\backslash(H\times^K P)$ with $K\backslash P$. This is the inverse of the equivalence $\cDb_H(\widetilde{X})\simto \cDb_K(X)$ denoted $\nu^*$ in~\cite[\S2.6.3]{bl}, which is isomorphic to $i^*\circ\For^H_K$ in our notation.

Consider the composition $\hamma^H_K \circ i_!:\cDb_K(X)\to\cDb_H(\widetilde{X})$. If $M$ is an object of $\cDb_K(X)$ and $P$ is a $K$-resolution of $X$, we have
\[
(\hamma^H_K i_! M)(H\times^K P)=(q_{H\times^K P})_!(\widetilde{i}^K_{H\times^K P})_! M(P) [2\dim(H/K)],
\]
where we have identified $(H\times^K P)\times_{\widetilde X} X$ with $P$. Since $q_{H\times^K P}\widetilde{i}^K_{H\times^K P}$ is identified with the identity map from $K\backslash P$ to itself, the composition isomorphism for $(\cdot)_!$ gives us an isomorphism $\hamma^H_K \circ i_!\natisom \sInd_K^H[2 \dim(H/K)]$. We depict this isomorphism as follows:
\[
\vc{\begin{tikzpicture}[smallcube]
\compuinit
% 2-cells
\node[cubef] at (0.5,0.5) {$\mIE$};
% 0-cells
\node (lu) at (0,1) {$\cDb_H(\widetilde{X})$};
\node (ru) at (1,1) {$\cDb_H(\widetilde{X})$};
\node (ld) at (0,0) {$\cDb_K(X)$};
\node (rd) at (1,0) {$\cDb_K(X)$};
% 1-cells
\draw[-,double distance=1.5pt] (lu) -- (ru);
\draw[<-] (lu) -- node[arr] {$\al \sInd_K^H[2\dim(H/K)]$} (ld);
\draw[<-] (ru) -- node[arl] {$\al \hamma^H_K i_!$} (rd);
\draw[-,double distance=1.5pt] (ld) -- (rd);
\end{tikzpicture}}
\]
From now on we omit the $\circ$ from the name of $\hamma^H_K \circ i_!$ since we regard it as a basic functor in its own right. Within this appendix, we consider both versions of the induction equivalence, $\sInd_K^H$ and $\hamma^H_K i_!$, using the former to help study the latter. In the main body of the paper, only $\hamma^H_K i_!$ appears.

%.....................................................................
\subsection{Notation for isomorphisms involving induction equivalence}
%.....................................................................

Continue with the setting of \S\ref{subsect:induction-equiv}.

%.....................................................................
\subsubsection{Transitivity of induction equivalence}
%.....................................................................

Suppose that $K\subset J\subset H$, and let $i_1:X\to J\times^K X$ and $i_2:J\times^K X\to \widetilde{X}$ be the inclusions. As usual, we identify $H\times^J (J\times^K X)$ with $H\times^K X=\widetilde{X}$. We have an obvious transitivity isomorphism for the $\sInd$ version of induction equivalence:
\[
\vc{\begin{tikzpicture}[stdtriangle]
\compuinit
% 0-cells
\node (lu) at (0,1) {$\al \cDb_H(\widetilde{X})$};
\node (r) at (1,1) {$\al \cDb_J(J\times^K X)$};
\node (ld) at (1,0) {$\al \cDb_K(X)$};
% 1-cells
\draw[<-] (lu) -- node[arl] {$\al \sInd^H_J$} (r);
\draw[<-] (lu) -- node[arr] {$\al \sInd^H_K$} (ld);
\draw[<-] (r) -- node[arl] {$\al \sInd^J_K$} (ld);
% 2-cell
\node[lttricelld] at (0.5+0.5*\tric,0.5+0.5*\tric) {\tiny$\ITr$};
\end{tikzpicture}}
\]
We can define an analogous transitivity isomorphism $\hamma^H_J(i_2)_! \circ \hamma^J_K(i_1)_! \natisom \hamma^H_K i_!$ using isomorphisms we have already defined:
\[
\vc{\begin{tikzpicture}[stdtriangle]
\compuinit
% 0-cells
\node (lu) at (0,1) {$\al \cDb_H(\widetilde{X})$};
\node (r) at (1,1) {$\al \cDb_J(J\times^K X)$};
\node (ld) at (1,0) {$\al \cDb_K(X)$};
% 1-cells
\draw[<-] (lu) -- node[arl] {$\al \hamma^H_J(i_2)_!$} (r);
\draw[<-] (lu) -- node[arr] {$\al \hamma^H_K i_!$} (ld);
\draw[<-] (r) -- node[arl] {$\al \hamma^J_K(i_1)_!$} (ld);
% 2-cell
\node[lttricelld] at (0.5+0.5*\tric,0.5+0.5*\tric) {\tiny$\ITr$};
\end{tikzpicture}}
\quad:=\quad
\vc{\begin{tikzpicture}[stdtriangle]
\compuinit
% 2-cells
\node[lttricelld] at (0.5+0.5*\tric, 1.5+0.5*\tric) {\tiny$\mTr$};
\node[lttricelld] at (1.5+0.5*\tric, 0.5+0.5*\tric) {\tiny$\Comp$};
\node[cubef] at (1.5,1.5) {$\mInt$};
% 0-cells
\node (lluu) at (0,2) {{\tiny $\cDb_H(\widetilde{X})$}};
\node (luu) at (1,2) {{\tiny $\cDb_J(\widetilde{X})$}};
\node (muu) at (2,2) {{\tiny $\cDb_J(J\times^K X)$}};
\node (lu) at (1,1) {{\tiny $\cDb_K(\widetilde{X})$}};
\node (mu) at (2,1) {{\tiny $\cDb_K(J \times^K X)$}};
\node (mm) at (2,0) {{\tiny $\cDb_K(X)$}};
% 1-cells
\draw[<-] (lluu) -- node[arl] {{\tiny $\hamma^H_J$}} (luu);
\draw[<-] (luu) -- node[arl] {{\tiny $(i_2)_!$}} (muu);
\draw[<-] (lu) -- node[arl] {{\tiny $(i_2)_!$}} (mu);
\draw[<-] (luu) -- node[arl] {{\tiny $\hamma^J_K$}} (lu);
\draw[<-] (muu) -- node[arl] {{\tiny $\hamma^J_K$}} (mu);
\draw[<-] (mu) -- node[arl] {{\tiny $(i_1)_!$}} (mm);
\draw[<-] (lluu) -- node[arr] {{\tiny $\hamma^H_K$}} (lu);
\draw[<-] (lu) -- node[arr] {{\tiny $i_!$}} (mm);
\end{tikzpicture}}
\]

%.....................................................................
\subsubsection{Integration and induction equivalence}
%.....................................................................

Suppose that $I$ is a closed subgroup of $H$ such that $H=IK$. We can identify $I\times^{I\cap K} X$ with $\widetilde{X}$. From the definitions, we have an obvious isomorphism:
\[
\vc{\begin{tikzpicture}[smallcube]
\compuinit
% 2-cells
\node[cubef] at (0.5,0.5) {$\mIEI$};
% 0-cells
\node (lu) at (0,1) {$\cDb_H(\widetilde{X})$};
\node (ru) at (1,1) {$\cDb_K(X)$};
\node (ld) at (0,0) {$\cDb_I(\widetilde{X})$};
\node (rd) at (1,0) {$\cDb_{I\cap K}(X)$};
% 1-cells
\draw[<-] (lu) -- node[arl] {$\al \sInd^H_K$} (ru);
\draw[<-] (lu) -- node[arr] {$\al \hamma^H_I$} (ld);
\draw[<-] (ru) -- node[arl] {$\al \hamma^K_{I\cap K}$} (rd);
\draw[<-] (ld) -- node[arr] {$\al \sInd^I_{I\cap K}$} (rd);
\end{tikzpicture}}
\]

We define an analogous isomorphism for the other version of induction equivalence:
\[
\vc{\begin{tikzpicture}[smallcube]
\compuinit
% 2-cells
\node[cubef] at (0.5,0.5) {$\mIEI$};
% 0-cells
\node (lu) at (0,1) {$\cDb_H(\widetilde{X})$};
\node (ru) at (1,1) {$\cDb_K(X)$};
\node (ld) at (0,0) {$\cDb_I(\widetilde{X})$};
\node (rd) at (1,0) {$\cDb_{I\cap K}(X)$};
% 1-cells
\draw[<-] (lu) -- node[arl] {$\al \hamma^H_K i_!$} (ru);
\draw[<-] (lu) -- node[arr] {$\al \hamma^H_I$} (ld);
\draw[<-] (ru) -- node[arl] {$\al \hamma^K_{I\cap K}$} (rd);
\draw[<-] (ld) -- node[arr] {$\al \hamma^I_{I\cap K} i_!$} (rd);
\end{tikzpicture}}
\quad:=\quad
\vc{\begin{tikzpicture}[smallcube]
\compuinit
% 2-cells
\node[lttricelld] at (0.5+0.5*\tric,0.5+0.5*\tric) {\tiny$\mTr$};
\node[rttricellu] at (0.5-0.5*\tric,0.5-0.5*\tric) {\tiny$\mTr$};
\node[cubef] at (1.5,0.5) {$\mInt$};
% 0-cells
\node (lu) at (0,1) {$\cDb_H(\widetilde{X})$};
\node (mu) at (1,1) {$\cDb_K(\widetilde{X})$};
\node (ru) at (2,1) {$\cDb_K(X)$};
\node (ld) at (0,0) {$\cDb_I(\widetilde{X})$};
\node (md) at (1,0) {$\cDb_{I\cap K}(\widetilde{X})$};
\node (rd) at (2,0) {$\cDb_{I\cap K}(X)$};
% 1-cells
\draw[<-] (lu) -- node[arl] {$\al \hamma^H_K$} (mu);
\draw[<-] (lu) -- node[arl] {$\al \hamma^H_{I\cap K}$} (md);
\draw[<-] (mu) -- node[arl] {$\al i_!$} (ru);
\draw[<-] (lu) -- node[arr] {$\al \hamma^H_I$} (ld);
\draw[<-] (mu) -- node[arl] {$\al \hamma^K_{I\cap K}$} (md);
\draw[<-] (ru) -- node[arl] {$\al \hamma^K_{I\cap K}$} (rd);
\draw[<-] (ld) -- node[arr] {$\al \hamma^I_{I\cap K}$} (md);
\draw[<-] (md) -- node[arr] {$\al i_!$} (rd);
\end{tikzpicture}}
\]

%.....................................................................
\subsubsection{Inverse image and induction equivalence}
%.....................................................................

Let $f:X\to Y$ be a morphism of $K$-varieties, $g:\widetilde{X}\to\widetilde{Y}$ the induced morphism of $H$-varieties, and $j:Y\to \widetilde{Y}$ the inclusion. Then we have a cartesian square
\[
\vc{\begin{tikzpicture}[vsmallcube]
\compuinit
% 2-cells
\node[cart] at (0.5,0.5) {};
% 0-cells
\node (lu) at (0,1) {$X$};
\node (ru) at (1,1) {$Y$};
\node (ld) at (0,0) {$\widetilde{X}$};
\node (rd) at (1,0) {$\widetilde{Y}$};
% 1-cells
\draw[->] (lu) -- node[arl] {$\al f$} (ru);
\draw[->] (lu) -- node[arr] {$\al i$} (ld);
\draw[->] (ru) -- node[arl] {$\al j$} (rd);
\draw[->] (ld) -- node[arr] {$\al g$} (rd);
\end{tikzpicture}}
\]

From the definitions, we have an obvious isomorphism
\[
\vc{\begin{tikzpicture}[smallcube]
\compuinit
% 2-cells
\node[cubef] at (0.5,0.5) {$\mIBC$};
% 0-cells
\node (lu) at (0,1) {$\cDb_H(\widetilde{X})$};
\node (ru) at (1,1) {$\cDb_K(X)$};
\node (ld) at (0,0) {$\cDb_H(\widetilde{Y})$};
\node (rd) at (1,0) {$\cDb_K(Y)$};
% 1-cells
\draw[<-] (lu) -- node[arl] {$\al \sInd^H_K$} (ru);
\draw[<-] (lu) -- node[arr] {$\al g^*$} (ld);
\draw[<-] (ru) -- node[arl] {$\al f^*$} (rd);
\draw[<-] (ld) -- node[arr] {$\al \sInd^H_K$} (rd);
\end{tikzpicture}}
\]

We define an analogous isomorphism for the other version of induction equivalence:
\[
\vc{\begin{tikzpicture}[smallcube]
\compuinit
% 2-cells
\node[cubef] at (0.5,0.5) {$\mIBC$};
% 0-cells
\node (lu) at (0,1) {$\cDb_H(\widetilde{X})$};
\node (ru) at (1,1) {$\cDb_K(X)$};
\node (ld) at (0,0) {$\cDb_H(\widetilde{Y})$};
\node (rd) at (1,0) {$\cDb_K(Y)$};
% 1-cells
\draw[<-] (lu) -- node[arl] {$\al \hamma^H_K i_!$} (ru);
\draw[<-] (lu) -- node[arr] {$\al g^*$} (ld);
\draw[<-] (ru) -- node[arl] {$\al f^*$} (rd);
\draw[<-] (ld) -- node[arr] {$\al \hamma^H_K j_!$} (rd);
\end{tikzpicture}}
\quad:=\quad
\vc{\begin{tikzpicture}[smallcube]
\compuinit
% 2-cells
\node[cubef] at (0.5,0.5) {$\mInt$};
\node[cubef] at (1.5,0.5) {$\BC$};
% 0-cells
\node (lu) at (0,1) {$\cDb_H(\widetilde{X})$};
\node (mu) at (1,1) {$\cDb_K(\widetilde{X})$};
\node (ru) at (2,1) {$\cDb_K(X)$};
\node (ld) at (0,0) {$\cDb_H(\widetilde{Y})$};
\node (md) at (1,0) {$\cDb_{K}(\widetilde{Y})$};
\node (rd) at (2,0) {$\cDb_{K}(Y)$};
% 1-cells
\draw[<-] (lu) -- node[arl] {$\al \hamma^H_K$} (mu);
\draw[<-] (mu) -- node[arl] {$\al i_!$} (ru);
\draw[<-] (lu) -- node[arr] {$\al g^*$} (ld);
\draw[<-] (mu) -- node[arl] {$\al g^*$} (md);
\draw[<-] (ru) -- node[arl] {$\al f^*$} (rd);
\draw[<-] (ld) -- node[arr] {$\al \hamma^H_{K}$} (md);
\draw[<-] (md) -- node[arr] {$\al j_!$} (rd);
\end{tikzpicture}}
\]

%.....................................................................
\subsubsection{Constant sheaf under induction equivalence}
\label{sss:const-induction-equiv}
%.....................................................................

It is clear from definitions that we have a canonical isomorphism $\sInd_K^H(\ubk_X^K) \cong \ubk_{\widetilde{X}}^H$. Using the isomorphism $\hamma_K^H i_! \natisom \sInd_K^H[2\dim(H/K)]$ we deduce a canonical isomorphism $\hamma_K^H i_!(\ubk_X^K) \cong \ubk_{\widetilde{X}}^H[2 \dim(H/K)]$. We depict the resulting isomorphisms of functors as follows:
\[
\vc{\begin{tikzpicture}[stdtriangle]
\compuinit
% 0-cells
\node (lu) at (0,1) {$\bb$};
\node (r) at (1,0.5) {$\cDb_K(X)$};
\node (ld) at (0,0) {$\cDb_H(\widetilde{X})$};
% 1-cells
\draw[->] (lu) -- node[arl] {$\al \ubk_X^K$} (r);
\draw[->] (lu) -- node[arr] {$\al \ubk_{\widetilde{X}}^H$} (ld);
\draw[->] (r) -- node[arl] {$\al \sInd_K^H$} (ld);
% 2-cell
\node[tricell] at (\tric,0.5) {\tiny$\Rel$};
\end{tikzpicture}}
\qquad
\vc{\begin{tikzpicture}[stdtriangle]
\compuinit
% 0-cells
\node (lu) at (0,1) {$\bb$};
\node (r) at (1,0.5) {$\cDb_K(X)$};
\node (ld) at (0,0) {$\cDb_H(\widetilde{X})$};
% 1-cells
\draw[->] (lu) -- node[arl] {$\al \ubk_X^K$} (r);
\draw[->] (lu) -- node[arr] {$\al \ubk_{\widetilde{X}}^H[2\dim (H/K)]$} (ld);
\draw[->] (r) -- node[arl] {$\al \hamma_K^H i_!$} (ld);
% 2-cell
\node[tricell] at (\tric,0.5) {\tiny$\Rel$};
\end{tikzpicture}}
\]

%---------------------------------------------------------------------
% Start of induction equivalence commutativity lemmas
%---------------------------------------------------------------------

\addtocounter{figure}{2}

%---------------------------------------------------------------------
\subsection{Compatibilities of transitivity of induction equivalence}
\label{ss:constant-transitivity-induction}
%---------------------------------------------------------------------

\begin{figure}
\begin{center}
Setting: \qquad $K \subset J \subset H$,\quad $n = 2\dim(J/K)$,\quad $m = 2\dim(H/K)$,\quad
$\xymatrix@1{X\ar[r]^-{i_1}&J \times^K X\ar[r]^-{i_2}&H \times^K X=\widetilde{X}}$,\quad $i=i_2i_1$

\subfigure[][]{\qc{\begin{tikzpicture}[stdprism]
\compuinit
% hidden 2-cells
\node[cuber] at (0.5,0.5,0) {$\mIE$};
\node[prismdf] at (1,0.5,\tric) {$\ITr$};
% outer 0-cells
\node (rlu) at (0,1,0) {$\cDb_H(\widetilde{X})$};
\node (rru) at (1,1,0) {$\cDb_H(\widetilde{X})$};
\node (fr) at (	1,0.5,1) {$\cDb_J(J\times^K X)$};
\node (rld) at (0,0,0) {$\cDb_K(X)$};
\node (rrd) at (1,0,0) {$\cDb_K(X)$};
% hidden 1-cells
\draw[liner,<-] (rru) -- node[arr,pos=.3] {$\al \sInd^H_K[m]$} (rrd);
% outer 1-cells
\draw[-,double distance=1.5pt] (rlu) -- (rru);
\draw[<-] (rlu) -- node[arr] {$\al \hamma^H_K i_!$} (rld);
\draw[<-] (rru) -- node[arl] {$\al \sInd^H_J[m-n]$} (fr);
\draw[-,double distance=1.5pt] (rld) -- (rrd);
\draw[<-] (fr) -- node[arl] {$\al \sInd^J_K[n]$} (rrd);
% visible 2-cells
\node[prismtf] at (0.5,0.75,0.5) {$\mIE$};
\node[prismbf] at (0.5,0.25,0.5) {$\mIE$};
\node[prismlf] at (0,0.5,\tric) {$\ITr$};
% visible 0- and 1-cells
\node (fl) at (0,0.5,1) {$\cDb_J(J\times^K X)$};
\draw[<-] (rlu) -- node[arl] {$\al \hamma^H_J(i_2)_!$} (fl);
\draw[<-] (fl) -- node[arl,pos=.3] {$\al \hamma^J_K(i_1)_!$} (rld);
\draw[-,double distance=1.5pt] (fl)  -- (fr);
\end{tikzpicture}}\label{lem:ITrcomparison}}

\subfigure[][]{\qc{\begin{tikzpicture}[stdtetr]
\compuinit
% hidden 2-cells
\node[tetrlr] at ({-\tric/2},0.5,\tric) {$\Rel$};
\node[tetrdr] at ({\tric/2},0.5,\tric) {$\Rel$};
% outer 0-cells
\node (ru) at (0,1,0) {$\bb$};
\node (fl) at (-0.5,0.5,1) {$\cDb_H(\widetilde{X})$};
\node (fr) at (0.5,0.5,1) {$\cDb_J(J \times^K X)$};
\node (rd) at (0,0,0) {$\cDb_K(X)$};
% hidden 1-cell
\draw[liner,->] (ru) -- node[arl,pos=.75] {$\al \ubk_X^K$} (rd);
% visible 2-cells
\node[tetrtf] at (0, {(1+\tric)/2}, {1-\tric}) {$\Rel$};
\node[tetrbf] at (0, {(1-\tric)/2}, {1-\tric}) {\tiny$\ITr$};
% visible 1-cells
\draw[<-] (fl) -- node[arl] {$\al \ubk_{\widetilde{X}}^H$} (ru);
\draw[<-] (fl) -- node[arr] {$\al \sInd_K^H$} (rd);
\draw[->] (ru) -- node[arl] {$\al \ubk_{J \times^K X}^J$} (fr);
\draw[->] (rd) -- node[arr] {$\al \sInd_K^J$} (fr);
\draw[<-] (fl) -- node[arl] {$\al \sInd_J^H$} (fr);
\end{tikzpicture}}\label{lem:Indcompositionconstant}}~
\subfigure[][]{\qc{\begin{tikzpicture}[stdtetr]
\compuinit
% hidden 2-cells
\node[tetrlr] at ({-\tric/2},0.5,\tric) {$\Rel$};
\node[tetrdr] at ({\tric/2},0.5,\tric) {$\Rel$};
% outer 0-cells
\node (ru) at (0,1,0) {$\bb$};
\node (fl) at (-0.5,0.5,1) {$\cDb_H(\widetilde{X})$};
\node (fr) at (0.5,0.5,1) {$\cDb_J(J \times^K X)$};
\node (rd) at (0,0,0) {$\cDb_K(X)$};
% hidden 1-cell
\draw[liner,->] (ru) -- node[arl,pos=.75] {$\al \ubk_X^K$} (rd);
% visible 2-cells
\node[tetrtf] at (0, {(1+\tric)/2}, {1-\tric}) {$\Rel$};
\node[tetrbf] at (0, {(1-\tric)/2}, {1-\tric}) {\tiny$\ITr$};
% visible 1-cells
\draw[<-] (fl) -- node[arl] {$\al \ubk_{\widetilde{X}}^H[m]$} (ru);
\draw[<-] (fl) -- node[arr] {$\al \hamma_K^H i_!$} (rd);
\draw[->] (ru) -- node[arl] {$\al \ubk_{J \times^K X}^J[n]$} (fr);
\draw[->] (rd) -- node[arr] {$\al \hamma_K^J (i_1)_!$} (fr);
\draw[<-] (fl) -- node[arl] {$\al \hamma_J^H (i_2)_!$} (fr);
\end{tikzpicture}}\label{lem:Gamma_!compositionconstant}}
\end{center}
\caption{Compatibilities of transitivity of induction equivalence}
\end{figure}
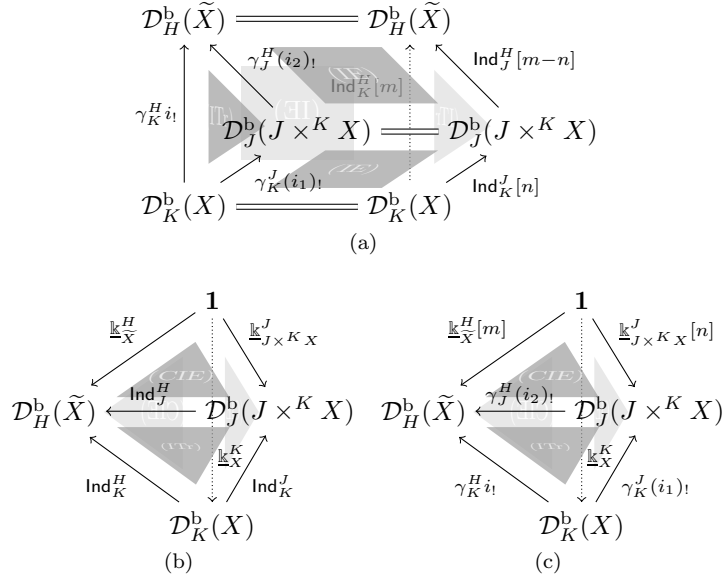

To prove part~\subref{lem:ITrcomparison}, fix a $K$-resolution $P$ of $X$ and consider the following commutative diagram:
\[
\xymatrix@C=1.5cm@R=18pt{
K\backslash P \ar@{^{(}->}[r]^-{(\widetilde{i_1})^K_{J \times^K P}} \ar@{^{(}->}[rd]_-{\widetilde{i}^K_{H \times^K P}} & K\backslash J \times^K P \ar@{^{(}->}[d]^-{(\widetilde{i_2})^K_{H \times^K P}} \ar@{->>}[r]^-{q_{J \times^K P}^{K \subset J}} & J \backslash J \times^K P \ar@{^{(}->}[d]^-{(\widetilde{i_2})^J_{H \times^K P}} \\
& K \backslash H \times^K P \ar@{->>}[r]^-{q_{H \times^K P}^{K \subset J}} \ar@{->>}[rd]_-{q_{H \times^K P}^{K \subset H}} & J \backslash H \times^K P \ar@{->>}[d]^-{q_{H \times^K P}^{J \subset H}}\\
&& H \backslash H \times^K P
}
\]
Denote by $\tau_K^J : K \backslash P \xrightarrow{\sim} J\backslash J \times^K P$, $\tau_J^H :  J\backslash J \times^K P \xrightarrow{\sim}  H\backslash H \times^K P$ and $\tau_K^H=\tau_J^H \tau_K^J$ the natural isomorphisms. The statement we must prove is equivalent to the commutativity of the diagram obtained by gluing the following two prisms, where all faces are labelled by $(\cdot)_!$ composition isomorphisms:
\[
\vc{\begin{tikzpicture}[longprism2]
\compuinit
% outer 0-cells
\node (rlu) at (0,1,0) {$\cDb(K\backslash J\times^K P)$};
\node (rru) at (1,1,0) {$\cDb(J\backslash J \times^K P)$};
\node (fr) at (	1,0.5,1) {$\cDb(J\backslash J \times^K P)$};
\node (rld) at (0,0,0) {$\cDb(K\backslash H\times^K P)$};
\node (rrd) at (1,0,0) {$\cDb(J\backslash H \times^K P)$};
% hidden 1-cells
\draw[->] (rru) -- node[arr,pos=.3] {$\al ((\widetilde{i_2})^J_{H \times^K P})_!$} (rrd);
% outer 1-cells
\draw[->] (rlu) -- node[arl] {$\al (q_{J \times^K P}^{K \subset J})_!$} (rru);
\draw[->] (rlu) -- node[arr] {$\al ((\widetilde{i_2})^K_{H \times^K P})_!$} (rld);
\draw[double distance=1.5pt,-] (rru) -- (fr);
\draw[->] (rld) -- node[arr] {$\al (q_{H \times^K P}^{K \subset J})_!$} (rrd);
\draw[->] (fr) -- node[arl] {$\al ((\widetilde{i_2})^K_{H \times^K P})_!$} (rrd);
% visible 0- and 1-cells
\node (fl) at (0,0.5,1) {$\cDb(K\backslash P)$};
\draw[<-] (rlu) -- node[arl] {$\al ((\widetilde{i_1})^K_{J \times^K P})_!$} (fl);
\draw[->] (fl) -- node[arl,pos=.3] {$\al (\widetilde{i}^K_{H \times^K P})_!$} (rld);
\draw[linef,->] (fl)  -- node[arr,pos=.3] {$\al (\tau_K^J)_!$} (fr);
\end{tikzpicture}}
\]
\[
\vc{\begin{tikzpicture}[longprism2]
\compuinit
% outer 0-cells
\node (rlu) at (0,1,0) {$\cDb(K\backslash H\times^K P)$};
\node (rru) at (1,1,0) {$\cDb(J\backslash H \times^K P)$};
\node (fr) at (	1,0.5,1) {$\cDb(J\backslash J \times^K P)$};
\node (rld) at (0,0,0) {$\cDb(H\backslash H\times^K P)$};
\node (rrd) at (1,0,0) {$\cDb(H\backslash H \times^K P)$};
% hidden 1-cells
\draw[->] (rru) -- node[arr,pos=.3] {$\al (q_{H \times^K P}^{J \subset H})_!$} (rrd);
% outer 1-cells
\draw[->] (rlu) -- node[arl] {$\al (q_{H \times^K P}^{K \subset J})_!$} (rru);
\draw[->] (rlu) -- node[arr] {$\al (q_{H \times^K P}^{K \subset H})_!$} (rld);
\draw[<-] (rru) -- node[arl] {$\al ((\widetilde{i_2})^K_{H \times^K P})_!$} (fr);
\draw[double distance=1.5pt,-] (rld) -- (rrd);
\draw[->] (fr) -- node[arl] {$\al (\tau_J^H)_!$} (rrd);
% visible 0- and 1-cells
\node (fl) at (0,0.5,1) {$\cDb(K\backslash P)$};
\draw[<-] (rlu) -- node[arl] {$\al (\widetilde{i}^K_{H \times^K P})_!$} (fl);
\draw[->] (fl) -- node[arl,pos=.3] {$\al (\tau_K^H)_!$} (rld);
\draw[linef,->] (fl)  -- node[arr,pos=.3] {$\al (\tau_K^J)_!$} (fr);
\end{tikzpicture}}
\]
Hence the result follows from Lemma \ref{lem:_!composition_!}.

Part~\subref{lem:Indcompositionconstant} is easy. By definition, the tetrahedron in part~\subref{lem:Gamma_!compositionconstant} is obtained by gluing the prism in part~\subref{lem:ITrcomparison} to the tetrahedron in part~\subref{lem:Indcompositionconstant} (with appropriate shifts included).

%---------------------------------------------------------------------
\subsection{Compatibilities of integration and induction equivalence}
%---------------------------------------------------------------------

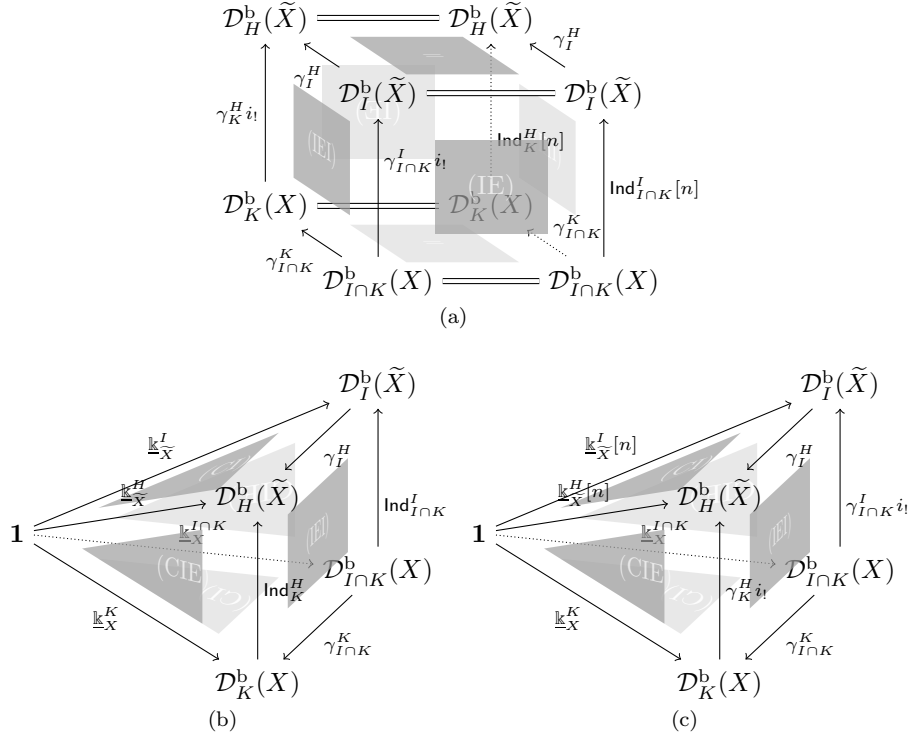
\begin{figure}
\begin{center}
Setting: \qquad $H=IK$, \quad $H/I$ contractible,\quad $n=2\dim(H/K)$, \quad $\xymatrix@1{X\ar[r]^-{i}&H \times^K X}=\widetilde{X}$

\subfigure[][]{\qc{\begin{tikzpicture}[stdcube]
\compuinit
% hidden 0-cell
\node (rrd) at (1,0,0) {$\cDb_K(X)$};
% hidden 2-cells
\node[cuber] at (0.5,0.5,0) {$\mIE$};
\node[cubed] at (1,0.5,0.5) {$\IEI$};
\node[cubeb] at (0.5,0,0.5) {$=$};
% outer 0-cells
\node (rlu) at (0,1,0) {$\cDb_H(\widetilde{X})$};
\node (rru) at (1,1,0) {$\cDb_H(\widetilde{X})$};
\node (fru) at (1,1,1) {$\cDb_I(\widetilde{X})$};
\node (frd) at (1,0,1) {$\cDb_{I\cap K}(X)$};
\node (fld) at (0,0,1) {$\cDb_{I\cap K}(X)$};
\node (rld) at (0,0,0) {$\cDb_K(X)$};
% hidden 1-cells
\draw[liner,<-] (rru) -- node[arl,pos=.7] {$\al \sInd^H_K[n]$} (rrd);
\draw[-,double distance=1.5pt] (rld) -- (rrd);
\draw[liner,<-] (rrd) -- node[arl] {$\al \hamma^K_{I\cap K}$} (frd);
% outer 1-cells
\draw[-,double distance=1.5pt] (rlu) -- (rru);
\draw[<-] (rru) -- node[arl] {$\al \hamma^H_I$} (fru);
\draw[<-] (fru) -- node[arl] {$\al \sInd^I_{I\cap K}[n]$}(frd);
\draw[-,double distance=1.5pt] (fld) -- (frd);
\draw[<-] (rld) -- node[arr] {$\al \hamma^K_{I\cap K}$} (fld);
\draw[<-] (rlu) -- node[arr] {$\al \hamma^H_K i_!$} (rld);
% visible 2-cells
\node[cubel] at (0,0.5,0.5) {$\IEI$};
\node[cubet] at (0.5,1,0.5) {$=$};
\node[cubef] at (0.5,0.5,1) {$\mIE$};
% visible 0- and 1-cells
\node (flu) at (0,1,1) {$\cDb_I(\widetilde{X})$};
\draw[<-] (rlu) -- node[arr,pos=.7] {$\al \hamma^H_I$} (flu);
\draw[-,double distance=1.5pt] (flu) -- (fru);
\draw[<-] (flu) -- node[arl,pos=.3] {$\al \hamma^I_{I\cap K}i_!$} (fld);
\end{tikzpicture}}\label{lem:IEIcomparison}}

\subfigure[][]{\qc{\begin{tikzpicture}[stdpyra]
\compuinit
% hidden 0-cell
\node (rrd) at (1,0,0) {$\cDb_{I\cap K}(X)$};
% hidden 2-cells
\node[pyrar] at (1-\tric,0.5,0.5*\tric) {$\Rel$};
\node[pyrab] at (1-\tric,0.5*\tric,0.5) {$\CInt$};
% outer 0-cells
\node (l) at (0,0.5,0.5) {$\bb$};
\node (rru) at (1,1,0) {$\cDb_I(\widetilde{X})$};
\node (frd) at (1,0,1) {$\cDb_K(X)$};
% hidden 1-cells
\draw[liner,->] (l) -- node[arl] {$\al \ubk_X^{I\cap K}$} (rrd);
% outer 1-cells
\draw[<-] (rru) -- node[arl,pos=.7] {$\al \sInd_{I\cap K}^I$} (rrd);
\draw[->] (rrd) -- node[arl] {$\al \hamma_{I\cap K}^K$} (frd);
\draw[->] (l) -- node[arl] {$\al \ubk_{\widetilde{X}}^I$} (rru);
\draw[->] (l) -- node[arr] {$\al \ubk_X^K$} (frd);
% visible 2-cells
\node[cubel,xscale=-1] at (1,0.5,0.5) {$\IEI$};
\node[pyraf] at (1-\tric,0.5,1-0.5*\tric) {$\Rel$};
\node[pyrat] at (1-\tric,1-0.5*\tric,0.5) {$\CInt$};
% visible 0- and 1-cells
\node (fru) at (1,1,1) {$\cDb_H(\widetilde{X})$};
\draw[->] (rru) -- node[arl] {$\al \hamma_I^H$} (fru);
\draw[<-] (fru) -- node[arl] {$\al \sInd_K^H$}(frd);
\draw[->] (l) -- node[arl,pos=.7] {$\al \ubk_{\widetilde{X}}^H$} (fru);
\end{tikzpicture}}\label{lem:IndGammaconstant}}~
\subfigure[][]{\qc{\begin{tikzpicture}[stdpyra]
\compuinit
% hidden 0-cell
\node (rrd) at (1,0,0) {$\cDb_{I\cap K}(X)$};
% hidden 2-cells
\node[pyrar] at (1-\tric,0.5,0.5*\tric) {$\Rel$};
\node[pyrab] at (1-\tric,0.5*\tric,0.5) {$\CInt$};
% outer 0-cells
\node (l) at (0,0.5,0.5) {$\bb$};
\node (rru) at (1,1,0) {$\cDb_I(\widetilde{X})$};
\node (frd) at (1,0,1) {$\cDb_K(X)$};
% hidden 1-cells
\draw[liner,->] (l) -- node[arl] {$\al \ubk_X^{I\cap K}$} (rrd);
% outer 1-cells
\draw[<-] (rru) -- node[arl,pos=.7] {$\al \hamma_{I\cap K}^I i_!$} (rrd);
\draw[->] (rrd) -- node[arl] {$\al \hamma_{I\cap K}^K$} (frd);
\draw[->] (l) -- node[arl] {$\al \ubk_{\widetilde{X}}^I[n]$} (rru);
\draw[->] (l) -- node[arr] {$\al \ubk_X^K$} (frd);
% visible 2-cells
\node[cubel,xscale=-1] at (1,0.5,0.5) {$\IEI$};
\node[pyraf] at (1-\tric,0.5,1-0.5*\tric) {$\Rel$};
\node[pyrat] at (1-\tric,1-0.5*\tric,0.5) {$\CInt$};
% visible 0- and 1-cells
\node (fru) at (1,1,1) {$\cDb_H(\widetilde{X})$};
\draw[->] (rru) -- node[arl] {$\al \hamma_I^H$} (fru);
\draw[<-] (fru) -- node[arl] {$\al \hamma_K^H i_!$}(frd);
\draw[->] (l) -- node[arl,pos=.7] {$\al \ubk_{\widetilde{X}}^H[n]$} (fru);
\end{tikzpicture}}\label{lem:Gamma_!Gammaconstant}}
\end{center}
\caption{Compatibilities of integration and induction equivalence}
\end{figure}

Part~\subref{lem:IEIcomparison} can be proved in the same way as Lemma~\ref{lem:ITrcomparison}. Part~\subref{lem:IndGammaconstant} is easy. By definition, the pyramid in part~\subref{lem:Gamma_!Gammaconstant} is obtained by gluing the cube in part~\subref{lem:IEIcomparison} to the pyramid in part~\subref{lem:IndGammaconstant} (with appropriate shifts included).

%---------------------------------------------------------------------
\subsection{Compatibilities of inverse image and induction equivalence}
%---------------------------------------------------------------------

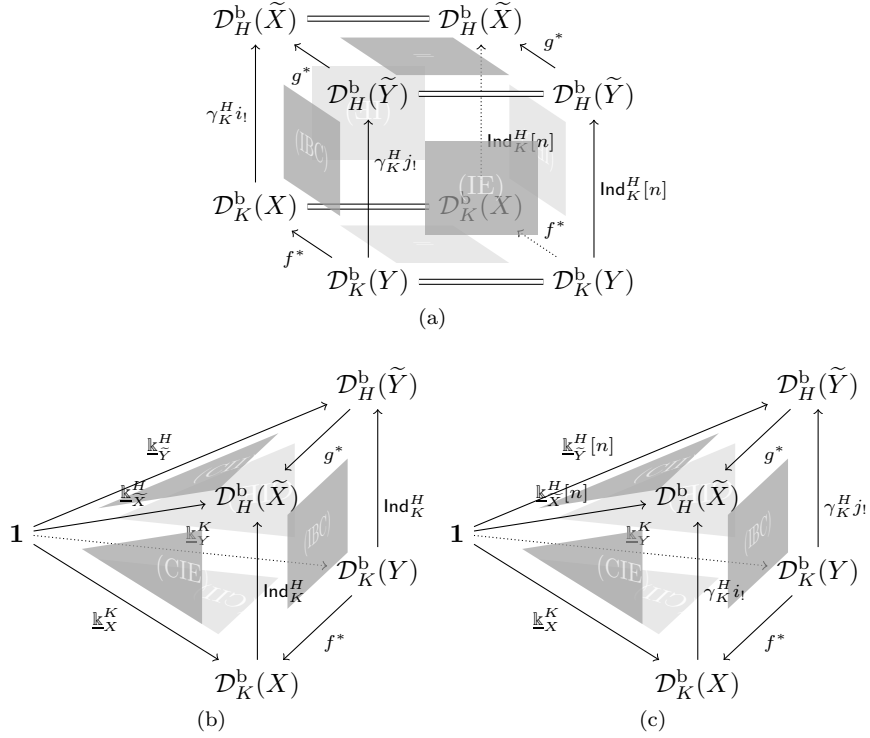
\begin{figure}
\begin{center}
Setting: \qquad $K \subset H$, \quad $\widetilde{X}=H\times^K X$, \quad
\qc{\begin{tikzpicture}[vsmallcube]
\compuinit
% 2-cells
\node[cart] at (0.5,0.5) {};
% 0-cells
\node (lu) at (0,1) {$X$};
\node (ru) at (1,1) {$Y$};
\node (ld) at (0,0) {$\widetilde{X}$};
\node (rd) at (1,0) {$\widetilde{Y}$};
% 1-cells
\draw[->] (lu) -- node[arl] {$\al f$} (ru);
\draw[->] (lu) -- node[arr] {$\al i$} (ld);
\draw[->] (ru) -- node[arl] {$\al j$} (rd);
\draw[->] (ld) -- node[arr] {$\al g$} (rd);
\end{tikzpicture}},\quad $n = 2\dim(H/K)$

\subfigure[][]{\qc{\begin{tikzpicture}[stdcube]
\compuinit
% hidden 0-cell
\node (rrd) at (1,0,0) {$\cDb_K(X)$};
% hidden 2-cells
\node[cuber] at (0.5,0.5,0) {$\mIE$};
\node[cubed] at (1,0.5,0.5) {$\IBC$};
\node[cubeb] at (0.5,0,0.5) {$=$};
% outer 0-cells
\node (rlu) at (0,1,0) {$\cDb_H(\widetilde{X})$};
\node (rru) at (1,1,0) {$\cDb_H(\widetilde{X})$};
\node (fru) at (1,1,1) {$\cDb_H(\widetilde{Y})$};
\node (frd) at (1,0,1) {$\cDb_{K}(Y)$};
\node (fld) at (0,0,1) {$\cDb_{K}(Y)$};
\node (rld) at (0,0,0) {$\cDb_K(X)$};
% hidden 1-cells
\draw[liner,<-] (rru) -- node[arl,pos=.7] {$\al \sInd^H_K[n]$} (rrd);
\draw[-,double distance=1.5pt] (rld) -- (rrd);
\draw[liner,<-] (rrd) -- node[arl] {$\al f^*$} (frd);
% outer 1-cells
\draw[-,double distance=1.5pt] (rlu) -- (rru);
\draw[<-] (rru) -- node[arl] {$\al g^*$} (fru);
\draw[<-] (fru) -- node[arl] {$\al \sInd^H_{K}[n]$}(frd);
\draw[-,double distance=1.5pt] (fld) -- (frd);
\draw[<-] (rld) -- node[arr] {$\al f^*$} (fld);
\draw[<-] (rlu) -- node[arr] {$\al \hamma^H_K i_!$} (rld);
% visible 2-cells
\node[cubel] at (0,0.5,0.5) {$\IBC$};
\node[cubet] at (0.5,1,0.5) {$=$};
\node[cubef] at (0.5,0.5,1) {$\mIE$};
% visible 0- and 1-cells
\node (flu) at (0,1,1) {$\cDb_H(\widetilde{Y})$};
\draw[<-] (rlu) -- node[arr,pos=.7] {$\al g^*$} (flu);
\draw[-,double distance=1.5pt] (flu) -- (fru);
\draw[<-] (flu) -- node[arl,pos=.3] {$\al \hamma^H_{K}j_!$} (fld);
\end{tikzpicture}}\label{lem:IBCcomparison}}

\subfigure[][]{\qc{\begin{tikzpicture}[stdpyra]
\compuinit
% hidden 0-cell
\node (rrd) at (1,0,0) {$\cDb_K(Y)$};
% hidden 2-cells
\node[pyrar] at (1-\tric,0.5,0.5*\tric) {$\Rel$};
\node[pyrab] at (1-\tric,0.5*\tric,0.5) {$\Cnst$};
% outer 0-cells
\node (l) at (0,0.5,0.5) {$\bb$};
\node (rru) at (1,1,0) {$\cDb_H(\widetilde{Y})$};
\node (frd) at (1,0,1) {$\cDb_K(X)$};
% hidden 1-cells
\draw[liner,->] (l) -- node[arl] {$\al \ubk_Y^K$} (rrd);
% outer 1-cells
\draw[<-] (rru) -- node[arl,pos=.7] {$\al \sInd_K^H$} (rrd);
\draw[->] (rrd) -- node[arl] {$\al f^*$} (frd);
\draw[->] (l) -- node[arl] {$\al \ubk_{\widetilde{Y}}^H$} (rru);
\draw[->] (l) -- node[arr] {$\al \ubk_X^K$} (frd);
% visible 2-cells
\node[cubel,xscale=-1] at (1,0.5,0.5) {$\IBC$};
\node[pyraf] at (1-\tric,0.5,1-0.5*\tric) {$\Rel$};
\node[pyrat] at (1-\tric,1-0.5*\tric,0.5) {$\Cnst$};
% visible 0- and 1-cells
\node (fru) at (1,1,1) {$\cDb_H(\widetilde{X})$};
\draw[->] (rru) -- node[arl] {$\al g^*$} (fru);
\draw[<-] (fru) -- node[arl] {$\al \sInd_K^H$}(frd);
\draw[->] (l) -- node[arl,pos=.7] {$\al \ubk_{\widetilde{X}}^H$} (fru);
\end{tikzpicture}}\label{lem:Ind^*constant}}~
\subfigure[][]{\qc{\begin{tikzpicture}[stdpyra]
\compuinit
% hidden 0-cell
\node (rrd) at (1,0,0) {$\cDb_K(Y)$};
% hidden 2-cells
\node[pyrar] at (1-\tric,0.5,0.5*\tric) {$\Rel$};
\node[pyrab] at (1-\tric,0.5*\tric,0.5) {$\Cnst$};
% outer 0-cells
\node (l) at (0,0.5,0.5) {$\bb$};
\node (rru) at (1,1,0) {$\cDb_H(\widetilde{Y})$};
\node (frd) at (1,0,1) {$\cDb_K(X)$};
% hidden 1-cells
\draw[liner,->] (l) -- node[arl] {$\al \ubk_Y^K$} (rrd);
% outer 1-cells
\draw[<-] (rru) -- node[arl,pos=.7] {$\al \hamma_K^H j_!$} (rrd);
\draw[->] (rrd) -- node[arl] {$\al f^*$} (frd);
\draw[->] (l) -- node[arl] {$\al \ubk_{\widetilde{Y}}^H[n]$} (rru);
\draw[->] (l) -- node[arr] {$\al \ubk_X^K$} (frd);
% visible 2-cells
\node[cubel,xscale=-1] at (1,0.5,0.5) {$\IBC$};
\node[pyraf] at (1-\tric,0.5,1-0.5*\tric) {$\Rel$};
\node[pyrat] at (1-\tric,1-0.5*\tric,0.5) {$\Cnst$};
% visible 0- and 1-cells
\node (fru) at (1,1,1) {$\cDb_H(\widetilde{X})$};
\draw[->] (rru) -- node[arl] {$\al g^*$} (fru);
\draw[<-] (fru) -- node[arl] {$\al \hamma_K^H i_!$}(frd);
\draw[->] (l) -- node[arl,pos=.7] {$\al \ubk_{\widetilde{X}}^H[n]$} (fru);
\end{tikzpicture}}\label{lem:Gamma_!^*constant}}
\end{center}
\caption{Compatibilities of inverse image and induction equivalence}
\end{figure}

The proof of~\subref{lem:IBCcomparison} is similar to that of Lemma~\ref{lem:ITrcomparison}, but using Lemma \ref{lem:_!composition^*} rather than Lemma \ref{lem:_!composition_!}. Part~\subref{lem:Ind^*constant} is easy. By definition, the pyramid in~\subref{lem:Gamma_!^*constant} is obtained by gluing the cube in~\subref{lem:IBCcomparison} to the pyramid in~\subref{lem:Ind^*constant} (with appropriate shifts included).

%---------------------------------------------------------------------
\subsection{Equivariance under a finite group action}
\label{ss:equivariance-finite}
%---------------------------------------------------------------------

Let $f : X \to Y$ be a morphism of $H$-varieties, and assume that we have an action of a finite group $A$ on $X$ which commutes with the $H$-action, and such that $f$ is $A$-equivariant for the trivial $A$-action on $Y$. Then we obtain a canonical action of $A$ on the object $f_! \ubk_X^H$ of $\cDb_H(Y)$, in which the action of $a \in A$ is given by the composition $f_! \ubk_X^H \mathrel{\overset{\CnstRes}{\cong}} f_! a^* \ubk_X^H \mathrel{\overset{\BC}{\cong}} f_! \ubk_X^H$. Here $a$ denotes the action of $a$ on $X$, and the base change is for the square
\[
\vc{\begin{tikzpicture}[vsmallcube2]
\compuinit
% 2-cells
\node[cart] at (0.5,0.5) {};
% 0-cells
\node (lu) at (0,1) {$X$};
\node (ru) at (1,1) {$X$};
\node (ld) at (0,0) {$Y$};
\node (rd) at (1,0) {$Y$};
% 1-cells
\draw[->] (lu) -- node[arl] {$\al a$} (ru);
\draw[->] (lu) -- node[arr] {$\al f$} (ld);
\draw[->] (ru) -- node[arl] {$\al f$} (rd);
\draw[->] (ld) -- node[arr] {$\al \id$} (rd);
\end{tikzpicture}}
\]
(This construction defines an action of $A$ by Lemmas~\ref{lem:^*compositionconstant} and~\ref{lem:^*composition_!}.)

Now, consider a closed subgroup $K \subset H$, a $K$-variety $X$, and an $H$-variety $Y$. As usual, let $\widetilde{X}=H\times^K X$ and let $i : X \to \widetilde{X}$ be the inclusion. Assume that we have
an $H$-equivariant morphism $g: \widetilde{X} \to Y$. Let $f:=g \circ i$; it is automatically $K$-equivariant.
Assume furthermore that a finite group $A$ acts on $X$ compatibly with $K$ and that $f$ is $A$-equivariant for the trivial $A$-action on $Y$. Then we have a natural $A$-action on $\widetilde{X}$, and $g$ is $A$-equivariant. In particular, we obtain $A$-actions on 
%the objects
$f_! \ubk_X^K \in \cDb_K(Y)$ and $g_! \ubk_{\widetilde{X}}^H \in \cDb_H(Y)$. Recall 
%that we have constructed an 
the isomorphism $\hamma_K^H i_! \ubk_X^K \overset{\CnstIE}{\cong} \ubk_{\widetilde{X}}^H[2 \dim(H/K)]$ from \S\ref{sss:const-induction-equiv}. 
Applying the functor $g_!$, this induces an isomorphism
\begin{equation}
\label{eqn:isom-action-A}
g_! \ubk_{\widetilde{X}}^H[2 \dim(H/K)] \overset{\CnstIE}{\cong} g_! \hamma_K^H i_! \ubk_X^K \mathrel{\overset{\mInt}{\cong}} \hamma_K^H g_! i_! \ubk_X^K \mathrel{\overset{\Co}{\cong}}\hamma_K^H f_! \ubk_X^K.
\end{equation}

\setcounter{thm}{\value{subsection}}
\addtocounter{thm}{-1}

\begin{lem}\label{lem:_!action}
Isomorphism \eqref{eqn:isom-action-A} is $A$-equivariant.
\end{lem}

\begin{proof}
Let $n = 2\dim (H/K)$.  The compatibility of \eqref{eqn:isom-action-A} with the action of $a \in A$ is equivalent to the commutativity of the diagram obtained by gluing the pyramid
\[
\vc{\begin{tikzpicture}[xscale=4,yscale=2,z={(-0.4,-0.6)}]
\compuinit
% hidden 0-cell
\node (rrd) at (1,0,0) {$\cDb_K(X)$};
% hidden 2-cells
\node[pyrar] at (1-\tric,0.5,0.5*\tric) {$\Rel$};
\node[pyrab] at (1-\tric,0.5*\tric,0.5) {$\Cnst$};
% outer 0-cells
\node (l) at (0,0.5,0.5) {$\bb$};
\node (rru) at (1,1,0) {$\cDb_H(\widetilde{X})$};
\node (frd) at (1,0,1) {$\cDb_K(X)$};
% hidden 1-cells
\draw[liner,->] (l) -- node[arl] {$\al \ubk_X^K$} (rrd);
% outer 1-cells
\draw[<-] (rru) -- node[arl,pos=.7] {$\al \hamma_K^H i_!$} (rrd);
\draw[->] (rrd) -- node[arl] {$\al a^*$} (frd);
\draw[->] (l) -- node[arl] {$\al \ubk_{\widetilde{X}}^H[n]$} (rru);
\draw[->] (l) -- node[arr] {$\al \ubk_X^K$} (frd);
% visible 2-cells
\node[cubel,xscale=-1] at (1,0.5,0.5) {$\IBC$};
\node[pyraf] at (1-\tric,0.5,1-0.5*\tric) {$\Rel$};
\node[pyrat] at (1-\tric,1-0.5*\tric,0.5) {$\Cnst$};
% visible 0- and 1-cells
\node (fru) at (1,1,1) {$\cDb_H(\widetilde{X})$};
\draw[->] (rru) -- node[arl] {$\al a^*$} (fru);
\draw[<-] (fru) -- node[arl] {$\al \hamma_K^H i_!$}(frd);
\draw[->] (l) -- node[arl,pos=.7] {$\al \ubk_{\widetilde{X}}^H[n]$} (fru);
\end{tikzpicture}}
\]
which is commutative by Lemma \ref{lem:Gamma_!^*constant} to the two cubes
\[
\vc{\begin{tikzpicture}[xscale=3,yscale=2,z={(0.5,-0.4)}]
\compuinit
% hidden 0-cell
\node (rrd) at (1,0,0) {$\cDb_K(Y)$};
% hidden 2-cells
\node[cuber] at (0.5,0.5,0) {$\Comp$};
\node[cubed] at (1,0.5,0.5) {$\BC$};
\node[cubeb] at (0.5,0,0.5) {$\BC$};
% outer 0-cells
\node (rlu) at (0,1,0) {$\cDb_K(\widetilde{X})$};
\node (rru) at (1,1,0) {$\cDb_K(Y)$};
\node (fru) at (1,1,1) {$\cDb_K(Y)$};
\node (frd) at (1,0,1) {$\cDb_K(Y)$};
\node (fld) at (0,0,1) {$\cDb_K(X)$};
\node (rld) at (0,0,0) {$\cDb_K(X)$};
% hidden 1-cells
\draw[liner,<-] (rru) -- node[arl,pos=.7] {$\al \id_!$} (rrd);
\draw[liner,->] (rld) -- node[arl,pos=.3] {$\al f_!$} (rrd);
\draw[liner,->] (rrd) -- node[arl] {$\al \id^*$} (frd);
% outer 1-cells
\draw[->] (rlu) -- node[arl] {$\al g_!$} (rru);
\draw[->] (rru) -- node[arl] {$\al \id^*$} (fru);
\draw[<-] (fru) -- node[arl,pos=.3] {$\al \id_!$}(frd);
\draw[->] (fld) -- node[arr] {$\al f_!$} (frd);
\draw[->] (rld) -- node[arr] {$\al a^*$} (fld);
\draw[<-] (rlu) -- node[arr] {$\al i_!$} (rld);
% visible 2-cells
\node[cubel] at (0,0.5,0.5) {$\BC$};
\node[cubet] at (0.5,1,0.5) {$\BC$};
\node[cubef] at (0.5,0.5,1) {$\Comp$};
% visible 0- and 1-cells
\node (flu) at (0,1,1) {$\cDb_K(\widetilde{X})$};
\draw[->] (rlu) -- node[arl,pos=.7] {$\al a^*$} (flu);
\draw[->] (flu) -- node[arr,pos=.3] {$\al g_!$} (fru);
\draw[<-] (flu) -- node[arl,pos=.3] {$\al i_!$} (fld);
\end{tikzpicture}}\!
\vc{\begin{tikzpicture}[xscale=3,yscale=2,z={(0.5,-0.4)}]
\compuinit
% hidden 0-cell
\node (rrd) at (1,0,0) {$\cDb_K(Y)$};
% hidden 2-cells
\node[cuber] at (0.5,0.5,0) {$\mInt$};
\node[cubed] at (1,0.5,0.5) {$\mInt$};
\node[cubeb] at (0.5,0,0.5) {$\BC$};
% outer 0-cells
\node (rlu) at (0,1,0) {$\cDb_H(\widetilde{X})$};
\node (rru) at (1,1,0) {$\cDb_H(Y)$};
\node (fru) at (1,1,1) {$\cDb_H(Y)$};
\node (frd) at (1,0,1) {$\cDb_K(Y)$};
\node (fld) at (0,0,1) {$\cDb_K(\widetilde{X})$};
\node (rld) at (0,0,0) {$\cDb_K(\widetilde{X})$};
% hidden 1-cells
\draw[liner,<-] (rru) -- node[arl,pos=.7] {$\al \hamma_K^H$} (rrd);
\draw[liner,->] (rld) -- node[arl,pos=.3] {$\al g_!$} (rrd);
\draw[liner,->] (rrd) -- node[arl] {$\al \id^*$} (frd);
% outer 1-cells
\draw[->] (rlu) -- node[arl] {$\al g_!$} (rru);
\draw[->] (rru) -- node[arl] {$\al \id^*$} (fru);
\draw[<-] (fru) -- node[arl] {$\al \hamma_K^H$}(frd);
\draw[->] (fld) -- node[arr] {$\al g_!$} (frd);
\draw[->] (rld) -- node[arr] {$\al a^*$} (fld);
\draw[<-] (rlu) -- node[arr] {$\al \hamma_K^H$} (rld);
% visible 2-cells
\node[cubel] at (0,0.5,0.5) {$\mInt$};
\node[cubet] at (0.5,1,0.5) {$\BC$};
\node[cubef] at (0.5,0.5,1) {$\mInt$};
% visible 0- and 1-cells
\node (flu) at (0,1,1) {$\cDb_H(\widetilde{X})$};
\draw[->] (rlu) -- node[arl,pos=.7] {$\al a^*$} (flu);
\draw[->] (flu) -- node[arr,pos=.3] {$\al g_!$} (fru);
\draw[<-] (flu) -- node[arl,pos=.3] {$\al \hamma_K^H$} (fld);
\end{tikzpicture}}
\]
which are commutative by Lemmas~\ref{lem:_!^*basechange_!} and~\ref{lem:basechangeGamma}, respectively.
\end{proof}

%%%%%%%%%%%%%%%%%%%%%%%%%%%%%%%%%%%%%%%%%%%%%%%%%%%%%%%%%%%%%%%%%%%%%%%%%%%%%%%%%%%%%%%%%%%%%%%%%%%%%%%


\begin{thebibliography}{AHJR2}

\bibitem[Ac]{a}
P.~Achar, \emph{Green functions via hyperbolic localization}, Doc.~Math.\ \textbf{16} (2011), 869--884.

\bibitem[AH]{ah}
P.~Achar, A.~Henderson, \emph{Geometric Satake, Springer correspondence and small representations}, Selecta Math.\ (N.S.) \textbf{19} (2013), no.~4, 949--986.

\bibitem[AHJR1]{ahjr}
P.~Achar, A.~Henderson, D.~Juteau, S.~Riche, \emph{Weyl group actions on the Springer sheaf}, Proc.\ Lond.\ Math.\ Soc.\ \textbf{108} (2014), no.~6, 1501--1528.

\bibitem[AHJR2]{genspring1}
P.~Achar, A.~Henderson, D.~Juteau, S.~Riche, \emph{Modular generalized Springer correspondence I: the general linear group}, to appear in J.\ Eur.\ Math.\ Soc.\ (JEMS), arXiv:1307.2702.

\bibitem[AHJR3]{genspring2}
P.~Achar, A.~Henderson, D.~Juteau, S.~Riche, \emph{Modular generalized Springer correspondence II: classical groups}, to appear in J.\ Eur.\ Math.\ Soc.\ (JEMS), arXiv:1404.1096.

\bibitem[AM]{am}
P.~Achar, C.~Mautner, \emph{Sheaves on nilpotent cones, Fourier transform, and a geometric Ringel duality}, to appear in Mosc.\ Math.\ J., arXiv:1207.7044.

\bibitem[BBD]{bbd}
A.~Be{\u\i}linson, J.~Bernstein, P.~Deligne, \emph{Faisceaux pervers}, in \emph{Analysis and topology on singular spaces, I (Luminy, 1981)}, Ast{\'e}risque \textbf{100} (1982), 5--171.

\bibitem[BD]{bd}
A.~Be{\u\i}linson, V.~Drinfeld, \emph{Quantization of Hitchin's integrable system and Hecke eigensheaves}, available at \texttt{http://www.math.uchicago.edu/$\sim$mitya/langlands.html}.

\bibitem[BL]{bl}
J.~Bernstein, V.~Lunts, \emph{Equivariant sheaves and
functors}, Lecture Notes in Math.~1578, Springer, 1994.

\bibitem[BF]{bf}
R.~Bezrukavnikov, M.~Finkelberg, \emph{Equivariant Satake category and Kostant-Whittaker reduction}, Mosc.~Math.~J.~\textbf{8} (2008), no.~1, 39--72, 183.

\bibitem[BFM]{bfm}
R.~Bezrukavnikov, M.~Finkelberg, I.~Mirkovi{\'c}, \emph{Equivariant homology and $K$-theory of affine Grassmannians and Toda lattices}, Compos.~Math.~\textbf{141} (2005), no.~3, 746--768. 

\bibitem[BV]{bv}
J.~M. Boardman, R.~M. Vogt, {\em Homotopy invariant algebraic structures on
  topological spaces}, Lecture Notes in Mathematics, vol. 347, Springer-Verlag,
  Berlin, 1973.
  
\bibitem[BM]{bm}
W.~Borho, R.~MacPherson,
\emph{Repr\'esentations des groupes de Weyl et homologie d'inter\-section pour les vari\'et\'es nilpotentes},
C.\ R.\ Acad.\ Sci.\ Paris S\'er.\ I Math.\ {\bf 292} (1981), no.~15, 707--710. 

\bibitem[Bra]{br} 
T.~Braden, \emph{Hyperbolic localization of intersection cohomology}, Transform.~Groups \textbf{8} (2003), 209--216.

\bibitem[Bro]{broer}
A.~Broer, {\em The sum of generalized exponents and Chevalley's restriction
  theorem for modules of covariants}, Indag.\ Math.\ (N.S.) {\bf 6} (1995),
  385--396.

\bibitem[Bry]{brylinski}
J.-L.~Brylinski, {\em Transformations canoniques, dualit\'e projective, th\'eorie de Lefschetz, transformations de Fourier et sommes trigonom\'etriques}, Ast\'erisque {\bf 140-141} (1986), 3--134.

\bibitem[CG]{CG}
N.~Chriss, V.~Ginzburg, \emph{Representation theory
    and complex geometry}, Bir\-kh{\"a}user, 1997.

\bibitem[De]{de}
P.~Deligne, \emph{Cohomologie {\`a} support propre}, Expos{\'e} XVII in \emph{Th{\'e}orie des topos et cohomologie {\'e}tale des sch{\'e}mas (SGA 4.3)}, Lecture Notes in Math.~305, Springer, 1973.

\bibitem[DM]{dm} 
P.~Deligne, J.~Milne, \emph{Tannakian categories}, in \emph{Hodge cycles, motives, and Shimura varieties}, Lecture Notes in Math.~900, Springer, 1982.

\bibitem[HKK]{hkk}
K.~A.~Hardie, K.~H.~Kamps, R.~W.~Kieboom, \emph{A homotopy $2$-groupoid of a Hausdorff space}, Applied Categorical Structures \textbf{8} (2000), 209--234.

\bibitem[Iv]{iv}
B.~Iversen, \emph{Cohomology of sheaves}, Springer, 1986.

\bibitem[Ja]{jantzen}
J.~C.~Jantzen, {\em Representations of algebraic groups}, 2nd ed., Mathematical
  Surveys and Monographs, no. 107, Amer.\ Math.\ Soc., Providence, RI, 2003.

\bibitem[Jo]{joyal}
A.~Joyal, {\em Quasi-categories and Kan complexes}, J. Pure Appl.\ Algebra {\bf
  175} (2002), 207--222.
    
\bibitem[Ju]{j}
D.~Juteau, \emph{Modular Springer correspondence, decomposition matrices and basic sets}, preprint, arXiv:1410.1471.

\bibitem[KaS]{kas}
M.~Kashiwara, P.~Schapira, \emph{Sheaves on manifolds}, Grundlehren der Mathematischen Wissenschaften 292, Springer, 1990.

\bibitem[KeS]{kellystreet}
G.~M.~Kelly, R.~Street, \emph{Review of the elements of $2$-categories}, in \emph{Category Seminar (Proc.\ Sem., Sydney, 1972/1973)},  Lecture Notes in Math.~420, Springer, Berlin, 1974. 

\bibitem[LZ]{lz}
Y.~Liu, W.~Zheng, {\em Enhanced six operations and base change theorem for
  sheaves on Artin stacks}, preprint, arXiv:1211.5948.

\bibitem[Lur]{lurie}
J.~Lurie, {\em Higher topos theory}, Annals of Mathematics Studies, vol. 170,
  Princeton University Press, Princeton, NJ, 2009.

\bibitem[L1]{lus:gp}
G.~Lusztig, {\em Green polynomials and singularities of unipotent classes},
  Adv.\ in Math.\ {\bf 42} (1981), 169--178.

\bibitem[L2]{lus:icc}
G.~Lusztig, {\em Intersection cohomology complexes on a reductive group},
  Invent.~Math.~{\bf 75} (1984), 205--272.

\bibitem[L3]{lus:charsh}
G.~Lusztig, {\em Character sheaves III}, Adv.\ in Math.\ {\bf 57} (1985), no.~3, 266--315.

\bibitem[MacL]{maclane}
S.~Mac~Lane, \emph{Categories for the working mathematician}, second edition, Graduate Texts in Mathematics 5, Springer-Verlag, New York, 1998.

\bibitem[Mau1]{mautner}
C.~Mautner, {\em Sheaf theoretic methods in modular representation theory},
  Ph.D. thesis, University of Texas at Austin, 2010.
  
\bibitem[Mau2]{mautner2} 
C.~Mautner, {\em A geometric Schur functor}, Selecta Math.\ (N.S.) \textbf{20} (2014), no.~4, 961--977.

\bibitem[MV1]{mv1}
I.~Mirkovi{\'c}, K.~Vilonen, \emph{Characteristic varieties of character sheaves}, Invent.~Math.~{\bf 93} (1988), 405--418.
  
\bibitem[MV2]{mv} 
I.~Mirkovi{\'c}, K.~Vilonen, \emph{Geometric Langlands duality and representations of algebraic groups over commutative rings}, Ann.~of Math.~(2) \textbf{166} (2007), 95--143.

\bibitem[MVy]{mvy}
I.~Mirkovi{\'c}, M.~Vybornov, \emph{Quiver varieties and Beilinson-Drinfeld
Grassmannians of type A}, preprint, arXiv:0712.4160.

\bibitem[P1]{power}
A.~J.~Power, \emph{A $2$-categorical pasting theorem}, J.~Algebra \textbf{129} (1990), 439--445.

\bibitem[P2]{powern}
A.~J.~Power, \emph{An $n$-categorical pasting theorem}, in \emph{Category theory (Como, 1990)},
Lecture Notes in Math.~1488, Springer, Berlin, 1991. 

\bibitem[R1]{reeder}
M.~Reeder, {\em Zero weight spaces and the Springer correspondence}, Indag.\
  Math.\ (N.S.) {\bf 9} (1998), 431--441.

\bibitem[R2]{reeder-additional}
M.~Reeder, {\it Small modules, nilpotent orbits, and motives of reductive groups},  Internat.\ Math.\ Res.\ Notices {\bf 1998},  no.~20, 1079--1101.

\bibitem[R3]{reeder2}
M.~Reeder, {\it Small representations and minuscule Richardson orbits},  Int.\ Math.\ Res.\ Not.\  {\bf 2002},  no.~5, 257--275. 

\bibitem[RSW]{rsw}
S.~Riche, W.~Soergel, G.~Williamson, \emph{Modular Koszul duality}, Compos.\ Math.\ \textbf{150} (2014), no.~2, 273--332.

\bibitem[Rou]{rouquier}
R.~Rouquier, {\em Categorification of $\mathfrak{sl}_2$ and braid groups}, in \emph{Trends in representation theory of algebras and related topics}, Contemp.\ Math.\ 406, A.M.S., Providence, RI, 2006.

\end{thebibliography}
\end{document}